\documentclass[11pt]{article}
\newcommand{\model}[1]{\texttt{#1}}
\usepackage[letterpaper,margin=1in]{geometry}
\usepackage{amssymb,amsmath,amsfonts,eurosym,geometry,ulem,graphicx, color,setspace,comment,footmisc,natbib,pdflscape,array}
\usepackage{subcaption}
\definecolor{bur}{RGB}{158,5,8}

\newcolumntype{L}[1]{>{\raggedright\let\newline\\arraybackslash\hspace{0pt}}m{#1}}
\newcolumntype{C}[1]{>{\centering\let\newline\\arraybackslash\hspace{0pt}}m{#1}}
\newcolumntype{R}[1]{>{\raggedleft\let\newline\\arraybackslash\hspace{0pt}}m{#1}}

\usepackage{mathtools}

\DeclarePairedDelimiter\floor{\lfloor}{\rfloor}

\usepackage[margin = 1.5cm, font = small, labelfont = bf, labelsep = period]{caption}

\usepackage{xcolor}
\usepackage{ulem}
\usepackage{graphicx}
\usepackage{enumitem}
\usepackage{algorithm}
\usepackage{algpseudocode}
\usepackage{bbm}
\newcommand{\norm}[1]{\left\lVert#1\right\rVert}

\usepackage[para,online,flushleft]{threeparttable}
\usepackage{booktabs,multirow,makecell,tabularx}
\definecolor{darkblue}{rgb}{0.0, 0.0, 0.55}
\definecolor{lm}{rgb}{0.95, 0.4, 0.8}
\usepackage[colorlinks=true,linkcolor=blue!80!black,citecolor=blue!80!black,urlcolor=magenta]{hyperref}
\usepackage{longtable}
\newtheorem{theorem}{Theorem}
\newtheorem{corollary}[theorem]{Corollary}
\newtheorem{proposition}[theorem]{Proposition}
\newtheorem{lemma}[theorem]{Lemma}
\newtheorem{remark}{Remark}
\newtheorem{example}{Example}
\newtheorem{assum}{Assumption}
\newtheorem{assert}{Property}
\newtheorem{definition}{Definition}

\newenvironment{proof}[1][Proof]{\noindent\textbf{#1.} }{\ \rule{0.5em}{0.5em}}

\newcommand{\Z}{\mathcal Z}
\usepackage{bbm}

\def\Expect{{\mathbb E}}
\def\Prob{{\mathbb P}}

\newcommand{\bm}{\boldsymbol}
\usepackage{booktabs}

\def\min{\mathop{\rm min}}
\def\max{\mathop{\rm max}}

\def\argmin{\mathop{\rm argmin}}
\def\sup{\mathop{\rm sup}}
\def\inf{\mathop{\rm inf}}

\newcommand{\V}[1]{{{\boldsymbol #1}}}

\newcommand{\Q}{\mathbb{Q}}
\newcommand{\x}{\V{x}}
\newcommand{\bxi}{\V{\xi}}
\newcommand{\bzeta}{\V{\zeta}}

\newcommand{\z}{\V{z}}

\newcommand{\btheta}{\V{\theta}}

\newcommand{\1}{\V{1}}
\newcommand{\Dist}{\mathcal{D}}

\renewcommand{\Re}{\mathbb{R}}
\newcommand{\quoteIt}[1]{``#1''}

\newcommand{\bt}{\V{t}}
\newcommand{\scalarvecxi}{{\eta}}
\newcommand{\vecxi}{\boldsymbol{\scalarvecxi}}
\newcommand{\vecz}{\boldsymbol{z}}
\newcommand{\DataN}{{\mathcal{S}}}
\newcommand{\median}{\mbox{median}}
\newcommand{\Var}{\mbox{Var}}

\usepackage{xcolor}

\newcommand{\removed}[1]{}
\newif \iffinal
\finalfalse
\iffinal
\newcommand{\EDmodified}[1]{{\color{magenta} #1}}
\newcommand{\EDcomments}[1]{{}}
\else
\newcommand{\EDmodified}[1]{{\color{magenta} #1}}
\newcommand{\USmodified}[1]{{\color{teal} #1}}
\newcommand{\UScomments}[1]{{\USmodified{Utsav commented: #1}}}
\newcommand{\EDcomments}[1]{{\EDmodified{Erick commented: #1}}}

\fi
\usepackage{tikz}
\usepackage{pgfplots}
\pgfplotsset{compat=1.10}
\usepackage{xcolor}
\usepackage{amsmath}

\DeclareRobustCommand{\revised}[1]{{\begingroup\color{blue}#1\endgroup}}
\DeclareRobustCommand{\correction}[1]{%
  {\begingroup\color[rgb]{1,0.2,0.2}#1\endgroup}%
}
\usepackage{bbm}
\renewcommand{\revised}[1]{#1}
\begin{document}
\title{
{\revised{Mitigating optimistic bias in entropic risk estimation and optimization}}
}
\author{ {Utsav Sadana}\thanks{CIRRELT, GERAD \& Department of Computer Science and Operations Research, Université de Montréal, Montréal, Québec, Canada, utsav.sadana@umontreal.ca.} \and Erick Delage\thanks{GERAD \& Department of Decision Sciences
HEC Montréal,
Montréal, Québec, Canada,
erick.delage@hec.ca.} \and Angelos Georghiou\thanks{Department of Business and Public Administration,
University of Cyprus
Nicosia, Cyprus,
georghiou.angelos@ucy.ac.cy.}}
\date{\today}
\maketitle
\begin{abstract}
\vspace{0in}

\revised{The entropic risk measure is widely used in high-stakes decision-making across economics, management science, finance, and safety-critical control systems because it captures tail risks associated with uncertain losses. However, when data are limited, the empirical entropic risk estimator, formed by replacing the expectation in the risk measure with a sample average,
underestimates true risk. We show that this negative bias grows superlinearly with the standard deviation of the loss for distributions with unbounded right tails. We further demonstrate that several existing bias reduction techniques developed for empirical risk either continue to underestimate entropic risk or substantially overestimate it, potentially leading to overly risky or overly conservative decisions.
To address this issue, we develop a parametric bootstrap procedure that is strongly asymptotically consistent and provides a controlled overestimation of entropic risk under mild assumptions. The method first fits a distribution to the data and then estimates the empirical estimator’s bias via bootstrapping. We show that the fitted distribution must satisfy only weak regularity conditions, and Gaussian mixture models offer a convenient and flexible choice within this class.
As an application, we introduce a distributionally robust optimization model for an insurance contract design problem that incorporates correlations in household losses. We show that selecting regularization parameters using standard cross-validation can lead to substantially higher out-of-sample risk for the insurer if the validation bias is not corrected. Our approach improves performance by recommending higher and more accurate premiums, thereby better reflecting the underlying tail risk.}
\bigskip
\end{abstract}

\section{Introduction}\label{sec:intro}
The purpose of a risk measure is to assign a real number to a random variable, representing the preference of a risk-averse decision maker towards different risky alternatives. For instance, when faced with multiple options, a decision maker might prefer a guaranteed loss of zero over an uncertain option, even if the latter has a strictly negative expected loss. While this behavior can be explained using the mean-variance criterion \citep{Markowitz_1952}, which balances the expected loss and its fluctuations around the mean, the entropic risk measure offers greater flexibility by incorporating higher moments of the loss distribution. \revised{In economics, management science, and finance,  a decision maker's preferences under uncertainty are often modeled using expected utility theory \citep{Von_Neumann_Morgenstern_1944}.  The entropic risk measure is the certainty equivalent  of the exponential utility function \citep{Ben-Tal_Teboulle_1986, Smith_Chapman_2023}, which represents the risk preferences of a decision maker exhibiting constant absolute risk aversion \citep[CARA --][]{Arrow_1971, Pratt_1964}. A key advantage of using entropic risk in multi-stage decision-making is \revised{that it is the only risk measure that is time consistent in the class of law-invariant convex risk measures \citep{Kupper_Schachermayer_2009}}.  Consequently, the Markowitz model of mean-variance portfolio optimization has been extended to the expected utility framework of \cite{Von_Neumann_Morgenstern_1944} in static and dynamic settings \citep{Merton1969, Merton1971, Samuelson1969}.   Markowitz \citep{Markowitz_2014} emphasizes expected-utility maximization as the guiding principle and discusses conditions under which mean-variance optimization provides a reasonable approximation to the expected-utility maximization problem. This view is supported by empirical studies
 examining the adequacy of the exponential utility specification across different domains \citep{Kirkwood2004}}.

The entropic risk measure is useful in high-stakes decision-making \revised{in static and dynamic problems}, where rare events and their associated extreme losses are a significant concern. 
\revised{ This has led to}  significant growth in research on exponential utility functions, which appear in the literature under various names, including entropic risk minimization, tilted empirical risk minimization, constant absolute risk aversion, and as special cases of more general shortfall risk measures and optimized certainty equivalent risk measures \citep{Ben-Tal_Teboulle_1986}.  \revised{Decision-making under uncertainty with entropic risk measure} is widespread, particularly 
in finance \citep{Föllmer_Schied_2002, Föllmer_Schied_2011, Smith_Chapman_2023},  portfolio selection \citep{Merton1969, Brandtner_Kürsten_2018, Markowitz_2014, BaruchZhang2022, Chen_Ramachandra_2024}, revenue management \citep{Lim_Shanthikumar_2007}, economics \citep{Svensson_Werner_1993}, operations management \citep{Choi_Ruszczyński_2011, HuangJuXing2023, Chen_Sim_2024}, robotics \citep{Nass_Belousov_2019},  statistics \citep{Li_Beirami_2023},  reinforcement learning \citep{Fei_Yang_2021, Hau_Petrik_2023}, risk-sensitive control \citep{Howard_Matheson_1972, Bäuerle_Jaśkiewicz_2024}, game theory \citep{eliashberg1978role, Saldi_Başar_2020}, and catastrophe insurance  pricing \citep{Bernard_Liu_2020}. 

\revised{Across these domains, underestimating risk can lead to decisions that %
are %
overly optimistic. %
In finance and portfolio selection, this may tilt allocations toward riskier assets; in revenue management and operations, it can produce %
capacity and
inventory plans that perform poorly under demand surges or disruptions; in robotics, reinforcement learning, and risk-sensitive control, it can yield policies that fail to meet safety or reliability targets; in catastrophe insurance pricing, it can result in premiums and capital buffers that are insufficient to cover tail losses; %
in game-theoretic settings, it can drive strategic agents toward equilibria that amplify systemic risk, for instance by underinvesting in protection against rare joint-loss events. These observations underscore that, in high-stakes applications%
, the accuracy of risk estimation is critical.}

Since the seminal work by \cite{Föllmer_Schied_2002}, which established the axiomatic foundations for convex risk measures, there has been growing interest in quantitative risk management using convex law-invariant risk measures, such as the entropic risk measure. %
Unlike coherent risk measures such as Conditional Value at Risk (CVaR),
convex law-invariant risk measures allow risk to vary nonlinearly with the size of a position.

To formally define the entropic risk measure, let  $\ell(\vecz,   \vecxi)$ represent the uncertain loss associated with an uncertain parameter $\vecxi \in \Xi \subseteq \Re^d$ \revised{and $\bm \z \subseteq\mathbb{R}^d$ a feasible decision.}
Then, the entropic risk associated with parameter $\vecxi$ is given by: 
   \begin{align}
 \rho%
 _{\Prob}(\ell(\vecz,   \vecxi)) := \begin{cases}
\frac{1}{\alpha}\log(\Expect_{\Prob}[\exp(\alpha \ell(\vecz,   \vecxi))]) & \text{ if } 
\alpha >0,\\
 \Expect_{\Prob}[\ell(\vecz,   \vecxi)] & \text{ if } 
\alpha =0,
\end{cases} \label{eq:trueentropicrisk}
\end{align}
where the loss $\ell(\vecz,   \vecxi)$ is transformed by the increasing and convex exponential function,  and $\alpha$ is the risk aversion parameter. This formulation expresses the entropic risk as the certainty equivalent of the expected disutility 
\(\mathbb{E}_{\mathbb{P}}[\exp(\alpha \ell(\vecz,   \vecxi))]\), reflecting the monetary value of the risk inherent in the uncertain outcome \(\ell(\vecz,   \vecxi)\).  By adjusting the risk-aversion parameter \(\alpha\), also known as the Arrow-Pratt measure of risk aversion, the decision maker’s sensitivity to extreme losses can be controlled. %

In real-world applications, the distribution $\Prob$ of the random variable $\vecxi$ is unknown, and decisions are often made using  historical realizations of random variable $\vecxi$  that are assumed to be independent and identically distributed (i.i.d.)~with distribution $\Prob$. Let the data set of $N$ historical observations be denoted by $\Dist_N = \{\hat{\vecxi}_1, \hat{\vecxi}_2, \cdots, \hat{\vecxi}_N\}$. %
A common approach to estimate the entropic risk is to replace the true distribution with the empirical distribution defined as $\hat{\Prob}_N%
:= \frac{1}{N} \sum_{i=1}^N \delta_{\hat{\vecxi}_i}(\vecxi)$, where  $\delta_{\vecxi}$ is a Dirac distribution at the point $\vecxi$. The empirical entropic risk measure is then given by: 
\begin{equation}
  \rho_{\hat{\Prob}_N}(\ell(\vecz,   \vecxi)) :=  \frac{1}{\alpha}\log\left(\frac{1}{N} \sum_{i=1}^N \exp(\alpha \ell(\vecz,   \hat{\vecxi}_i))\right).\label{eq:empce} 
\end{equation}
Since the logarithm function is  strongly concave, Jensen's inequality implies that the empirical entropic risk strictly underestimates the true entropic risk:
\begin{align}
   \Expect [\rho_{\hat{\Prob}_N}(\ell(\vecz,   \vecxi))]\;&=\;\Expect\left[\frac{1}{\alpha}\log\left(\frac{1}{N} \sum_{i=1}^N \exp(\alpha \ell(\vecz,   \hat{\vecxi}_i))\right)\right]\notag \\&<  \frac{1}{\alpha}\log\left(\Expect\left[\frac{1}{N} \sum_{i=1}^N \exp(\alpha \ell(\vecz,   \hat{\vecxi}_i))\right]\right)=\rho_\Prob(\ell(\vecz,   \vecxi)), 
\end{align}
unless $\mbox{Var}(\ell(\vecz,   \vecxi))=0$,  
and where the expectation is taken with respect to the randomness of the data $\mathcal{D}_N$.
A fundamental challenge in high-stakes decision-making lies in accurately estimating risk. Even with a large number of samples, the empirical entropic risk can significantly underestimate the true risk, especially for \revised{decisions with high risk exposure.} %
This challenge is demonstrated in the following example.

\begin{example}\label{example:risk_gamma}
\begin{figure}[t]
    \centering
\includegraphics[width=0.5\linewidth]{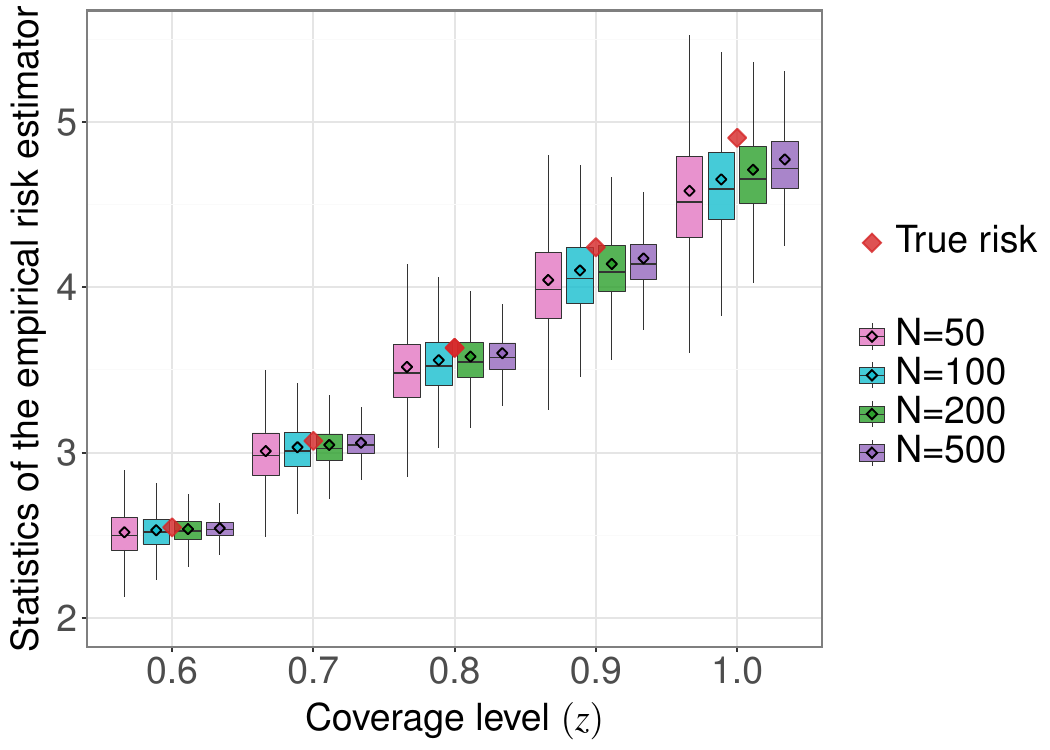}
   \caption{\revised{Statistics of the empirical risk for different values of  \revised{$z \in \{0.6, 0.7, 0.8, 0.9, 1\}$} and training sample sizes $N\in \{50, 100, 200, 500\}$ over $10000$ repetitions. The true risk is given by \revised{$(-15/2)\log(1-0.48z)$.}}}
    \label{fig:risk_gamma}
\end{figure}
    \revised{In an insurance pricing problem, the insurer aims to determine the minimum premium $\pi$ at which they can insure against a certain loss $\eta$. Let the risk aversion parameter of insurer be $\alpha=2$. %
    Assuming a coverage $z$ for the loss $\eta$, the loss of the insurer if they charge a premium $\pi$ is a random variable given by $\revised{z}\scalarvecxi-\pi$. Thus, the minimum premium at which the insurer insures the risk should be such that the entropic risk of the insurer from insuring is at most equal to $0$, i.e., $\frac{1}{\alpha}\log\left(\Expect_\Prob[\exp(\alpha (z\scalarvecxi-\pi))]\right) \leq 0$.
    On rearranging the terms, one can show that the minimum premium equals the entropic risk associated with the loss $\revised{z}\scalarvecxi$, $\pi \geq \frac{1}{\alpha}\log\left(\Expect_\Prob\left[\exp(\alpha \revised{z}\scalarvecxi)\right]\right)$, also called the exponential premium \citep{Gerber_1974}.
Suppose that the loss follows a Gamma distribution $\Gamma(\kappa, \lambda)$  \citep[see][]{Fu_Moncher_2004, Bernard_Liu_2020} with shape parameter $\kappa$ and scale parameter $\lambda$. The moment-generating function of a $\Gamma$-distributed random variable is known in closed form which allows us to analytically compute the optimal premium
$\pi^* = \frac{1}{\alpha}\log\left(\left(1-\lambda \revised{z}\alpha\right)^{-\kappa}\right)$ if $\lambda<1/\revised{(z\alpha)}$.
Suppose an insurer has access to $N\in \{50, 100, 200, 500\}$ samples of the losses which are generated from a $\Gamma(15, 0.24)$-distribution. %
   We use empirical distribution $\hat{\Prob}_N$ over $N$ samples to estimate the entropic risk for different coverage levels %
  $ z \in \{0.6, 0.7, 0.8, 0.9,  1\}$. Figure~\ref{fig:risk_gamma} presents statistics of the distribution of the empirical risk estimator as a function of $N$ and \revised{$z$}. %
      We can see that a larger coverage exposes the insurer to significant underestimation of the risk of the loss $z\scalarvecxi$ even for the relatively large sample size $N = 500$.}

\end{example}

\subsection{Our contributions}

The contribution of this paper is to propose a scheme that produces \revised{estimators, which under suitable assumptions on the tails of the loss $\ell(\vecz, \vecxi)$,  provably remove the negative bias in the empirical entropic risk estimator.  Our approach constructs overestimators of entropic risk while ensuring that the induced conservatism does not lead to overly conservative decisions.  More specifically, we} %
  propose a bias correction term $\delta(\Dist_N)$ that employs bootstrapping, using a distribution fitted to  the data $\Dist_N$,  to get 
\begin{align}
 \rho_{\Prob}(\ell(\vecz, \vecxi)) \approx \text{median}\left[\rho_{\hat{\Prob}_N}(\ell(\vecz, \vecxi)) + \delta(\Dist_N)\right],
\end{align}
 where the median statistic is taken with respect to the randomness in $\Dist_N$. 
We establish mild conditions under which $\delta(\Dist_N)\rightarrow 0$ almost surely as $N\rightarrow \infty$. %
In particular, basing the bootstrap on a constrained maximum likelihood estimation (MLE) of the sampling distribution will yield an asymptotically consistent estimator. 
Unfortunately, our empirical experiments establish that a bootstrap correction based on MLE fails to adequately address the underestimation of entropic risk. %
Instead, we provide two procedures to fit  ``bias-aware'' distributions  that take into account the entropic risk estimation bias caused by tail events. The first one involves  a distribution  matching technique  that tries to fit the entropic risk estimator's distribution itself, and the second uses a simple mixture distribution with a component dedicated to fitting the tail of the empirical distribution. %

Going beyond the estimation of entropic risk, we study the entropic risk minimization problem. Solving the sample average approximation (SAA) of the entropic risk minimization problem is known to produce a second source of bias, also known as the optimizer's curse \citep{Smith_Winkler_2006}. Distributionally robust optimization (DRO) is widely used to address the optimistic bias of SAA policies because the decision maker is protected against perturbations in the empirical distribution that lie in a distributional ambiguity set. %
To tune the radius of the ambiguity set, a typical approach is to use K-Fold cross validation (CV). We use our bias mitigation procedure to estimate the validation performance of the resulting decisions.
To demonstrate the effectiveness of our approach, we conduct a case study on an insurance pricing problem. %
The true distribution of losses, which may be correlated across households, is unknown to both the insurer and the households. The insurer addresses this uncertainty by solving a distributionally robust insurance pricing problem, determining the coverage to offer and the premium to charge each household. Our results show that the insurer can achieve a significant improvement in out-of-sample entropic risk compared to the ``traditional" K-Fold CV procedure, which selects the radius of the distributional ambiguity set solely by evaluating decisions on validation loss scenarios.   %

Our contributions can be described as follows:
\begin{enumerate}
\item On the theoretical side, \revised{we prove that in the finite sample regime, most available estimation procedures  (e.g., empirical entropic risk, optimizer's information criterion (OIC), Delta Method, non-parametric bootstrap and double bootstrap) %
can severely underestimate the entropic risk of $\ell(\vecz, \vecxi)$ when its variance $\mbox{Var}(\ell(\vecz, \vecxi))$ is large. %
Namely, the size of negative bias grows at least superlinearly with respect to $\sqrt{\mbox{Var}(\ell(\vecz, \vecxi))}$.  In practice, this implies that, when used in optimization (or decision-making), the former class of estimators are likely to steer the decision maker towards higher variance alternatives in order to draw value from the estimation bias. One notable exception is an estimation procedure based on leave-one-out cross validation \citep{stone1974cross, arlot2010survey} %
with a guaranteed positive bias, yet the latter can be shown to grow exponentially with respect to $\mbox{Var}(\ell(\vecz, \vecxi))$, and thus only promotes highly conservative actions.}%
\item \revised{We present a strongly asymptotically consistent procedure to debias the empirical entropic risk estimator in a way that ensures eventual over-estimation of risk as variance grows, if the tails of $\ell(\vecz,\vecxi)$ are lighter-than-Gaussian, while preserving overestimation to be at most linear in $\mbox{Var}(\ell(\vecz, \vecxi))$. This is done by estimating the size of the empirical risk bias using a parametric bootstrap procedure that fits a Gaussian mixture model (GMM) to the empirical loss distribution, and compensating for this estimated bias.}

\item  \revised{%
We show that straightforward maximum likelihood estimation (MLE) can empirically yield a bias correction that is smaller than the true bias of the empirical risk estimator, even with a large training set. To address this issue, we introduce two fitting procedures that empirically produce more conservative bias estimates. The first approach minimizes a novel loss function designed to fit the distribution of the empirical entropic risk, using an end-to-end differentiable pipeline that leverages automatic differentiation for scalable learning. The second approach adopts a simpler strategy based on closed-form fitting of a parametric distribution that exploits the tail behavior of the empirical distribution, making it highly scalable and easy to deploy.
}Our methods could be of independent interest for debiasing more general risk measures.
  
\item On the application side, our work contributes toward data-driven designing of insurance premium pricing and coverage policies.
To the best of our knowledge, this is the first time that a distributionally robust version of the well-known risk-averse insurance pricing  problem \citep{Bernard_Liu_2020} is introduced in the literature. Our model takes into account the different risk aversion attitudes of the insurer and households as well as the systemic risk  associated with \revised{correlated events such as floods.} 
\end{enumerate}

The paper is organized as follows.  Section \ref{sec:property_entropic_risk} discusses the properties of the entropic risk measure. \revised{Section \ref{sec:asymptotic:effect:variance_on_bias} provides theoretical results on the asymptotic effect of variance of $\ell(\vecz, \vecxi)$ on the estimation bias of several existing estimators.} %
Section \ref{sec:bias_correct} provides a bias correction procedure to mitigate the underestimation problem.  In Section \ref{sec:dro}, we study the entropic risk minimization problem using the DRO framework. In Section \ref{sec:DRIP}, we introduce the distributionally robust insurance pricing problem and provide numerical results in Section \ref{sec:num_ins_price}. The conclusions are given in Section \ref{sec:conclusions}. \revised{Appendix \ref{sec:lit_review}  reviews related works on estimating risk measures,  correcting optimistic bias associated with solving SAA problem, and insurance pricing, situating our contributions within these three research streams, and clarifying how our approach differs from existing methods; Appendices \ref{append:property_entropic_risk}, \ref{appendix:2}, and \ref{append:bias_correct} present the proofs of the results in Sections \ref{sec:property_entropic_risk}, \ref{sec:asymptotic:effect:variance_on_bias}, and \ref{sec:bias_correct}, respectively. Appendix \ref{sec:appenddro} provides the theory and proofs to support the DRO model introduced in Section \ref{sec:dro}. Finally, additional details and results omitted from the paper can be found in Appendix \ref{append:details}. }

 \paragraph{Notations:} $[m]$ denotes the set of integers $\{1, 2, \cdots, m\}$.
 $\|\cdot \|_{*}$ denotes the dual norm of $\|\cdot \|$. \( \delta_\zeta \) is the Dirac distribution at the point \( \zeta\), and $\log(\cdot)$ refers to the natural logarithm. \revised{We use standard asymptotic notation $\mathcal O(\cdot)$, $\Omega(\cdot)$,  and $\omega(\cdot)$ as in \cite{Cormen2009}. 
For functions $f,g : \mathbb{R}_{+} \to \mathbb{R}$, we write 
$f(\sigma) = \mathcal O(g(\sigma))$ if there exist constants $c>0$ and $\bar{\sigma}>0$ such that 
$f(\sigma) \le c\,g(\sigma)$ for all $\sigma \ge \bar{\sigma}$, 
$f(\sigma) = \Omega(g(\sigma))$ if there exist constants $c>0$ and $\bar{\sigma}>0$ such that $c\,g(\sigma) \le f(\sigma)$ for all $\sigma \ge \bar{\sigma}$, 
and $f(\sigma) = \omega(g(\sigma))$ if for every $c>0$, there exists $\bar{\sigma}>0$ such that 
$c\,g(\sigma) \leq f(\sigma)$ for all $\sigma \ge \bar{\sigma}$. 
$\Lambda_{\Q}(t):=\log(\Expect_\Q[\exp(t\zeta)])$ is the cumulant generating function of $\zeta\sim\mathbb{Q}$.}%

\section{Properties of entropic risk measure}\label{sec:property_entropic_risk}

Let $(\Omega, \mathcal{F}, \Prob)$ be a probability space and let $\mathcal{L}^{p}:=\mathcal{L}^{p}(\Omega, \mathcal{F}, \Prob)$ denote the space of real-valued measurable functions, $X: \Omega\rightarrow \mathbb{R}$ such that $\Expect[|X|^p] <\infty$, for some  $p\geq 1$.  %
The entropic risk measure is a convex, law invariant risk measure \citep{Föllmer_Schied_2002}, thus satisfying the following definition.
\begin{definition}\label{def:entropicrisk}  A functional $\rho \colon \mathcal{L}^{p}\rightarrow \bar{\mathbb{R}}$, where $\bar{\mathbb{R}}:=\mathbb{R}\cup\{\infty\}$,  is a convex law-invariant risk measure if
\begin{enumerate}
\item[$(a)$]  $\rho(X- m) = \rho(X) - m$ for all $X\in \mathcal{L}^{p}$ and $m \in \mathbb{R}$ and $\rho(0) = 0$. 
\item[$(b)$] $\rho(X) \leq \rho(X')$ if $X\leq X'$ almost surely (a.s.) for all $X \in \mathcal{L}^{p}$.
\item[$(c)$] $\rho(\lambda X+ (1 - \lambda)X') \leq \lambda \rho(X) + (1 - \lambda)\rho(X')$ for $\lambda \in [0,1]$ and for all $X, X' \in \mathcal{L}^{p}$.
\item[$(d)$] $\rho(X) = \rho(X')$ for all $X, X' \in \mathcal{L}^{p}$ such that $X=X'$ in distribution.
\end{enumerate} 
\end{definition}
Condition $(a)$, also known as  cash-invariance property, states that $m$ is the minimum amount that should be added to a risky position to make it acceptable to a regulator. Condition $(b)$, ensures monotonicity, meaning lower losses are preferable. Condition $(c)$, convexity, ensures that diversification reduces risk. Lastly, condition $(d)$, law invariance, states that two random variables with the same distribution should have equal risk. 

Letting $\vecz\in\mathcal{Z}\subseteq\mathbb{R}^d$ be a decision vector, $\vecxi:\Omega\rightarrow\Re^d$ be a random vector, and $\ell(\vecz,   \vecxi)\in\mathcal{L}^p$ the random loss that it produces, we will further impose the following assumption to ensure that the entropic risk of $\ell(\vecz,   \vecxi)$ is finite, together with the mean and variance of $\exp(\alpha \ell(\vecz,   \vecxi))$ for all $\vecz\in\mathcal{Z}$. 
\begin{assum}\label{assum:tail_bound2}
    For all $\vecz\in\mathcal{Z}$, the tails of $\ell(\vecz,   \vecxi)$ are exponentially bounded: %
\[
\Prob(|\ell(\vecz,   \vecxi)| > a) \leq G \exp(-a\alpha C), \quad \forall a\geq 0 ,
\]
for some $G>0$ and $C> 2$. Equivalently, the moment-generating function $\Expect[\exp(t\ell(\vecz,   \vecxi))]\in\mathbb{R}$ for all $t\in(-\alpha C,\alpha C)$ for some $C> 2$, see Lemma \ref{lemma:equiv:mgf:tail} in Appendix \ref{append:property_entropic_risk} for a proof of equivalence.
\end{assum}

Assumption \ref{assum:tail_bound2} further restricts the space of loss functions in Definition \ref{def:entropicrisk} in order to work with random variables that are \quoteIt{well-behaved} from the point of view of entropic risk estimation at a risk tolerance level of $\alpha$. 
Indeed, our assumptions will ensure that the empirical estimator is asymptotically consistent for all $\vecz\in\mathcal{Z}$. We note that our assumption relates to $\mathcal{L}_M$, the set of random variables with finite-valued moment generating functions, through the following inclusion: $\mathcal{L}^{\infty} \subseteq \mathcal{L}_{M} \subseteq \mathcal{L}_{\alpha} \subseteq \mathcal{L}^{p}$ with $\mathcal{L}_\alpha$ as the set of random variables in $\mathcal{L}\revised{^p}$ that satisfy Assumption \ref{assum:tail_bound2}.  

\begin{lemma}\label{thm:boundedExpVar2}
    Under Assumption \ref{assum:tail_bound2}, $\Expect[\exp(\alpha \ell(\vecz,   \vecxi))]\in \left[\exp(-\revised{\frac{G}{C}}),\,\frac{G}{C-1}+\revised{1}\right]$ and $\text{Var}[\exp(\alpha \ell(\vecz,   \vecxi))]\in \left[0,\,\frac{2G}{C-2}+\revised{1}\right]$.
\end{lemma}

\section{\revised{Asymptotic effect of variance on estimation bias}} \label{sec:asymptotic:effect:variance_on_bias}
In this section, we theoretically analyze the asymptotic effects of the variance of $\ell(\vecz,\vecxi)$ on the estimation bias. %
To do so, we make the following simplifying assumption that the decision affects only the mean and the standard deviation of the loss, see \cite{meyer1987two} and Appendix \ref{sec:appendix:location:scale} for a detailed discussion of the location-scale optimization problem. %

\begin{assum}\label{ass:meanVarControl}(Location-scale optimization problem) 
The distributions of feasible random losses can be represented as a location-scale %
optimization problem:
    \begin{align}
\{ \ell(\boldsymbol{z}, \vecxi) : \vecz\in\mathcal{Z}\} \stackrel{F}{=} \{\mu+\sigma \xi : (\mu,\sigma)\in\mathcal{A}\},
\end{align}
where $\stackrel{F}{=}$ compares the distributions of the random variables in the two sets, for some fixed $\xi:\Omega\rightarrow \Re$, with $\Expect[\xi]=0$ and $\Expect[\xi^2]=1$, and some $\mathcal{A}\subset \mathbb{R}\times\mathbb{R}_+$.
\end{assum}
The above assumption is satisfied, for example, in a single asset portfolio optimization (or insurance pricing problem), where $\ell(z,\scalarvecxi):=z \scalarvecxi$ with $z\in\mathcal{Z}\subseteq \mathbb{R}_+$ as the size of the risky investment (or coverage ratio), while $\scalarvecxi$ captures the excess return of the asset with respect to the risk free rate (or covered loss). The location-scale %
control problem is obtained by letting $\xi:=(\scalarvecxi-\Expect[\scalarvecxi])/\sqrt{\mbox{Var}(\scalarvecxi)}$ so that $\ell(z,\scalarvecxi)=\Expect[\scalarvecxi]z + \sqrt{\mbox{Var}(\scalarvecxi)}z\xi$. While this can generalize to multi-asset environment with radially symmetric assets distribution \citep{fang1990spherically}, %
it is clear that in general, the set of feasible random losses might not satisfy Assumption \ref{ass:meanVarControl}. %
Yet, by studying this class of problems, we can shed light on the role that variance plays on risk underestimation %
and identify ways of preventing the estimator from perversely incentivizing larger risk exposure.  The next section identifies conditions under which  the asymptotic growth rate of cumulant generating function of the true loss distribution dominates the growth rate of the empirical estimator as a function of \(\mathrm{Var}(\ell(\vecz,\vecxi))\), and establishes that common bias correction procedures fail to address this issue.%

\removed{

In a classical single asset portfolio optimization problem, let 
In a portfolio setting, let $\bm \eta \in \mathbb{R}^d$ denote the vector of asset losses and suppose $\bm \eta = \bar{\bm \mu} + \bar{\bm \sigma}\xi$,
where $\bar{\bm \ell} \in \mathbb{R}^d$, $\bar{\bm \sigma} \in \mathbb{R}^d$ are deterministic and $\xi$ is a scalar random variable
with $\mathbb{E}[\xi]=0$ and $\operatorname{Var}(\xi)=1$, independent of $\bm z$.
The portfolio loss $\ell(\bm z,\bm \eta) := \bm z^\top \bm \eta$ then satisfies $
    \ell(\bm z,\xi) = \bm z^\top \bar{\bm \ell} + (\bm z^\top \bar{\bm \sigma})\,\xi
                = \mu(\bm z) + \sigma(\bm z)\,\xi,$
with $\mu(\bm z)=\bm z^\top \bar{\bm \ell}$ and $\sigma(\bm z)=\bm z^\top \bar{\bm \sigma}$.

In an insurance problem with $d$ households, the insurer chooses coverages $\bm z \in [0,1]^d$ and premiums $\pi(\bm z)$, and their loss is $ \ell(\bm z, \bm \eta) := \bm z^\top \bm \eta - \mathbf{1}^\top \pi(\bm z)$ which can be simplified as follows
\[
            \ell(\bm z,\xi)    = (\bm z^\top \bar{\bm \ell} - \mathbf{1}^\top \pi(\bm z)) + (\bm z^\top \bar{\bm \sigma})\,\xi
                = \mu(\bm z) + \sigma(\bm z)\,\xi,
\]
with $\mu(\bm z) = \bm z^\top \bar{\bm \ell} - \mathbf{1}^\top \pi(\bm z)$ and $\sigma(\bm z)=\bm z^\top \bar{\bm \sigma}$.
Thus, under a one-factor specification, both models are of the form
$\ell(\bm z,\xi)=\mu(\bm z)+\sigma(\bm z)\,\xi$ with scalar $\xi$.
}

\subsection{Asymptotic effect of variance on empirical risk estimator}\label{sec:asymptotic:saa}

We start our discussion by studying the effects of the variance of the random loss, manipulated using $z\in \mathcal{Z}$, on the quality of the empirical risk estimate. %
The restricted space of random losses defined in Assumption \ref{ass:meanVarControl} allows us to study the effect of $\sigma$ on estimation error of $\rho_{\Prob^\xi}(\mu+\sigma\xi)$ to draw insights on the role of $\mbox{Var}(\ell(\vecz, \vecxi))$ on the estimation error of $\rho_{\Prob}(\ell(\vecz, \vecxi))$. %
For instance, one can easily show that the estimation error of the risk measured by a \emph{coherent risk measure $\varrho$} is linear in the standard deviation of $\ell(\z,\vecxi)$: 
\[\varrho_{\Prob}(\ell(\vecz, \vecxi)) - \mathbb{E}[\varrho_{\hat{\Prob}_N}(\ell(\vecz, \vecxi))] =\varrho_{\Prob^\xi}(\mu+\sigma\xi) - \Expect[\varrho_{\hat{\Prob}_N^\xi}(\mu+\sigma\xi)]=\sigma\left(\varrho_{\Prob^\xi}(\xi) - \Expect[\varrho_{\hat{\Prob}_N^\xi}(\xi)]\right),\]
where $(\mu,\sigma)\in\mathcal{A}$ are the mean and standard deviation of $\ell(\vecz,\vecxi)$, $\Prob^\xi$ and $\hat{\Prob}_N^\xi$ are, respectively, the distribution of $\xi$ and the empirical distribution of $\xi$ using $N$ i.i.d.~samples, and the last equality follows from the properties (translation invariance and positive homogeneity) of coherent risk measures. More generally, the estimation error is linear in $\sigma$ for translation-invariant and positive-homogeneous risk measures \citep{landsman2012translation}, which include distortion risk measures such as Conditional Tail Expectation (CTE), CVaR and Value-at-Risk (VaR). 

Our first result states that the size of underestimation is actually nonlinear for entropic risk when the random variable has an unbounded right tail.

\begin{proposition}\label{prop:bias:saa:superlinear}
    Let Assumption \ref{ass:meanVarControl} be satisfied and the loss $\ell(\vecz, \vecxi)$ have an unbounded right tail. Then:
    \[\rho_{\Prob}(\ell(\vecz, \vecxi)) - \mathbb{E}[\rho_{\hat{\Prob}_N}(\ell(\vecz, \vecxi))] =\omega(\sqrt{\mbox{Var}(\ell(\vecz, \vecxi))}).\]
\end{proposition}

Proposition~\ref{prop:bias:saa:superlinear} indicates that, for random losses with sufficiently large variance, the degree of risk underestimation grows at a rate that is at least superlinear in the standard deviation of $\ell(\vecz,\vecxi)$. A particularly relevant special case with an unbounded right tail arises when $\ell(\vecz,\vecxi)$ follows a normal distribution. The following result shows that, in this setting, the magnitude of underestimation eventually grows at a rate that is at least linear in the variance of the loss (rather than in its standard deviation).
\begin{proposition}\label{prop:saa:normalLB}
    Let Assumption \ref{ass:meanVarControl} be satisfied, and the loss $\ell(\vecz, \vecxi)$ be normally distributed. Then:
    \[\rho_{\Prob}(\ell(\vecz, \vecxi)) - \mathbb{E}[\rho_{\hat{\Prob}_N}(\ell(\vecz, \vecxi))] =\Omega(\mbox{Var}(\ell(\vecz, \vecxi))).\]
\end{proposition}

As a second example, we also establish that the asymptotic lower bound becomes unbounded when $\ell(\z,\vecxi)\stackrel{F}{=} \mu+\sigma \xi$ with $\xi$ following a Laplace distribution, i.e. with a density $f_{\xi}(x)=\frac{1}{\sqrt{2}}e^{-\sqrt{2}|x|}$. %

\begin{proposition}\label{prop:saa:laplaceLB}
    Let Assumption \ref{ass:meanVarControl} be satisfied for some $\xi$ following a Laplace distribution, then $\rho_{\Prob}(\ell(\vecz, \vecxi)) - \mathbb{E}[\rho_{\hat{\Prob}_N}(\ell(\vecz, \vecxi))]
\rightarrow \infty$ as $\mbox{Var}(\ell(\vecz, \vecxi))\rightarrow 2/\alpha^2$.%
\end{proposition}
  Proposition \ref{prop:saa:laplaceLB} reveals a fundamental challenge in bias correction for the entropic risk measure under Laplace-distributed losses: even when the distribution family is known, small errors in variance estimation can cause the bias estimate to become unbounded. This instability underscores the practical difficulty of obtaining reliable bias-corrected estimates in such settings. %
In what follows, we will be interested in measuring or correcting for the underestimation bias. This will require us to assume stronger conditions than exponentially bounded tails. 

\begin{definition}%
\citep{Vershynin2025}\label{assum:subgaussian}
    The tails of $X$ are subgaussian, if there exists some $\mathfrak c_0>0$ such that 
    \[\mathbb P(|X|\ge a) \le 2\exp\Big(-\frac{a^2}{\mathfrak c_0^2}\Big) \qquad \forall a\ge 0.\]
\end{definition}

Our next proposition indicates that when the distribution of $\ell(\z,\vecxi)$ has subgaussian tails, then the estimation error %
eventually grows at most at a linear rate with respect to the variance of $\ell(\z,\vecxi)$.

\begin{proposition}\label{prop:saa:UB}
    Let Assumption \ref{ass:meanVarControl} be satisfied and the loss $\ell(\vecz, \vecxi)$ have subgaussian tails  (see Definition \ref{assum:subgaussian}), then 
    \[|\rho_{\Prob}(\ell(\vecz, \vecxi)) - \mathbb{E}[\rho_{\hat{\Prob}_N}(\ell(\vecz, \vecxi))]| = \mathcal O(\mbox{Var}(\ell(\vecz, \vecxi))).\]
\end{proposition}

\begin{remark}
    We note that it might appear that the relevance of asymptotic results in terms of variance is of limited interest when $\mathcal{A}$ is bounded, given that there exists no sequence of $\z$ such that $\mbox{Var}(\ell(\z,\vecxi))\rightarrow \infty$. Yet, given that the entropic risk measure \quoteIt{magnifies} the scale of the random loss proportionally to $\alpha$, any asymptotic result in $\mbox{Var}(\ell(\z,\vecxi))$ also implies the existence of a risk-level for which the asymptotic bound becomes relevant over $\mathcal{A}$. For example, Lemma \ref{lem:var_linear_underest_short} in the Appendix adapts Proposition \ref{prop:saa:normalLB} to conclude that  that there is a level of risk aversion that leads to an underestimation error that grows linearly with respect to the variance of the losses within the feasible set. %
    
    \removed{ For example, letting $\rho^{{\correction{\bar{\alpha}}}}(\cdot)$ denote the risk measure using $\bar{\alpha}$, Proposition \ref{prop:saa:normalLB} implies that given any arbitrary $\bar{\sigma}>0$, there exists a $C>0$ and a $\bar{\alpha}\geq \alpha$ for which a normally distributed loss $\ell(\z,\vecxi)$  with variance $\sigma^2>\bar{\sigma}^2$ will have:
        \begin{align*}
    \rho^{{\bar{\alpha}}}_{\Prob}(\ell(\vecz, \vecxi)) - \mathbb{E}[\rho^{{\bar{\alpha}}}_{\hat{\Prob}_N}(\ell(\vecz, \vecxi))] &= \rho^{{\bar{\alpha}}}_{\Prob^\xi}(\mu+\sigma\xi) - \mathbb{E}[\rho^{{\bar{\alpha}}}_{\hat{\Prob}_N^\xi}(\mu+\sigma\xi)]\\
    &= \frac{\alpha}{\bar{\alpha}}\left(\rho_{\Prob^\xi}(\frac{\bar{\alpha}}{\alpha}(\mu+\sigma\xi)) - \mathbb{E}[\rho_{\hat{\Prob}_N^\xi}(\frac{\bar{\alpha}}{\alpha}(\mu+\sigma\xi))]\right) \\
    &\geq \frac{\alpha}{\bar{\alpha}} C \left(\frac{\bar{\alpha}}{\alpha}\sigma\right)^2 = C\frac{\bar{\alpha}}{\alpha} \mbox{Var}(\ell(\vecz, \vecxi)).
    \end{align*}}
\end{remark}

\subsection{Asymptotic effect of variance on common bias correction procedures}

In this section, we study how increasing variance affects the asymptotic bias of common bias correction procedures proposed in the literature. %
We show that central limit theorem based corrections, such as the Optimizer’s Information Criterion (OIC) and Delta Method, as well as the nonparametric bootstrap and double bootstrap \citep{Kim_2010}, exhibit a negative bias that grows superlinearly with the standard deviation of the loss. A notable exception is leave-one-out cross-validation (LOOCV), which instead overestimates the entropic risk at a rate that is at least exponential in the standard deviation of $\ell(\vecz,\vecxi)$.

\subsubsection{Bias correction based on central limit theorem (CLT)}

Several approaches rely on CLT %
to devise asymptotically unbiased estimators. 
Assuming finite second moment for $\exp(\alpha \ell(\vecz, \vecxi))$, CLT combined with the Delta Method \citep{Van_der_Vaart_2000} gives $\sqrt{N}(\rho_{\hat{\Prob}_N}(\ell(\z,\vecxi)) - \rho_\Prob(\ell(\z,\vecxi))) \overset{d}{\to} %
\mathcal{N}(0, \frac{\text{Var}(\exp(\alpha\ell(\z,\vecxi)))}{\alpha^2  (\mathbb{E}(\exp(\alpha \ell(\z,\vecxi))))^2})$ as $N\rightarrow \infty$. %
This approximation is centered at zero and does not capture the bias in the empirical entropic risk estimator. A second-order Taylor expansion based on the Delta Method yields the leading bias term and motivates the bias-corrected estimator \citep{Horowitz_2019}, see Lemma \ref{lemma_delta_bias} in Appendix~\ref{proof:derive:oic:delta} for the derivation:
\[\rho_{\model{Delta}}:=\rho_{\hat{\Prob}_N}(\ell(\z,\vecxi)) +\frac{\text{Var}_{\hat{\Prob}_N}(\exp(\alpha\ell(\z,\vecxi)))}{2\alpha N (\mathbb{E}_{\hat{\Prob}_N}(\exp(\alpha \ell(\z,\vecxi))))^2}. \]

Another closely related estimator is obtained by noticing that entropic risk measure can be equivalently written as  an optimized certainty equivalent risk measure \citep{Ben-Tal_Teboulle_1986}, i.e., 
\begin{align}
   \rho_\Prob(\ell(\z,\vecxi)) = \inf_t  \Expect_\Prob[h(t, \ell(\z,\vecxi))], \label{eq:oce:rep}
\end{align}
where  $h(t,\zeta) = t+ \frac{1}{\alpha}\exp(\alpha (\zeta-t))-\frac{1}{\alpha}$. 
The empirical risk estimator is equivalently obtained by solving the Sample-Average Approximation (\model{SAA}) of problem \eqref{eq:oce:rep}: \[\rho_\model{SAA} := \inf_t  \Expect_{\hat{\Prob}_N}[h(t, \ell(\z,\vecxi))]=\rho_{\hat{\Prob}_N}(\ell(\z,\vecxi)),\]   where $\Prob$ is replaced with $\hat{\Prob}_N$. It is well-known that decisions based on the SAA can suffer from optimizer's curse, leading to an optimistic bias \citep{Smith_Winkler_2006}. 
To correct the first-order optimistic bias associated with the \model{SAA}, \cite{Iyengar_Lam_2023} introduced an  estimator based on the optimizer's information criterion\footnote{It is to be noted that the \model{OIC} estimator is designed to address the optimizer's curse, while we aim to debias the empirical entropic risk estimator.} (see Lemma \ref{lemma_oic} in Appendix~\ref{proof:derive:oic:delta} for the derivation):
   \[\rho_{\model{OIC}}:= \rho_{\hat{\Prob}_N}(\ell(\z,\vecxi))+\frac{\text{Var}_{\hat{\Prob}_N}(\exp(\alpha \ell(\z,\vecxi)))}{\alpha N(\mathbb{E}_{{\hat{\Prob}_N}}[\exp(\alpha \ell(\z,\vecxi))])^2}.\]

By exploiting the fact that the ratio of the sample variance to the squared mean of $\exp(\alpha \ell(\z,\vecxi))$ grows at most linearly with the sample size, the following proposition establishes that both $\rho_{\model{Delta}}$ and $\rho_{\model{OIC}}$ suffer from the same issues as $\rho_{\hat{\Prob}_N}(\ell(\vecz, \vecxi))$ in terms of the influence of $\mbox{Var}(\ell(\x,\vecxi))$ on risk underestimation.
\begin{proposition}\label{prop:oic:delta}
Propositions \ref{prop:bias:saa:superlinear}, \ref{prop:saa:normalLB}, \ref{prop:saa:laplaceLB}, and \ref{prop:saa:UB} hold verbatim when $\rho_{\hat{\Prob}_N}(\ell(\vecz, \vecxi))$  is replaced by either $\rho_{\model{Delta}}$ or $\rho_{\model{OIC}}$.

\end{proposition}

\subsubsection{Non-parametric bootstrapping}\label{sec:commonBS}
\begin{algorithm}[htb]
\caption{Non-parametric bootstrap bias correction \label{bootstrap_non_param}}
\begin{algorithmic}[1]
\Function{NonParametricBootstrapBiasCorrection}{$\mathcal{S}, M$}
        \State $\hat{\Prob}_N \gets$ Empirical distribution of loss scenarios $\mathcal{S}$
        \For{$n \gets 1$ \textbf{to} $M$}
            \State $\hat{\Prob}_{N,N} \gets  \text{Draw } N \text{ i.i.d.  samples  from } \hat{\Prob}_N$ 
            \State $\rho_n \gets \rho_{\hat{\Prob}_{N,N}}(\ell(\vecz, \vecxi))$  
        \EndFor
        \State $\hat{\delta}_N(\hat{\Prob}_N) \gets \text{mean}(\{\rho_{\hat{\Prob}_N}(\ell(\vecz, \vecxi))-\rho_n\}_{n=1}^M)$ \label{bias:corr2}
        \State \Return $\delta_N(\hat{\Prob}_N)$
\EndFunction
\end{algorithmic}
\end{algorithm}
 A typical approach in the literature to devise an unbiased estimator is to use non-parametric bootstrapping. %
In \cite{Kim_2010}, a sample-based risk estimator is calculated by repeatedly sampling $N$ observations with replacement from the empirical distribution of loss scenarios $\mathcal{S}_N:=\{\ell(\vecz, \hat{\vecxi}_1), \cdots, \ell(\vecz, \hat{\vecxi}_N)\}$ as shown in Algorithm \ref{bootstrap_non_param}.   For each bootstrap sample, the empirical distribution $\hat{\mathbb P}_{N,N}$  of the bootstrapped sample is constructed and the risk $\rho_{\hat{\mathbb P}_{N,N}}(\ell(\z,\vecxi))$ is computed. The bias is estimated as $\Expect
[\rho_{\hat{\Prob}_{N}}(\ell(\z,\vecxi))]-\rho_{\Prob}
(\ell(\z,\vecxi))\approx \Expect[\rho_{\hat{\Prob}_{N,N}}(\ell(\z,\vecxi))|\hat{\Prob}_N]-\rho_{\hat{\Prob}_N}(\ell(\z,\vecxi))$, and used to correct for the bias of the empirical estimate $\rho_{\hat{\Prob}_N}(\ell(\z,\vecxi))$. Effectively, the $\model{BS}$ estimator takes the form:
\[\rho_{\texttt{BS}}:=\rho_{\hat{\Prob}_N}(\ell(\z,\vecxi))-(\Expect[\rho_{\hat{\Prob}_{N,N}}(\ell(\z,\vecxi))|\hat{\Prob}_N]-\rho_{\hat{\Prob}_N}(\ell(\z,\vecxi)))=2\rho_{\hat{\Prob}_N}(\ell(\z,\vecxi))-\Expect[\rho_{\hat{\Prob}_{N,N}}(\ell(\z,\vecxi))|\hat{\Prob}_N].\]
By Jensen's inequality, we know that $\Expect[\rho_{\hat{\Prob}_{N,N}}(\ell(\z,\vecxi))|\hat{\Prob}_N]\leq \rho_{\hat{\Prob}_N}(\ell(\z,\vecxi))$ almost surely, hence $\rho_{\texttt{BS}}\geq\rho_{\hat{\Prob}_N}(\ell(\z,\vecxi))$ almost surely. %

\cite{Kim_2010} also proposes a double bootstrapping (\model{DBS}) procedure that further compensates for the bias in estimating the bias of the empirical estimator of risk. This one takes the form of:
\[\rho_{\texttt{DBS}}:=3\rho_{\hat{\Prob}_N}(\ell(\z,\vecxi))-3\Expect[\rho_{\hat{\Prob}_{N,N}}(\ell(\z,\vecxi))|\hat{\Prob}_N]+\Expect[\rho_{\hat{\Prob}_{N,N,N}}(\ell(\z,\vecxi))|\hat{\Prob}_N],\]
where $\hat{\Prob}_{N,N,N}$ is the empirical distribution of $N$ samples drawn from $\hat{\Prob}_{N,N}$.

Unfortunately,  both \model{BS} and \model{DBS} estimators suffer almost surely the same underestimation issue as the SAA estimator. %
\begin{proposition}\label{prop:bs:dbs}
Propositions \ref{prop:bias:saa:superlinear}, \ref{prop:saa:normalLB}, \ref{prop:saa:laplaceLB}, and \ref{prop:saa:UB} hold verbatim when $\rho_{\hat{\Prob}_N}(\ell(\vecz, \vecxi))$  is replaced by either $\rho_{\texttt{BS}}$ or $\rho_{\texttt{DBS}}$.

\end{proposition}

\subsubsection{Leave-one out cross validation}
To  mitigate the underestimation of the optimal value $ \rho_\Prob(\ell(\z,\vecxi))$ by the empirical (or \model{SAA}) estimator,  leave-one out cross validation (CV) is proposed in the literature \citep{stone1974cross, arlot2010survey}.
Let $\hat{\mathbb{P}}_{N_{-i}}$ denote the empirical distribution without the $i$th scenario, and let $\hat{t}_{-i}$ denote the optimal solution of \eqref{eq:oce:rep} in which $\hat{\mathbb{P}}_{N_{-i}}$ is used instead of $\hat{\Prob}_N$. The   estimator is then defined as: 
\[\rho_{\model{\model{LOOCV}}}:=  \frac{1}{N}\sum_{i=1}^{N} h(\hat{t}_{-i}, \ell(\z,\hat{\vecxi}_i)) = \frac{1}{N}\sum_{i=1}^{N} \left(\hat{t}_{-i}+\frac{1}{\alpha}\left(\exp(\alpha \ell(\z,\hat{\vecxi}_i) -\hat{t}_{-i}))-1\right)\right).\]
Since $\hat{t}_{-i}$ is a feasible solution of \eqref{eq:oce:rep}, we have $\rho_{\Prob}(\ell(\z,\vecxi)) \leq  \Expect_{\Prob}[h(\hat{t}_{-i}, \ell(\z,\vecxi))]$ almost surely with respect to the randomness of $\hat{t}_{-i}$ for all $i\in [N]$. Thus, 
\begin{equation}\label{eq:loocvposBias}
\begin{array}{l@{\,}l}
\rho_{\Prob}(\ell(\z,\vecxi)) &\,\leq \displaystyle\,\Expect\left[\frac{1}{N}\sum_{i=1}^N\Expect_{\Prob}[h(\hat{t}_{-i}, \ell(\z,\vecxi))]\right]=\displaystyle\frac{1}{N}\sum_{i=1}^N\Expect\left[\Expect_{\Prob}[h(\hat{t}_{-i}, \ell(\z,\vecxi))]\right]\\
 &= \,\displaystyle\frac{1}{N}\sum_{i=1}^N\Expect\left[\Expect[h(\hat{t}_{-i}, \ell(\z,\hat{\vecxi}_i))]|\hat{\mathbb{P}}_{N_{-i}}]\right] \, =\,\displaystyle \frac{1}{N}\sum_{i=1}^N\Expect[h(\hat{t}_{-i}, \ell(\z,\hat{\vecxi}_i))]= \displaystyle\, \Expect [\rho_{\model{\model{LOOCV}}}],  
\end{array}
\end{equation}
where the first equality follows from linearity of expectation, the second is due to the independence of $\hat{\vecxi}_i\sim \Prob$ and $\{\hat{\vecxi}_j\}_{j\in [N]_{-i}}$ and the third follows from the law of iterated expectations.
Thus, $\rho_{\model{\model{LOOCV}}}$ is a positively biased estimator of $\rho_{\Prob}(\ell(\z,\vecxi))$ and thus resolves the issue of underestimation of risk. Unfortunately, the next proposition establishes that it comes at the price of the magnitude of the positive bias growing exponentially, instead of linearly, with respect to the standard deviation of $\ell(\z,\vecxi)$ %
when the tails of $\ell(\z,\vecxi)$ are light enough.%

\begin{proposition}\label{thm:loocv_overestimate}
    If Assumption \ref{ass:meanVarControl} is satisfied and $\ell(\z,\vecxi)$ has subgaussian tails, then $|\rho_{\Prob}(\ell(\vecz, \vecxi)) - \mathbb{E}[\rho_{\model{LOOCV}}]| =\Omega(\exp(k_1\sqrt{\text{Var}(\ell(\vecz, \vecxi))}))$ %
    and $|\rho_{\Prob}(\ell(\vecz, \vecxi)) - \mathbb{E}[\rho_{\model{LOOCV}}]| =\mathcal O (\exp(k_2\text{Var}(\ell(\vecz, \vecxi))))$ %
    for some $k_1>0$ and $k_2>0$. %
\end{proposition}

In summary, estimators designed to mitigate the underestimation of risk by the empirical risk estimator are ill-suited for the entropic risk measure, as they either inherit the same negative bias or become overly conservative. Because accurately estimating the bias is challenging in practice, a desirable estimator should intentionally overestimate the true entropic risk while keeping the level of overestimation controlled.

\section{Bias mitigation  using bias-aware parametric bootstrapping}\label{sec:bias_correct}
 In this section, we introduce our proposed estimators designed to address the underestimation problem associated with the empirical entropic risk estimator,
 $\rho_{\hat{\Prob}_N}(\ell(\z,\vecxi))$. The true bias is given by $\Expect[\rho_{\hat{\Prob}_N}(\ell(\z,\vecxi))]-\rho_\Prob(\ell(\z,\vecxi))$, where the expectation is taken with respect to the randomness of $\hat{\Prob}_N$. Since the true distribution $\Prob$ is unknown, the exact bias cannot be determined. 
 Instead, we propose a modification to the classical bootstrap algorithm described in Section \ref{sec:commonBS}. Namely,  we first fit a distribution $\Q_N$ using the $N$ i.i.d.\ loss scenarios $\DataN_N:=\{\ell(\z,\hat{\vecxi}_1), \ell(\z,\hat{\vecxi}_2), \cdots, \ell(\z,\hat{\vecxi}_N)\}$ and then repeatedly sample from $\Q_N$, instead of resampling from the empirical distribution.  We will demonstrate  that a properly fitted distribution allows one to control the effect of variance on underestimation, while still ensuring that the estimator is  strongly asymptotically consistent. Nevertheless,  we will observe empirically that fitting $\Q_N$ using MLE does not convincingly resolve the underestimation issue, which is why we introduce bias-aware procedures to better fit the data, see Sections~\ref{sec:RiskMatching} and \ref{sec:extremes}. 

\subsection{The bias mitigation procedure}
Similar to Assumption~\ref{assum:tail_bound2}, the following assumption ensures that the mean and variance of $\exp(\alpha\zeta)$ are finite when $\zeta$ is drawn from the fitted distribution $\Q_N$.

\begin{assum}\label{assum:fit_tail_bound}
    Suppose that the tails of $\zeta\sim\Q_N$ are almost surely uniformly exponentially bounded. Namely, with probability one with respect to the sampling process and the fitting procedure of $\Q_N$, there exists some $G>0$ and some $C>2$ and \revised{$\bar{N}>0$} such that $\zeta\sim\Q_N$ satisfies Assumption \ref{assum:tail_bound2} for all \revised{$N\geq \bar{N}$}. 
\end{assum}
This assumption is a reasonable one to make when  Assumption \ref{assum:tail_bound2} is satisfied given that it simply requires one to restrict the class of eligible distributions for $\Q_N$ so that it captures similar properties as $\ell(\z,\vecxi)$. %
In practice, this assumption is satisfied by properly defining the set of models used to estimate $\Q_N$ from the loss scenarios $\mathcal{S}$. %

Next, we introduce the necessary notation to describe our estimation procedure. Let $\rho_{\Q_N}(\zeta)$ denote the entropic risk for the distribution $\Q_N$. Let $\hat{\Q}_{N,N}$ represent the empirical distribution of $N$ values drawn i.i.d.\ from the estimated distribution $\Q_N$, and $\rho_{\hat{\Q}_{N,N}}(\zeta)$ denote the corresponding empirical entropic risk. Notice that since $\Q_N$ was estimated using $N$ loss scenarios, it is itself random, thus $\rho_{\mathbb Q_N}(\zeta)$ and $\rho_{\hat{\mathbb Q}_{N,N}}(\zeta)$ are random variables as well.  Our proposed estimator for the needed bias correction of the entropic risk is given by: %
\begin{equation}
    \delta_N(\Q_N):= \mbox{median}\left(\rho_{\Q_N}(\zeta)-\rho_{\hat{\Q}_{N,N}}(\zeta)|\Q_N\right),
\end{equation}
where the median is taken with respect to randomness of samples drawn from $\Q_{N}$ to produce $\hat{\Q}_{N,N}$.
Algorithm~\ref{alg:bootstrap_with_GMM} summarizes a Monte-Carlo approximation based procedure for calculating $\delta_N(\Q_N)$. %
Note that as the number of Monte-Carlo samples increases, the sampling approximation of the parametric bootstrap bias correction estimate improves, i.e. $\hat{\delta}_N(\mathbb Q_N)$ converges to the true estimate $\delta_N(\mathbb Q_N)$.
\begin{algorithm}[htb]
\caption{Parametric Bootstrap bias correction  \label{alg:bootstrap_with_GMM}}
\begin{algorithmic}[1]
\Function{ParametricBootstrapBiasCorrection}{$\mathcal{S}, M$}
        \State $\Q_N \gets$ Fit a distribution to the loss scenarios $\mathcal{S}$
        \For{$n \gets 1$ \textbf{to} $M$}
            \State $\hat{\Q}_{N,N} \gets  \text{Draw } N \text{ i.i.d.  samples  from } \Q_N$ 
            \State $\rho_n \gets \rho_{\hat{\Q}_{N,N}}(\zeta)$  
        \EndFor
        \State $\hat{\delta}_N(\Q_N) \gets \mbox{median}\left(\{\rho_{\Q_N}(\zeta)-\rho_n\}_{n=1}^M\right)$ \label{bias:corr}
        \State \Return $\hat{\delta}_N(\Q_N)$
\EndFunction
\end{algorithmic}
\end{algorithm}

In the next theorem, we show that %
as the number of training samples $N\rightarrow\infty$, the bias-adjusted empirical risk almost surely converges to the true entropic risk.
\begin{theorem}\label{thm:botstrapconsistent}
    Under Assumptions \ref{assum:tail_bound2} and \ref{assum:fit_tail_bound}, the estimator $\rho_{\hat{\Prob}_N}(\ell(\z,\vecxi))+\delta_N(\Q_N)$ is strongly asymptotically consistent with respect to $\rho_\Prob(\ell(\z,\vecxi))$.
\end{theorem}

\revised{The proof involves two key steps. The first step is to establish that the empirical entropic risk converges to the true risk almost surely. %
The second step involves showing that the bias correction term, $\delta_N(\Q_N)$,  converges to zero almost surely. 
We note that our approach does not rely on asymptotic or parametric (Gaussian, for instance) assumptions made in the literature to correct the bias \citep{Kim_Hardy_2007, siegel2021profit, Troop_Godin_2021}.  %

We now introduce %
an assumption on the tails of loss distribution that will allow us to guarantee that our proposed estimator eventually overestimates the risk as variance of $\ell(\z,\vecxi)$ increases. The main idea is as follows. If one wishes to control the effect of variance of losses, which entropic risk grows at rate $\mathcal O(\mbox{Var}(\ell(\z,\vecxi))^{\bar{p}/2}%
)$, then they should fit a distribution $\mathbb Q_N$ associated to an entropic risk $\rho_{\Q_N}(\zeta)$ that grows at rate $\Omega(\mbox{Var}(\zeta)^{\bar{q}/2}%
)$, with $\bar{q}>\bar{p}$ 
in order for the bias of the bias-adjusted empirical risk to grow at the rate $\Omega(\mbox{Var}(\ell(\z,\vecxi))^{\bar{q}/2})$. Specifically, we will consider $1<\bar{p}<\bar{q}=2$.

    \begin{definition}[Lighter-than-Gaussian tails]\label{def:lighter-than-gauss} 
   We say that a random variable $X$  has lighter-than-Gaussian tails if there exists a $q>2$ and $G, C>0$  such that 
    \[\mathbb{P}(|X|\ge a)\le G\exp\left(-Ca^{q}\right), \forall a\ge0.\]
    It implies that there exists a $\bar p\in(1,2)$%
    and a $\nu>0$ such that $\Expect[\exp(tX)]\leq 2\exp(\nu t^{\bar{p}})$ for all $t\geq 0$, see Lemma~\ref{lemma:equivalence:mgf_tails} in Appendix \ref{append:bias_correct} for the proof.
    \end{definition}    
Next, we will describe properties that need to hold for $\mathbb Q_N$ almost surely in order for the variance effect to be mitigated.

\begin{assert}[Affine equivariance of fitting procedure]\label{property:affEquiFit}
    $\Q_N$ is equivariant to affine transformations of the loss sample set $\DataN$. Namely, %
    the fitting procedure is such that, for all $(a,b)\in\mathbb{R}\times\mathbb{R}_+$, if $\zeta\sim \Q_N$, fitted on $\{\ell(\z,\hat{\vecxi}_i)\}_{i=1}^N$, and $\zeta'\sim \Q_N'$, fitted on $\{a+b\ell(\z, \hat{\vecxi}_i)\}_{i=1}^N$, then $a+b \zeta%
    \stackrel{F}{=}\zeta'%
    $.%
\end{assert}

\begin{assert}[Empirical mean matched]\label{property:unifBoundedMeanFit}
The mean of $\zeta\sim\Q_N$  is almost surely matched to the empirical mean $\hat{\mu}_N:=\Expect_{\hat{\Prob}_N}[\ell(\z,\vecxi)]$. Namely, with probability one, we have that $\Expect_{\Q_N}[\zeta]= \hat{\mu}_N$.
\end{assert}

\begin{assert}[Uniformly Heavy Gaussian Standardized  Right Tails]\label{property:unifHeavyGaussFit}
Letting $\zeta\sim\Q_N$ and $(\hat{\mu}_N,\hat{\sigma}_N)$ the mean and variance of $\hat{\Prob}_N$, the right tail of $\tilde{\xi}:=(\zeta-\hat{\mu}_N)/\hat{\sigma}_N$ is uniformly heavy Gaussian. Namely, there exists  some $\mathfrak{c}_1, \mathfrak{c}_2, \bar{a}> 0$ such that with probability one, we have $\zeta\sim\Q_N$ satisfies $\Prob_{\Q_N}((\zeta-\hat{\mu}_N)/\hat{\sigma}_N\geq a)\geq \mathfrak{c}_1\exp(-a^2/\mathfrak{c}_2^2)$ for all $a\geq \bar{a}$. %
\end{assert}

\begin{assert}[Uniformly Subgaussian Standardized Tails]\label{property:unifSubGaussFit}
Letting $\zeta\sim\Q_N$ and $(\hat{\mu}_N,\hat{\sigma}_N)$ the mean and variance of $\hat{\Prob}_N$, the standardized variable $\tilde{\xi}:=(\zeta-\hat{\mu}_N)/\hat{\sigma}_N$  
has uniformly subgaussian tails. Namely, there exists some $\mathfrak{c}_3>0$ such that with probability one, we have $\Prob_{\Q_N}(|(\zeta-\hat{\mu}_N)/\hat{\sigma}_N|\geq a)\leq 2\exp(-a^2/\mathfrak{c}_3^2)$ for all $a\geq 0$. %
\end{assert}

The above properties jointly control the tail behavior of the standardized losses in a way that yields a controlled overestimator of the entropic risk. Affine equivariance (Property \ref{property:affEquiFit}) links the fit on original losses to the fit on standardized losses: standardizing the data affects the fitted model only through the corresponding affine map. As a result,  the bias correction can be analyzed on the standardized scale and expressed explicitly in terms of the growth of the standardized cumulant generating function. Property \ref{property:unifBoundedMeanFit} controls the mean so that the conservatism induced by Property \ref{property:unifHeavyGaussFit} is not offset by an uncontrolled mean shift. Indeed, Property \ref{property:unifHeavyGaussFit} enforces a sufficiently heavy right tail for the fitted model, yielding a cumulant generating function whose growth is $\Omega(\Var(\ell(\vecz, \vecxi)))$ and thus producing systematic overestimation when the true loss is lighter-than-Gaussian. Finally, Property \ref{property:unifSubGaussFit} prevents this overestimation from becoming excessive by imposing subgaussian tails on the standardized fit, ensuring the estimation error grows at most linearly in the variance, i.e., $\mathcal O(\Var(\ell(\vecz, \vecxi)))$,  and providing the exponential moment bounds needed for strong asymptotic consistency.

\begin{theorem}\label{thm:our_estimator_overestimate_omega}\label{thm:our_estimator_overestimate_O}
    Let Assumption \ref{ass:meanVarControl} be satisfied and the loss $\ell(\vecz, \vecxi)$ have lighter-than-Gaussian tails. Further let $\Q_N$ and its fitting procedure satisfy properties \ref{property:affEquiFit}, \ref{property:unifBoundedMeanFit}, and \ref{property:unifHeavyGaussFit}.   Then:
    \begin{equation}
    \Expect[\rho_{\hat{\Prob}_N}(\ell(\z,\vecxi))+\delta_N(\Q_N)] - \rho_{\Prob}(\ell(\vecz, \vecxi)) = \Omega(\mbox{Var}(\ell(\vecz, \vecxi))).\label{eq:thm:OmegaVar}
    \end{equation}
    If $\Q_N$ additionally satisfies Property \ref{property:unifSubGaussFit}, then 
    \begin{equation}
    |\rho_{\Prob}(\ell(\vecz, \vecxi))-(\Expect[\rho_{\hat{\Prob}_N}(\ell(\z,\vecxi))+\delta_N(\Q_N)])| = \mathcal O(\mbox{Var}(\ell(\vecz, \vecxi))), %
    \label{eq:thm:OVar}
    \end{equation}    
    and $\rho_{\hat{\Prob}_N}(\ell(\z,\vecxi))+\delta_N(\Q_N)$ is strongly asymptotically consistent. %
\end{theorem}

The theorem above shows that, when $\mbox{Var}(\ell(\vecz,\vecxi))$ is sufficiently large and $\ell(\vecz, \vecxi)$ has lighter-than-Gaussian tails, our bias-corrected estimator is guaranteed to overestimate the entropic risk. This resolves the underestimation exhibited by the empirical entropic-risk estimator and by standard bias correction methods. Unlike LOOCV, this overestimation does not come at a high price: the estimation error grows at the same rate as the empirical risk estimator (i.e., it remains $\mathcal O(\mbox{Var}(\ell(\vecz, \vecxi)))$). Moreover, the proposed parametric-bootstrap estimator is strongly consistent, i.e., it converges to the true entropic risk almost surely as $N\to\infty$, even when the fitted model $\Q_N$ is misspecified. This allows one to search over convenient distributional families that admit closed-form expression of entropic risk and efficient sampling, so that the bootstrap bias correction can be computed cheaply. In practice, the chosen family should also be flexible enough to capture multimodality in the loss distribution.}

\subsection{\texorpdfstring{Employing a Gaussian Mixture Model as $\Q_N$}{\Q N is a Gaussian Mixture Model}}\label{sec:Q:gaussian}

Among the options for choosing the distribution $\Q_N$, we utilize a Gaussian Mixture Model (GMM), $\Q^\btheta$, with parameters $\btheta:= (\bm{\pi}, \bm{\mu}, \bm{\sigma})$, 
 where $\mathfrak J:=\{1, \cdots, J\}$ denotes the index set of mixture components, $\bm \pi =(\pi_{\mathfrak j})_{\mathfrak j\in \mathfrak J} \in\mathbb R^{J}$ denotes the weights %
 and $\bm \mu =(\mu_{\mathfrak j})_{\mathfrak j\in \mathfrak J}\in\mathbb R^{J}$ and $\bm \sigma=(\sigma_{\mathfrak j})_{\mathfrak j\in \mathfrak J}\in\mathbb R^{J}$ denote the means and standard deviations of the mixtures, respectively.  %
There are  several %
 advantages for using GMM. \revised{First, GMMs are universal density approximators, meaning they can approximate any smooth density up to any arbitrary accuracy given sufficient number of components \citep{goodfellow2016deep}}, and second, the moment-generating function of a random variable $\zeta \sim \Q^\btheta$ exists for all $\alpha$, and thus  the entropic risk $\rho_{\mathbb Q^\btheta}(\zeta) = \frac{1}{\alpha}\log\left(\sum_{\mathfrak j\in \mathfrak J}\pi_{\mathfrak j}\exp(\alpha\mu_{\mathfrak j} + \frac{\alpha^2}{2} \sigma_{\mathfrak j}^2)\right)$ can be obtained in closed form. This eliminates the need to estimate the entropic risk through simulation in step \ref{bias:corr} of Algorithm \ref{alg:bootstrap_with_GMM}. 
Further, one can efficiently sample from a GMM, and the sampling operation can be made differentiable. %
 Finally, we will show that a  GMM properly fitted to standardized losses in $\bar{\DataN}_N$  satisfies  Properties \ref{property:affEquiFit}-\ref{property:unifSubGaussFit}.%

\revised{Our calibration procedure can be described generically as Algorithm \ref{alg:fittingGMM}. We note that step \ref{algstep:standardize} is a common step of standardization of the data. This is followed with a less common but reasonable step \ref{algstep:meanmatch}, which ensures that $\Q^{\btheta^*}$ matches the empirical mean in $
\bar{\DataN}_N$.

\begin{algorithm}[htb]
\caption{Fitting a parametric distribution to $\DataN$ to satisfy properties  \ref{property:affEquiFit} and \ref{property:unifBoundedMeanFit}}\label{alg:fittingGMM}
\begin{algorithmic}[1]
\Function{AffineEquivariantMeanPreservingCalibration}{$\DataN$, $\{\Q^\btheta\}_{\btheta\in\Theta}$}
        \State $\hat{\mu}_N\gets (1/N)\sum_{\zeta\in\DataN} \zeta$
        \State $\hat{\sigma}_N \gets \sqrt{(1/N)\sum_{\zeta\in\DataN} (\zeta-\hat{\mu}_N)^2}$
        \State $\bar{\DataN}_N \gets \{\zeta_i'\}_{i=1}^N$ with $\zeta_i' := (\zeta_i-\hat{\mu}_N)/\hat{\sigma}_N$ for all $\zeta_i\in\DataN$ \label{algstep:standardize}
        \State $\bar{\Theta} \gets \{\btheta\in\Theta| \Expect_{\Q^\btheta}[\xi']= %
        0\}$\label{algstep:meanmatch}%
        \State Fit a distribution with parameter in $\bar{\Theta}$ to $\bar{\DataN}_N$ to obtain $\Q^{\btheta^*}$ \label{algstep:fit}
        \State Identify as $\Q_N$ the distribution of $\hat{\mu}_N+\hat{\sigma}_N\xi'$, with $\xi'\sim \Q^{\btheta^*}$
        \State \Return $\Q_N$
\EndFunction
\end{algorithmic}
\end{algorithm}

Equipped with the GMM model and calibration procedure described in Algorithm \ref{alg:fittingGMM}, we obtain our main result.

\begin{theorem}\label{thm:GMM}
  Let Assumption \ref{ass:meanVarControl} be satisfied and the loss $\ell(\vecz, \vecxi)$ have lighter-than-Gaussian tails.  If the class of GMM model considered $\{\Q^\btheta\}_{\btheta\in\Theta}$ satisfies the following conditions:
    \begin{enumerate}
\item there exists a bound $B>0$ such that $\max_{\mathfrak j\in \mathfrak J} \sigma_{\mathfrak j} \leq B$ and $\max_{ \mathfrak j \in \mathfrak J} |\mu_{\mathfrak j}| \leq B$ for all $(\pi,\mu,\sigma)\in\Theta$,
        \item there exists  bounds $\epsilon>0$ and $B>-\infty$ %
        and some ${\mathfrak j}\in \mathfrak J$, such that $\pi_{{\mathfrak j}}\geq\epsilon$, $\mu_{{\mathfrak j}}\ge B$, 
        and $\sigma_{{\mathfrak j}} \geq \epsilon$ for all $(\pi,\mu,\sigma)\in\Theta$, 
    \end{enumerate}    
  then $\rho_{\hat{\Prob}_N}(\ell(\z,\vecxi))+\delta_N(\Q_N)$ with $\Q_N$ obtained from Algorithm \ref{alg:fittingGMM} is asymptotically consistent and  satisfies 
     equations \eqref{eq:thm:OmegaVar} and \eqref{eq:thm:OVar}.
\end{theorem}

The proof of the above theorem proceeds by showing that the %
GMM 
fitting procedure described in Algorithm \ref{alg:fittingGMM} satisfies properties \ref{property:affEquiFit}-\ref{property:unifSubGaussFit}; Theorem \ref{thm:our_estimator_overestimate_omega} then applies directly. Consequently, the proposed parametric-bootstrap procedure based on a fitted GMM yields an overestimator of the entropic risk. In the next subsections, we will propose three practical procedures for fitting a GMM to the standardized losses in step \ref{algstep:fit} of Algorithm \ref{alg:fittingGMM}.}

\subsubsection{Learning GMM by Maximum Likelihood Estimation }

The natural approach for fitting the parameters $\btheta$ of a GMM  is to use MLE, typically achieved via the Expectation-Maximization (EM) algorithm \citep{Dempster_Laird_1977}.  The following example demonstrates that with limited samples, %
if we use the EM algorithm to fit a GMM, 
the underestimation persists even for large sample sizes. Note that the true loss distribution assumed in the following example does not satisfy Definition \ref{def:lighter-than-gauss}, and thus the theoretical results do not apply.

 \begin{example}\label{ex:bias_correction}
Consider the problem of estimating the entropic risk of %
loss $\eta$ that follows a Gaussian mixture model with two components
$\eta \sim \text{GMM}(\V{\pi}, \V{\mu}, \V{\sigma})$, $\V{\pi} =
    [0.7\; 0.3]$, $ 
   \V{\mu} =
    [0.5\; 1]$, and  $\V{\sigma} =
    [2\; 1]$, and $\alpha=3$.
       To obtain Figure \ref{fig:bias correction}, we draw $N$ i.i.d.~samples from $\text{GMM}(\V{\pi}, \V{\mu}, \V{\sigma})$ where $N\in\{10^2, 10^3, 10^4, 10^5\}$.
 The true bias correction is obtained by first computing the true entropic risk 
and then subtracting the  expected empirical entropic risk, $\bar{\rho} = \frac{1}{10000}\sum_{i=1}^{10000}\frac{1}{\alpha} \log\left(\frac{1}{N}\sum_{j=1}^N\exp(\alpha\hat{\eta}_{j}^{(i)}) \right)$, obtained by bootstrapping with $10000$ repetitions.\footnote{We repeat this procedure $100$ times, compute the estimate of true entropic risk, and the $95\%$ confidence interval for the true entropic risk is contained in the marker drawn on Figure \ref{fig:bias correction} for the true entropic risk.}
We fit a GMM $\Q^\btheta$ to the samples using the EM algorithm and use Algorithm \ref{alg:bootstrap_with_GMM} to estimate the bias. The boxplots are plotted by resampling 100 times from $\text{GMM}(\V{\pi}, \V{\mu}, \V{\sigma})$. %
The bias of the empirical risk estimation decays at a logarithmic (in N) rate.
We see that this procedure
underestimates the true bias for finite number of samples. Also, we can observe that as the number of training samples increases, the bias estimated by the fitted GMM converges toward~$0$. 
\begin{figure}[t]
    \centering
\revised{\includegraphics[width=0.5\linewidth]{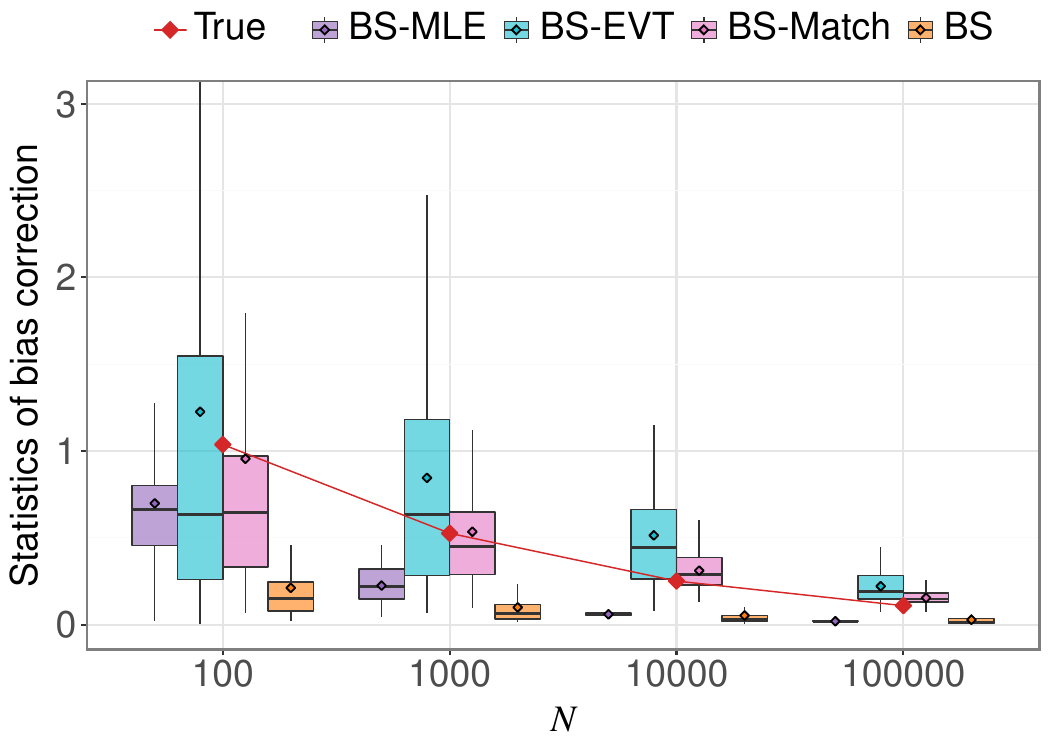}}
    \caption{Statistics of bias correction estimated from non-parametric bootstrap and parametric bootstrap obtained by fitting a GMM by MLE (\model{BS-MLE}), entropic risk matching (\model{BS-Match}) and tail fitting (\model{BS-EVT}) followed by bootstrapping over $100$ resampling from the underlying distribution.}
    \label{fig:bias correction}
\end{figure}
\end{example}
\revised{\cite{jin2016local} show that, even for equally-weighted mixtures of well-separated spherical Gaussians with $J \ge 3$ components, the population (infinite-sample) likelihood admits bad local maxima, and that EM  with random initialization converges to bad critical points with probability at least $1-\exp(-\Omega(J))$. Having observed that the bias of the empirical risk estimator is governed by the cumulant generating function and the distribution of the maxima of N samples drawn from the underlying loss distribution,  we next propose alternative distribution fitting strategies.} %
We fit a distribution such that $N$ i.i.d.\ samples drawn from it replicate the bias observed in $N$ i.i.d.\ samples from the true distribution $\Prob$. We refer to this approach as ``bias-aware" distribution matching. This concept is inspired by ``decision-aware'' learning methods in contextual optimization problems \citep{Elmachtoub_Grigas_2021, Donti_Amos_2017, Grigas_Qi_Shen_2023, Sadana_Chenreddy_2025}, where statistical accuracy is deliberately traded for improved decision outcomes.

 \subsubsection{Learning GMM by Entropic Risk Matching}\label{sec:RiskMatching}

In this section, \revised{we introduce an alternative loss function to learn the parameters of a GMM, replacing the likelihood function. %
Given a candidate GMM $\Q^{\btheta}$, we compare (i) the sampling distribution of the empirical entropic-risk estimator computed from $n$ i.i.d.\ draws from $\Q^{\btheta}$ to (ii) the empirical sampling distribution obtained by computing the same estimator from $n$ draws taken from the scenario set $\mathcal S$. 
By minimizing this discrepancy, we aim to fit a model whose finite-sample entropic-risk bias matches that observed in the data.
The approach is based on the following proposition, which establishes that the sampling distribution of the entropic risk uniquely identifies the underlying loss distribution.%

\begin{proposition}\label{prop:identifiable}
 Let $\{\hat{\zeta}_i\}_{i=1}^n$ and $\{\hat{\zeta}_i'\}_{i=1}^n$ be $n$ i.i.d. random variables drawn respectively from $\mathbb{Q}$ and $\mathbb{Q}'$. Let $\hat{\rho}_{\mathbb Q, n}:= \frac{1}{\alpha}\log((1/n)\sum_{i=1}^n\exp(\alpha \hat{\zeta}_i))$ and $\hat{\rho}_{\mathbb Q', n}:= \frac{1}{\alpha}\log((1/n)\sum_{i=1}^n\exp(\alpha \hat{\zeta}_i'))$. We must have that $\hat{\rho}_{\mathbb Q, n}\stackrel{F}{=} \hat{\rho}_{\mathbb Q', n}$ if and only if $\hat{\zeta}_i\stackrel{F}{=} \hat{\zeta}'_i$.
\end{proposition}}

Algorithm~\ref{alg:riskmatching} %
 is used to learn the parameters $\btheta$ of the GMM.
\revised{To construct the empirical distribution of entropic risk for $n$ samples drawn from the empirical standardized loss scenarios $\bar{\mathcal{S}}$}, \revised{$\bar{\mathcal{S}}$} is divided into $B$ bins, with each bin containing $n = N / B$ scenarios. The entropic risk is computed for each bin, forming the set $\mathcal{R}_{\bar{\mathcal S}}$. The corresponding empirical distribution, $\hat{\mathbb{P}}_{\mathcal{R}_{\bar{\mathcal S}}}$, over the set $\mathcal{R}_{\bar{\mathcal S}}$, captures the variability of entropic risk across the $B$ bins. For the former distribution, with a fixed $\btheta$, $B' \times n$ i.i.d.\ samples are drawn from $\Q^{\btheta}$ and divided into $B'$ bins. The entropic risk is then computed for the scenarios in each bin, yielding the set $\mathcal{R}_{\btheta}$. The corresponding empirical distribution, $\hat{\mathbb{P}}_{\mathcal{R}_\btheta}$, captures the variability of entropic risk  across the $B'$ bins.
 
Next, the algorithm compares the empirical distribution $\hat{\mathbb{P}}_{\mathcal{R}_{\bar{\mathcal S}}}$ with the model-based distribution $\hat{\mathbb{P}}_{\mathcal{R}_\btheta}$ using the following Wasserstein distance:
\begin{align*}
\mathcal{W}^2\left(\hat{\mathbb{P}}_{\mathcal{R}_{\bar{\mathcal S}}},\hat{\mathbb{P}}_{\mathcal{R}_\btheta}\right) = \left(\int_{0}^1 \lvert F_{\mathcal{R}_{\bar{\mathcal S}}}^{-1}(q)-F_{\mathcal{R}_\btheta}^{-1}(q) \rvert^2 dq\right)^{1/2},
\end{align*}
where $F_{\mathcal{R}_{\bar{\mathcal S}}}^{-1}$ and $F_{\mathcal{R}_\btheta}^{-1}$ are quantile functions associated with sets $\mathcal{R}_{\bar{\mathcal S}}$ and $\mathcal{R}_\btheta$, respectively. This distance quantifies the discrepancy between the two distributions. The algorithm iteratively adjusts the GMM parameters to minimize this distance. It uses gradient descent to update the parameters $\btheta_t$ at each iteration as shown in step \ref{grad_descent},
where $\gamma$ is the step size. \revised{The parameters are projected onto the feasible space $\bar{\Theta}$ in each iteration, thereby ensuring that the mean of $\zeta \sim \Q^{\btheta}$ is $0$}. 

To enable computation of the gradients of the Wasserstein distance with respect to $\btheta$, the algorithm employs differentiable sampling techniques that enable automatic differentiation through the sampling process (see Algorithm \ref{alg:samplegmm} in Appendix~\ref{appen:samplegmm}). This is based on reparameterization approach \citep{Kingma_Salimans_2015}, which allows stochastic sampling operations to be expressed in a differentiable manner. The iterative process continues until the Wasserstein distance $\mathcal{W}^2\left(\hat{\mathbb{P}}_{\mathcal{R}_{\bar{\mathcal S}}},\hat{\mathbb{P}}_{\mathcal{R}_{\btheta_t}}\right)$ falls below a predefined convergence threshold $\epsilon$, or until a maximum number of iterations $T$ is reached. Further details of the algorithm can be found in Appendix \ref{appen:fit_dist}.

Even though computing the Wasserstein distance between distribution of losses has a worst-case complexity $O(B'\log(B'))$ \citep{Kolouri_Pope_2019}, there is a significant total cost associated with the gradient descent procedure described in Algorithm \ref{alg:riskmatching}. In the next section, we provide a semi-analytic procedure to learn a two-component GMM that can account for the tail scenarios.

\begin{algorithm}
\caption{Fit GMM \revised{in step \ref{algstep:fit} of Algorithm \ref{alg:fittingGMM}} by entropic risk matching \label{alg:riskmatching}}
\begin{algorithmic}[1]
\Function{\model{BS-Match}}{$\bar{\mathcal{S}}, \bar{\Theta}, J$}
  \State Divide standardized loss scenarios in $\bar{\mathcal{S}}$ into $B$ bins, each of size $n$
\State Compute the entropic risk in $B$ bins, forming the empirical distribution $\hat{\mathbb{P}}_{\mathcal{R}_{\bar{\mathcal S}}}$%
  \State $\btheta_0 \gets \textsc{EM}(\bar{\mathcal S},J)$ \revised{\Comment{EM algorithm ensures that $\btheta_0 \in \bar\Theta$}} %
    \State Initialize the iteration counter $t \gets 0$ and $\mathfrak{D} \gets \infty$
  \While{$\mathfrak{D}>\epsilon$ and $t<T$} 
  \State Draw $B'\times n$ i.i.d.~samples from $\Q^{\btheta_t}$, split  into $B'$ bins
\State Compute entropic risk in each bin, forming $\hat{\mathbb{P}}_{\mathcal{R}_{\btheta_t}}$
    \State \label{grad_descent} Update GMM parameters: $\V{\theta}_{t+1} \gets \V{\theta}_t - \gamma \nabla_{\V{\theta}_t} \mathcal{W}^2\left(\hat{\mathbb{P}}_{\mathcal{R}_{\bar{\mathcal S}}}, \hat{\mathbb{P}}_{\mathcal{R}_{\btheta_t}}\right)$
    \State Project $\V{\theta}_{t+1}$ onto \revised{$\bar{\Theta}$ and} the feasible region of a GMM \revised{described in Theorem \ref{thm:GMM}}  %
    \State Update distance: $\mathfrak{D} \gets \mathcal{W}^2\left(\hat{\mathbb{P}}_{\mathcal{R}_{\bar{\mathcal S}}},\hat{\mathbb{P}}_{\mathcal{R}_{\btheta_t}}\right)$
    \State Increment iteration counter: $t \gets t + 1$ 
  \EndWhile
  \State \Return $\Q^{\btheta_t}$
\EndFunction
\end{algorithmic}
\end{algorithm}

 \subsubsection{Learning GMM by matching the extremes}\label{sec:extremes}

\revised{The next procedure to fit GMM in step \ref{algstep:fit} of Algorithm \ref{alg:fittingGMM} is motivated by an %
upper bound on the bias that depends on the maxima of the samples from the underlying loss distribution and its cumulant generating function. In particular, Lemma~\ref{lemma_almst_sure} in the Appendix~\ref{appendix:2} implies that for  $\ell(\vecz,\vecxi)\stackrel{F}{=}\mu+\sigma\xi$, $\Expect[\rho_{\hat{\Prob}_N}(\ell(\z,\vecxi))] -\rho_\Prob(\ell(\z,\vecxi))   \leq   \sigma \Expect[M_N] - (1/\alpha)\Expect[\Lambda_{\Prob^{\xi}}(\alpha \sigma )] %
$ where $M_N =\max_{1\leq i \leq N} \hat{\xi}_i$ with $\hat{\xi}_i=(\ell(\z,\hat{\vecxi}_i)-\mu)/\sigma\sim\Prob^\xi$ independently.  Motivated by this inequality, we propose to fit one of the GMM components of the $\mathbb{Q}^\btheta$   to the distribution of $M_N$.}
\revised{From the extreme value theory (EVT), the maxima of distributions, %
such as Gaussian, Gamma, and Laplace, fall in the Gumbel maximum domain of attraction \citep{embrechts1997modelling, de_Haan_Ferreira_2006}. This means that, after an appropriate centering and scaling that depends on the sample size, the distribution of the maximum converges to a Gumbel limit. Because both the target losses and the Gaussian family share this same limiting extreme-value behavior, using the maximum of a fitted normal distribution yields an analytically tractable surrogate for the right tail. } %

We construct an equally-weighted two-component GMM to represent the loss distribution (see Algorithm~\ref{alg:fitGMMevt} in the Appendix \ref{appen:fit_dist}). The first component aims to estimate the distribution of maxima $M_n$ of the loss scenarios. To this end, first  we divide the \revised{standardized} loss scenarios in  $\bar{\mathcal{S}}$ into $B$ bins, each of size $n = N/B$. We store the maximum within each bin in set $\mathcal{M}$, where  $F_\mathcal{M}$ denotes the cdf of scenarios in $\mathcal{M}$. This approach is typically referred to as the block maxima method \citep{de_Haan_Ferreira_2006}.  Motivated by the extreme value theory, the algorithm estimates the parameters of the first component by matching the cdf of the maximum of $n$ i.i.d. samples from $\mathcal{N}(\mu^e, \sigma^e)$, denoted by $\Phi_{\mu^e, \sigma^e}^n$ to  $F_\mathcal{M}$. 
The calculation of the parameters $(\mu^e, \sigma^e)$ can be done in a semi-analytic way as discussed in Appendix \ref{appen:fit_dist}. For the second component of the GMM, we set the mean to $-\mu^e$ 
and the standard deviation to 0. This choice ensures that the mean of the overall GMM is equal to the mean of the loss scenarios $\mu_{\bar{\mathcal{S}}}$\revised{, thus ensuring that the conditions in Algorithm \ref{alg:fittingGMM} are satisfied, and that Theorem  \ref{thm:GMM} holds}.  %

\revised{Proposition~\ref{prop:saa:laplaceLB} showed that for Laplace-distributed losses, the negative bias becomes unbounded as $\mbox{Var}(\ell(\vecz,\vecxi)) \to 2/\alpha^2$. %
In contrast, by construction, the estimation error of the proposed two-component GMM is asymptotically bounded by $\mathcal O(\mbox{Var}(\ell(\vecz, \vecxi)))$. Thus, one cannot expect the two-component GMM to fully remove the negative bias of the empirical risk estimator. %
 Designing an estimator that provably overestimates the entropic risk uniformly over all exponentially bounded distributions (e.g., Laplace distribution) is therefore nontrivial and requires fitting a distribution with tails at least as heavy as the underlying distribution, together with guarantees that any induced conservatism  remains controlled. 
 Controlled overestimation with GMMs is  guaranteed for lighter-than-gaussian tails, with \model{BS-EVT} providing a practical bias mitigation tool for subgaussian losses as illustrated in examples~\ref{ex:bias_correction} and \ref{ex:projects}. The same holds true for \model{BS-Match}.

}

\subsection{Benchmarking bias correction estimators}\label{benchmark_section}
\revised{In this section, we review several methods already discussed for estimating entropic risk and show through a numerical example how the variance of investments affect their estimation bias.} %
The Median-of-Means (\model{MoM}) estimator is constructed by dividing the loss scenarios into $\floor{\sqrt{N}}$ blocks, calculating the entropic risk  within each block, and then taking the median of these entropic risk values \citep{Lugosi_Mendelson_2019}. Thus, the \model{MoM} also suffers from the underestimation issue inherent in the empirical risk estimator.

Similar to Example~\ref{ex:bias_correction}, the true loss distribution in the following example does not satisfy Definition~3, and therefore the theoretical results do not apply.
  \begin{example}\label{ex:projects}
       Consider a project selection problem with three projects.  Let  $\xi \sim \text{GMM}\left(\V{\pi}, \V{\mu}, \V{\sigma}\right) $  with 5 components  (see Appendix \ref{append:param:example} for parameter values of the GMM).  Each project is affected by the same random variable $\xi$.
       Suppose the losses associated with the three projects are given by $z\xi$ with $z =0.4, 0.6, 0.8$, respectively. Let   the  risk aversion parameter be $\alpha=3$. The true entropic risk admits a closed-form expression
   $\rho(\ell(z, \eta)) = (1/3) \log(\sum_{\mathfrak j=1}^5\pi_{\mathfrak j}\exp\left(3z\mu_{\mathfrak j} +(9/2)(\sigma_{\mathfrak j}z)^2\right))$.  To evaluate each estimator, we draw $100$ instances with a sample size of $N=10000$ from the $\text{GMM}\left(\V{\pi}, \V{\mu}, \V{\sigma}\right)$. 
In Figure \ref{fig:project}, we observe the behavior of the bootstrap   (Algorithm~\ref{alg:bootstrap_with_GMM}) in conjunction with the entropic risk matching (Algorithm~\ref{alg:riskmatching}) denoted by \model{BS-Match}, and the bootstrap   (Algorithm~\ref{alg:bootstrap_with_GMM}) in conjunction with the extremes matching  (Algorithm~\ref{alg:fitGMMevt}) denoted by \model{BS-EVT}. 
It can be seen that the true entropic risk of project 1 is lower than that  of project 3. Furthermore project $3$ has a higher standard deviation of $0.8\sqrt{\text{Var}(\xi)}$, while that of project 1 is $0.4\sqrt{\text{Var}(\xi)}$. Nevertheless, several  estimators, nonparametric bootstrap (\model{BS}), sample–average approximation (\model{SAA}), median-of-means (\model{MOM}), optimizer’s information criterion (\model{OIC}), and maximum likelihood (\model{BS-MLE}), systematically underestimate the entropic risk of the riskier project~3. As a result, these procedures tend to make project~3 appear more attractive than project~1 from a risk-adjusted perspective. In contrast, our proposed estimators, \model{BS-Match} and \model{BS-EVT}, substantially inflate the estimated risk of the high-variance projects relative to the empirical estimator, and are therefore more informative about the true risk exposure faced by the decision maker.
Finally, the \model{LOOCV} estimator is guaranteed to exhibit overestimation of risk in expectation, yet our numerical results highlight an important limitation: its overestimation grows exponentially in the standard deviation of the loss. Consistent with this behavior, the empirical distribution of the \model{LOOCV} estimates over the 100 experiments is highly skewed, with the median estimate lying significantly below the true entropic risk.

\begin{figure}[t]
    \centering
\revised{\includegraphics[width=0.5\linewidth]{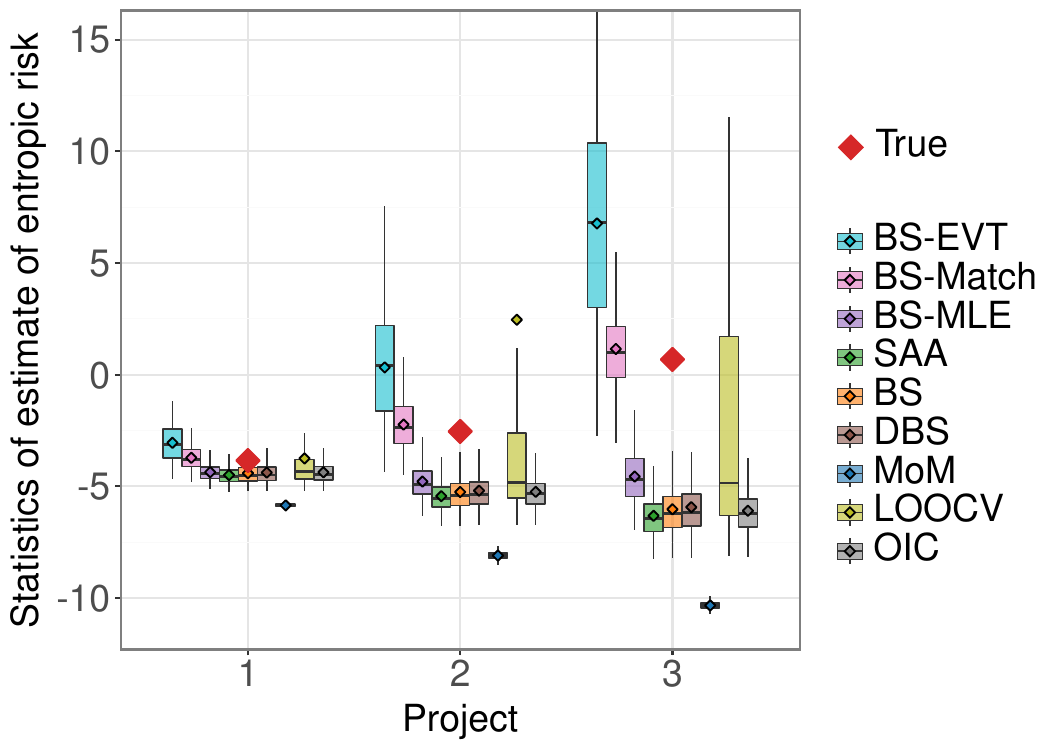}}
    \caption{Statistics of the estimates of the true entropic risk obtained from different models for each  project.}
    \label{fig:project}
\end{figure}

\end{example}

In the next section, we will show that the proposed procedures for  mitigating estimation bias can also be applied to mitigate the optimistic bias when solving entropic risk minimization problems. As discussed earlier, optimistic bias occurs due to lack of data, in which case regularization type techniques (such as distributionally robust optimization) are employed to correct the bias. By providing better estimates of the validation risk, we can more accurately calibrate the hyperparameters compared to traditional CV methods.

\section{\revised{Bias-aware cross validation in entropic risk optimization}}\label{sec:dro}

Entropic risk minimization considers the following problem:
\begin{align}\label{eq:true:opt}
   \rho^* =  \min_{\z \in \Z} \rho_\Prob(\ell(\z, \vecxi)):= \frac{1}{\alpha} \log \left(\mathbb{E}_{\Prob}[\exp(\alpha \ell(\z,\vecxi))]\right).
\end{align}
As the true underlying distribution $\Prob$ is typically unknown, it is common practice to replace it with the empirical distribution $\hat\Prob_N$, solving the corresponding SAA problem:%
\begin{equation}\label{eq:saa:opt}
       \rho_{\text{\model{SAA}}}^* = \min_{\z \in \Z} \rho_{\hat{\Prob}_N}(\ell(\z, \vecxi)):= \frac{1}{\alpha} \log \left(\mathbb{E}_{\hat{\Prob}_N}[\exp(\alpha \ell(\z,\vecxi))]\right).
\end{equation}
Under certain assumptions, it can be shown that $\rho_{\text{\model{SAA}}^*} \rightarrow \rho^*$ as \( N \) grows, as shown in Proposition \ref{lemma:convrate:saa} in Appendix~\ref{sec:appenddro}.  
In the limited data setting, the risk produced by solving problem \eqref{eq:saa:opt}   underestimates the true risk $\rho^*$ due to overfitting on the empirical distribution. Distributionally robust optimization (DRO) is one of the approaches to mitigate the optimistic bias of \model{SAA} by robustifying decisions against perturbations in the empirical distribution \citep{Delage_Ye_2010, Wiesemann_Kuhn_2014, Mohajerin_Esfahani_Kuhn_2018, Rahimian_Mehrotra_2022}.  It is assumed that nature perturbs the empirical distribution within a distributional ambiguity set $\mathcal{B}(\epsilon)$ containing all distributions $\Q$ that are at a ``distance'' $\epsilon\geq 0$ away from the empirical distribution $\hat{\Prob}_N$, so as to maximize the entropic risk of the decision maker, while decision maker aims to minimize the worst-case risk resulting in the following min-max problem:%
\begin{equation}\label{eq:DRO}
\rho_{\texttt{DRO}}^*:=\min_{\z\in \mathcal{Z}}\sup_{\Q \in \mathcal{B}(\epsilon)}\frac{1}{\alpha} \log \left(\mathbb{E}_{\Q}[\exp(\alpha \ell(\z,\vecxi))]\right). 
\end{equation}
In Theorem \ref{thm:converge:dro} in Appendix~\ref{sec:appenddro}, we show that $\rho_{\model{DRO}}^* \rightarrow \rho^*$ in probability. Furthermore, we show that type-$\infty$ Wasserstein ambiguity set is a suitable choice  for problem~\eqref{eq:DRO}.
The type-$\infty$ Wasserstein ambiguity set $\mathcal{B}_\infty(\epsilon)$  of radius $\epsilon\geq 0$, can be defined as follows:
\begin{equation} 
\mathcal{B}_\infty(\epsilon) := \left\{
         \Q \in \mathcal{M}(\Xi)\lvert  \Q\left\{\vecxi \in \Xi\right\} = 1,    \mathcal{W}_\infty (\Q, \hat{\Prob}_{N})\leq \epsilon \right\}, \label{def:wass_infty}
\end{equation}
where the Wasserstein distance is defined as
\begin{align*}
\mathcal{W}_\infty\left(\Prob_1,\Prob_{2}\right):=\inf_{\pi\in \mathcal{M}(\Xi\times\Xi)}\left\{\text{ess.sup}\|{\bzeta}_1-{\bzeta}_2\| \,\pi (d\bzeta_1,d\bzeta_2)\right\}.
\end{align*}
Here, $\pi$ is a joint distribution of the random vectors $\bzeta_1$ and $\bzeta_2$ with marginals $\Prob_1$ and $\Prob_2$, respectively, $\operatorname*{ess\,sup}$ denotes the essential supremum, and  $\| \cdot\|$ denotes the norm.

A common approach to select the radius $\epsilon$ of the ambiguity set is through $K$-fold CV. For each $\epsilon\in\mathcal E$, $K$-fold CV aims to estimate the true performance of policy $\z^*(\hat{\mathbb P}_{N},\epsilon)$ resulting from problem~\eqref{eq:DRO}, i.e.,   $\rho_{\mathbb P}(\ell(\z^*(\hat{\mathbb P}_{N},\epsilon), \vecxi))$, and subsequently select the  $\epsilon$ that minimizes this risk. The approach divides the dataset into $K$ folds. For each fold, we optimize the DRO model on
$K-1$ folds and evaluate the solution's performance on the remaining fold, repeating this process for all folds.  Specifically, for each candidate value of $\epsilon$, the model in problem~\eqref{eq:DRO} is solved using the training data $\hat{\mathbb P}_{-k}^K$ from all folds except the $k$-th fold to determine $\z^*(\hat{\mathbb P}_{-k}^K,\epsilon)$ which is then evaluated on the validation data to obtain $\rho_{\vecxi\sim\hat{\mathbb P}_k^K}(\ell(\z^*(\hat{\mathbb P}_{-k}^K,\epsilon), \vecxi))$, where $\hat{\mathbb P}_k^K$ denotes the empirical distribution of scenarios in fold $k$. The resulting estimator for a given radius $\epsilon$ is then given by $\rho_{k\sim U(K)}(\rho_{\vecxi\sim\hat{\mathbb P}_k^K}(\ell(\z^*(\hat{\mathbb P}_{-k}^K,\epsilon), \vecxi)))$, where $U(K)$ is the  uniform distribution over the set $\{1,2,\ldots,K\}$.  Since the goal is to minimize  risk, we choose  $\epsilon$ that minimizes the validation risk, i.e., $\epsilon^* = \argmin_{\epsilon\in\mathcal E}\rho_{k\sim U(K)}(\rho_{\vecxi\sim \hat{\mathbb P}_k^K}(\ell(\z^*(\hat{\mathbb P}_{-k}^K,\epsilon), \vecxi)))$. 
However, for each value of radius $\epsilon$ and choice of $K$, the following proposition shows that  the  entropic risk estimator     $\rho_{k\sim U(K)}(\rho_{\vecxi\sim\hat{\mathbb P}_k^K}(\ell(\z^*(\hat{\mathbb P}_{-k}^K,\epsilon), \vecxi)))$   based on the K-fold CV, underestimates the entropic risk of the policy constructed using $N(1-\frac{1}{K})$ data points.%
\begin{proposition}\label{prop:k-fold}
    Given $\epsilon\in\mathcal E$, 
     \begin{equation}\label{Kfold}
        \Expect[\rho_{k\sim U(K)}(\rho_{\vecxi\sim \hat{\mathbb P}_k^K}(\ell(\z^*(\hat{\mathbb P}_{-k}^K,\epsilon), \vecxi)))]<\rho_{\mathbb P}(\ell(\z^*(\hat{\mathbb P}_{N(1-\frac{1}{K})},\epsilon), \vecxi)).
         \end{equation} 
\end{proposition}
The proof of the above proposition follows from Jensen's inequality and tower property of entropic risk measure which states that $\rho(\bzeta)=\rho(\rho(\bzeta|\bzeta'))$ for random variables $\bzeta', \bzeta$. Note that this property is satisfied only by  entropic risk measure  in the family of  law-invariant risk measures \citep{Kupper_Schachermayer_2009}.
Notice that for large values of $K<N$, $N(1-\frac{1}{K})$ approaches $N$, thus the right-hand-side of \eqref{Kfold} mimics the performance of the  solution that uses all $N$ data points, that is, $\rho_{\mathbb P}(\ell(\z^*(\hat{\mathbb P}_{N},\epsilon), \vecxi))$.
To mitigate the underestimation of the entropic risk, we propose using the bias-aware bootstrap procedure described in Algorithm~\ref{alg:bootstrap_with_GMM} together with Algorithm \ref{alg:fittingGMM}, where GMM fitting in step \ref{algstep:fit}  either  uses maximum likelihood estimation, or Algorithm~\ref{alg:riskmatching} based on entropic risk matching,  or  Algorithm~\ref{alg:fitGMMevt} based on extreme value theory.  Algorithm~\ref{alg:radiustune} describes our proposed approach for selecting the optimal $\epsilon$, with Algorithm~\ref{alg:KFold} in Appendix \ref{sec:appenddro} describing the $K$-fold CV step. One can recover the traditional biased CV  procedure by setting $\delta = 0$ in line \ref{corrct:radius} of Algorithm~\ref{alg:radiustune}.

As we will see in the following sections, the solution based on our proposed approach significantly outperforms traditional CV procedure. 

\begin{algorithm}[H]
\caption{Radius selection for DRO \label{alg:radiustune}}
\begin{algorithmic}[1]
\Function{RadiusTuning}{$\Dist_N , K, M$}
    \For{$\epsilon \in \mathcal{E}$}
       \State $\mathcal{S}, \hat{\rho} \gets \text{K-foldCV}(K, \Dist_N, \epsilon)$ 
            \State \fbox{$\delta \gets \text{BootstrapBiasCorrection}(\mathcal{S}, M)$} \Comment{ Algorithm~\ref{alg:bootstrap_with_GMM} }\label{corrct:radius}
            \State $\rho(\epsilon) \gets \hat{\rho}+ \delta$ 
    \EndFor
    \State $\epsilon^* \gets \arg\min_{\epsilon \in \mathcal{E}} \rho(\epsilon)$ \label{alg:step:opt_radius}
\EndFunction
\end{algorithmic}
\end{algorithm}

\section{Distributionally robust insurance policy}\label{sec:DRIP}
The US National Flood Insurance Program (NFIP) provides flood coverage at subsidized premiums but faces significant challenges due to the large, correlated losses it insures against. These losses often result in claims exceeding the cumulative premiums collected over time \citep{Marcoux_H_Wagner_2023}. Consequently, the NFIP currently operates with a deficit exceeding \$20 billion and is compelled to consider raising premiums \citep{Marcoux_H_Wagner_2023}. However, higher premiums often deter households from purchasing coverage. This reluctance stems from how individuals perceive risk, which is frequently shaped by empirical losses rather than statistical estimates \citep{Kousky_Cooke_2012}. As a result, households tend to underestimate the risks associated with rare events. Demand for insurance, therefore, typically spikes only after catastrophic disasters \citep{Gallagher_2014}. In other words, individuals who have not experienced a catastrophic flood event are more likely to underestimate the associated risks. Surveys indicate that people exposed to flood risk but without firsthand experience of similar disasters often exhibit overly optimistic views about the threats posed by climate change. This optimism has been linked to houses in high flood-risk areas being overvalued by 6–9\% \citep{Bakkensen_Barrage_2021}. Furthermore, NFIP premium \revised{subsidy} reductions, combined with advances in risk estimation and flood risk mapping, have been shown to decrease house prices in high-risk areas \citep{Hino_Burke_2021}. Such behavioral responses are not unique to flood insurance markets. \cite{Herrnstadt_Sweeney_2024} use a difference-in-differences method to show that the prices of houses within 500 feet of a gas pipeline in the Bay Area dropped by \$383 per household following the deadly 2010 pipeline explosion in San Francisco. Moreover, residents’ perceptions of risk increased significantly above the empirical average for several years after the explosion. Interestingly, however, the prices of properties located 2,000 feet from the pipeline remained unaffected, despite being classified as high-risk. This discrepancy underscores the impact of firsthand experiences of catastrophic events on risk perception. \revised{This directly impacts the premium and coverage policies that are acceptable to a household. Consequently, an insurer's assessment of their risk exposure needs to take into account the households' risk perception.}

We consider an insurance pricing problem, with one risk-averse insurer and $M$ representative risk-averse households. The proposed model can account for correlated losses and asymmetry in the perception of risk measured by the empirical loss distributions at the household level.
Let $\alpha_h$ denote the risk aversion of household $h$, and let $\alpha_0$ represent the insurer's risk aversion parameter. The uncertain loss faced by household $h$ is represented by $\xi_h$. The insurer offers a policy $(z_h, \pi_h)$ to household $h$, where the indemnity function $z_h \xi_h$ specifies the coverage provided for their loss $\xi_h$, and $\pi_h$ is the corresponding premium paid by household $h$. Consequently, the net loss faced by household $h$ under this policy is given by $ \pi_h+(1-z_h)\xi_h$. Let $\hat{\Prob}_{h, N}$ be the empirical distribution of losses faced by household $h$. The insurer's demand response model assumes that household $h$ will accept the policy $(z_h, \pi_h)$ if the empirical entropic risk with insurance is less than the entropic risk without insurance. This condition is expressed by the following constraint:
\begin{equation}\label{demand_response}\rho^{\alpha_h}_{\hat{\Prob}_{h,N}}\left(\pi_h+(1-z_h)\xi_h \right) \leq \rho^{\alpha_h}_{\hat{\Prob}_{h,N}}\left(\xi_h \right).
 \end{equation}

The insurer aims to minimize risk across the policies offered to all households. Accordingly, the insurer's true entropic risk is given by $\rho^{\alpha_0}_{\mathbb P}\left(\boldsymbol{z}^\top \bxi - \1^\top \bm\pi \right)$, where $\1$ is the vector of ones of the appropriate dimension. Since the true joint distribution, $\mathbb{P}$, of losses across all households with marginals $\mathbb P_h$, is unknown, the insurer replaces it with the empirical distribution $\hat{\Prob}_N$ and solves the following optimization problem to determine the policies offered to the $M$ households:
\begin{equation}\label{SAA_insurance}
\begin{array}{lll}
    \min \quad & \rho^{\alpha_0}_{\hat{\Prob}_N}\left( \boldsymbol{z}^\top \bxi- \1^\top \bm \pi \right)\\
    \text{s.t.} & \bm \pi\in\mathbb R_+^M,\, \bm z\in[0,1]^M\\
    &\rho^{\alpha_h}_{\hat{\Prob}_{h,N}}\left(\pi_h+(1-z_h)\xi_h\right) \leq \rho^{\alpha_h}_{\hat{\Prob}_{h,N}}\left( \xi_h \right) &\forall h \in [M].
    \end{array}
\end{equation}
As discussed in the previous section, decisions based on the empirical data can be optimistically biased. To address this, the insurer solves the following distributionally robust insurance pricing problem, which minimizes the worst-case entropic risk:
\begin{equation*}
\begin{array}{lll}
    \min &\displaystyle\sup_{\Q \in \mathcal{B}_{\infty}(\epsilon)} \; \rho^{\alpha_0}_{\Q}\left( \boldsymbol{z}^\top \bxi-\1^\top \bm \pi \right)\\
    \text{s.t.} & \bm \pi\in\mathbb R_+^M,\, \bm z\in[0,1]^M\\
    & \rho^{\alpha_h}_{\hat{\Prob}_{h,N}}\left(\pi_h+(1-z_h)\xi_h \right) \leq \rho^{\alpha_h}_{\hat{\Prob}_{h,N}}\left( \xi_h \right) &\forall h \in [M],
\end{array}
\end{equation*}
where $\Q$ lies in the type-$\infty$ Wasserstein ambiguity set given in \eqref{def:wass_infty}. Since the loss function is linear, it follows from  Corollary~\ref{Corollary:DRO} in Appendix~\ref{sec:appenddro}  that the   problem can be reformulated as the following regularized exponential cone program:
\begin{equation}\label{insurance_regularized}
\begin{array}{lll}
    \min & \rho^{\alpha_0}_{\hat{\Prob}_N}\left( \boldsymbol{z}^\top \bxi-\1^\top \bm \pi \right) + \epsilon\|\V{z}\|_{*}\\
    \text{s.t.}  & \bm \pi\in\mathbb R_+^M,\, \bm z\in[0,1]^M\\
    & \rho^{\alpha_h}_{\hat{\Prob}_{h,N}}\left( \pi_h+(1-z_h)\xi_h \right) \leq \rho^{\alpha_h}_{\hat{\Prob}_{h,N}}\left(\xi_h \right) &\forall h \in [M].
\end{array}
\end{equation}

The proposed setting is motivated by \cite{Bernard_Liu_2020}, who assumes that both the insurer and households are expected utility maximizers, with complete information on the true loss distributions for each household and their risk aversion parameters. Our approach relaxes the assumption of known loss distributions by providing only samples of the loss distribution to both the insurer and the households.  While \cite{Bernard_Liu_2020} impose some additional assumptions to analytically characterize pricing and coverage decisions under partially correlated risks, our data-driven method formulates the problem as a tractable exponential cone program \eqref{SAA_insurance} which needs to be solved numerically. Additionally, the insurer's robustified (DRO) problem \eqref{insurance_regularized} retains this tractability. 
At  optimality, the constraints hold with equality due to the monotonicity of entropic risk measure,  that is,   $\rho^{\alpha_h}_{\hat{\Prob}_{h,N}}\left( \pi_h+(1-z_h)\xi_h \right) = \rho^{\alpha_h}_{\hat{\Prob}_{h,N}}\left(\xi_h \right) $ which can be equivalently written as: 
\revised{\begin{equation}\label{policy}
\pi_h = \rho^{\alpha_h}_{\hat{\Prob}_{h,N}}\left(\xi_h \right) -\rho^{\alpha_h}_{\hat{\Prob}_{h,N}}\left((1-z_h)\xi_h \right) = \frac{1}{\alpha_h} \log \left(\frac{\Expect_{\hat{\Prob}_{h,N}}[\exp(\alpha_h \xi_h)]}{\Expect_{\hat{\Prob}_{h,N}}[\exp(\alpha_h (1-z_h)\xi_h)]} \right).
\end{equation}}
The demand response model \eqref{demand_response}, which links premiums to coverage, accounts for the asymmetry in risk perception between households, who rely on $\hat{\mathbb{P}}_{h, N}$, and the insurer, who uses $\hat{\mathbb{P}}_{N}$. This flexible framework adapts to households' evolving responses as additional information is incorporated through observed loss events, thereby updating the empirical distributions. 
It is worth noting that our study could easily be adapted to
accommodate alternative demand response models as long as the insurer's valuations of the insurance is a concave function of coverage. Moreover, the regulatory constraints enforcing minimum coverage requirements for each household could easily be integrated.

\section{Numerical Experiments}\label{sec:num_ins_price}

The numerical experiments conducted in this section demonstrate the effectiveness of our proposed distributionally robust insurance pricing model under various conditions. Our main goal is to evaluate how different calibration methods for the radius $\epsilon$ influence the insurer's out-of-sample entropic risk and investigate the structure of the optimal policies $(\bm z,\bm \pi)$ offered to the households. Each household’s loss follows a Gamma distribution $ \Gamma(\kappa_h, \lambda_h)$, with shape $\kappa_h$ and scale $\lambda_h$ parameters specific to each household \( h \in [M] \). 
The correlation among the losses of different households is modeled using a Gaussian copula, with $\Sigma = r \mathbf{1}\mathbf{1}^\top + (1 - r) I$, 
where $r$ controls the amount of correlation among the different households, \( \mathbf{1} \) is a vector of all ones, and \( I \) is the identity matrix, \revised{see \cite{genest2007metaelliptical} and \cite{shi2020regression} for a detailed discussion on flood risk modeling using copulas and gamma distributions.}

In the following experiments, we consider $M = 5$ households with risk aversion parameters $\alpha_1, \alpha_2, \alpha_3, \alpha_4, \alpha_5 = 2.9, 2.7, 2.5, 2.3, 2.1$ and set the insurer's risk aversion parameter to $\alpha_0=2$. In  Algorithm~\ref{alg:radiustune}, we select the radius of the ambiguity set $\epsilon$ from set $\mathcal E$ which contains twenty equally-spaced values in the interval $[0, 6]$. 
To generate an instance, we sample $N$ loss scenarios for each household and evaluate the out-of-sample performance by generating \revised{$10^7$} i.i.d.\ data points (test data). This procedure is repeated over $100$ instances to obtain the statistics presented in the subsequent sections.

We consider five different calibration methods. Models \model{BS-Match} and \model{BS-EVT} use the $5$-fold CV procedure in
Algorithm~\ref{alg:radiustune} together with Algorithm~\ref{alg:riskmatching} and   Algorithm~\ref{alg:fitGMMevt} within the bootstrapping procedure, respectively, i.e., $K = 5$. \model{CV} model corresponds to the traditional $5$-fold CV where we set $\delta = 0$ in step \ref{corrct:radius} in Algorithm~\ref{alg:radiustune}.  The model labeled as \model{Oracle} uses the test data \revised{to calibrate} the radius $\epsilon$ at the validation step. Finally, model \model{SAA} solves problem~\eqref{SAA_insurance} and does not involve any calibration.

All experiments described in the following sections were conducted in Python. The MOSEK 10.1 solver was used to solve exponential cone programs, while the entropic risk matching was performed on a GPU using the POT library \citep{Flamary_Courty_2021}.
\subsection{Case with mild correlation}
\begin{figure}[htbp]
    \centering
    \begin{subfigure}[b]{0.49\linewidth}
        \centering
        \includegraphics[width=0.85\linewidth]{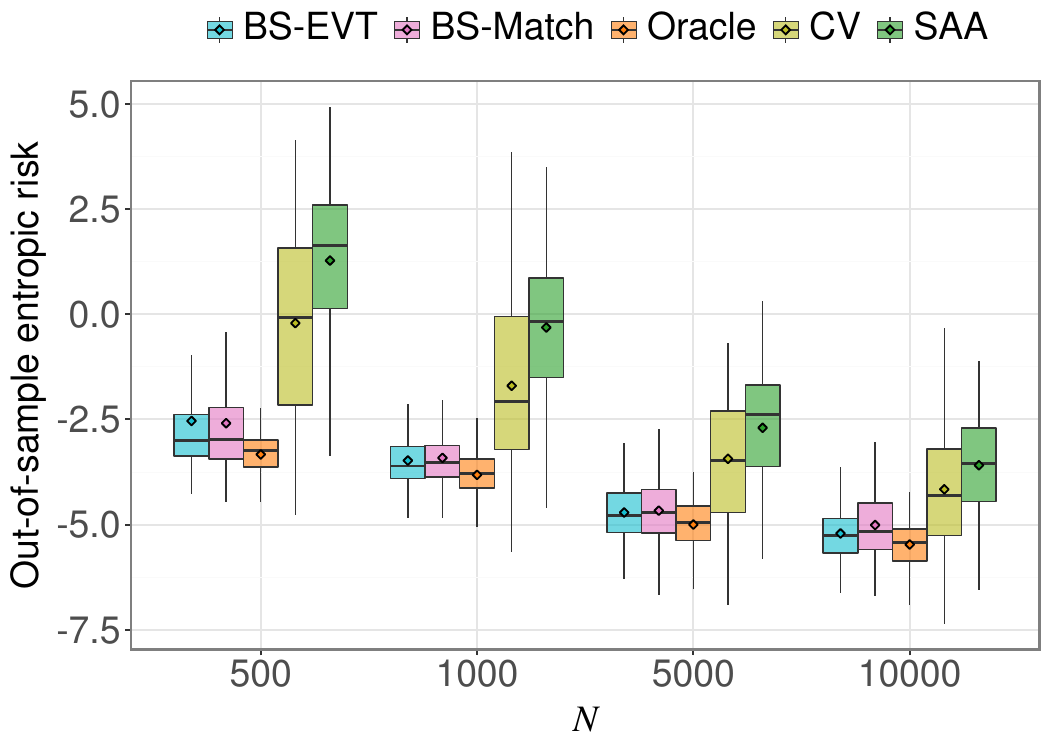}
        \caption{Insurer’s out-of-sample entropic risk.}
        \label{fig:risk_with_N}
    \end{subfigure}
    \hfill
    \begin{subfigure}[b]{0.49\linewidth}
        \centering
        \includegraphics[width=0.85\linewidth]{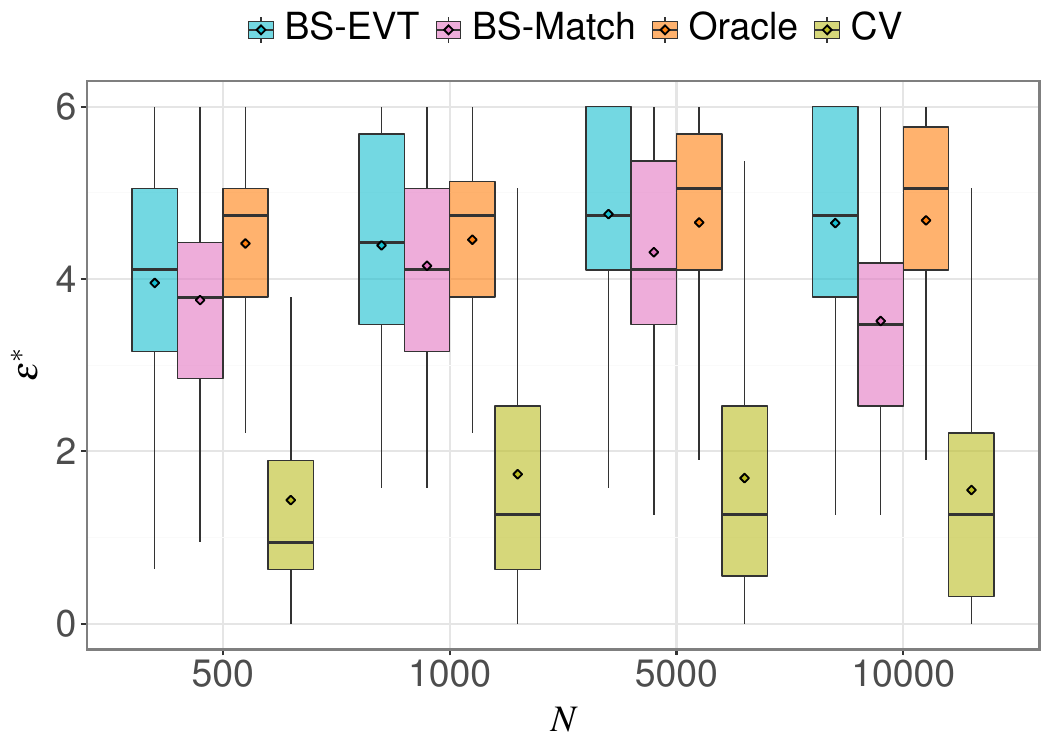}
        \caption{
        Optimal radius $\epsilon^*$ selected by different methods.}
        \label{fig:radius_with_N}
    \end{subfigure}
    \caption{Comparison of the effects of training sample size $N$ on out-of-sample entropic risk (left) and optimal radius $\epsilon^*$ (right). Boxplots present the statistics after $100$ resampling of datasets, and diamonds \revised{present} the mean for each $N$.}
    \label{fig:combined_N}
\end{figure}

\begin{figure}[t]
    \centering
\includegraphics[width=0.75\linewidth]{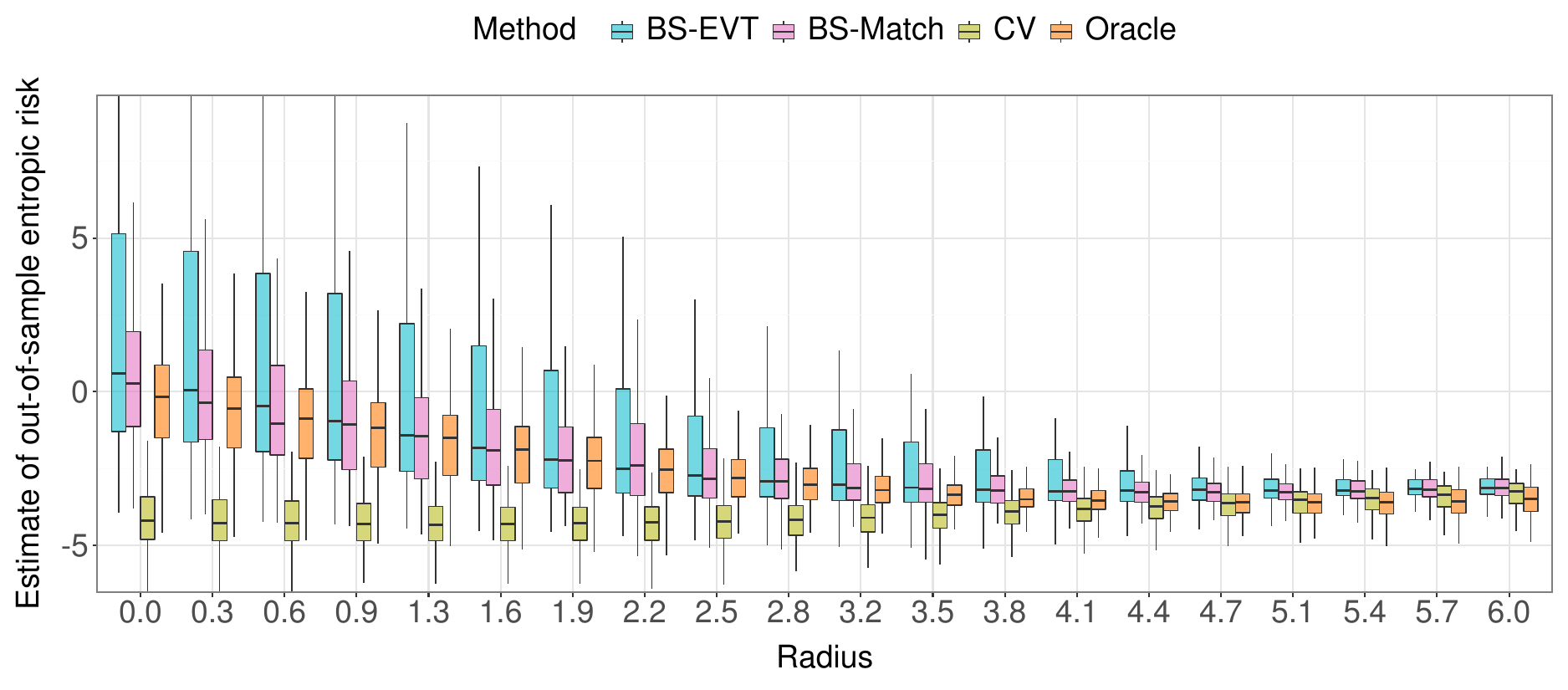}
    \caption{Statistics of entropic risk estimators for different radius and $N=1000$ after $100$ resampling of datasets} 
    \label{fig:est_risk_with_eps_1000}
\end{figure}

In the first experiment, we examine the effect of sample size on the out-of-sample risk observed by the insurer.
In the base scenario, all households have a common Gamma-distributed marginal loss distribution, \revised{$\Gamma(\kappa,\theta)=\Gamma(10,0.45)$ (in thousands of dollars)}, with a correlation coefficient $r = 0.5$ \revised{implying a mean loss of about 4.5K on a 100K property, and a coefficient of variation $\text{CV}=1/\sqrt{\kappa}=0.32$ which is a low-dispersion baseline to represent a mild-loss regime \citep{NRCan_FFDEG_2025_web} for %
pluvial flood losses compared to coastal floods \citep{nelson2025pluvial}}. %
This configuration allows us to %
focus on the effects of sample size $N$ and correlation coefficient $r$ on the insurer’s decisions, controlling for variability from differing marginal distributions. Similar insights hold for cases with heterogeneous marginal distributions among households, discussed in Appendix~\ref{append:diff_marginal}. The results are summarized in Figures~\ref{fig:combined_N}-\ref{fig:premium_with_N_house}.  Figure~\ref{fig:risk_with_N} shows that both \model{BS-Match} and \model{BS-EVT} consistently outperform \model{CV} and \model{SAA} across different sample sizes, while achieving an out-of-sample entropic risk similar to the oracle-based calibration. This can be explained by looking at the ambiguity radius $\epsilon^*$ chosen by each method. In Figure~\ref{fig:radius_with_N}, we observe that \model{CV} typically chooses $\epsilon^*$ values that are significantly lower than the optimal choice, while the \model{BS-Match} and \model{BS-EVT} choices are closer to optimal. This discrepancy results from \model{CV}'s estimation procedure, which underestimates the true entropic risk for each $\epsilon$ (as discussed in Section~\ref{sec:dro}), leading to \revised{selecting} an overly optimistic $\epsilon^*$ in step \ref{alg:step:opt_radius} in Algorithm~\ref{alg:radiustune}. This can be seen in Figure \ref{fig:est_risk_with_eps_1000} where we plot the variation in the estimate of the out-of-sample entropic risk with the radius $\epsilon$ for each model with $N=1000$ (similar plots are obtained in Appendix \ref{append:sec:estimate_out} for $N\in \{500, 5000, 10000\}$). In contrast, \model{BS-Match} and \model{BS-EVT} better estimate the trend in the variation of true entropic risk with $\epsilon$, thus making a more informed choice for  $\epsilon^*$. 

\begin{figure}[t]
    \centering
    \begin{subfigure}[htbp]{0.49\linewidth}
        \centering
        \includegraphics[width=0.85\linewidth]{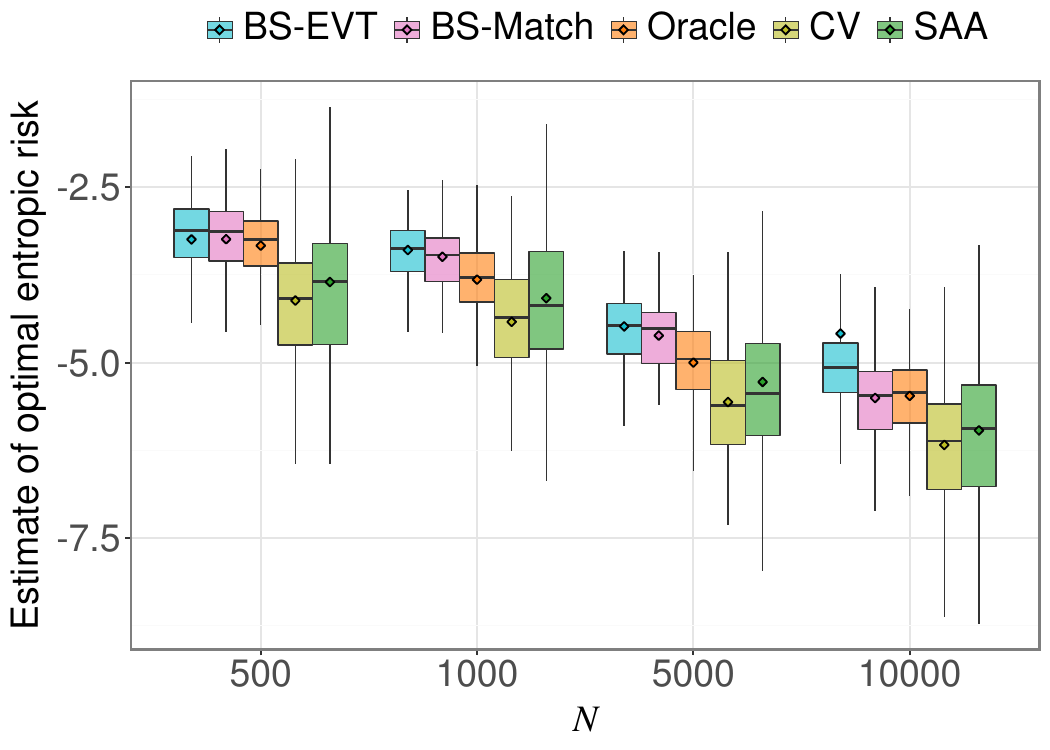}
        \caption{Insurer’s estimate of the optimal entropic risk.}
        \label{fig:est_risk_with_N}
    \end{subfigure}
    \hfill
    \begin{subfigure}[htbp]{0.49\linewidth}
        \centering
        \includegraphics[width=0.85\linewidth]{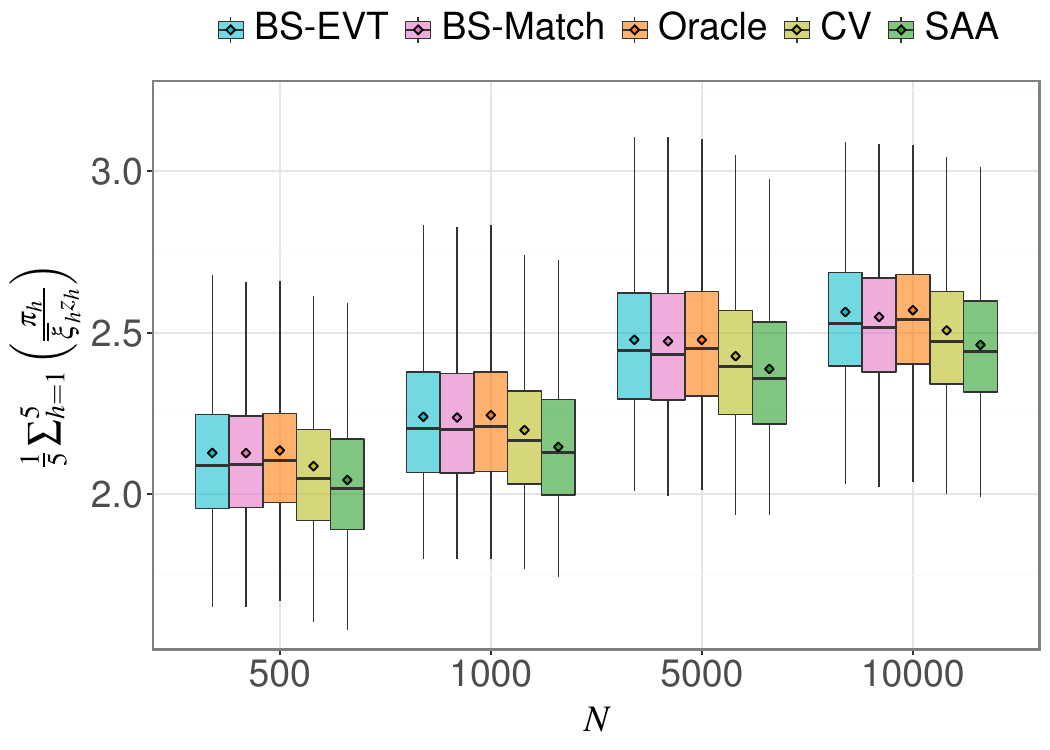}
        \caption{Average 
    optimal premium per unit of expected coverage.}
       \label{fig:premium_N}
    \end{subfigure}
    \caption{Comparison of the effects of number of training samples $N$ on insurer's estimate of optimal entropic risk and average premium per unit of expected coverage.}
    \label{fig:combined_N_est_premium}
\end{figure}

Figure~\ref{fig:est_risk_with_N} depicts the estimation of the optimal entropic risk computed by each method. The results have a similar interpretation as Figure~\ref{fig:project} in Section~\ref{benchmark_section}. 
We observe a significant underestimation of the true entropic risk by \model{CV} and \model{SAA}, while the estimates of the \model{BS-Match} and \model{BS-EVT} stay close to the estimates of the optimal calibration \model{Oracle}. We observe that as the sample size increases, the optimal risk decreases, and the same holds for each method’s estimate of the optimal risk.
This is because it is optimal for households to pay higher premiums due to the increase in their estimate of the risk of their respective loss.
  \begin{figure}[b]
    \centering
\includegraphics[width=0.6\linewidth]{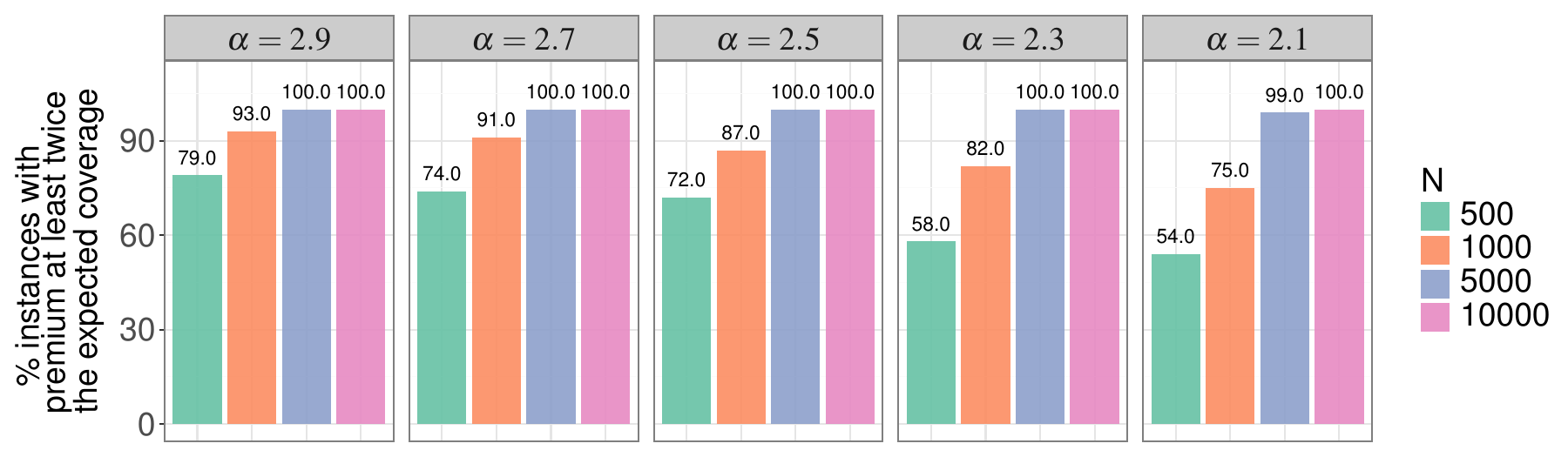}
    \caption{Effect of the number of training samples on the proportion of instances where premium exceeds twice the expected coverage for \model{BS-Match}. Households become less risk averse as we go from the left of the panel to the right one. }
    \label{fig:premium_with_N_house}
\end{figure}
Consequently, the insurer can charge higher premiums per unit expected coverage  as a function of $N$, see Figure \ref{fig:premium_N} 
where the average premium per unit expected coverage across the $5$ households is given $(1/5)\sum_{h=1}^5\pi_h/(\bar{\xi}_hz_h)$ with $\bar{\xi}_h$ as the expected loss of household $h$. We observe an increase in this ratio as $N$ increases, which can be explained by the observation that both the insurer and households become more capable of accurately estimating their true risk, thus enabling the insurer to extract higher premiums from households for the same coverage level. Moreover, \model{CV} and \model{SAA} charge a lower premium per unit of expected coverage compared to our proposed approaches because they underestimate the risk. Namely, both methods are overly optimistic regarding how to correct the estimation error due to sampling, effectively using an $\epsilon$ that is too low. 

To illustrate the variation in the households' willingness to pay for insurance as a function of the number of training samples, we use as a proxy the proportion of instances where the premium is at least twice the expected coverage. For \model{BS-Match}, Figure~\ref{fig:premium_with_N_house} illustrates that highly risk-averse households pay higher premiums per unit of expected coverage more frequently, even when the number of training samples is low. Additionally, as  $N$ increases, the proportion approaches $100\%$ for all households.

\begin{figure}[t]
    \centering
    \begin{subfigure}[b]{0.49\linewidth}
        \centering
        \includegraphics[width=0.85\linewidth]{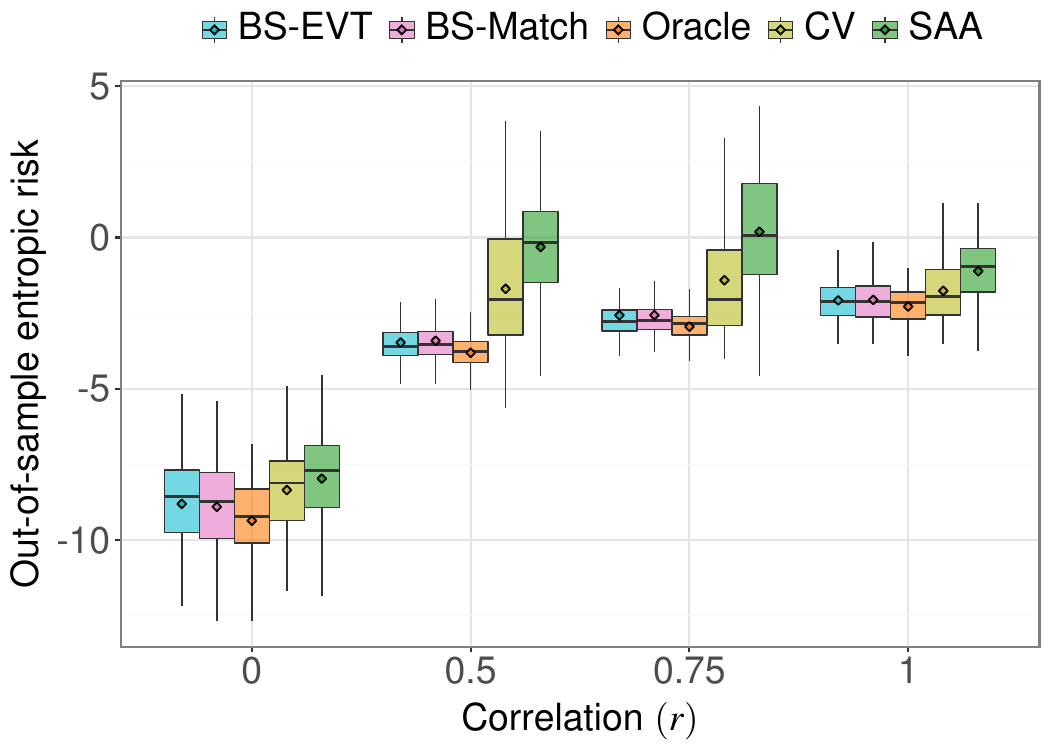}
        \caption{Insurer’s out-of-sample entropic risk.}
        \label{fig:risk_with_r}
    \end{subfigure}
    \hfill
    \begin{subfigure}[b]{0.49\linewidth}
        \centering
        \includegraphics[width=0.85\linewidth]{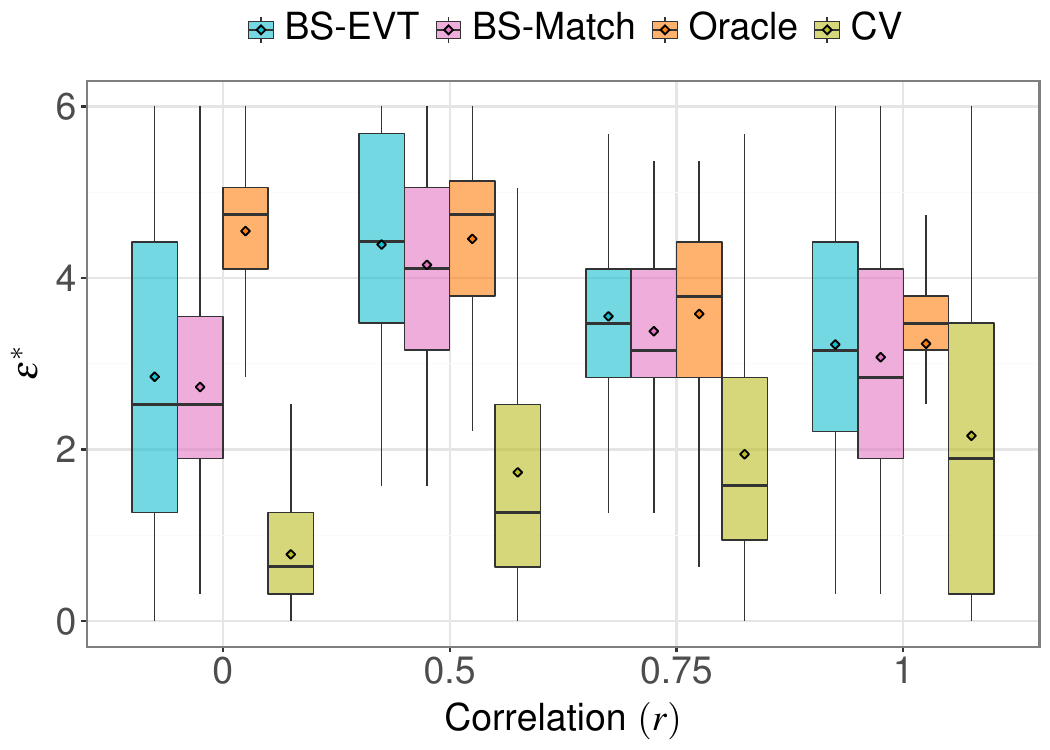}
        \caption{Optimal radius $\epsilon^*$ selected by different methods.}
        \label{fig:radius_with_r}
    \end{subfigure}
    \caption{Comparison of the effects of correlation coefficient $r$ on out-of-sample entropic risk (left) and optimal radius $\epsilon^*$ (right).}
    \label{fig:combined_r}
\end{figure}

\begin{figure}[htbp]
    \centering
    \begin{subfigure}[t]{0.49\linewidth}
        \centering
        \includegraphics[width=0.85\linewidth]{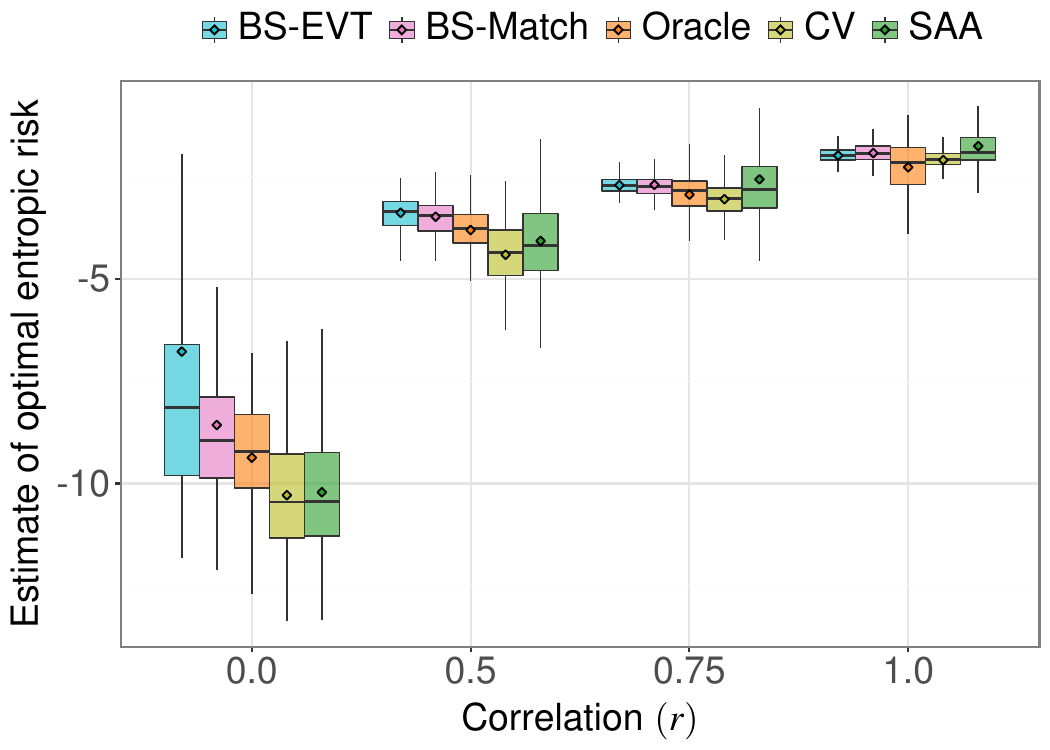}
        \caption{Insurer’s estimate of the optimal  entropic risk.}
        \label{fig:est_risk_with_r}
    \end{subfigure}
    \hfill
    \begin{subfigure}[t]{0.49\linewidth}
        \centering
        \includegraphics[width=0.85\linewidth]{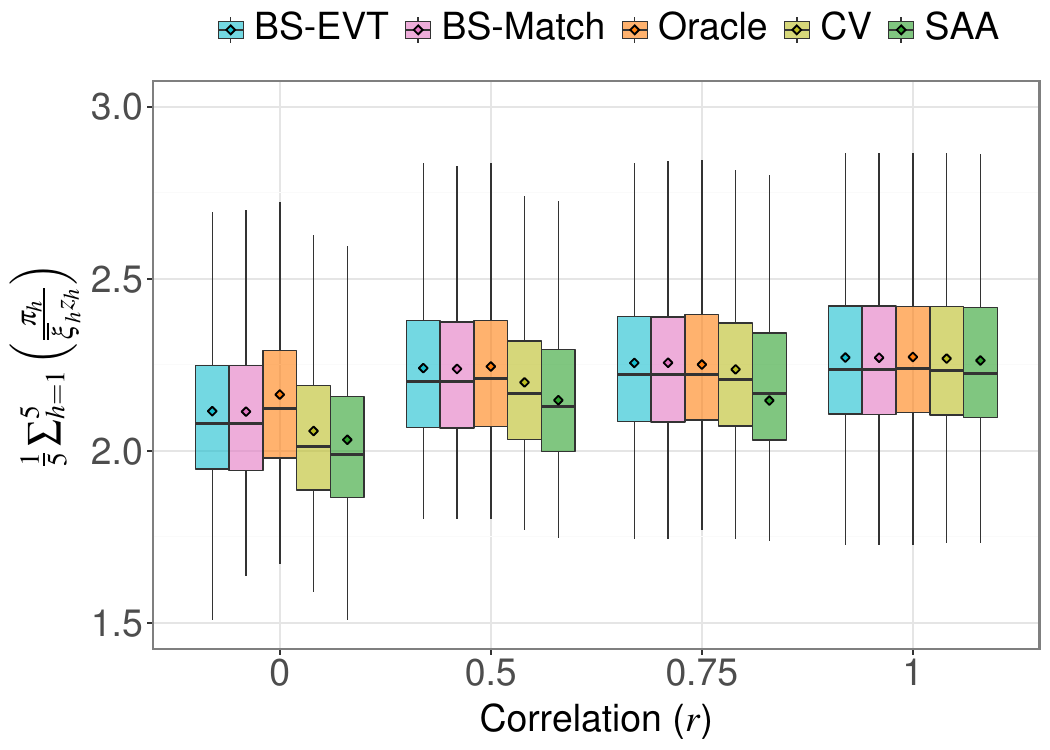}
        \caption{Average 
        optimal premium per unit of expected coverage.}
       \label{fig:premium_with_r}
    \end{subfigure}
    \caption{Comparison of the effects of correlation coefficient $r$ on estimate of optimal entropic risk and premium per unit expected coverage.}
    \label{fig:combined_r_est_premium}
\end{figure}

\begin{figure}[htbp]
    \centering
    \begin{subfigure}[t]{\linewidth}
        \centering
        \includegraphics[width=0.85\linewidth]{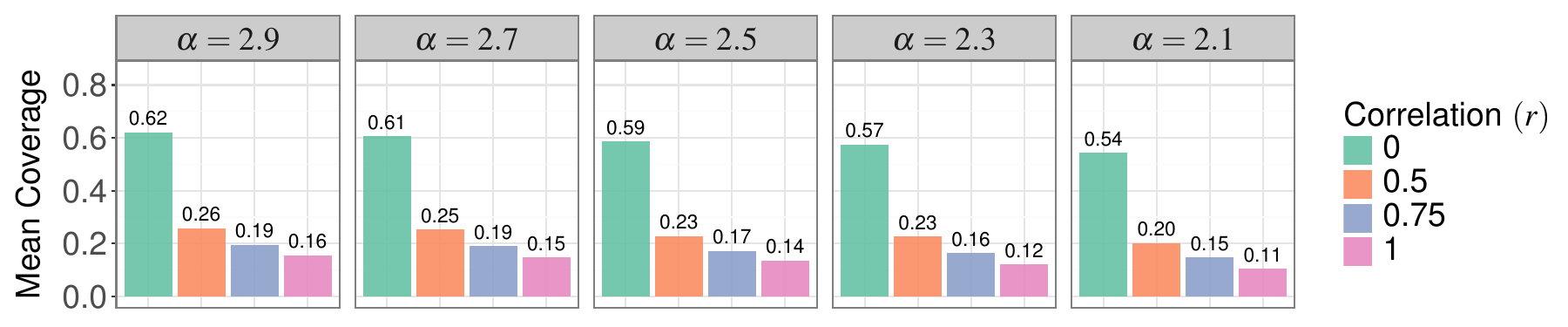}
        \caption{\model{BS-Match}.}
        \label{fig:cov_r_house}
    \end{subfigure}

    \vspace{0.5cm} %

    \begin{subfigure}[t]{\linewidth}
        \centering
        \includegraphics[width=0.85\linewidth]{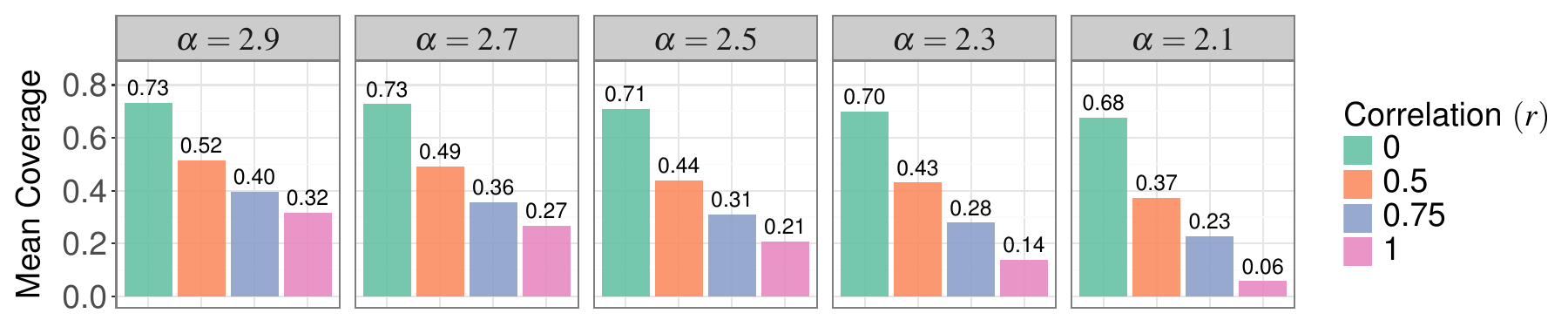}
        \caption{\model{SAA}. }
        \label{fig:cov_r_house_saa}
    \end{subfigure}

    \caption{Effect of correlation on the coverage proportion offered to households
        averaged across $100$ instances. Each panel represents a different household. Households become less risk-averse as we go from the left panel to the right one.}
    \label{fig:cov_comparison}
\end{figure}
\subsection{Effect of correlations}

In the second experiment, we fix the sample size at \( N = 1000 \). Similar to the previous experiment, we assume that all households share a common marginal loss distribution modeled by a Gamma distribution, \(\Gamma(10, 0.45)\). However, in this experiment, we vary the parameter \( r \), which controls the pairwise correlation of losses between households. The correlation coefficient ranges from 
$0$ (independent losses) to $1$ (comonotone losses). Figures~\ref{fig:combined_r}-\ref{fig:premium_with_r_house} summarize the results. 
First, note that similar insights about the effectiveness of our proposed approaches can also be observed in this setting.  Indeed, Figure~\ref{fig:radius_with_r} shows that  both \model{CV} and \model{SAA} are over-optimistic, leading them to select lower $\epsilon^*$ values than those chosen by our proposed approaches.
With regards to the behavior in terms of $r$, Figure \ref{fig:risk_with_r} demonstrates that the out-of-sample entropic risk initially increases with the correlation coefficient $r$, but eventually stabilizes. This trend is intuitive: higher correlation among households' losses means that extreme loss events are more likely to occur simultaneously, increasing the insurer's risk exposure. Figure \ref{fig:est_risk_with_r} shows the estimates of the optimal entropic risk produced by each model. 
\texttt{BS-EVT} and \texttt{BS-Match} overestimate the optimal entropic risk of the insurer compared to \model{Oracle}. %
Figure \ref{fig:premium_with_r} shows that the average optimal premium per unit of expected coverage also increases with the correlation coefficient. This reflects the insurer's response to higher risk by charging higher premiums to compensate for the increased likelihood of large, simultaneous payouts. However, there is a diminishing return effect; beyond a certain point, further increases in correlation do not lead to significantly higher premiums per unit of expected coverage. As the correlation between household losses increases, the benefits of risk pooling diminish, so the insurer reduces coverage levels significantly to reduce the risk exposure, see  Figure \ref{fig:cov_r_house} and \ref{fig:cov_r_house_saa}. While Figures  \ref{fig:cov_r_house} and \ref{fig:cov_r_house_saa} show that households receive more coverage with \texttt{SAA}  than \texttt{BS-Match}, this high coverage exposes the insurer to higher risks than the optimal coverage in the case of highly correlated losses.  Figure \ref{fig:premium_with_r_house} demonstrates that as the correlation $r$ across households increases, the proportion of instances with  premiums exceeding twice the expected coverage also increase, and this effect is more pronounced for more risk-averse households. The high risk-averse households secure greater coverage by paying high premiums to reduce their risk exposure.

\begin{figure}[htbp]
    \centering
\includegraphics[width=0.6\linewidth]{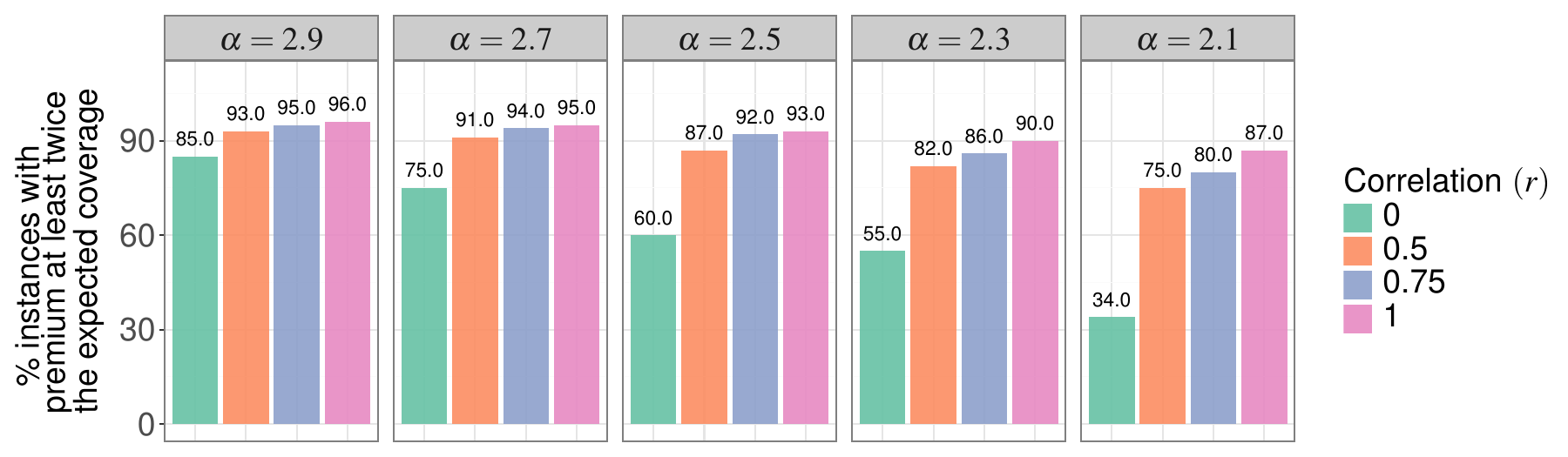}
    \caption{Effect of correlation $r$ on proportion of instances where premium exceeds two times the expected coverage. Each panel represents a different household.  Households become less risk averse as we go from the left panel to the right one.}
    \label{fig:premium_with_r_house}
\end{figure}

\section{\revised{Conclusions and Future Work}}\label{sec:conclusions}

This paper highlights a critical managerial challenge: existing estimators designed to correct the negative bias of empirical risk estimator can be ineffective for the widely-studied entropic risk measure in the finite-sample regime, leading to distorted decisions. In practice, this bias can push managers toward riskier actions by understating downside exposure, or toward unnecessarily conservative choices by overstating risk. Either outcome can negatively affect pricing, coverage design, and investment decisions in risk-sensitive applications. 

We  propose a new estimator that converges almost surely to the true entropic risk.  Under certain light-tailed assumptions on the loss distribution, our theoretical results show that the estimator eventually eliminates the negative finite-sample bias inherent in the empirical risk estimator as positions become riskier. From a managerial perspective, the resulting estimator provides the flexibility to choose a distribution that satisfies fairly mild conditions and ensures conservative risk assessment.  Our theoretical results show that one such family of distributions is a Gaussian mixture model (GMM) that offers the right balance: it eliminates the negative bias of the empirical estimator with an estimation error that is proportional to the variance of the loss, captures multimodal loss distributions, admits closed-form expressions for entropic risk, and enables efficient sampling for bias estimation through bootstrapping.
Even when we deviate from the light-tailed assumptions on the loss distributions, our estimators significantly mitigate the optimistic bias. %
This translates into decision-making with lower risk exposure without requiring precise knowledge of the true loss distribution.

We also demonstrate that optimistic bias persists when entropic risk is embedded within optimization models, such as those used for pricing and investment decisions. %
We illustrate the value of our bias-aware procedure through an insurance pricing application involving risk-averse insurers and households facing correlated losses. The results show that bias-aware DRO models yield premium and coverage policies with lower realized entropic risk than those obtained using conventional validation technique. From a policy and managerial standpoint, the findings underscore the importance of explicitly accounting for estimation uncertainty when pricing insurance against correlated losses %
such as those driven by climate change. %
Overall, our results suggest that accurate and conservative risk estimation is not merely a statistical concern, but a central managerial lever. By explicitly addressing finite-sample bias, decision makers can design pricing and investment policies that reduce risk exposure, especially in the presence of correlated losses. %

There are several avenues for extending this research. First, further work is needed to develop estimation procedures that can provably overestimate the entropic risk for loss distributions with subgaussian tails, for example, by modifying the fitting step in the GMM-based implementations of our three estimators, while preserving tractability and providing controlled overestimation guarantees.  Another direction is to extend the proposed procedure to account for both high variance and negative bias of the empirical risk estimator. A third direction is to investigate  how the proposed bias correction procedures can be adapted to reduce estimation bias for other convex risk measures, e.g., optimized certainty equivalent (OCE) risk measures and utility-based shortfall risk (UBSR) measures. %
Finally, exploring the effectiveness of these bias corrections in multi-stage settings, where entropic risk measure is %
widely used, such as control theory and reinforcement learning, would be valuable.  %

\section*{Acknowledgments}
The first author's research is supported by the NSERC, Canada, grant RGPIN-2024-05067. This research is supported by compute resources
provided by Calcul Quebec  and the Digital Research
Alliance of Canada. We thank Tianyu Wang for his helpful comments on an earlier version of the paper.

\setlength{\bibsep}{0pt}
\bibliographystyle{apalike}
\bibliography{biblio}

@article{Arrow_1963, title={Uncertainty and the welfare economics of medical care}, volume={53}, number={5}, journal={American Economic Review}, author={Arrow, Kenneth J.}, year={1963}, pages={941--973} }

@book{Arrow_1971, address={Chicago}, title={Essays in the Theory of Risk Bearing}, publisher={Markham Publishing Co.}, place={Chicago}, author={Arrow, Kenneth J.}, year={1971} }

@article{Bakkensen_Barrage_2021, title={{Going Underwater? Flood Risk Belief Heterogeneity and Coastal Home Price Dynamics}}, volume={35}, number={8}, journal={The Review of Financial Studies}, author={Bakkensen, Laura A and Barrage, Lint}, year={2022}, pages={3666--3709} }

@article{Bartl_Mendelson_2022, title={{On Monte-Carlo methods in convex stochastic optimization}}, volume={32}, number={4}, journal={The Annals of Applied Probability}, author={Bartl, Daniel and Mendelson, Shahar}, year={2022}, pages={3146--3198} }

@article{Bäuerle_Jaśkiewicz_2024, title={{Markov decision processes with risk-sensitive criteria: an overview}}, volume={99}, number={1}, journal={Mathematical Methods of Operations Research}, publisher={Springer}, author={Bäuerle, Nicole and Jaśkiewicz, Anna}, year={2024}, pages={141--178} }

@inproceedings{Beirami_Razaviyayn_2017, title={On Optimal Generalizability in Parametric Learning}, volume={30}, booktitle={Advances in Neural Information Processing Systems},  author={Beirami, Ahmad and Razaviyayn, Meisam and Shahrampour, Shahin and Tarokh, Vahid}, editor={I. Guyon and U. Von Luxburg and S. Bengio and H. Wallach and R. Fergus and S. Vishwanathan and R. Garnett}, year={2017}, pages={3456--3466}, publisher={Curran Associates, Inc.} }

@article{Ben-Tal_denHertog_2015, title={Deriving robust counterparts of nonlinear uncertain inequalities}, volume={149}, number={1}, journal={Mathematical Programming}, author={Ben-Tal, Aharon and den Hertog, Dick and Vial, Jean-Philippe}, year={2015}, pages={265--299} }

@article{Ben-Tal_Teboulle_1986, title={Expected Utility, Penalty Functions, and Duality in Stochastic Nonlinear Programming}, volume={32}, number={11}, journal={Management Science}, publisher={INFORMS}, author={Ben-Tal, Aharon and Teboulle, Marc}, year={1986}, pages={1445--1466} }

@article{Bernard_Liu_2020, title={Optimal insurance in the presence of multiple policyholders}, volume={180}, journal={Journal of Economic Behavior \& Organization}, author={Bernard, Carole and Liu, Fangda and Vanduffel, Steven}, year={2020}, pages={638--656} }

@article{Bernard_Tian_2010, title={Insurance Market Effects of Risk Management Metrics}, volume={35}, number={1}, journal={The Geneva Risk and Insurance Review}, author={Bernard, Carole and Tian, Weidong}, year={2010}, pages={47--80} }

@article{Bertsimas_Shtern_2023, title={A Data-Driven Approach to Multistage Stochastic Linear Optimization}, volume={69}, number={1}, journal={Management Science}, author={Bertsimas, Dimitris and Shtern, Shimrit and Sturt, Bradley}, year={2023}, pages={51--74} }

@inproceedings{Bousquet_Elisseeff_2000, title={Algorithmic stability and generalization performance}, volume={13}, booktitle={Advances in Neural Information Processing Systems}, author={Bousquet, Olivier and Elisseeff, André}, year={2000}, editor={Leen, T and Dietterich, T and Tresp, V}, pages={196--202}, publisher={MIT Press, Cambridge, MA} }

@article{Brandtner_Kürsten_2018, title={Entropic risk measures and their comparative statics in portfolio selection: Coherence vs. convexity}, volume={264}, number={2}, journal={European Journal of Operational Research}, publisher={Elsevier}, author={Brandtner, Mario and Kürsten, Wolfgang and Rischau, Robert}, year={2018}, pages={707--716} }

@article{Cai_Tan_2008, title={{Optimal reinsurance under VaR and CTE risk measures}}, volume={43}, number={1}, journal={Insurance: Mathematics and Economics}, author={Cai, Jun and Tan, Ken Seng and Weng, Chengguo and Zhang, Yi}, year={2008}, pages={185--196} }

@article{Catoni_2012, title={{Challenging the empirical mean and empirical variance: A deviation study}}, volume={48}, number={4}, journal={Annales de l’Institut Henri Poincaré, Probabilités et Statistiques}, author={Catoni, Olivier}, year={2012}, pages={1148--1185} }

@article{Chan_Van_Parys_2024, title={From Distributional Robustness to Robust Statistics: A Confidence Sets Perspective}, journal={arXiv preprint arXiv:2410.14008}, author={Chan, Gabriel and Van Parys, Bart and Bennouna, Amine}, year={2024}, note={Accessed on Dec 19, 2024. \url{https://arxiv.org/abs/2410.14008}} }

@article{Chen_He_2024, title={An exponential cone programming approach for managing electric vehicle charging}, volume={72}, number={5}, journal={Operations Research}, publisher={INFORMS}, author={Chen, Li and He, Long and Zhou, Yangfang}, year={2024}, pages={2215--2240} }

@article{meyer1987two,
  title={Two-moment decision models and expected utility maximization},
  author={Meyer, Jack},
  journal={The American economic review},
  pages={421--430},
  year={1987},
  publisher={JSTOR}
}

@article{jin2016local,
  title={Local maxima in the likelihood of gaussian mixture models: Structural results and algorithmic consequences},
  author={Jin, Chi and Zhang, Yuchen and Balakrishnan, Sivaraman and Wainwright, Martin J and Jordan, Michael I},
  journal={Advances in neural information processing systems},
  volume={29},
  year={2016}
}

@article{landsman2012translation,
  title={Translation-invariant and positive-homogeneous risk measures and optimal portfolio management in the presence of a riskless component},
  author={Landsman, Zinoviy and Makov, Udi},
  journal={Insurance: Mathematics and Economics},
  volume={50},
  number={1},
  pages={94--98},
  year={2012},
  publisher={Elsevier}
}

@unpublished{NRCan_FFDEG_2025_web,
  author       = {{Natural Resources Canada}},
  title        = {Federal Flood Damage Estimation Guidelines for Buildings and Infrastructure},
  institution    = {Government of Canada},
year={2025},
note={Accessed on Dec 28, 2025. \url{https://natural-resources.canada.ca/science-data/science-research/natural-hazards/flood-mapping/federal-flood-damage-estimation-guidelines-buildings-infrastructure}}
}

@article{shi2020regression,
  title={Regression for copula-linked compound distributions with applications in modeling aggregate insurance claims},
  author={Shi, Peng and Zhao, Zifeng},
  journal={The Annals of Applied Statistics},
  volume={14},
  number={1},
  pages={357--380},
  year={2020},
  publisher={JSTOR}
}

@article{schuhmacher2021justifying,
  title={Justifying mean-variance portfolio selection when asset returns are skewed},
  author={Schuhmacher, Frank and Kohrs, Hendrik and Auer, Benjamin R},
  journal={Management Science},
  volume={67},
  number={12},
  pages={7812--7824},
  year={2021},
  publisher={INFORMS}
}

@incollection{fang1990spherically,
  title={Spherically and elliptically symmetric distributions},
  author={Fang, Kai-Tai and Kotz, Samuel and Ng, Kai Wang},
  booktitle={Symmetric Multivariate and Related Distributions},
  pages={26--68},
  year={1990},
  publisher={Springer}
}

@article{Merton1969,
  author  = {Merton, Robert C.},
  title   = {Lifetime Portfolio Selection under Uncertainty: The Continuous-Time Case},
  journal = {The Review of Economics and Statistics},
  year    = {1969},
  volume  = {51},
  number  = {3},
  pages   = {247--257}
}

@article{genest2007metaelliptical,
  title={Metaelliptical copulas and their use in frequency analysis of multivariate hydrological data},
  author={Genest, Christian and Favre, A-C and B{\'e}liveau, Julie and Jacques, Christine},
  journal={Water Resources Research},
  volume={43},
  number={9},
  year={2007},
  publisher={Wiley Online Library}
}

@article{Merton1971,
  author  = {Merton, Robert C.},
  title   = {Optimum Consumption and Portfolio Rules in a Continuous-Time Model},
  journal = {Journal of Economic Theory},
  year    = {1971},
  volume  = {3},
  number  = {4},
  pages   = {373--413}
}

@article{Kirkwood2004,
  author  = {Kirkwood, Craig W.},
  title   = {Approximating Risk Aversion in Decision Analysis Applications},
  journal = {Decision Analysis},
  year    = {2004},
  volume  = {1},
  number  = {1},
  pages   = {51--67},
  doi     = {10.1287/deca.1030.0007}
}

@article{Samuelson1969,
  author  = {Samuelson, Paul A.},
  title   = {Lifetime Portfolio Selection by Dynamic Stochastic Programming},
  journal = {The Review of Economics and Statistics},
  year    = {1969},
  volume  = {51},
  number  = {3},
  pages   = {239--246},
  doi     = {10.2307/1926559}
}

@book{Cormen2009,
  title     = {Introduction to Algorithms},
  author    = {Cormen, Thomas H. and Leiserson, Charles E. and Rivest, Ronald L. and Stein, Clifford},
  publisher = {MIT Press},
  address   = {Cambridge, MA},
  year      = {2009},
  edition   = {3}
}

@article{eliashberg1978role,
  title={The role of attitude toward risk in strictly competitive decision-making situations},
  author={Eliashberg, Jehoshua and Winkler, Robert L},
  journal={Management Science},
  volume={24},
  number={12},
  pages={1231--1241},
  year={1978},
  publisher={INFORMS}
}

@article{Chen_Ramachandra_2024, title={{Robust Data-Driven CARA Optimization}}, journal={arXiv preprint arXiv:2107.06714v3}, author={Chen, Li and Ramachandra, Arjun and Rujeerapaiboon, Napat and Sim, Melvyn}, year={2024}, note={Accessed on Dec 19, 2024. \url{https://arxiv.org/pdf/2107.06714}} }

@article{Chen_Sim_2024, title={{Robust CARA optimization}}, journal={Operations Research}, publisher={INFORMS}, author={Chen, Li and Sim, Melvyn}, year={2024}, note={Forthcoming} }

@article{Cheung_Sung_2014, title={Risk-Minimizing Reinsurance Protection For Multivariate Risks}, volume={81}, number={1}, journal={The Journal of Risk and Insurance}, publisher={[American Risk and Insurance Association, Wiley]}, author={Cheung, K. C. and Sung, K. C. J. and Yam, S. C. P.}, year={2014}, pages={219--236} }

@article{Choi_Ruszczyński_2011, title={A multi-product risk-averse newsvendor with exponential utility function}, volume={214}, number={1}, journal={European Journal of Operational Research}, author={Choi, Sungyong and Ruszczyński, Andrzej}, year={2011}, pages={78--84} }

@article{popoviciu1935equations,
  title={Sur les {\'e}quations alg{\'e}briques ayant toutes leurs racines r{\'e}elles},
  author={Popoviciu, Tiberiu},
  journal={Mathematica},
  volume={9},
  number={129--145},
  pages={20},
  year={1935}
}

@article{BaruchZhang2022,
  author  = {Shmuel Baruch and Xiaodi Zhang},
  title   = {The Distortion in Prices Due to Passive Investing},
  journal = {Management Science},
  year    = {2022},
  volume  = {68},
  number  = {8},
  pages   = {6219--6234},
  doi     = {10.1287/mnsc.2021.4114}
}

@article{HuangJuXing2023,
  author  = {Yu Huang and Nengjiu Ju and Hao Xing},
  title   = {Performance Evaluation, Managerial Hedging, and Contract Termination},
  journal = {Management Science},
  year    = {2023},
  volume  = {69},
  number  = {8},
  pages   = {4953--4971},
  doi     = {10.1287/mnsc.2022.4533}
}

@article{Adcock2014MeanVarianceSkew,
  author       = {Adcock, Chris J.},
  title        = {Mean-variance-skewness efficient surfaces, Stein's lemma and the multivariate extended skew-normal distribution},
  journal      = {European Journal of Operational Research},
  year         = {2014},
  volume       = {234},
  number       = {2},
  pages        = {392--401},
  issn         = {0377-2217},
  doi          = {10.1016/j.ejor.2013.09.024}
}

@article{meyer1992sufficient,
  title={Sufficient conditions for expected utility to imply mean-standard deviation rankings: empirical evidence concerning the location and scale condition},
  author={Meyer, Jack and Rasche, Robert H},
  journal={The Economic Journal},
  volume={102},
  number={410},
  pages={91--106},
  year={1992},
  publisher={Oxford University Press, Oxford, UK}
}

@article{embrechts1997modelling,
  title={Modelling Extremal Events},
  author={Embrechts, P and Kl{\"u}ppelberg, C and Mikosch, T},
  journal={Applications of mathematics},
  year={1997}
}

@book{Feller1971,
  author    = {William Feller},
  title     = {An Introduction to Probability Theory and Its Applications, Vol. II},
  edition   = {2nd},
  publisher = {Wiley},
  year      = {1971}
}

@book{Vershynin2025,
title={High-Dimensional Probability
An Introduction with Applications in Data Science}, author={Vershynin, Roman}, publisher={Cambridge University Press}, address={Cambridge, UK}, year={2018}}

@article{Horowitz_2019, title={Bootstrap Methods in Econometrics}, volume={11},  journal={Annual Review of Economics}, publisher={Annual Reviews}, author={Horowitz, Joel L.}, year={2019}, pages={193--224}}

@book{de_Haan_Ferreira_2006, title={Extreme Value Theory: An Introduction}, publisher={Springer, Berlin}, author={de Haan, Laurens and Ferreira, Ana}, year={2006} }

@article{Delage_Ye_2010, title={{Distributionally Robust Optimization Under Moment Uncertainty with Application to Data-Driven Problems}}, volume={58}, number={3}, journal={Operations Research}, author={Delage, Erick and Ye, Yinyu}, year={2010}, pages={595--612} }

@article{Dempster_Laird_1977, title={{Maximum Likelihood from Incomplete Data via the EM Algorithm}}, volume={39}, number={1}, journal={Journal of the Royal Statistical Society: Series B (Methodological)}, publisher={Wiley Online Library}, author={Dempster, Arthur P and Laird, Nan M and Rubin, Donald B}, year={1977}, pages={1--22} }

@incollection{Donti_Amos_2017, title={Task-based end-to-end model learning in stochastic optimization}, volume={30}, booktitle={Advances in Neural Information Processing Systems}, publisher={Curran Associates, Inc.}, author={Donti, Priya and Amos, Brandon and Kolter, J Zico}, editor={Guyon, I and von Luxburg, U and
Bengio, S and Wallach, H and Fergus, R and Vishwanathan,  S and Garnett, R}, year={2017}, pages={5484--5494} }

@article{Elmachtoub_Grigas_2021, title={Smart “predict, then optimize”}, journal={Management Science}, publisher={INFORMS}, author={Elmachtoub, Adam N and Grigas, Paul}, year={2021}, volume = {68}, number={1}, pages={9--26} }

@inproceedings{Fei_Yang_2021, title={{Exponential Bellman Equation and Improved Regret Bounds for Risk-Sensitive Reinforcement Learning}}, booktitle={Advances in Neural Information Processing Systems}, author={Fei, Yingjie and Yang, Zhuoran and Chen, Yudong and Wang, Zhaoran}, editor={Ranzato, M. and Beygelzimer, A. and Dauphin, Y. and Liang, P.S. and Vaughan, J. Wortman}, volume={34}, year={2021}, pages={20436--20446} }

@article{Flamary_Courty_2021, title={{POT: Python Optimal Transport}}, volume={22}, number={78}, journal={Journal of Machine Learning Research}, author={Flamary, Rémi and Courty, Nicolas and Gramfort, Alexandre and Alaya, Mokhtar Z and Boisbunon, Aurélie and Chambon, Stanislas and Chapel, Laetitia and Corenflos, Adrien and Fatras, Kilian and Fournier, Nemo and others}, year={2021}, pages={1--8} }

@article{Föllmer_Schied_2002, title={Convex measures of risk and trading constraints}, volume={6}, journal={Finance and Stochastics}, publisher={Springer}, author={Föllmer, Hans and Schied, Alexander}, year={2002}, pages={429--447} }

@book{Föllmer_Schied_2011, title={Stochastic finance: an introduction in discrete time}, publisher={Walter de Gruyter, Berlin}, edition={4th}, author={Föllmer, Hans and Schied, Alexander}, year={2016} }

@article{Fu_Moncher_2004, title={{Severity Distributions for GLMs: Gamma or Lognormal? Evidence from Monte Carlo Simulations}}, journal={Casualty Actuarial Society Discussion Paper Program}, publisher={Citeseer}, author={Fu, Luyang and Moncher, Richard B}, year={2004}, pages={149--230} }

@article{Gallagher_2014, title={{Learning about an Infrequent Event: Evidence from Flood Insurance Take-Up in the United States}}, volume={6}, number={3}, journal={American Economic Journal: Applied Economics}, author={Gallagher, Justin}, year={2014}, pages={206--233}}

@article{Gao_Kleywegt_2023, title={{Distributionally Robust Stochastic Optimization with Wasserstein Distance}}, volume={48}, number={2}, journal={Mathematics of Operations Research}, publisher={INFORMS}, author={Gao, Rui and Kleywegt, Anton}, year={2023}, pages={603--655} }

@article{Gerber_1974, title={On additive premium calculation principles}, volume={7}, number={3}, journal={ASTIN Bulletin: The Journal of the IAA}, publisher={Cambridge University Press}, author={Gerber, Hans U}, year={1974}, pages={215--222} }

@incollection{Horowitz_2001, title={The Bootstrap}, volume={5}, booktitle={Handbook of Econometrics}, publisher={Elsevier}, author={Horowitz, Joel L.}, editor={Heckman, James J. and Leamer, Edward}, year={2001}, pages={3159--3228} }

@article{siegel2021profit,
  title={Profit estimation error in the newsvendor model under a parametric demand distribution},
  author={Siegel, Andrew F and Wagner, Michael R},
  journal={Management Science},
  volume={67},
  number={8},
  pages={4863--4879},
  year={2021},
  publisher={INFORMS}
}

@book{goodfellow2016deep,
  title={Deep learning},
  author={Goodfellow, Ian and Bengio, Yoshua and Courville, Aaron},
  year={2016},
  publisher={MIT press},
address ={Cambridge, MA}
}

@article{stone1974cross,
  title={Cross-validatory choice and assessment of statistical predictions},
  author={Stone, Mervyn},
  journal={Journal of the Royal Statistical Society: Series B (Methodological)},
  volume={36},
  number={2},
  pages={111--133},
  year={1974},
  publisher={Wiley Online Library}
}

@article{arlot2010survey,
  title={A survey of cross-validation procedures for model selection},
  author={Arlot, Sylvain and Celisse, Alain},
  journal={Statistics Surveys},
  volume={4},
  pages={40--79},
  year={2010}
}

@article{nelson2025pluvial,
  title={Pluvial flood impacts and policyholder responses throughout the United States},
  author={Nelson-Mercer, Benjamin and Kim, Taeho and Tran, Vinh Ngoc and Ivanov, Valeriy},
  journal={npj Natural Hazards},
  volume={2},
  number={1},
  pages={8},
  year={2025},
  publisher={Nature Publishing Group},
address={London, UK}
}

@article{Grigas_Qi_Shen_2023, title={{Integrated Conditional Estimation-Optimization}}, journal={arXiv preprint arXiv:2110.12351v4}, author={ Qi, Meng and Grigas, Paul and Shen, Max}, year={2023}, note={Accessed on Dec 19, 2024. \url{https://arxiv.org/pdf/2110.12351}} }

@article{Gupta_Huang_2022, title={Debiasing in-sample policy performance for small-data, large-scale optimization}, journal={Operations Research}, publisher={INFORMS}, author={Gupta, Vishal and Huang, Michael and Rusmevichientong, Paat}, year={2024}, pages={848--870}, volume={72}, number={2} }

@inproceedings{Hau_Petrik_2023, title={{Entropic Risk Optimization in Discounted MDPs}}, booktitle={International Conference on Artificial Intelligence and Statistics}, publisher={PMLR}, author={Hau, Jia Lin and Petrik, Marek and Ghavamzadeh, Mohammad}, editor={Francisco Ruiz and Jennifer Dy and Jan-Willem van de Meentyear}, year={2023},volume={206}, pages={47--76} }

@article{Herrnstadt_Sweeney_2024, title={Housing Market Capitalization of Pipeline Risk: Evidence from a Shock to Salience and Awareness}, journal={Land Economics}, publisher={University of Wisconsin Press}, author={Herrnstadt, Evan and Sweeney, Richard L}, year={2024}, note={Forthcoming} }

@article{Hino_Burke_2021, title={The effect of information about climate risk on property values}, volume={118}, number={17}, journal={Proceedings of the National Academy of Sciences}, publisher={National Acad Sciences}, author={Hino, Miyuki and Burke, Marshall}, year={2021}, pages={1--9} }

@article{Howard_Matheson_1972, title={{Risk-sensitive Markov decision processes}}, volume={18}, number={7}, journal={Management science}, publisher={INFORMS}, author={Howard, Ronald A and Matheson, James E}, year={1972}, pages={356–369} }

@article{Hu_Hong_2012, title={{Kullback-Leibler Divergence Constrained Distributionally Robust Optimization}}, journal={Optimization Online}, author={Hu, Zhaolin and Hong, L Jeﬀ}, year={2012}, note={Accessed on Dec 19, 2024. \url{https://optimization-online.org/2012/11/3677/ } }}

@inproceedings{Ito_Yabe_2018, title={{Unbiased Objective Estimation in Predictive Optimization}}, booktitle={International Conference on Machine Learning}, publisher={PMLR}, author={Ito, Shinji and Yabe, Akihiro and Fujimaki, Ryohei}, editor={J Dy and A Krause}, year={2018}, pages={2176--2185} }

@article{Iyengar_Lam_2023, title={Optimizer’s Information Criterion: Dissecting and Correcting Bias in Data-Driven Optimization}, journal={arXiv preprint arXiv:2306.10081}, author={Iyengar, Garud and Lam, Henry and Wang, Tianyu}, year={2023}, note={Accessed on Dec 19, 2024. \url{https://arxiv.org/pdf/2306.10081}} }

@inproceedings{Jang_Gu_2017, title={{Categorical Reparameterization with Gumbel-Softmax}}, booktitle={International Conference on Learning Representations}, publisher={Curran Associates, Inc.}, author={Jang, Eric and Gu, Shixiang and Poole, Ben}, volume={3}, year={2017}, pages={1920--1931} }

@article{Jiang_Chen_2020, title={Rates of convergence of sample average approximation under heavy tailed distributions}, journal={Optimization Online}, author={Jiang, Jie and Chen, Zhiping and Yang, Xinmin}, year={2020}, note= { Accessed on Dec 19, 2024. \url{https://optimization-online.org/wp-content/uploads/2020/06/7849.pdf}} }

@article{Kaluszka_2004a, title={{An extension of Arrow’s result on optimality of a stop loss contract}}, volume={35}, number={3}, journal={Insurance: Mathematics and Economics}, author={Kaluszka, Marek}, year={2004}, pages={527--536} }

@article{Kaluszka_2004b, title={{Mean-Variance Optimal Reinsurance Arrangements}}, journal={Scandinavian Actuarial Journal}, publisher={Taylor \& Francis}, author={Kaluszka, Marek}, pages={28--41}, volume={2004}, number={1}, year={2004} }

@article{Kim_2010, title={Bias correction for estimated distortion risk measure using the bootstrap}, volume={47}, number={2}, journal={Insurance: Mathematics and Economics}, author={Kim, Joseph H. T.}, year={2010}, pages={198--205} }

@article{Kim_Hardy_2007, title={{Quantifying and Correcting the Bias in Estimated Risk Measures}}, volume={37}, number={2}, journal={ASTIN Bulletin: The Journal of the IAA}, author={Kim, Joseph Hyun Tae and Hardy, Mary R.}, year={2007}, pages={365--386}}

@inproceedings{Kingma_Salimans_2015, title={{Variational Dropout and the Local Reparameterization Trick}}, volume={28}, booktitle={Advances in Neural Information Processing Systems}, publisher={Curran Associates, Inc.}, author={Kingma, Durk P and Salimans, Tim and Welling, Max}, editor={Cortes, C. and Lawrence, N. and Lee, D. and Sugiyama, M. and Garnett, R.}, year={2015} }

@inproceedings{Kolouri_Pope_2019, title={{Sliced Wasserstein Auto-Encoders}}, volume={5}, booktitle={International Conference on Learning Representations}, publisher={Curran Associates, Inc.}, author={Kolouri, Soheil and Pope, Phillip E and Martin, Charles E and Rohde, Gustavo K}, year={2019}, pages={3481--3499} }

@article{Kousky_Cooke_2012, title={{Explaining the Failure to Insure Catastrophic Risks}}, volume={37},  journal={The Geneva Papers on Risk and Insurance-Issues and Practice}, author={Kousky, Carolyn and Cooke, Roger}, year={2012}, pages={206--227} }

@inproceedings{Kuiper_Yang_2024, title={{Distributionally Robust Optimization as a Scalable Framework to Characterize Extreme Value Distributions}}, booktitle={The Conference on Uncertainty in Artificial Intelligence}, author={Kuiper, W. and Yang, W. and Hassan, A. and Ng, Y. and Bidkhori, H. and Blanchet, J. and Tarokh, V.}, volume={244}, editor={Kiyavash, Negar and Mooij, Joris M.}, publisher={PMLR}, year={2024} }

@article{Kupper_Schachermayer_2009, title={Representation results for law invariant time consistent functions}, volume={2}, number={3}, journal={Mathematics and Financial Economics}, publisher={Springer}, author={Kupper, Michael and Schachermayer, Walter}, year={2009}, pages={189--210} }

@article{L.A._Bhat_2022, title={{A Wasserstein Distance Approach for Concentration of Empirical Risk Estimates}}, volume={23}, number={238}, journal={Journal of Machine Learning Research}, author={L.A., Prashanth and Bhat, Sanjay P.}, year={2022}, pages={1--61} }

@article{Lam_Mottet_2017, title={Tail Analysis Without Parametric Models: A Worst-Case Perspective}, volume={65}, number={6}, journal={Operations Research}, author={Lam, Henry and Mottet, Clementine}, year={2017}, pages={1696--1711} }

@article{Li_Beirami_2023, title={{On Tilted Losses in Machine Learning: Theory and Applications}}, volume={24}, number={142}, journal={Journal of Machine Learning Research}, author={Li, Tian and Beirami, Ahmad and Sanjabi, Maziar and Smith, Virginia}, year={2023}, pages={1--79} }

@article{Lim_Shanthikumar_2007, title={Relative Entropy, Exponential Utility, and Robust Dynamic Pricing}, volume={55}, number={2}, journal={Operations Research}, author={Lim, Andrew E. B. and Shanthikumar, J. George}, year={2007}, pages={198--214} }

@article{Lugosi_Mendelson_2019, title={{Mean Estimation and Regression Under Heavy-Tailed Distributions: A Survey}}, volume={19}, number={5}, journal={Foundations of Computational Mathematics}, author={Lugosi, Gábor and Mendelson, Shahar}, year={2019}, pages={1145–1190}, language={en} }

@inproceedings{Maddison_Mnih_2017, title={{The Concrete Distribution: A Continuous Relaxation of Discrete Random Variables}}, booktitle={International Conference on Learning Representations (ICLR)}, author={Maddison, Chris J and Mnih, Andriy and Teh, Yee Whye}, volume={3}, year={2017} }

@article{Marcoux_H_Wagner_2023, title={{Fifty Years of US Natural Disaster Insurance Policy}}, journal={Handbook of Insurance}, author={Marcoux, Kendra and H Wagner, Katherine R}, year={2023}, note={Accessed on Dec 19, 2024. \url{https://www.krhwagner.com/papers/NaturalDisasterInsuranceHandbook.pdf}} }

@article{Markowitz_1952, title={Portfolio selection}, volume={7}, number={1}, journal={The Journal of Finance}, author={Markowitz, Harry}, year={1952}, pages={77--91} }

@article{Markowitz_2014, title={Mean–variance approximations to expected utility}, volume={234}, number={2}, journal={European Journal of Operational Research}, author={Markowitz, Harry}, year={2014}, pages={346--355} }

@book{McNeil_Frey_2005, title={{Quantitative Risk Management: Concepts, Techniques and Tools}}, publisher={Princeton University Press}, author={McNeil, Alexander J and Frey, Rüdiger and Embrechts, Paul}, year={2005} }

@article{Mohajerin_Esfahani_Kuhn_2018, title={Data-driven distributionally robust optimization using the {Wasserstein} metric: Performance guarantees and tractable reformulations}, volume={171}, number={1–2}, journal={Mathematical Programming}, author={Mohajerin Esfahani, Peyman and Kuhn, Daniel}, year={2018}, pages={115--166} }

@inproceedings{Nass_Belousov_2019, address={Macau, China}, title={Entropic Risk Measure in Policy Search}, booktitle={2019 IEEE/RSJ International Conference on Intelligent Robots and Systems (IROS)}, publisher={IEEE Press}, author={Nass, David and Belousov, Boris and Peters, Jan}, year={2019}, pages={1101--1106} }

@article{Pratt_1964, title={Risk aversion in the small and in the large}, volume={32}, number={1/2}, journal={Econometrica}, author={Pratt, John W}, year={1964}, pages={122--136} }

@article{Rahimian_Mehrotra_2022, title={Frameworks and results in distributionally robust optimization}, volume={3}, journal={Open Journal of Mathematical Optimization}, author={Rahimian, Hamed and Mehrotra, Sanjay}, year={2022}, pages={1--85} }

@article{Sadana_Chenreddy_2025, title={A survey of contextual optimization methods for decision-making under uncertainty}, volume={320}, number={2}, journal={European Journal of Operational Research}, author={Sadana, Utsav and Chenreddy, Abhilash and Delage, Erick and Forel, Alexandre and Frejinger, Emma and Vidal, Thibaut}, year={2025}, pages={271--289} }

@article{Saldi_Başar_2020, title={{Approximate Markov-Nash equilibria for discrete-time risk-sensitive mean-field games}}, volume={45}, number={4}, journal={Mathematics of Operations Research}, publisher={Informs}, author={Saldi, Naci and Başar, Tamer and Raginsky, Maxim}, year={2020}, pages={1596--1620} }

@article{Shalev-Shwartz_2012, title={{Online learning and online convex optimization}}, volume={4}, number={2}, journal={Foundations and Trends® in Machine Learning}, publisher={Now Publishers, Inc.}, author={Shalev-Shwartz, Shai and others}, year={2012}, pages={107--194} }

@book{Shapiro_Dentcheva_2009, title={{Lectures on Stochastic Programming: Modeling and Theory}}, publisher={Society for Industrial and Applied Mathematics}, author={Shapiro, Alexander and Dentcheva, Darinka and Ruszczyński, Andrzej}, year={2009} }

@article{Siegel_Wagner_2023, title={{Technical Note--Data-Driven Profit Estimation Error in the Newsvendor Model}}, journal={Operations Research}, author={Siegel, Andrew F. and Wagner, Michael R.}, year={2023}, volume={71}, number={6}, pages={2146--2157} }

@article{Smith_Winkler_2006, title={{The Optimizer’s Curse: Skepticism and Postdecision Surprise in Decision Analysis}}, volume={52}, number={3}, journal={Management Science}, author={Smith, James E. and Winkler, Robert L.}, year={2006}, pages={311--322} }

@article{Smith_Chapman_2023, title={{On Exponential Utility and Conditional Value-at-Risk as risk-averse performance criteria}}, journal={IEEE Transactions on Control Systems Technology}, publisher={IEEE}, author={Smith, Kevin M and Chapman, Margaret P}, year={2023}, volume={31}, number={6}, pages ={2555--2570} }

@article{Svensson_Werner_1993, title={Nontraded assets in incomplete markets: Pricing and portfolio choice}, volume={37}, number={5}, journal={European Economic Review}, publisher={Elsevier}, author={Svensson, L. E. O. and Werner, Ingrid M}, year={1993}, pages={1149--1168} }

@inproceedings{Troop_Godin_2021, title={Bias-corrected peaks-over-threshold estimation of the CVaR}, booktitle={Uncertainty in Artificial Intelligence}, publisher={PMLR}, author={Troop, Dylan and Godin, Frédéric and Yu, Jia Yuan}, year={2021}, volume={161},  editor = 	 {de Campos, Cassio and Maathuis, Marloes H.}, pages={1809--1818} }

@book{Van_der_Vaart_2000, title={{Asymptotic Statistics}},  publisher={Cambridge University Press}, place={Cambridge, UK}, author={Van der Vaart, Aad W}, year={2000} }

@book{Von_Neumann_Morgenstern_1944, title={{Theory of games and economic behavior}}, publisher={Princeton University Press}, author={Von Neumann, John and Morgenstern, Oskar}, year={1944} }

@article{Wiesemann_Kuhn_2014, title={Distributionally robust convex optimization}, volume={62}, number={6}, journal={Operations Research}, publisher={INFORMS}, author={Wiesemann, Wolfram and Kuhn, Daniel and Sim, Melvyn}, year={2014}, pages={1358--1376} }

 \appendix
\renewcommand{\thesection}{\Alph{section}}
\renewcommand{\thesubsection}{\Alph{section}.\arabic{subsection}}
\section{Literature Review} \label{sec:lit_review}
\subsection{Risk estimation}
Quantitative risk measurement often relies on precise estimation of risk measures commonly used in finance and actuarial science \citep{McNeil_Frey_2005}. Calculating risk for multidimensional random variables can be challenging due to the need for complex integrals, often approximated using Monte Carlo simulation. When the underlying distribution isn't directly accessible and only limited samples are available, Monte Carlo-based risk estimators tend to underestimate the actual risk. \cite{Kim_Hardy_2007} addressed this by using \revised{double-}bootstrap to correct bias in Value at Risk (VaR) and Conditional Tail Expectation (CTE) estimates, with \cite{Kim_2010} extending this approach to general distortion risk measures. \revised{The authors note that for distortion risk measures, the empirical risk estimate is a linear combination of order statistics, for which the bootstrap and double bootstrap estimates can be analytically computed.} %

Several approaches have been proposed in the literature for estimating tail risks. \cite{Lam_Mottet_2017} introduce a distributionally robust optimization based method to construct worst-case bounds on tail risk, assuming the density function is convex beyond a certain threshold. Extreme Value Theory (EVT) is commonly used to estimate tail risk measures, such as CVaR, by fitting a Generalized Pareto Distribution to values exceeding a threshold. \cite{Troop_Godin_2021} develop an asymptotically unbiased CVaR estimator by correcting the bias in the estimates obtained via maximum likelihood. In \cite{Kuiper_Yang_2024}, the authors derive DRO-based estimators for  EVT statistics to account for model misspecification due to scenarios outside the asymptotic tails. While these methods focus on upper-tail risk, they are not directly applicable to the entropic risk measure.   Instead of fitting an extreme value distribution, we utilize a parametric two-component Gaussian Mixture Model (GMM), which provides a closed-form expression for the entropic risk. %

 One related field of research is to derive concentration bounds on the risks estimates depending on whether the random variable is subgaussian, subexponential or heavy-tailed.  For optimized certainty equivalent risk measures that are Hölder continuous, \cite{L.A._Bhat_2022} link estimation error to the Wasserstein distance between the empirical and true distributions, for which concentration bounds are available. While the Central Limit Theorem (CLT) ensures asymptotic convergence of the sample average to the true mean, this guarantee doesn't always hold for finite samples unless the tails are Gaussian or subgaussian \citep{Catoni_2012, Bartl_Mendelson_2022}. Robust statistics literature offers alternative estimators, like the median-of-means (MoM) estimator \citep{Lugosi_Mendelson_2019}, which ensure the estimator is close to the true mean with high confidence. However, these approaches differ from our focus, which is on constructing estimators with minimal bias.

\subsection{Correcting optimistic bias}
 
Our work on entropic risk minimization relates to correcting the optimistic bias of SAA policies to achieve true decision performance \citep{Smith_Winkler_2006, Beirami_Razaviyayn_2017}. SAA is analogous to empirical risk minimization in machine learning, where the goal is to minimize empirical risk. Methods such as DRO, hold-out, and K-fold CV are used to correct this bias. These approaches involve partitioning data into training, validation, and test sets, and then selecting the hyperparameter that results in the smallest validation risk \citep{Bousquet_Elisseeff_2000}. Our hyper-parameter selection employs the debiased validation risk.

Several methods have been proposed to correct the bias of SAA in linear optimization problems under the assumption that the true data distribution is Gaussian. For example, \cite{Ito_Yabe_2018} introduce a perturbation technique that generates parameters around the true values under Gaussian error assumptions to achieve an asymptotically unbiased estimator of the true loss. Similarly, \cite{Gupta_Huang_2022} derive estimators for the out-of-sample performance of in-sample optimal policies under a Gaussian distribution and offer extensions to approximate Gaussian cases, leveraging the structure of linear optimization problems. However, extending these methods to our nonlinear problem is challenging. Moreover, our objective is to find optimal policies with low out-of-sample risk rather than merely estimating this risk. Since CV risk could be biased estimate of the out-of-sample risk,  we employ our bias correction procedure to correct it, enabling appropriate calibration of the regularization parameter. Importantly, our approach does not rely on Gaussian assumptions for the uncertain parameter or the structure of the objective function.

In \cite{Siegel_Wagner_2023}, the authors analytically characterize the bias in SAA policies for a data-driven newsvendor problem, providing an asymptotically debiased profit estimator by leveraging the asymptotic properties of order statistics. \revised{In \cite{siegel2021profit}, the authors assume a parametric form for the demand distribution, and correct the asymptotic bias in the estimate of maximized estimated profit in a newsvendor problem.} \cite{Iyengar_Lam_2023} introduce an Optimizer's Information Criterion (OIC) to correct bias in SAA policies, generalizing the approach by \cite{Siegel_Wagner_2023} to other loss functions. However, OIC requires access to the gradient, Hessian, and influence function of the decision rule, which can be challenging to obtain in general constrained optimization problems. Moreover, 
the form of optimal policy is known for risk-neutral newsvendor problems but not for entropic risk minimization problems.

\subsection{Insurance pricing}

The design of insurance contracts has been widely studied since the foundational work of Arrow \citep{Arrow_1963, Arrow_1971}. Under the assumption that premiums are proportional to the policy's actuarial value, it has been shown that an expected utility-maximizing policyholder will choose full coverage above a deductible. Various extensions of Arrow's model have been proposed to account for the risk aversion of both the insured and the insurer, using criteria such as mean-variance \citep{Kaluszka_2004a, Kaluszka_2004b}, Value at Risk (VaR), and Tail VaR \citep{Cai_Tan_2008}. \cite{Bernard_Tian_2010} incorporate regulatory constraints on the insurer’s insolvency risk through VaR. \cite{Cheung_Sung_2014} extend these models to multiple policyholders with fully dependent risks (comonotonicity), where the insurer utilizes convex law-invariant risk measures. \cite{Bernard_Liu_2020} further explore different levels of dependence among policyholders, with both insurers and policyholders %
using exponential utility functions. However, these studies typically assume that the loss distribution is known.
We extend the model proposed by \cite{Bernard_Liu_2020} to account for ambiguity regarding the true loss distribution when only a limited number of samples are available.

\noindent
\section{Proof of results in Section \ref{sec:property_entropic_risk}}\label{append:property_entropic_risk}

\begin{lemma}\label{lemma:equiv:mgf:tail}
The following conditions are equivalent:
\begin{enumerate}
    \item There exist some constants  $G > 0$ and $C > 2$ such that $\Prob(|\ell(\vecz, \vecxi)| > a) \leq G \exp(-a \alpha C), \,\forall a \geq 0$,
\item The moment-generating function of $\ell(\vecz, \vecxi)$ satisfies $\Expect[\exp(t \ell(\vecz, \vecxi))] \in \mathbb{R}$ for all $t \in (-\alpha C, \alpha C)$, for some $C > 2$.
\end{enumerate}
\end{lemma}

\begin{proof}
The property $\Prob(|\ell(\vecz, \vecxi)| > a) \leq G \exp(-a\alpha C),\, \forall a\geq 0$
for some $G>0$ and $C>2$, implies that when $t\in (-\alpha C,\alpha C)$, $C>2$, we have that:
\begin{equation}
\begin{aligned}\label{bound_mgf_gamma}
    0 &\leq \Expect[\exp(t\ell(\vecz, \vecxi))]\leq \Expect[\exp(|t||\ell(\vecz, \vecxi)|)]=  \int_0^{\infty} \Prob(\exp(|t||\ell(\vecz, \vecxi)|)> x) dx\\&= \int_{0}^{\infty} \Prob\left(|\ell(\vecz, \vecxi)|> \frac{\log(x)}{|t|}\right) dx\leq 1+\int_{1}^{\infty} \Prob\left(|\ell(\vecz, \vecxi)|> \frac{\log(x)}{|t|}\right) dx\\&= 1+\int_1^{\infty} G\exp\left(-\frac{\alpha C\log(x)}{|t|}\right) dx= 1+\int_1^{\infty} Gx^{-\frac{\alpha C}{|t|}}dx=1+\frac{G|t|}{\alpha C- |t|}\\&\leq \frac{\alpha C- |t|}{\alpha C- |t|}+ \frac{G\alpha C}{\alpha C- |t|}\le \frac{\alpha C}{\alpha C- |t|}+\frac{G\alpha C}{\alpha C- |t|}= \frac{{\tilde{G}}\alpha C}{\alpha C- |t|},
\end{aligned}
\end{equation}
where {$\tilde{G}:=G+1$}, and we %
exploit the fact that $\log(x)\leq 0$ for $x\in (0,1]$ and $|t|<\alpha C$.

\noindent
Alternatively, $\Expect[\exp(t\ell(\vecz, \vecxi))]\in\mathbb{R}$ for all $t\in(-\alpha C,\alpha C)$ for some $C> 2$ implies that for any $a\geq 0$, 
\[\Prob(\ell(\vecz, \vecxi)>a)=\Prob(\exp(\alpha C\ell(\vecz, \vecxi))>\exp(\alpha a C ))\leq \frac{\Expect[\exp(\alpha C \ell(\vecz, \vecxi))]}{\exp(\alpha a C)}=G^+\exp(-\alpha a C),\]
where  we have used Markov's inequality to obtain the upper bound and $G^+:=\Expect[\exp(\alpha C \ell(\vecz, \vecxi))]\in\mathbb{R}$. A similar argument holds for $\Prob(-\ell(\vecz, \vecxi)>a)\leq G^-\exp(-\alpha aC)$ with  $G^-:=\Expect[\exp(-\alpha C \ell(\vecz, \vecxi))]\in\mathbb{R}$. Thus, by the union bound, we have:
\[\Prob(|\ell(\vecz, \vecxi)| > a) \leq (G^++G^-) \exp(-a\alpha C), \quad \forall a\geq 0.\]
\end{proof}

\subsection{Proof of Lemma \ref{thm:boundedExpVar2} }
\begin{proof}
Using the layer cake representation of a non-negative, real-valued measurable function, we have that:
    \begin{align*}
        \Expect[\exp(\alpha \ell(\vecz, \vecxi))]&= \int_0^{\infty} \Prob(\exp(\alpha \ell(\vecz, \vecxi))> x) dx\\
        & = \int_{{-\infty}} ^{\infty} \Prob(\exp(\alpha \ell(\vecz, \vecxi))> \exp(\alpha y)) \alpha \exp(\alpha y) dy\\
        & = \int_{{-\infty}}^{\infty} \Prob(\ell(\vecz, \vecxi)> y) \alpha \exp(\alpha y) dy\\
        &{= \alpha\int_{-\infty}^{0} \mathbb{P}(\ell(\vecz, \vecxi)> y) \exp(\alpha y) \,dy+\alpha \int_{0}^{\infty} \mathbb{P}(\ell(\vecz, \vecxi)> y) \exp(\alpha y) \,dy}\\
&\leq{\alpha\int_{-\infty}^{0}  \exp(\alpha y) \,dy +} \int_0^\infty  \Prob(\ell(\vecz, \vecxi)> y) \alpha \exp(\alpha y) dy  \\
        & \leq { 1+} \int_0^{\infty} \Prob(|\ell(\vecz, \vecxi)|> y) \alpha \exp(\alpha y) dy \\
        &\leq { 1+} \alpha \int_0^{\infty} G \exp(-(C-1)\alpha y) dy={ 1+} \frac{G}{C-1} %
    \end{align*}
    where  the last inequality follows from Assumption \ref{assum:tail_bound2}, and $C>2$ is used to obtain the final result. 
From Jensen's inequality, we also have that:
    \begin{align*}
        \log(\Expect[\exp(\alpha \ell(\vecz, \vecxi))]) &\geq \Expect[\log(\exp(\alpha \ell(\vecz, \vecxi)))] = \Expect[\alpha \ell(\vecz, \vecxi)]\geq  -\alpha \Expect[ |\ell(\vecz, \vecxi)|]\\
        & = -\alpha \int_{0}^\infty \Prob(| \ell(\vecz, \vecxi)| \geq y) dy 
          \geq  - \alpha\int_{0}^\infty G\exp({-C y \alpha}) dy =  -G/C,
    \end{align*}
    where, the second inequality comes from $x\ge -|x|,\, \forall x$, and the  last inequality follows from Assumption \ref{assum:tail_bound2}.
    Hence, we have that $\Expect[\exp(\alpha \ell(\vecz, \vecxi))]\geq \exp({-G/C})$.
        
    Similarly,
    \begin{align*}
        \Expect[\exp(2\alpha \ell(\vecz, \vecxi))]
        & \ {=} \  \int_{-\infty}^{\infty} \Prob(\exp(2\alpha \ell(\vecz, \vecxi))> \exp(2\alpha y)) 2\alpha \exp(2\alpha y) dy\\
        & \ \leq \  {1+} \int_0^{\infty} \Prob(|\ell(\vecz, \vecxi)|> y) 2 \alpha \exp(2\alpha y) dy\\
        & \ \leq \ {1+} 2\alpha \int_0^{\infty} G \exp(-(C-2)\alpha y) dy={1+} \frac{2G}{C-2} %
    \end{align*}
    where we used $C>2$ to obtain the last equality, see Assumption \ref{assum:tail_bound2}.
    Hence,
    \[0\leq \text{Var}(\exp(\alpha \ell(\vecz, \vecxi)))=\Expect[\exp(2 \alpha\ell(\vecz, \vecxi))]-\left(\Expect[\exp(\alpha \ell(\vecz, \vecxi))]\right)^2\leq\Expect[\exp(2\alpha \ell(\vecz, \vecxi))]\leq {1}+\frac{2G}{C-2}%
    .\]
    \end{proof}

\section{\revised{Proof of the results in Section \ref{sec:asymptotic:effect:variance_on_bias}}}\label{appendix:2}
In this section, we will provide proofs of all the propositions given in Section \ref{sec:asymptotic:effect:variance_on_bias}. The next lemma will allow us to economize notation in the subsequent proofs:  the bias associated with estimating the loss $\ell(\vecz, \vecxi)$ of a decision $\vecz$ is equal to the bias associated with estimating it through its location-scale-equivalent $\mu+\sigma \xi$ representation. Hence, we may work equivalently with estimating the risk of the random loss $\mu+\sigma \xi$ with $\xi\sim\Prob^\xi$ based on $\hat{\Prob}_N^\xi$, the empirical distribution constructed from $\{\hat{\xi}_i\}_{i=1}^N$ drawn i.i.d. from $\Prob^\xi$. %

\subsection{Some useful lemmas}

\begin{lemma}\citep[ Proposition 2.5.2]{Vershynin2025}\label{lemma:cgf:subgaussian:centered}
    If a random variable $\zeta\sim\Q$, with $\Expect_\Q[\zeta]=0$, has subgaussian tails (see Definition \ref{assum:subgaussian}) then there exists $
\nu_0>0$ such that its cumulant generating function satisfies:
   \[     \Lambda_{\Q}(t)\le 
\nu_0^2t^2    \;\; \forall t\in \mathbb{R}. \]
\end{lemma}

\begin{lemma}%
\label{recur_lemma}
   Let $\{\hat{\zeta}_1,\dots,\hat{\zeta}_N\}$ be the set of equiprobable scenarios captured by an empirical distribution $\Q$, and let $(a,b)\in\mathbb{R}\times\mathbb{R}_+$, then
    \begin{equation}\label{eq:recurr_lemma}a+b \max_{1\leq j \leq N}\hat{\zeta}_j - \log(N)/\alpha \leq \rho_{\Q}(a+b\zeta) \leq a+b \max_{1\leq j \leq N}\hat{\zeta}_j.\end{equation}
\end{lemma}
\begin{proof}
  We bound  the entropic risk using the relation $a+b\hat{\zeta}_i \leq  \max_{1\le j\le N} a+b\hat{\zeta}_j\, \forall  i=\{1, \cdots, N\}$: 
  \begin{equation}\label{eq:upper_bound_emp_risk}
    \begin{aligned}
      \rho_{\Q}(a+b\zeta) &=  (1/\alpha) \log \left((1/N)\sum_{i=1}^N \exp(\alpha (a+b\hat{\zeta}_i))\right) \\&\leq (1/\alpha) \log \left((1/N)\sum_{i=1}^N  \exp(\alpha \max_{1\leq j \leq N} a+b\hat{\zeta}_j)\right)\\
        &= \max_{1\leq j\leq N} a + b \hat{\zeta}_j = a + b \max_{1\leq j\leq N} \hat{\zeta}_j.
        \end{aligned}
        \end{equation}
        We use the relation $\sum_{i=1}^N  y_i \geq \max_{1\leq j\leq N}y_j$, for all $i$, for non-negative  $y_i$ values, to obtain:
    \begin{equation}\label{eq:lower_bound_emp_risk}
    \begin{aligned}
    \rho_{\Q}(a+b\zeta)&=(1/\alpha)\log((1/N)\sum_{i=1}^N \exp(\alpha (a+b\hat{\zeta}_{j}))) \\&\geq (1/\alpha)\log((1/N)\max_{1\leq j\leq N} \exp(\alpha (a+b\hat{\zeta}_{j})))\\
    &= \max_{1\leq j\leq N}a+b\hat{\zeta}_{j} - \log(N)/\alpha \\& = a+b \max_{1\leq j\leq N} \hat{\zeta}_j - \log(N)/\alpha.
    \end{aligned}
        \end{equation}
\end{proof}

\begin{lemma}\label{lemma_almst_sure}
     If Assumption \ref{ass:meanVarControl} is satisfied, then 
\begin{subequations}
\begin{align}
&\rho_{\hat{\Prob}_N}(\ell(\z,\vecxi)) \leq \rho_\Prob(\ell(\z,\vecxi)) -   (1/\alpha)\Lambda_{\Prob^{\xi}}(\alpha\sigma) +\sigma M_N, \mbox{ a.s.},\label{eq:almost_sure}
\\
   & \rho_{\hat{\Prob}_N}(\ell(\z,\vecxi)) \geq \rho_\Prob(\ell(\z,\vecxi)) -(1/\alpha)\Lambda_{\Prob^{\xi}}(\alpha \sigma) + \sigma M_N-\log(N)/\alpha, \mbox{ a.s.},
\end{align}
where $M_{N}:=\max_{1\leq j \leq N}\hat{\xi}_j$, with $\hat{\xi}_j=(\ell(\z,\hat{\vecxi}_j)-\mu)/\sigma\sim\Prob^\xi$ independently.
\end{subequations}
\end{lemma}
\begin{proof}
From Lemma \ref{recur_lemma}, we have that
\begin{equation}\label{eq:upper_bound_emp_risk_cumulant}
    \begin{aligned}
      \rho_{\hat{\Prob}_N}(\ell(\z,\vecxi)) &=\rho_{\hat\Prob_N}\left(\mu +\sigma\frac{\ell(\vecz,\vecxi)-\mu}{\sigma}\right)\\ &\leq  \mu+\sigma M_N\\
        &= \mu + \sigma M_N + \rho_\Prob(\ell(\z,\vecxi))-\rho_\Prob(\ell(\z,\vecxi))\\
        &= \rho_\Prob(\ell(\z,\vecxi)) +  \sigma M_N - \rho_{\Prob^\xi}(\sigma \xi) = \rho_\Prob(\ell(\z,\vecxi)) +  \sigma M_N - (1/\alpha)\Lambda_{\Prob^{\xi}}(\alpha \sigma),
    \end{aligned}
    \end{equation}
    where third equality is obtained from the translation invariance of entropic risk measure, i.e., $\rho_\Prob(\ell(\z,\vecxi))=\rho_{\Prob^\xi}(\mu+\sigma\xi)=\mu+\rho_{\Prob^\xi}(\sigma \xi)$.
    Similarly from Lemma \ref{recur_lemma}, we have
\begin{equation}\label{eq:lower_bound_emp_risk_cumulant}
    \begin{aligned}
              \rho_{\hat{\Prob}_N}(\ell(\z,\vecxi))  &\geq \mu+\sigma M_N - \log(N)/\alpha \\ &= \mu + \sigma M_N  - \log(N)/\alpha + \rho_\Prob(\ell(\z,\vecxi))-\rho_\Prob(\ell(\z,\vecxi))\\
        &= \rho_\Prob(\ell(\z,\vecxi)) +  \sigma M_N  - \log(N)/\alpha - \rho_{\Prob^\xi}(\sigma \xi) \\
        &= \rho_\Prob(\ell(\z,\vecxi)) +  \sigma M_N  - \log(N)/\alpha - (1/\alpha)\Lambda_{\Prob^{\xi}}(\alpha \sigma),
    \end{aligned}
    \end{equation}
    where we used translation invariance of entropic risk measure, i.e., $\rho_{\Prob^\xi}(\mu+\sigma\xi)=\mu+\rho_{\Prob^\xi}(\sigma \xi)$, to obtain the second equality.
\end{proof}
\subsection{\texorpdfstring{Detailed derivation of $\rho_{\model{Delta}}$ and $\rho_{\model{OIC}}$}{derive}}\label{proof:derive:oic:delta}
\begin{lemma}\label{lemma_delta_bias}
    Suppose that Assumption \ref{assum:tail_bound2} holds with $C>4$. 
    Then, the first-order bias correction based on the Taylor series expansion of empirical risk is given by 
     \[\rho_{\model{Delta}}:= \rho_{\hat{\Prob}_N}(\ell(\z,\vecxi))+\frac{\text{Var}_{\hat{\Prob}_N}(\exp(\alpha \ell(\z,\vecxi)))}{2N\alpha (\mathbb{E}_{{\hat{\Prob}_N}}[\exp(\alpha \ell(\z,\vecxi))])^2}.\]
\end{lemma}
\begin{proof}
   Let $\bar{Y}_N= \Expect_{\hat{\Prob}_N}[\exp(\alpha \ell(\z,\vecxi))]$. With $C>4$, $\exp(\alpha \ell(\z,\vecxi))$ has finite fourth order moments (see Assumption \ref{assum:tail_bound2}). Then the %
   Taylor series expansion of $\frac{1}{\alpha}\log(y)$ around $\bar{y}:=\Expect\left[\exp(\alpha \ell(\z,\vecxi))\right]=\Expect[\bar{Y}_N]$ given in \citep[Equation 3.2]{Horowitz_2001} gives that: %
\begin{align*}
  &  \rho_{\hat{\Prob}_N}(\ell(\z,\vecxi))=\frac{1}{\alpha}\log(\bar{Y}_N) =     \frac{1}{\alpha}\log(\bar{y}) + (1/(\alpha \bar{y})) (\bar{Y}_N-\bar{y})-(1/(2\alpha \bar{y}^2)) (\bar{Y}_N-\bar{y})^2+r_N\\
      & \implies  \Expect[\rho_{\hat{\Prob}_N}(\ell(\z,\vecxi))] - \rho_\Prob(\ell(\z,\vecxi))= \Expect[\frac{1}{\alpha}\log(\bar{Y}_N)] -     \frac{1}{\alpha}\log(\bar{y}) = -\frac{1}{2\alpha \bar{y}^2}\Expect[(\bar{Y}_N-\bar{y})^2]+ \mathcal O(N^{-2})\\
      &\qquad\qquad = -\frac{1}{2\alpha N (\mathbb{E}_\Prob[\exp(\alpha \ell(\z,\vecxi))])^2}\mbox{Var}_\Prob(\exp(\alpha \ell(\z,\vecxi) ))+ \mathcal O(N^{-2}).
\end{align*}
where $\Expect[r_N]=\mathcal O(N^{-2})$. We take the expectation and variance over the empirical distribution $\hat{\Prob}_N$, to obtain the bias correction. %
\end{proof}

    \begin{lemma}\label{lemma_oic}\citep{Iyengar_Lam_2023}
   The data-driven estimator of entropic risk based on the optimizer's information criterion (OIC) is given by:
   \[\rho_{\model{OIC}}:= \rho_{\hat{\Prob}_N}(\ell(\vecz, \vecxi))+\frac{\text{Var}_{\hat{\Prob}_N}(\exp(\alpha \ell(\vecz, \vecxi)))}{N\alpha (\mathbb{E}_{{\hat{\Prob}_N}}[\exp(\alpha \ell(\vecz, \vecxi))])^2}.\]
\end{lemma}
\begin{proof}

From Theorem 1 in \cite{Iyengar_Lam_2023}, it follows that for a loss function $h(t,\zeta)$ with optimal decision $t^*=\rho_{\Prob}(\ell(\vecz, \vecxi))$ and SAA decision $\hat{t}=\rho_{\hat{\Prob}_N}(\ell(\vecz, \vecxi))$:
    \[
\Expect[\Expect_{\Prob}[h(\hat{t}, \ell(\vecz, \vecxi))]] =      \Expect[h(\hat{t}, \ell(\vecz, \vecxi))]  \underbrace{-\frac{1}{N}\Expect_{\Prob}[\nabla_{t} h(t^*,\ell(\vecz, \vecxi)) \text{IF}({\ell(\z,\vecxi)}%
)]}_{\delta_{\model{OIC}}} + o\left(\frac{1}{N}\right),\]
where expectation is with respect to the randomness of $\Dist_N$. For $h(t,\ell(\vecz, \vecxi)) = t+ \frac{1}{\alpha}\exp(\alpha (\ell(\vecz, \vecxi)-t))-\frac{1}{\alpha}$, we know that $\nabla_{t}h(t^*,\ell(\vecz, \vecxi))=1-\exp(\alpha(\ell(\vecz, \vecxi)-t^*))$ and $\nabla^2_{t,t}h(t^*,\ell(\vecz, \vecxi))=\alpha\exp(\alpha(\ell(\vecz, \vecxi)-t^*))$. The influence function in the expression of $\delta_{\model{OIC}}$ is obtained as follows: 
{
\begin{align*}
    \text{IF}( {\ell(\z,\vecxi)}%
    ) \,&=\,  -\left(\Expect_\Prob[\nabla^2_{t,t}h(t^*,\ell(\vecz, \vecxi))]\right)^{-1}\nabla_{t}h(t^*,\ell(\vecz, \vecxi)) \\ 
    &= \,  - \frac{1-\exp(\alpha(\ell(\vecz, \vecxi)-t^*))}{\Expect_{\Prob}[\alpha\exp(\alpha(\ell(\vecz, \vecxi)-t^*))]} =  -(1/\alpha)\left(1-\exp(\alpha(\ell(\vecz, \vecxi)-t^*))\right),
\end{align*}
since $\Expect_\Prob[\exp(\alpha(\ell(\vecz, \vecxi)-t^*))]=1$.
Next, we substitute the value of $\text{IF}( \ell(\z,\vecxi)%
)$ and $\nabla_{t}h(t^*,\ell(\vecz, \vecxi))$ to obtain  the bias of a decision $t^*$:
\begin{equation*}
\begin{aligned}\delta_{\texttt{OIC}} \, = \,  &- \frac{1}{N}\Expect_{\Prob}[\nabla_{t} h(t^*,\ell(\vecz, \vecxi)) \text{IF}({\ell(z,\vecxi)}%
)] =\,\frac{1}{ N \alpha} \Expect_\Prob\left[(1-\exp(\alpha\ell(\vecz, \vecxi)-\alpha t^*))^2\right]\\
\, =\,&\frac{1}{N\alpha} \exp(-2 \alpha t^*)\Expect_\Prob\left[\left(\Expect_\Prob[\exp(\alpha\ell(\vecz, \vecxi))]-\exp(\alpha\ell(\vecz, \vecxi))\right)^2\right]\\
\, =\,&\frac{1}{N\alpha} \left(\frac{\Expect_\Prob\left[(\Expect_{\Prob}[\exp(\alpha \ell(\vecz, \vecxi))]-\exp(\alpha \ell(\vecz, \vecxi))\right)^2]}{\left(\Expect_{\Prob}[\exp(\alpha \ell(\vecz, \vecxi))]\right)^2}\right)\\
\, =\, & \frac{\text{Var}_{\Prob}(\exp(\alpha \ell(\vecz, \vecxi)))}{N\alpha (\mathbb{E}_{\Prob}[\exp(\alpha \ell(\vecz, \vecxi))])^2},
\end{aligned}
\end{equation*}
\removed{
\begin{equation*}
\begin{aligned}\delta_{\texttt{OIC}} \, = \,  &- \frac{1}{N}\Expect_{\Prob}[\nabla_{t} h(t^*,\ell(\vecz, \vecxi)) \text{IF}({\ell(z,\vecxi)}%
)]\\
\, =\,&\frac{1}{N} \left(\frac{\Expect_\Prob\left[(1-\exp(\alpha(\ell(\vecz, \vecxi)-t^*)))^2\right]}{\alpha \Expect_{\Prob}[\exp(\alpha (\ell(\vecz, \vecxi)-t^*))]}\right)\\
\, =\,&\frac{1}{N} \left(\frac{\exp(-2 \alpha t^*)\Expect_\Prob\left[\left(\Expect_\Prob[\exp(\alpha\ell(\vecz, \vecxi))]-\exp(\alpha\ell(\vecz, \vecxi))\right)^2\right]}{\alpha \exp(-\alpha t^*) \Expect_{\Prob}[\exp(\alpha (\ell(\vecz, \vecxi)))]}\right)\\
\, =\,&\frac{1}{N} \left(\frac{\Expect_\Prob\left[(\Expect_{\Prob}[\exp(\alpha \ell(\vecz, \vecxi))]-\exp(\alpha \ell(\vecz, \vecxi))\right)^2]}{\alpha \left(\Expect_{\Prob}[\exp(\alpha \ell(\vecz, \vecxi))]\right)^2}\right)\\
\, =\, & \frac{\text{Var}_{\Prob}(\exp(\alpha \ell(\vecz, \vecxi)))}{N\alpha (\mathbb{E}_{\Prob}[\exp(\alpha \ell(\vecz, \vecxi))])^2},
\end{aligned}
\end{equation*}
}
where the fourth equality comes from $\Expect_\Prob[\exp(\alpha(\ell(\vecz, \vecxi)))]=\exp(\alpha t^*)$.}
Since $\Prob$ is not known, \cite{Iyengar_Lam_2023} replace $\Prob$ with $\hat{\Prob}$ to obtain their estimator, \[\rho_{\model{OIC}}:= \hat{t}+\text{Var}_{\hat{\Prob}_N}(\exp(\alpha \ell(\vecz, \vecxi)))/(N\alpha (\mathbb{E}_{{\hat{\Prob}_N}}[\exp(\alpha \ell(\vecz, \vecxi))])^2).\]

\end{proof}
\subsection{Proof of Proposition \ref{prop:bias:saa:superlinear}}
\begin{proof}
From 
Lemma \ref{lemma_almst_sure}, we have that
\begin{align}\label{eq:bias:superlinear}
    \mathbb{E}[\rho_{\hat{\Prob}_N}(\ell(\z,\vecxi))] \leq \rho_\Prob(\ell(\z,\vecxi)) + \sigma\mathbb{E}[M_N] - (1/\alpha)\Lambda_{\Prob^{\xi}}(\alpha\sigma)    
\end{align}
    This proof relies on showing that the cumulant generating function grows faster than any linear function in $t$,   $\Lambda_{\Prob^{\xi}}(t) = \omega(t)$, due to the unbounded right tail of $\xi$. In particular, the unbounded tail implies that $\Prob(\xi\geq y)=\Prob((\ell(\z,\vecxi)-\mu)/\sigma\geq y) = \Prob(\ell(\z,\vecxi)\geq \mu+\sigma y)>0$ for all $y$. Let us consider an arbitrary $k>0$, and fix $\bar{x}:=2k$ and $\bar{t}:=-\log\left(\Prob(\xi\geq \bar{x})\right)/k$. One can confirm that for all $t\geq \bar{t}$, we have:
    \begin{equation}\label{eq:CGFomega}
    \begin{aligned}
       \Lambda_{\Prob^{\xi}}(t)=\log(\mathbb{E}[\exp(t\xi)]) &=\log(\mathbb{E}[\exp(t\xi)\mathbbm{1}_{\xi\geq \bar{x}}]+\mathbb{E}[\exp(t\xi)\mathbbm{1}_{\xi<\bar{x}}])\\&\;\geq \log(\exp(t\bar{x})\Prob(\xi\geq \bar{x}))\\&\;\;=2kt + \log(\Prob(\xi\geq \bar{x}))\\&\;\;\;\geq kt + k\bar{t} + \log(\Prob(\xi\geq \bar{x})) \\&\;\;\;= kt, 
    \end{aligned}
    \end{equation}
    where $\xi\sim\Prob^\xi$. Thus, $\Lambda_{\Prob^{\xi}}(t) = \omega(t)$.
    
    Equipped with $\Lambda_{\Prob^{\xi}}(t) = \omega(t)$, we can obtain our claim. Namely, given any $k>0$, we know that for $\bar{k}:=k+\mathbb{E}[M_N]$, there exists a $\bar{t}$ such that $\Lambda_{\Prob^{\xi}}(t) \geq \bar{k}t$ for all $t\geq \bar{t}$. Hence, letting $\bar{\sigma}:=\bar{t}/\alpha$, we thus get that for all $\sigma\geq \bar{\sigma}$:
\begin{align}\label{Lambda_diff_max}
 (1/\alpha)\Lambda_{\Prob^{\xi}}(\alpha\sigma) -  \sigma\mathbb{E}[M_N] &\geq  (1/\alpha)(k+\mathbb{E}[M_N])(\alpha \sigma) - \sigma\mathbb{E}[M_N] = k\sigma,
\end{align}
which combined with \eqref{eq:bias:superlinear} gives 
\[
\rho_{\Prob}(\ell(\vecz, \vecxi)) - \mathbb{E}[\rho_{\hat{\Prob}_N}(\ell(\vecz, \vecxi))] \geq (1/\alpha)\Lambda_{\Prob^{\xi}}(\alpha\sigma) -  \sigma\mathbb{E}[M_N]   \geq k\sigma.\]
Since $k>0$ was chosen arbitrarily, we conclude that Proposition \ref{prop:bias:saa:superlinear} holds.
\end{proof}

\subsection{Proof of Proposition \ref{prop:saa:normalLB}}
\begin{proof}
    Similar to the proof of Proposition \ref{prop:bias:saa:superlinear}, from Lemma \ref{lemma_almst_sure}, we have that:
    \begin{align*}
     \rho_{\Prob}(\ell(\vecz, \vecxi)) - \mathbb{E}[\rho_{\hat{\Prob}_N}(\ell(\vecz, \vecxi))] %
     &\geq   (1/\alpha)\Lambda_{\Prob^{\xi}}(\alpha\sigma) - \sigma\mathbb{E}[ M_N]= \alpha\sigma^2/2 - \sigma\mathbb{E}[M_N] =\Omega(\sigma^2),
         \end{align*}
         where the second to the last equality is obtained using the cumulant generating function of the standard normal distribution associated to $\xi$.
\end{proof}

\subsection{Proof of Proposition \ref{prop:saa:laplaceLB}}
\begin{proof}
    Similar to the proof of Proposition \ref{prop:bias:saa:superlinear}, from Lemma \ref{lemma_almst_sure}, we have that:
    \begin{align*}
       \rho_{\Prob}(\ell(\vecz, \vecxi)) - \mathbb{E}[\rho_{\hat{\Prob}_N}(\ell(\vecz, \vecxi))] %
       \geq   (1/\alpha)\Lambda_{\Prob^{\xi}}(\alpha\sigma) - \sigma\mathbb{E}[ M_N],
       \end{align*}
       where $\xi\sim\Prob^\xi$ follows a Laplace distribution with density $f(x):=(1/\sqrt{2})\exp(-\sqrt{2}|x|)$.
The exponential of the cumulant generating function of $\xi$ is given as follows:
\begin{align*}
\exp(\Lambda_{\Prob^{\xi}}(t))&=\frac{1}{\sqrt{2}}\int_{-\infty}^\infty \exp(tx)\exp(-\sqrt{2}|x|)dx=\frac{1}{\sqrt{2}}\left(\int_0^\infty \exp((t-\sqrt{2})x)dx + \int_0^\infty \exp((-t-\sqrt{2})x)dx\right)    \\
&=\frac{1}{\sqrt{2}} \left(\frac{1}{\sqrt{2}-t} +\frac{1}{t+\sqrt{2}}\right)=\frac{2}{2-t^2}
\end{align*}
where, the two integrals are finite only if $t^2<2$. %
Hence, $\Lambda_{\Prob^{\xi}}(\alpha \sigma)=\log(2)-\log(2-\alpha^2 \sigma^2)$ for all $\sigma^2 < 2/\alpha^2$, and otherwise infinite. Therefore, $\Lambda_{\Prob^{\xi}}(\alpha \sigma) \rightarrow \infty$ as $\sigma^2 \rightarrow 2/\alpha^2$, which gives:
\[ \rho_{\Prob}(\ell(\vecz, \vecxi)) - \mathbb{E}[\rho_{\hat{\Prob}_N}(\ell(\vecz, \vecxi))] \rightarrow \infty \text{ as } \mbox{Var}(\ell(\vecz, \vecxi)) \rightarrow 2/\alpha^2.\]
\end{proof}
\subsection{Proof of Proposition \ref{prop:saa:UB}}
\begin{proof}
    From Lemma \ref{lemma_almst_sure}, we have that:
\begin{equation}\label{eq:bias_cumulant}
\begin{aligned}
    \mathbb{E}[\rho_{\hat{\Prob}_N}(\ell(\z,\vecxi))] 
    &\geq \rho_{\Prob}(\ell(\z,\vecxi)) + \sigma\mathbb{E}[M_N] - \log(N)/\alpha-(1/\alpha)\Lambda_{\Prob^{\xi}}(\alpha\sigma)\\
    &\geq \rho_{\Prob}(\ell(\z,\vecxi)) + \sigma\mathbb{E}[M_N] - \log(N)/\alpha-(1/\alpha)
\nu_0^2 (\alpha \sigma)^2,
\end{aligned}
\end{equation}
where the second inequality follows from Lemma \ref{lemma:cgf:subgaussian:centered}. %
We rearrange to obtain:
\[\rho_{\Prob}(\ell(\vecz, \vecxi)) - \mathbb{E}[\rho_{\hat{\Prob}_N}(\ell(\vecz, \vecxi))] \leq 
\nu_0^2 \alpha \sigma^2 -\sigma \mathbb{E}[M_N] + \log(N)/\alpha=\mathcal O(\sigma^2). \]
Since $\rho_{\Prob}(\ell(\vecz, \vecxi)) \ge \mathbb{E}[\rho_{\hat{\Prob}_N}(\ell(\vecz, \vecxi))]$ from Jensen's inequality, we have:
\[|\rho_{\Prob}(\ell(\vecz, \vecxi)) - \mathbb{E}[\rho_{\hat{\Prob}_N}(\ell(\vecz, \vecxi))]| = \mathcal O(\sigma^2). \]
\end{proof}

 \begin{lemma}\label{lem:var_linear_underest_short}
    Let Assumption \ref{ass:meanVarControl} be satisfied, the loss $\ell(\vecz, \vecxi)$ be normally distributed, and $(\bar{\mu},\bar{\sigma})\in\mathcal{A}$ with $\bar{\sigma}>0$. Then, there exists $\alpha>0$ and some $C>0$, such that 
\[
 \rho^\alpha_{\Prob}(\ell(\vecz, \vecxi)) - \mathbb{E}[\rho^\alpha_{\hat{\Prob}_N}(\ell(\vecz, \vecxi))] \geq C{\alpha} \mbox{Var}(\ell(\vecz, \vecxi))
\]
for all $\z$ such that $\mbox{Var}(\ell(\vecz, \vecxi))\geq \bar{\sigma}{^2}$.
\end{lemma}
\begin{proof}
Let $\bar{\alpha}:=1$, then by Proposition~\ref{prop:saa:normalLB},  we have that for some $C_1>0$ and some $\sigma_1>0$:
\[\rho^{\bar{\alpha}}_{\Prob^\xi}(\mu+\sigma\xi) - \mathbb{E}[\rho^{\bar{\alpha}}_{\hat{\Prob}_N^\xi}(\mu+\sigma\xi)] \geq C_1\sigma^2\]
 for all $\sigma\geq \sigma_1$. Let $\alpha := %
 \sigma_1/\bar{\sigma}$ then for all $\z$ such that $\mbox{Var}(\ell(\z,\vecxi))\geq \bar{\sigma}^2$:
    \begin{align*}
    \rho^\alpha_{\Prob}(\ell(\vecz, \vecxi)) - \mathbb{E}[\rho^\alpha_{\hat{\Prob}_N}(\ell(\vecz, \vecxi))] &= \rho^\alpha_{\Prob^\xi}(\mu+\sigma\xi) - \mathbb{E}[\rho^\alpha_{\hat{\Prob}_N^\xi}(\mu+\sigma\xi)]\\
    &= \frac{1}{\alpha}\left(\rho^{\bar{\alpha}}_{\Prob^\xi}(\alpha(\mu+\sigma\xi)) - \mathbb{E}[\rho^{\bar{\alpha}}_{\hat{\Prob}_N^\xi}(\alpha(\mu+\sigma\xi))]\right) \\
    &\geq \frac{1}{\alpha} C_1 \left(\alpha\sigma\right)^2 = C_1\alpha \mbox{Var}(\ell(\vecz, \vecxi)),
    \end{align*}
    where $\mu:=\Expect[\ell(\z,\vecxi)]$ and $\sigma:=\sqrt{\mbox{Var}(\ell(\z,\vecxi))}$, and where we used the fact that $\alpha\sigma =%
    (\sigma_1/\bar{\sigma}) \bar{\sigma} = \sigma_1$.
\end{proof}
\subsection{Proof of Proposition \ref{prop:oic:delta}}

\begin{proof}
\removed{
From Popoviciu’s inequality \citep{popoviciu1935equations}, we have that
\begin{equation*}
  \Var_{\hat{\mathbb{P}}_N}(\exp(\alpha \ell(\z,\vecxi)))\le   \frac{(\max_{1\le i\le N}\exp(\alpha \ell(\z,\hat{\vecxi}_i))-\min_{1\le i\le N}\exp(\alpha \ell(\z,\hat{\vecxi}_i)))^2}{4} \le \frac{(\max_{1\le i\le N} \exp(\alpha \zeta_i))^2}{4}.
\end{equation*}
Dividing the above inequality by $(\sum_{i=1}^N\exp(\alpha \ell(\z,\hat{\vecxi}_i))/N)^2>0$, we obtain:
\[
\frac{\Var_{\hat{\mathbb{P}}_N}(\exp(\alpha \zeta))}{((1/N)\sum_{i=1}^N\exp(\alpha \zeta_i))^{\,2}}\leq \frac{\max_{1\le i\le N} \exp(\alpha \zeta_i))^2/4}{((1/N)\sum_{i=1}^N\exp(\alpha \zeta_i))^{\,2}}
\le \frac{(\max_{1\le i\le N}\exp(\alpha \zeta_i))^2/4}{(\max_{1\le i\le N}\exp(\alpha \zeta_i)/N)^2}=\frac{N^2}{4},
\]
where we exploited that $(\sum_{i=1}^N\exp(\alpha \zeta_i)/N)^2\ge (\max_{1\le i\le N}\exp(\alpha \zeta_i)/N)^2$.
}
From Popoviciu’s inequality \citep{popoviciu1935equations}, we have that
\begin{align*}
  \Var_{\hat{\mathbb{P}}_N}(\exp(\alpha \ell(\z,\vecxi)))&\le   \frac{(\max_{1\le i\le N}\exp(\alpha \ell(\z,\hat{\vecxi}_i))-\min_{1\le i\le N}\exp(\alpha \ell(\z,\hat{\vecxi}_i)))^2}{4} \\
  &\le \frac{(\max_{1\le i\le N} \exp(\alpha \ell(\z,\hat{\vecxi}_i)))^2}{4}\leq \frac{(\sum_{i=1}^N\exp(\alpha \ell(\z,\hat{\vecxi}_i)))^2}{4}\\
  &=(N^2/4)(\Expect_{\hat{\mathbb{P}}_N}[\exp(\alpha \ell(\z,\vecxi))])^2.
\end{align*}
where we exploited again that $\sum_{i=1}^N\exp(\alpha y_i)\ge \max_{1\le i\le N}\exp(\alpha y_i)$ for any set $\{y_i\}_{i=1}^N$.
Dividing the above inequality by $(\Expect_{\hat{\mathbb{P}}_N}[\exp(\alpha \ell(\z,\vecxi))])^2>0$, we obtain:
\[
\frac{\Var_{\hat{\mathbb{P}}_N}(\exp(\alpha \ell(\z,\vecxi)))}{(\Expect_{\hat{\mathbb{P}}_N}[\exp(\alpha \ell(\z,\vecxi))])^2}\leq %
\frac{N^2}{4}.
\]
Thus, we obtain the following almost sure bound on the two estimators:
\begin{align*}
\rho_{\model{OIC}}&=\rho_{\hat{\Prob}_N}(\ell(\z,\vecxi))+\frac{1}{\alpha N}\,
\frac{\mathrm{Var}_{\hat{\mathbb{P}}_N}(\exp(\alpha\ell(\z,\vecxi)))}
     {\big(\mathbb{E}_{\hat{\mathbb{P}}_N}[\exp(\alpha\ell(\z,\vecxi))]\big)^2}
\le \rho_{\hat{\Prob}_N}(\ell(\z,\vecxi))+\frac{N}{{4}\alpha}
\\
\rho_{\model{Delta}} &=\rho_{\hat{\Prob}_N}(\ell(\z,\vecxi))+ \frac{1}{2\alpha N}\,
\frac{\mathrm{Var}_{\hat{\mathbb{P}}_N}(\exp(\alpha\ell(\z,\vecxi)))}
     {\big(\mathbb{E}_{\hat{\mathbb{P}}_N}[\exp(\alpha\ell(\z,\vecxi))]\big)^2}
\le \rho_{\hat{\Prob}_N}(\ell(\z,\vecxi))+\frac{N}{{8}\alpha}.
\end{align*}
Clearly, $\rho_{\model{OIC}}\geq \rho_{\hat{\Prob}_N}(\ell(\z,\vecxi))$ and $\rho_{\model{Delta}}\geq \rho_{\hat{\Prob}_N}(\ell(\z,\vecxi))$. 

We will represent the $\model{Delta}$ and $\model{OIC}$ estimators in a unified way by introducing a generic estimator $\model{M}\in\{\model{Delta},\model{OIC}\}$. Specifically, we set $\kappa_{\model{Delta}}=8$ and $\kappa_{\model{OIC}}=4$. With this notation, we summarize our findings as:
\begin{align}\label{eq:ub:oic}
\rho_{\hat{\Prob}_N}(\ell(\z,\vecxi)) \leq \rho_{\model{M}}
\le \rho_{\hat{\Prob}_N}(\ell(\z,\vecxi))+\frac{N}{\kappa_{\model{M}} \alpha}.
\end{align}

Taking expectation on all sides of the two  inequalities in \eqref{eq:ub:oic} before also subtracting $\rho_{\Prob}(\ell(\vecz, \vecxi))$, we obtain:
\begin{align*}
\rho_{\Prob}(\ell(\vecz, \vecxi))  -\Expect[\rho_{\hat{\Prob}_N}(\ell(\vecz, \vecxi))]- \frac{N}{\kappa_{\model{M}}\alpha} \leq \rho_{\Prob}(\ell(\vecz, \vecxi)) - \mathbb{E}[\rho_{\model{M}}]  \leq  \rho_{\Prob}(\ell(\vecz, \vecxi))  -\Expect[\rho_{\hat{\Prob}_N}(\ell(\vecz, \vecxi))],
\end{align*}
which also implies that 
\[|\rho_{\Prob}(\ell(\vecz, \vecxi)) - \mathbb{E}[\rho_{\model{M}}]| \leq |\rho_{\Prob}(\ell(\vecz, \vecxi))  -\Expect[\rho_{\hat{\Prob}_N}(\ell(\vecz, \vecxi))]| + \frac{N}{\kappa_{\model{M}}\alpha}\]
Since the term $N/(k_{\model{M}}\alpha)$ is constant with respect to $(\mu,\sigma)$, we thus conclude that the asymptotic properties of  $\rho_{\Prob}(\ell(\vecz, \vecxi))  -\Expect[\rho_{\hat{\Prob}_N}(\ell(\vecz, \vecxi))]$ and $|\rho_{\Prob}(\ell(\vecz, \vecxi))  -\Expect[\rho_{\hat{\Prob}_N}(\ell(\vecz, \vecxi))]|$ carry through to $\rho_{\Prob}(\ell(\vecz, \vecxi)) - \mathbb{E}[\rho_{\model{M}}]$ and $|\rho_{\Prob}(\ell(\vecz, \vecxi)) - \mathbb{E}[\rho_{\model{M}}]|$ respectively. Namely, propositions \ref{prop:bias:saa:superlinear}, \ref{prop:saa:normalLB}, \ref{prop:saa:laplaceLB}, and \ref{prop:saa:UB} hold with $\rho_{\hat{\Prob}_N}(\ell(\vecz, \vecxi))$ replaced with both $\rho_{\model{Delta}}$ or $\rho_{\model{OIC}}$.

\removed{
\noindent
We have that $\rho_{\model{M}}(\ell(\vecz, \vecxi))\geq \rho_{\hat{\Prob}_N}(\ell(\vecz, \vecxi)) \implies \Expect[\rho_{\model{M}}(\ell(\vecz, \vecxi))]\geq \Expect[\rho_{\hat{\Prob}_N}(\ell(\vecz, \vecxi))]$. We take expectations on both sides in \eqref{eq:ub:oic} to bound the gap between the expected value of empirical risk estimator and estimator $\model{M}$:
\begin{align*}
     &\Expect[\rho_{\model{M}}(\ell(\vecz, \vecxi))]\le \Expect[\rho_{\hat{\Prob}_N}(\ell(\vecz, \vecxi))]+\frac{N}{\kappa_\model{M}\alpha} \\ &\implies |\Expect[\rho_{\model{M}}(\ell(\vecz, \vecxi))]- \Expect[\rho_{\hat{\Prob}_N}(\ell(\vecz, \vecxi))]|\le \frac{N}{\kappa_{\model{M}}\alpha}. 
\end{align*}
Equipped with the above inequality and with the estimation error of the empirical risk estimator equal to $\mathcal O(\sigma^2)$ for subgaussian random variables,  the estimation error of \model{M} is also $\mathcal O(\sigma^2)$ using the triangle inequality:
\begin{align*}
|\rho_{\Prob}(\ell(\vecz, \vecxi)) - \mathbb{E}[\rho_{\model{M}}(\ell(\vecz, \vecxi))]| &\le   |\rho(\ell(\vecz, \vecxi))-\Expect[\rho_{\hat{\Prob}_N}(\ell(\vecz, \vecxi))]-\Expect[\rho_{\model{M}}(\ell(\vecz, \vecxi))]+ \Expect[\rho_{\hat{\Prob}_N}(\ell(\vecz, \vecxi))]|\\
  &\leq  |\rho(\ell(\vecz, \vecxi))-\Expect[\rho_{\hat{\Prob}_N}(\ell(\vecz, \vecxi))]|+|\Expect[\rho_{\model{M}}(\ell(\vecz, \vecxi))]- \Expect[\rho_{\hat{\Prob}_N}(\ell(\vecz, \vecxi))]|\\
  & \leq \mathcal O(\sigma^2)+\frac{N}{\kappa_{\model{M}}\alpha} = \mathcal O(\sigma^2).
\end{align*}
}
\end{proof}

\subsection{Proof of Proposition \ref{prop:bs:dbs}}
\begin{proof}
 Starting with the first-level bootstrap, let $\tilde{i}(j)_{j=1}^N$ denote the sampled indices with $\tilde{i}(j)\sim U(\{1,\dots,N\})$ independently, which is responsible for the set $\{\hat{\vecxi}_{\tilde{i}(j)}\}_{j=1}^N$ behind the  empirical distribution $\hat{\Prob}_{N,N}$. Conditional on the first-level bootstrap procedure, consider the second-level bootstrap sample  $\{\hat{\vecxi}_{\tilde{i}(\tilde{j}(k))}\}_{k=1}^N$, with each $\tilde{j}(k)\sim U(\{1,\dots,N\})$ independently and independently of $\tilde{i}$ responsible for the second-level bootstrap empirical distribution $\hat{\mathbb{P}}_{N,N, N}$. %
    
    From Lemma \EDmodified{\ref{recur_lemma}}, %
    we have that almost surely:
        \begin{equation}\label{eq:bound_dbs}
    \begin{aligned}
         &\mu+\sigma M_{N} -\log(N)/\alpha \le 
         \rho_{\hat{\mathbb{P}}_{N}}(\ell(\vecz, \vecxi)) \le \mu+\sigma M_{N},\\
    &\mu+\sigma M_{N,N} - \log(N)/\alpha \le 
    \rho_{\hat{\mathbb{P}}_{N,N}}(\ell(\vecz, \vecxi)) \le \mu+\sigma M_{N,N} ,\\
    &\mu+\sigma M_{N,N, N} -\log(N)/\alpha  \le
    \rho_{\hat{\mathbb{P}}_{N,N, N}}(\ell(\vecz, \vecxi)) \le \mu+\sigma M_{N,N, N},
    \end{aligned}
        \end{equation}
 where $M_N := \max_{1\le i\le N} \hat{\xi}_i$, with $\hat{\xi}_i:=(\ell(\z,\hat{\vecxi}_i)-\mu)/\sigma$, for all $i=1,\dots,N$, and independently drawn from $\Prob^\xi$, while
 \begin{align*}
     M_{N,N}&:= \max_{1\leq j\leq N} (\ell(\z,\hat{\vecxi}_{\tilde{i}(j)})-\mu)/\sigma = \max_{1\leq j\leq N}\hat{\xi}_{\tilde{i}(j)}\\
     M_{N,N,N}&:= \max_{1\leq k\leq N} (\ell(\z,\hat{\vecxi}_{\tilde{i}(\tilde{j}(k))})-\mu)/\sigma = \max_{1\leq k\leq N}\hat{\xi}_{\tilde{i}(\tilde{j}(k))}
 \end{align*}
  
  Both the bootstrap estimator with $(\nu_1, \nu_2, \nu_3)=(2, 1, 0)$ and the double bootstrap estimator with $(\nu_1, \nu_2, \nu_3)=(3, 3, 1)$ can be written in a unified form by introducing  a generic estimator $\model{M}$:
\begin{align}\label{eq:unify:M}
\rho_{\texttt{M}}:=\nu_1\rho_{\hat{\Prob}_N}(\ell(\vecz, \vecxi)) -\nu_2\Expect[\rho_{\hat{\mathbb{P}}_{N,N}}(\ell(\vecz, \vecxi))|\hat{\Prob}_N] +\nu_3\Expect[\rho_{\hat{\mathbb{P}}_{N,N, N}}(\ell(\vecz, \vecxi))|\hat{\Prob}_{N}],
\end{align}
with $\nu_1\geq 1$, $\nu_2\geq0$, $\nu_3\geq 0$ and $\nu_1-\nu_2+\nu_3=1$. Based on \eqref{eq:bound_dbs}, 
we have almost surely that
    \begin{align}\label{eq:bs:bound}
   \rho_{\texttt{M}}%
        &\leq \rho_{\hat{\mathbb{P}}_{N}}(\ell(\vecz, \vecxi))+(\nu_1-1)\rho_{\hat{\Prob}_N}(\ell(\vecz, \vecxi))- \nu_2\Expect[\mu+\sigma M_{N,N} - \log(N)/\alpha|\hat{\Prob}_N] +\nu_3\Expect[\mu+\sigma M_{N,N,N}|\hat{\Prob}_N] \notag\\
        &\leq \rho_{\hat{\mathbb{P}}_{N}}(\ell(\vecz, \vecxi))+\mu(\nu_1-\nu_2+\nu_3-1) -\nu_2\sigma  \Expect[M_{N,N} |\hat{\Prob}_N] + \nu_3\sigma  \Expect[M_{N,N, N} |\hat{\Prob}_N] \notag\\&\qquad+ \nu_2\log(N)/\alpha +(\nu_1-1)\sigma M_{N}\notag\\
        &= \rho_{\hat{\mathbb{P}}_{N}}(\ell(\vecz, \vecxi)) +((\nu_1-1) M_{N} + \nu_3 \Expect[M_{N,N, N} |\hat{\Prob}_N] -\nu_2  \Expect[M_{N,N} |\hat{\Prob}_N])\sigma + \nu_2\log(N)/\alpha.
     \end{align}
        We take expectations on both sides in the above, subtract $\rho(\ell(\vecz, \vecxi))$ from both sides and rearrange to obtain:
        \begin{align}\label{eq:rhoDBSLB}
         \rho_{\Prob}&(\ell(\vecz, \vecxi)) - \mathbb{E}[\rho_{\texttt{M}}]\geq \notag\\
         &\rho(\ell(\vecz, \vecxi))-\Expect[\rho_{\hat{\mathbb{P}}_{N}}(\ell(\vecz, \vecxi))] -((\nu_1-1) \Expect[M_{N}] -\nu_2  \Expect[M_{N,N}] + \nu_3 \Expect[M_{N,N, N}] )\sigma - \nu_2\log(N)/\alpha,
        \end{align}
        which add a term that is only linear in $\sigma$ to the bias obtained from the empirical estimator.
        Equipped with the asymptotic properties of the empirical risk estimator, namely i) $\rho(\ell(\vecz, \vecxi))  -\Expect[\rho_{\hat{\Prob}_N}(\ell(\vecz, \vecxi))] =\omega(\sigma)$ for distributions with unbounded right tail, ii) $\rho(\ell(\vecz, \vecxi))  -\Expect[\rho_{\hat{\Prob}_N}(\ell(\vecz, \vecxi))] =\Omega(\sigma)$ for normal distribution,  and iii) $\rho(\ell(\vecz, \vecxi))  -\Expect[\rho_{\hat{\Prob}_N}(\ell(\vecz, \vecxi))]\rightarrow \infty$ as $\sigma^2 \rightarrow 2/\alpha^2$ for Laplace distribution, we have that    Propositions \ref{prop:bias:saa:superlinear}, \ref{prop:saa:normalLB}, and \ref{prop:saa:laplaceLB} also hold for \model{BS} and \model{DBS} estimators.

From \eqref{eq:bound_dbs}, we also have that:
\begin{align*}%
   \rho_{\texttt{M}} &\geq \rho_{\hat{\mathbb{P}}_{N}}(\ell(\vecz, \vecxi))+(\nu_1-1)(\mu+\sigma M_N - \log(N)/\alpha)- \nu_2\Expect[\mu+\sigma M_{N,N} |\hat{\Prob}_N] \notag\\&\qquad+\nu_3\Expect[\mu+\sigma M_{N,N,N}- \log(N)/\alpha|\hat{\Prob}_N] \notag\\
   &= \rho_{\hat{\mathbb{P}}_{N}}(\ell(\vecz, \vecxi))+((\nu_1-1)M_N-\nu_2 \Expect[M_{N,N}|\Prob_N] + \nu_3 \Expect[M_{N,N,N}|\Prob_N])\sigma - (\nu_1+\nu_3-1)\log(N)/\alpha.
\end{align*}
This implies that
\begin{align*}
\rho_{\Prob}&(\ell(\vecz, \vecxi)) -\Expect[\rho_{\model{M}}]\\
&\leq \rho_{\Prob}(\ell(\vecz, \vecxi)) - \Expect[\rho_{\hat{\mathbb{P}}_{N}}(\ell(\vecz, \vecxi))] -((\nu_1-1)\Expect[M_N]-\nu_2 \Expect[M_{N,N}] + \nu_3 \Expect[M_{N,N,N}])\sigma \\&\qquad +(\nu_1+\nu_3-1)\log(N)/\alpha,     
\end{align*}
and thus combining with \eqref{eq:rhoDBSLB}, we get that
\[|\rho_{\Prob}(\ell(\vecz, \vecxi)) - \mathbb{E}[\rho_{\model{M}}]|\leq |\rho_{\Prob}(\ell(\vecz, \vecxi)) - \Expect[\rho_{\hat{\mathbb{P}}_{N}}(\ell(\vecz, \vecxi))]| + c_1\sigma + c_2,\]
where $c_1:=|(\nu_1-1)\Expect[M_N]-\nu_2 \Expect[M_{N,N}] + \nu_3 \Expect[M_{N,N,N}]|$ and $c_2:=(\nu_1+\nu_2+
\nu_3-1)\log(N)/\alpha$. We conclude that when the loss has subgaussian tails, since Proposition \ref{prop:saa:UB} states that $|\rho_{\Prob}(\ell(\vecz, \vecxi)) - \Expect[\rho_{\hat{\mathbb{P}}_{N}}(\ell(\vecz, \vecxi))]|=\mathcal O(\sigma^2)$, it must be that:
\[|\rho_{\Prob}(\ell(\vecz, \vecxi)) - \mathbb{E}[\rho_{\model{M}}]|\leq \mathcal O(\sigma^2) + \mathcal O(\sigma) = \mathcal O(\sigma^2).\]

Thus both \model{BS} and $\model{DBS}$ estimators satisfy Proposition \ref{prop:saa:UB}.

\end{proof}

\subsection{Proof of Proposition \ref{thm:loocv_overestimate}}
\begin{proof}
We have almost surely that:
  \begin{align*}
\rho_{\model{LOOCV}}
&= \frac{1}{N}\sum_{i=1}^N \left( \hat{t}_{-i} + \frac{1}{\alpha}\Big(\exp\big(\alpha(\ell(\z,\hat{\vecxi}_i)-\hat{t}_{-i})\big)-1\Big)\right)\\
&\geq \frac{1}{N}\sum_{i=1}^N \left( \mu+\sigma M_{N}^{-i} - \frac{\log(N-1)}{\alpha} + \frac{1}{\alpha}\Big(\exp\big(\alpha(\mu+\sigma\hat{\xi}_i-\mu-\sigma M_{N}^{-i})\big)-1\Big)\right)\\
&=\mu- \frac{\log(N-1)}{\alpha}-\frac{1}{\alpha}+\sigma\frac{1}{N}\sum_{i=1}^N M_{N}^{-i}  + \frac{1}{\alpha N}\sum_{i=1}^N\exp\big(\alpha\sigma(\hat{\xi}_i- M_{N}^{-i})\big)\\
&\geq\mu- \frac{\log(N-1)}{\alpha}-\frac{1}{\alpha}+\sigma\frac{1}{N}\sum_{i=1}^N M_{N}^{-i}  + \frac{1}{\alpha N}\max_{1\le i\le N}\exp\big(\alpha\sigma(\hat{\xi}_i- M_{N}^{-i})\big)\\
&=\mu- \frac{\log(N-1)}{\alpha}-\frac{1}{\alpha}+\sigma\frac{1}{N}\sum_{i=1}^N M_{N}^{-i}  + \frac{1}{\alpha N}\exp\big(\alpha\sigma \mathfrak{M}_N\big),   
  \end{align*}
  where $\hat{\xi}_i:=(\ell(\z,\hat{\vecxi}_i)-\mu)/\sigma$, $M_N^{-i}:=\max_{j\neq i} \hat{\xi}_j$, $\mathfrak{M}_N:=\max_i (\hat{\xi}_i- M_{N}^{-i})$, and the inequality comes from  Lemma~\ref{recur_lemma}.

  This implies that:
  \begin{align*}
\Expect[\rho_{\model{LOOCV}}]
&\geq \mu- \frac{\log(N-1)}{\alpha}-\frac{1}{\alpha}+\sigma\Expect[\frac{1}{N}\sum_{i=1}^N M_{N}^{-i}]  + \Expect[\frac{1}{\alpha N}\exp\big(\alpha\sigma \mathfrak{M}_N\big)]\\
&\geq \mu- \frac{\log(N-1)}{\alpha}-\frac{1}{\alpha}+\sigma\Expect[\frac{1}{N}\sum_{i=1}^N M_{N}^{-i}]  + \frac{1}{\alpha N} \exp(\alpha\sigma \Expect[\mathfrak{M}_N])\\
&= \mu + \Omega(\exp(k_1\sigma))
  \end{align*}
  for $k_1:=\alpha\Expect[\mathfrak{M}_N] >0$, and where we applied Jensen's inequality.  
  
  Note that $\Expect[\mathfrak{M}_N]>0$ is due to $\mbox{Var}(\xi)=1$. %
  Specifically, the proof will identify a threshold $\bar y$ such that, with positive probability, exactly one of the $N$ observations exceeds $\bar y$ while the remaining $N-1$ do not. %
  Indeed, a strictly positive variance implies that there must be some $\bar{y}\in\mathbb{R}$ such that $\Prob(\xi\leq \bar{y})\in(0,\,1)$, otherwise $\xi$ is deterministic with a variance of zero. Furthermore, we can exploit the fact that $\mathfrak{M}_N=\hat{\xi}_{(1)}-\hat{\xi}_{(2)}\geq 0$ with $\hat{\xi}_{(i)}$ the $i$-th largest sample in the set and observe that:
  \[\Prob(\mathfrak{M}_N >0) = \Prob(\hat{\xi}_{(1)}>\hat{\xi}_{(2)}) \geq \Prob(\hat{\xi}_{(1)}>\bar{y} \,\&\, \hat{\xi}_{(2)}\leq \bar{y}) >0,\]
  since $\Prob(\hat{\xi}_{(1)}>\bar{y} \,\&\, \hat{\xi}_{(2)}\leq \bar{y})$ is the probability that $N-1$ success is observed among $N$ experiments with success probability of $\Prob(\xi\leq \bar{y})\in(0,\,1)$. We thus can conclude that
  \[\Expect[\mathfrak{M}_N] \geq \Expect[\mathfrak{M}_N|\mathfrak{M}_N >0] \Prob(\mathfrak{M}_N >0)>0.\]
Finalizing our proof %
for the lower bound on estimation error, %
we obtain that
  \begin{align*}
      \Expect[\rho_{\model{LOOCV}}] &= \mu + \Omega(\exp(k_1\sigma)) + \rho_\Prob(\ell(\z,\vecxi))-\rho(\ell(\z,\vecxi))\\
      &= \mu + \Omega(\exp(k_1\sigma)) + \rho_\Prob(\ell(\z,\vecxi))-\mu -\frac{1}{\alpha}\Lambda_{\Prob_\xi}(\alpha\sigma)\\
      &\geq \Omega(\exp(k_1\sigma)) + \rho_\Prob(\ell(\z,\vecxi)) - \frac{\nu_0(\alpha\sigma)^2}{\alpha}\\
      &= \Omega(\exp(k_1\sigma)) + \rho_\Prob(\ell(\z,\vecxi)),
      \end{align*}
 where first inequality follows from  Lemma \ref{lemma:cgf:subgaussian:centered}. %
 Thus, we obtain the estimation error of the \model{LOOCV} estimator \[|\rho_{\Prob}(\ell(\vecz, \vecxi)) - \mathbb{E}[\rho_{\model{LOOCV}}]|\geq\Expect[\rho_{\model{LOOCV}}] -\rho_\Prob(\ell(\vecz, \vecxi)) = \Omega(\exp(k_1\sigma)). \]
 
 Next, we show that the estimation error has an upper bound that is exponential in the variance $\sigma^2$, that is, $ | \rho(\ell(\vecz, \vecxi)) - \Expect[\rho_{\model{LOOCV}}]| = \mathcal O(\exp(k_2\sigma^2))$ for some $k_2>0$.  We have almost surely that:
  \begin{align*}
\rho_{\model{LOOCV}}
&= \frac{1}{N}\sum_{i=1}^N  \left(\hat{t}_{-i} + \frac{1}{\alpha}\Big(\exp\big(\alpha(\ell(\vecz, \hat{\vecxi}_i)-\hat{t}_{-i})\big)-1\Big)\right)\\
&\leq \frac{1}{N}\sum_{i=1}^N  \left(\mu+\sigma M_{N}^{-i}+ \frac{1}{\alpha}\exp\big(\alpha(\ell(\vecz, \hat{\vecxi}_i)-\hat{t}_{-i})\big)\right) \\
&\leq \mu+\frac{1}{N}\sum_{i=1}^N  \sigma M_{N}^{-i}  + \frac{N-1}{\alpha N}\sum_{i=1}^N\exp\big(\alpha(\sigma(\hat{\xi}_i- M_{N}^{-i}))\big),
  \end{align*}
  where the first inequality comes from  Lemma~\ref{recur_lemma}. %
  The upper bound in the second inequality follows by combining  Lemma~\ref{recur_lemma} with $\ell(\vecz, \hat{\vecxi}_i) = \mu+\sigma \hat{\xi}_i$ yielding $\ell(\vecz, \hat{\vecxi}_i) - \hat t_{-i}
\le \sigma(\hat\xi_i - M_N^{-i}) + \log(N-1)/\alpha.$

Let $i^*$ be the index of the largest observation in $\{\hat\xi_1, \cdots, \hat\xi_N\}$, hence $ \hat\xi_{(1)}=\hat\xi_{i^*}$, $\hat\xi_{(2)} = M_N^{-i^*}$. Then for any $i\neq i^*$, we have $\hat\xi_i \leq M_N^{-i}$ and hence $\exp(\alpha\sigma(\hat\xi_i - M_N^{-i})) \le 1$. Thus, we have that 
\begin{align*} \sum_{i=1}^N &\exp(\alpha\sigma(\hat\xi_i - M_N^{-i})) = \sum_{i\neq i^*} \exp(\alpha\sigma(\hat\xi_i - M_N^{-i})) + \exp(\alpha\sigma(\hat\xi_{(1)} - \hat\xi_{(2)}))
\le
N-1 + \exp(\alpha\sigma \mathfrak M_N)
\end{align*}
Moreover, one can show that:
\begin{align*}
    \exp(\alpha\sigma  \mathfrak M_N) &= \exp(\alpha\sigma ( \hat\xi_{(1)} - \hat\xi_{(2)}))
\le \exp(\alpha\sigma (|\hat\xi_{(1)}| + |\hat\xi_{(2)}|))
\le  \exp(2\alpha \sigma \max_{1\le j\le N}|\hat\xi_j|)\\
&\leq \sum_{j=1}^N \exp(2\alpha \sigma |\hat\xi_j|) \leq \sum_{j=1}^N (\exp(2\alpha \sigma \hat\xi_j)+\exp(-2\alpha \sigma \hat\xi_j))
\end{align*}

Thus, we have that:
\begin{align*}
\mathbb E[\rho_{\model{LOOCV}}] &\leq \mu
+ \sigma\,\mathbb E\Big[\frac{1}{N}\sum_{i=1}^N M_N^{-i}\Big]
+ \frac{(N-1)^2}{\alpha N}
+ \frac{N-1}{\alpha N}\Expect[\exp(\alpha \sigma \mathfrak M_N)]
\\
&\leq\mu + \sigma\,\mathbb E\Big[\frac{1}{N}\sum_{i=1}^N M_N^{-i}\Big]
+ \frac{(N-1)^2}{\alpha N}
+ \frac{(N-1)}{\alpha N}\sum_{j=1}^N \mathbb E\big[\exp(2\alpha\sigma \hat\xi_j)+\exp(-2\alpha\sigma \hat\xi_j)\big]\\
&\leq\mu
+ \sigma\,\mathbb E\Big[\frac{1}{N}\sum_{i=1}^N M_N^{-i}\Big]
+ \frac{(N-1)^2}{\alpha N}
+ \frac{2(N-1)}{\alpha}\exp(4\nu_0^2\alpha^2\sigma^2)\\
&
= \mu + \mathcal{O}(\exp(k_2\sigma^2)),
\end{align*}
with $k_2:=4\nu_0^2\alpha^2>0$ and where
 the third inequality follows from Lemma \ref{lemma:cgf:subgaussian:centered}.
Finally, we can obtain the upper bound on the estimation error of the $\model{LOOCV}$ estimator:
\begin{align*}
    |\rho_{\Prob}(\ell(\vecz, \vecxi)) - \mathbb{E}[\rho_{\model{LOOCV}}]|&= \mathbb{E}[\rho_{\model{LOOCV}}]- \rho_{\Prob}(\ell(\vecz, \vecxi))= \mu + \mathcal{O}(\exp(k_2\sigma^2)) - \mu - (1/\alpha)\Lambda_{\Prob^\xi}(\sigma\alpha)\\
    &\leq \mathcal{O}(\exp(k_2\sigma^2)), 
\end{align*}
where the first equality follows from \eqref{eq:loocvposBias}, and where we exploited $\Lambda_{\Prob^\xi}(t)=\log(\Expect_{\Prob^\xi}[\exp(t\xi)])\geq \Expect_{\Prob^\xi}[\log(\exp(t\xi))]=t\Expect_{\Prob^\xi}[\xi]=0$ by the concavity of the logarithm function.
\end{proof}

\section{Proof of results in Section \ref{sec:bias_correct}}\label{append:bias_correct}

We use the notation $\ell(\vecz, \vecxi)=\mu+\sigma\xi$ for some $(\mu,\sigma)\in\mathbb{R}\times\mathbb{R}_+$ and some $\xi\sim \Prob^\xi$ with mean of 0 and variance of 1. 
The proof of Theorem \ref{thm:our_estimator_overestimate_omega} relies on the following 
lemmas.

\subsection{\revised{Some useful lemmas}}
 \begin{lemma}\label{lemma:eq:tail:cgf:lighter-than-gauss}
             If a random variable $X$ satisfies Definition~\ref{def:lighter-than-gauss}, then it is a subgaussian random variable, that is, there exists a constant $\mathfrak{c}_2>0, %
             $ such that
\[\mathbb P(|X|\ge a)\le  2\exp(-a^2/\mathfrak{c}_2^2)\; \forall a\ge 0.\]%
    \end{lemma}
    \begin{proof}

Fix $G>0$, $C>0$, and $q>2$ from Definition \ref{def:lighter-than-gauss}, and let $\tilde G:=\max(G, 1)$,  $a_0:=\max(1, (2\log\tilde G/C)^{1/q})$, and 
$\mathfrak{c}_2^2:=\max(2/C, a_0^2/\log 2)$.
If $a\ge a_0$, then $\log\tilde G\le (C/2)a_0^{q}\le (C/2)a^{q}$, hence
\[\Prob(|X|\ge a)%
\le \tilde G \exp(-C a^{q})\le \exp(-(C/2)a^{q})\le \exp(-(C/2)a^{2})
\le \exp(-a^{2}/\mathfrak{c}_2^{2})\le 2\exp(-a^{2}/\mathfrak{c}_2^{2}),\]
where we used $a\ge1\Rightarrow a^{q}\ge a^{2}$ and $\mathfrak{c}_2^2\ge 2/C$.
If $0\le a\le a_0$, then $\Prob(|X|\ge a)\le1$ and $\mathfrak{c}_2^2\ge a_0^2/\log 2$ gives
$2\exp(-a^{2}/\mathfrak{c}_2^{2})\ge 2\exp(-a_0^{2}/\mathfrak{c}_2^{2})\ge1$.
Therefore $\Prob(|X|\ge a)\le 2\exp(-a^{2}/\mathfrak{c}_2^{2})$ for all $a\ge0$.
    \end{proof}%
    \begin{lemma}\label{lemma:equivalence:mgf_tails}
    
If a random variable $X$ has lighter-than-Gaussian tails (see Definition \ref{def:lighter-than-gauss}), then %
     there exists a $\bar{p}\in(1,2)$ and a $\nu>0$ such that $\Expect[\exp(tX)]\leq 2\exp(\nu t^{\bar{p}})$ for all $t\geq 0$.
    \end{lemma}
    \begin{proof}
We have:
\[
\mathbb{P}(|X|>a)
\le G \exp\bigl(-Ca^q\bigr).
\]

We now show that this tail bound implies $\mathbb{E}[\exp(tX)]\le 2\exp(\nu t^{\bar{p}})$ for some ${\bar{p}}\in(1,2)$ and $\nu>0$.

\textbf{Step 1: MGF bound.} %
 For $t \ge 0$, using the layer cake representation of a non-negative, real-valued measurable function $\exp(tX)$ and the tail bound:
\begin{align*}
\mathbb{E}[\exp(tX)]
&= \int_0^{\infty} \mathbb{P}(\exp(tX)> x) \,dx = \int_0^{\infty} \mathbb{P}(tX >\log x) \,dx 
= t\int_{-\infty}^{\infty} \mathbb{P}(X > a) \exp(ta) \,da %
\\
&= t\int_{-\infty}^{0} \mathbb{P}(X > a) \exp(ta) \,da+t\int_{0}^{\infty} \mathbb{P}(X > a) \exp(ta) \,da\\
&\leq  t\int_{-\infty}^0 1 \exp(ta)da +t\int_0^\infty \mathbb{P}(|X|>a) \exp(ta) \,da  \\
&\leq  1+ Gt\int_0^\infty \exp\bigl(ta-Ca^{q}\bigr)\,da,
\end{align*}
where we substituted $x=\exp(ta)$ to obtain the third equality.

\textbf{Step 2: Young's inequality.}
Let ${\bar{p}} = q/(q-1)$, so that $1/{\bar{p}}+1/q=1$ and ${\bar{p}} \in (1,2)$ since $q>2$. By Young's inequality, $\eta \gamma \le \eta^{{\bar{p}}}/{\bar{p}} + \gamma^{q}/q$ for any $\eta>0$ and $\gamma>0$. Letting $\eta:=t/C^{1/q}$ and $\gamma:=C^{1/q}a$:
$ta \le \frac{t^{\bar{p}}}{\bar{p} C^{{\bar{p}}/q}}+\frac{Ca^q}{q} \implies ta -Ca^q \leq \frac{t^{\bar{p}}}{\bar{p} C^{{\bar{p}}/q}}+Ca^q(1/q-1)$. %
From this, we get 
\[Gt\int_0^\infty \exp(ta-Ca^q)\,da
\le Gt \exp\Bigl(\frac{t^{\bar{p}}}{{\bar{p}}\,C^{{\bar{p}}/q}}\Bigr)\int_0^\infty \exp(-C(1-1/q)a^q)da=G C_2 t\exp(C_1 t^{\bar{p}}),\]
with  %
$C_1:=(\bar{p}C^{{\bar{p}}/q})^{-1}>0$, $C_2:=\int_0^\infty \exp\left(-C_3a^q\right)da$, and $C_3:=C(1-1/q)>0$.%

\textbf{Step 3: $C_2$ is finite}. %
To verify that the integral in $C_2$ is finite, substitute $u = C_3a^q>0$, so that $a = C_3^{-1/q}u^{1/q}$ and $da = \frac{u^{(1/q)-1}}{qC_3^{1/q}}du$:
\begin{align*}
C_2=\int_0^\infty \exp\!\left(-C_3a^q\right)\,da
&= \frac{1}{qC_3^{1/q}}\int_0^\infty \exp(-u)u^{1/q-1}\,du  = \frac{1}{qC_3^{1/q}}\Gamma(1/q) < \infty,
\end{align*}
where the Gamma function $\Gamma(1/q)$ is finite since $1/q > 0$.
Thus, we have that:
\[\mathbb{E}[\exp(tX)]\leq 1+GC_2t\exp(C_1t^{\bar{p}}) = 1+ \tilde Gt\exp(C_1 t^{\bar{p}})\]
where $\tilde G = G C_2$.

\textbf{Step 4: Absorbing the linear term.}
We will show that $\tilde Gt\exp(C_1 t^{\bar{p}}) \leq \exp(\nu t^{\bar{p}})$ with $\nu:=\tilde G^{\bar{p}}+C_1>0$. First, in the case that $t\in [0,1/\tilde G]$, we have
\[\tilde Gt\exp(C_1t^{\bar{p}})\leq \exp(C_1 t^{\bar{p}}) \leq \exp((\tilde G^{\bar{p}}+C_1)t^{\bar{p}}).\]
Next, if $t\geq 1/\tilde G$, then
\[\tilde Gt\exp(C_1t^{\bar{p}})\leq  \exp(\tilde Gt+C_1 t^{\bar{p}})\leq \exp((\tilde Gt)^{\bar{p}}+C_1 t^{\bar{p}}) = \exp((\tilde G^{\bar{p}}+C_1)t^{\bar{p}}).\]
where we first exploit the convexity of $f(x):=\exp(x)\geq f(0)+xf'(0)=1+x\geq x$, and later the convexity of $f(x):=x^{\bar{p}}\geq f(1)+(x-1)f'(1) = 1+{\bar{p}}(x-1)\geq 1+x-1=x$ when $x\geq 1$, since ${\bar{p}}>1$. Thus, we have that $\Expect[\exp(tX)]\leq  1+\exp(\nu t^{\bar{p}}) \leq 2\exp(\nu t^{\bar{p}}) \;\forall t\geq 0$, where the second inequality is due to $\exp(\nu t^{\bar{p}})\geq 1$ for $t\ge 0$ and $\nu>0$.
\end{proof}
    
\begin{lemma}\label{thm:biasCorrectBound}
    Let $\hat{\xi}_i:=(\ell(\z,\hat{\vecxi}_i)-\mu)/\sigma$ for all $i=1,\dots,N$. Given that Property \ref{property:affEquiFit} is satisfied, then, almost surely, we have that
    \[\rho_{\hat{\Prob}_N}(\ell(\z,\vecxi))+\delta_N(\Q_N) \geq \mu - \log(N)/\alpha +\sigma (M_N - \median(M_{N,\mathbb{Q}_N^\xi}|\mathbb{Q}_N^\xi)) + (1/\alpha)\Lambda_{\mathbb{Q}_N^\xi}(\sigma\alpha)\]
    and
    \[\rho_{\hat{\Prob}_N}(\ell(\z,\vecxi))+\delta_N(\Q_N) \leq \mu + \log(N)/\alpha +\sigma (M_N - \median(M_{N,\mathbb{Q}_N^\xi}|\mathbb{Q}_N^\xi)) + (1/\alpha)\Lambda_{\mathbb{Q}_N^\xi}(\sigma\alpha)\]
    where $M_{N,\mathbb{Q}_N^\xi}:=\max_{1\leq j \leq N}\tilde{\xi}_j$, with each $\tilde{\xi}_j\sim \mathbb{Q}_N^\xi$ with $\mathbb{Q}_N^\xi$ the distribution that is fitted on $\{\hat{\xi}_i\}_{i=1}^N$, and $M_N := \max_{1\le i\le N}\hat{\xi}_i$. %
\end{lemma}

\begin{proof}
Based on Property \ref{property:affEquiFit}, we have that
\begin{equation}\label{eq:risk_Q}
    \rho_{\Q_N}(\zeta) = \rho_{\Q_N^\xi}(\mu+\sigma \xi) = \mu+(1/\alpha)\Lambda_{\mathbb{Q}_N^\xi}(\sigma\alpha)\mbox{ a.s.}.
\end{equation}
Also, given $\mathbb{Q}_N^\xi$ and letting $\hat{\mathbb{Q}}_{N,N}^\xi$ be the empirical distribution of $\{\tilde{\xi}_j\}_{j=1}^N$ drawn from $\mathbb{Q}_N^\xi$, we almost surely have  %
\begin{align}
\rho_{\hat{\mathbb{Q}}_{N,N}^\xi}(\mu+\sigma\xi) &= \frac{1}{\alpha}\log\left(\frac{1}{N}\sum_{j=1}^N\exp(\alpha (\mu+\sigma \tilde{\xi}_j))\right)\notag\\
&\leq \frac{1}{\alpha}\log\left(\frac{1}{N}\sum_{j=1}^N\max_{1\leq i\le N}\exp(\alpha (\mu+\sigma \tilde{\xi}_i))\right) = \mu+\sigma M_{N,\mathbb{Q}_N^\xi}.\label{eq:risk_Q_N}
\end{align}
Hence, together we get:
    \begin{align*}
    \rho_{\hat{\Prob}_N}&(\ell(\z,\vecxi))+\delta_N(\Q_N) =
        \rho_{\hat{\Prob}_N}(\ell(\z,\vecxi))+ \median(\rho_{\Q_N}(\zeta) - \rho_{\hat{\mathbb{Q}}_{N,N}}(\zeta)|\Q_N) \\
        &= \rho_{\hat{\Prob}_N}(\ell(\z,\vecxi))+ \rho_{\Q_N}(\zeta) -\median(\rho_{\hat{\mathbb{Q}}_{N,N}}(\zeta)|\Q_N)\\
        &= \rho_{\hat{\Prob}_N}(\ell(\z,\vecxi))+ \rho_{\Q_N}(\zeta) -\median(\rho_{\hat{\mathbb{Q}}_{N,N}^\xi}(\mu+\sigma \xi)|\mathbb{Q}_N^\xi)\\ 
        &\geq \mu+\sigma M_N - \log(N)/\alpha  + \mu+(1/\alpha)\Lambda_{\mathbb{Q}_N^\xi}(\sigma\alpha)- \median( \mu+\sigma M_{N,\mathbb{Q}_N^\xi}|\mathbb{Q}_N^\xi)\\
        &= \mu - \log(N)/\alpha   + \sigma (M_N- \median(M_{N,\mathbb{Q}_N^\xi}|\mathbb{Q}_N^\xi)) + (1/\alpha)\Lambda_{\mathbb{Q}_N^\xi}(\sigma\alpha),
    \end{align*}
    where the first, third, and last
    equalities exploit affine equivariance of $\median$, the second equality follows from Property \ref{property:affEquiFit}, and the first inequality uses Lemma \ref{recur_lemma}, \eqref{eq:risk_Q} and \eqref{eq:risk_Q_N}.
    
Similarly, given $\Q_N^\xi$, we have almost surely that %
\begin{align}
\rho_{\hat{\mathbb{Q}}_{N,N}}(\zeta) &= (1/\alpha)\log((1/N)\sum_{j=1}^N\exp(\alpha (\mu+\sigma \tilde{\xi}_j)))
\geq(1/\alpha)\log((1/N)\max_{1\le j \le N}\exp(\alpha (\mu+\sigma \tilde{\xi}_j)))\notag\\
&= \mu + \sigma \max_{1\le j \le N} \tilde{\xi}_j-\log(N)/\alpha = \mu + \sigma M_{N,\mathbb{Q}_N^\xi}-\log(N)/\alpha.\label{eq:risk_Q_N2}
\end{align}
Hence, together with Lemma \ref{recur_lemma}, we get:
    \begin{align*}
        \rho_{\hat{\Prob}_N}&(\ell(\z,\vecxi))+\delta_N(\Q_N) = \rho_{\hat{\Prob}_N}(\zeta)+ \rho_{\Q_N}(\zeta) -\median(\rho_{\hat{\mathbb{Q}}_{N,N}}(\zeta)|\Q_N)\\
        &\leq \mu+\sigma M_N + \mu +(1/\alpha)\Lambda_{\mathbb{Q}_N^\xi}(\sigma\alpha) 
    - \median(\mu+\sigma M_{N,\mathbb{Q}_N^\xi}- \log(N)/\alpha |\mathbb{Q}_N^\xi)\\
            &= \mu + \log(N)/\alpha   + \sigma (M_N-\median( M_{N,\mathbb{Q}_N^\xi}|\mathbb{Q}_N^\xi))+ (1/\alpha)\Lambda_{\mathbb{Q}_N^\xi}(\sigma\alpha),%
    \end{align*}
    where all equalities exploit affine equivariance of $\median$ and the first inequality uses Lemma \ref{recur_lemma}, \eqref{eq:risk_Q} and \eqref{eq:risk_Q_N2}.%
    
\end{proof}

\begin{lemma}\label{thm:noGaussBound}
    If $\zeta\sim\Q$, with  $\Expect_{\Q}[\zeta] = 0$, then $\Lambda_{\Q}(t)\geq 0$ for all $t\geq 0$.
\end{lemma}
\begin{proof}
From Jensen's inequality, we obtain:
    \begin{align*}
        \Lambda_{\Q}(t)&= \log(\Expect_{\Q}[\exp(t \zeta)]) \geq \Expect_{\Q}[\log(\exp(t\zeta))]=\Expect_{\Q}[t\zeta]=t\Expect_{\Q}[\zeta] =%
        0.
    \end{align*}

\end{proof}

\begin{lemma}\label{thm:superGaussBound}
    If $\zeta\sim\Q$ and
    there exist some $\mathfrak{c}_1,\mathfrak{c}_2,\bar{t}>0$ such that $\Prob(\zeta\geq t)\geq \mathfrak{c}_1\exp(-t^2/\mathfrak{c}_2^2)$ for all $t\geq \bar{t}$, 
    then $\Lambda_{\Q}(t)\geq t^2\mathfrak{c}_2^2/4+\log(t\mathfrak{c}_1 \mathfrak{c}_2\sqrt{\pi/4})$ for all $t\geq 2\bar{t}/\mathfrak{c}_2^2$.
\end{lemma}

\begin{proof}
Using the layer cake representation of a non-negative, real-valued measurable functions, we obtain:
    \begin{align*}
        \exp(\Lambda_{\Q}(t)) &= \Expect_{\Q}[\exp(t\zeta)] = \int_0^{\infty} \mathbb{P}_{\Q}(\exp(t\zeta)\geq x) \,dx %
        \\
&= \int_{-\infty}^{\infty} t\mathbb{P}_{\Q}(\zeta \geq y) \exp(ty) \,dy \quad \text{(substitute $x = \exp(ty)$)} \\&\geq \int_{\bar{t}}^\infty t\exp(ty)\Prob_{\Q}(\zeta\geq y)dy\\
        &\geq \int_{\bar{t}}^\infty t\exp(ty)\mathfrak{c}_1\exp(-y^2/\mathfrak{c}_2^2)dy \quad (\text{substitute $ty-y^2/\mathfrak{c}_2^2=t^2\mathfrak{c}_2^2/4-(1/\mathfrak{c}_2^2)(y- t\mathfrak{c}_2^2/2)^2$})\\
        &= t\mathfrak{c}_1\exp(t^2\mathfrak{c}_2^2/4)\int_{\bar{t}}^\infty \exp(-(1/\mathfrak{c}_2^2)(y- t\mathfrak{c}_2^2/2)^2)dy\\
        &\geq t\mathfrak{c}_1\exp(t^2\mathfrak{c}_2^2/4)\int_{t\mathfrak{c}_2^2/2}^\infty \exp(-(1/\mathfrak{c}_2^2)(y- t\mathfrak{c}_2^2/2)^2)dy\\
        &=t\mathfrak{c}_1\exp(t^2\mathfrak{c}_2^2/4)\int_{0}^\infty \exp(-m^2/\mathfrak{c}_2^2)dm \quad \text{(substitute $m=y-t\mathfrak{c}_2^2/2$)} \\
        &=t\mathfrak{c}_1\exp(t^2\mathfrak{c}_2^2/4)\int_{0}^\infty \exp(-z^2)\mathfrak{c}_2dz \quad \text{(substitute $z = m/\mathfrak{c}_2$)}\\
        &=t\mathfrak{c}_1\mathfrak{c}_2\sqrt{\pi/4}\exp(t^2\mathfrak{c}_2^2/4)\left((2/\sqrt{\pi})\int_{0}^\infty \exp(-z^2)dz\right)\\
        &=t\mathfrak{c}_1\mathfrak{c}_2\sqrt{\pi/4}\exp(t^2\mathfrak{c}_2^2/4)
    \end{align*}
    where the second inequality follows from our assumed lower bound on $\Prob_\Q(\zeta\geq y)$, the third inequality is obtained using the relation $t\mathfrak{c}_2^2/2\geq \bar{t}$, and where we used $(2/\sqrt{\pi})\int_{0}^\infty \exp(-z^2)dz=1$ to obtain the last equality.
\end{proof}

\begin{lemma}\label{thm:bigOmegaLoss}
    Let Assumption \ref{ass:meanVarControl} be satisfied. Further let $\Q_N$ and its fitting procedure satisfy properties \ref{property:affEquiFit}, \ref{property:unifBoundedMeanFit}, and \ref{property:unifHeavyGaussFit}. Then:
    \[\Expect[\rho_{\hat{\Prob}_N}(\ell(\z,\vecxi))+\delta_N(\Q_N)]  = \mu+ \Omega(\mbox{Var}(\ell(\vecz, \vecxi))).\]
\end{lemma}
\begin{proof}
Let $\mathbb{Q}_N^\xi$ be the distribution that is fitted on $\{\hat{\xi}_i\}_{i=1}^N$, with $\hat{\xi}_i:=(\ell(\z,\hat{\vecxi}_i)-\mu)/\sigma$, and $\tilde{\mathbb{Q}}_N^{{\xi}}$ be the distribution that is fitted on $\{\tilde{\xi}_i\}_{i=1}^N$, with $\tilde{\xi}_i:=(\ell(\vecz,\hat{\vecxi}_i)-\hat{\mu}_N)/\hat{\sigma}_N$ using the mean $\hat{\mu}_N$ and standard deviation $\hat{\sigma}_N$ of empirical distribution $\hat{\mathbb{P}}_N$. We start by establishing a relation between $\Lambda_{\mathbb{Q}_N^{\xi}}(t)$ and $\Lambda_{\tilde{\mathbb{Q}}_N^{\xi}}(t)$ that follows from Property \ref{property:affEquiFit}. Namely, we first observe that:
\[\tilde{\xi}_i:=(\ell(\z,\hat{\vecxi}_i)-\hat{\mu}_N)/\hat{\sigma}_N=(\mu+\sigma\hat{\xi}_i-\hat{\mu}_N)/\hat{\sigma}_N=(\hat{\xi}_i-\tilde{\mu}_N)/\tilde{\sigma}_N\]
where 
\[\tilde{\mu}_N:=\frac{\hat{\mu}_N-\mu}{\sigma}=(1/N)\sum_{i=1}^N\hat{\xi}_i,\]
and 
\[\tilde{\sigma}_N:=\hat{\sigma}_N/\sigma=\sqrt{(1/N)\sum_{i=1}^N(\hat{\xi}_i-(1/N)\sum_{j=1}^N\hat{\xi}_j)^2}.\]
Both $\tilde{\mu}_N$ and $\tilde{\sigma}_N$ only depend on the sample $\{\hat{\xi}_i\}_{i=1}^N$ drawn i.i.d. from $\Prob^\xi$.

Next, we exploit the affine equivariance %
property of the fitting procedure and the entropic risk to express the cumulant generating function fitted on the data to that fitted to the standardized data. Letting $\xi'\sim \mathbb{Q}_N^\xi$ and $\tilde{\xi}'\sim \tilde{\mathbb{Q}}_N^{{\xi}}$, since Property %
\ref{property:affEquiFit} states that $\hat{\xi}_i = \tilde{\mu}_N+\tilde{\sigma}_N\tilde{\xi}_i$ implies that $\xi' \stackrel{F}{=} \tilde{\mu}_N+\tilde{\sigma}_N\tilde{\xi}'$, we must have:
\begin{align}\label{eq:affine_equiv}
\Lambda_{\mathbb{Q}_N^{\xi}}(t)=\log(\Expect_{\mathbb{Q}_N^{\xi}}[\exp(t\xi')]) &= \log(\Expect_{\tilde{\mathbb{Q}}_N^{\xi}}[\exp(t(\tilde{\mu}_N+\tilde{\sigma}_N\tilde{\xi}'))]) = t\tilde{\mu}_N+\Lambda_{\tilde{\mathbb{Q}}_N^{{\xi}}}(t\tilde{\sigma}_N).
\end{align}

Turning to our main objective, we employ Lemma \ref{thm:biasCorrectBound} to get that
    \begin{align}
        \Expect[\rho_{\hat{\Prob}_N}(\ell(\z,\vecxi))&+\delta_N(\Q_N)]%
         \geq \Expect[\mu - \log(N)/\alpha +\sigma (M_N - \mbox{median}(M_{N,\mathbb{Q}_N^\xi}|\Q_N^\xi)) + (1/\alpha)\Lambda_{\mathbb{Q}_N^\xi}(\sigma\alpha)]\notag\\
        &= \mu - \log(N)/\alpha + \sigma \Expect[ M_N - \mbox{median}(M_{N,\mathbb{Q}_N^\xi}|\Q_N^\xi%
        )] + (1/\alpha)\Expect[\Lambda_{\mathbb{Q}_N^\xi}(\sigma\alpha)].\label{eq:overestimate_bias}
        \end{align}

Exploiting Property \ref{property:unifHeavyGaussFit}, we obtain that $\tilde{\Q}_N^\xi$ has a uniformly heavy standardized right Gaussian tail. Hence, there exists some $\mathfrak{c}_1, \mathfrak{c}_2, \bar{t} >0$ such that with probability one  $\mathbb P_{\tilde{\Q}_N^\xi}(\xi\geq t)\geq \mathfrak{c}_1\exp(-t^2/\mathfrak{c}_2^2)$ for all $t\geq \bar{t}$. Now, let $\bar{\sigma}$ such that $\Prob(\bar{\sigma}\geq 2\bar{t}/(\alpha\tilde{\sigma}_N \mathfrak{c}_2^2))>0$ and $\mathcal{E}$ be the event that $\bar{\sigma}\geq 2\bar{t}/(\alpha\tilde{\sigma}_N \mathfrak{c}_2^2)$ and $\bar{\mathcal E}$ denotes its complement. If $\sigma\geq \bar{\sigma}$, we can use \eqref{eq:affine_equiv} to express the expected cumulant generating function as follows:
        \begin{align*}
        \Expect[\Lambda_{\mathbb{Q}_N^\xi}(\sigma\alpha)]  
        &= \left(\Expect[\sigma\alpha\tilde{\mu}_N]+\Prob(\mathcal{E}) \Expect[\Lambda_{\tilde{\mathbb{Q}}_N^\xi}(\sigma\alpha\tilde{\sigma}_N)|\mathcal{E}]+\Prob(\bar{\mathcal{E}})\Expect[\Lambda_{\tilde{\mathbb{Q}}_N^\xi}(\sigma\alpha\tilde{\sigma}_N)|\bar{\mathcal{E}}]\right)\\  
        &\geq \left(\Expect[\sigma\alpha\tilde{\mu}_N]+\Prob(\mathcal{E}) \Expect[(\sigma\alpha\tilde{\sigma}_N \mathfrak{c}_2)^2/4+\log(\mathfrak{c}_1 \mathfrak{c}_2\sigma\alpha\tilde{\sigma}_N \sqrt{\pi/4})|\mathcal{E}]\right),        
        \end{align*}
        where in the inequality, the bound on $\Expect[\Lambda_{\tilde{\mathbb{Q}}_N^\xi}(\sigma\alpha\tilde{\sigma}_N)|\mathcal{E}]$  comes from Lemma \ref{thm:superGaussBound}, and where\\ 
        $\Expect[\Lambda_{\tilde{\mathbb{Q}}_N^\xi}(\sigma\alpha\tilde{\sigma}_N)|\bar{\mathcal{E}}]\ge%
        0$ comes from Lemma \ref{thm:noGaussBound} given that Property \ref{property:unifBoundedMeanFit} implies that $\Expect_{\tilde{\Q}_N^\xi}[\xi]=\Expect_{\Q_N}[(\zeta-\hat{\mu}_N)/\hat{\sigma}_N]= 0$ almost surely.
        Next, we substitute the cumulant generating function in \eqref{eq:overestimate_bias}
        \begin{align*}
         &\Expect[\rho_{\hat{\Prob}_N}(\zeta)+ \delta_N(\Q_N)] \\
         &\geq \mu - \log(N)/\alpha +(1/\alpha) \Prob(\mathcal{E})\Expect[\log(\mathfrak{c}_1 \mathfrak{c}_2\alpha\tilde{\sigma}_N \sqrt{\pi/4})|\mathcal{E}]+\sigma \Expect[M_N - \mbox{median}%
         (M_{N,\mathbb{Q}_N^\xi}|\Q_N^\xi%
         )+\tilde{\mu}_N] \\
        &\qquad\qquad+ (1/\alpha)\Prob(\mathcal{E})\Expect[(\alpha\tilde{\sigma}_N \mathfrak{c}_2)^2/4|\mathcal{E}]\sigma^2+\Prob(\mathcal{E})\log(\sigma)/\alpha\\                 
        &=\mu+\Omega(\sigma^2).
    \end{align*}

\end{proof}

\removed{
\correction{\begin{lemma}\label{lemma:tail:cgf}
Let $\zeta\sim \Q_N$ have subgaussian standardized tails (see Property \ref{property:unifSubGaussFit}),  then for some $K_N, b_N>0$, $\mathbb E [\exp(t\tilde{\xi})] \le \exp(b_N+K_N^2t^2)\; \forall t\in \mathbb{R}$ with  $\tilde{\xi}:= (\zeta-\hat{\mu}_N)/\hat{\sigma}_N$.
\end{lemma}
\begin{proof}
$\tilde{\xi}-\Expect[\tilde{\xi}]$ is subgaussian \citep[Lemma 2.6.8]{Vershynin2025}. For a centered random variable $\tilde{\xi}-\Expect[\tilde{\xi}]$ with subgaussian tails, the moment generating function is given in \citet[][see Proposition 2.5.2]{Vershynin2025}: $\mathbb E [\exp(t(\tilde{\xi}-\Expect[\tilde{\xi}]))]= \exp(H_N^2t^2)\; \forall t\in \mathbb{R}$ and for some $H_N>0$. Equipped with this result, we will bound the cumulant generating function of $\tilde{\xi}$. We have that
\[
\mathbb E|\tilde{\xi}|
= \int_0^\infty \mathbb P(|\tilde{\xi}|\ge a)\,da
\le \int_0^\infty 2e^{-a^2/k_N^2}\,da= \int_0^\infty 2k_Ne^{-y^2}dy
\le C_0 k_N.
\]
 Therefore, for all $t\in \mathbb R$,
\[
\mathbb E [\exp(t\tilde{\xi})]
= \exp(t\mathbb E(\tilde{\xi}))\mathbb E [\exp(t(\tilde{\xi}-\Expect[\tilde{\xi}]))]
\le \exp\big(C_0 k_N|t| + H_N^2 t^2\big).
\]
For all $t\in \mathbb{R}$, we have that $C_0 k_N|t| \le H_N k_N^2 t^2 + \frac{C_0^2}{2H_N}$,
hence
\[
\mathbb E [\exp(t\tilde{\xi})]%
\le \exp\left({\frac{C_0^2}{2H_N}}
+ {\Big(H_N^2+H_Nk_N^2\Big)} t^2\right).
\]
where $K_N = \sqrt{H_N^2+H_Nk_N^2}>0$ and $b_N=\frac{C_0^2}{2H_N}>0$.
\end{proof}}
}

\begin{lemma}\label{thm:bigOLoss}
  Let Assumption \ref{ass:meanVarControl} be satisfied and the loss $\ell(\vecz, \vecxi)$ have lighter-than-Gaussian tails. Then there exists a $\bar{p}<2$ such that:
    \[\rho_{\Prob}(\ell(\vecz, \vecxi))  = \mu + \mathcal{O} (\mbox{Var}(\ell(\vecz, \vecxi))^{{\bar{p}}/2}).\]
\end{lemma}

\begin{proof}
We start with showing that $\xi=(\ell(\vecz, \vecxi)-\mu)/\sigma$ must also have lighter-than-Gaussian tails. Namely, given that $\ell(\vecz, \vecxi)$ has lighter-than-Gaussian tails, we must have that for all $a\ge 0$:
\begin{align*}
\mathbb P(|\xi|\ge a)&=\mathbb P(|\ell(\vecz, \vecxi)-\mu|\ge \sigma a)
\le \mathbb P(|\ell(\vecz, \vecxi)|\ge \sigma a-|\mu|) 
\end{align*}
Let $a_0:=\frac{2|\mu|}{\sigma}$, $\tilde{C}:=C\sigma^q/2^q,\, \tilde{G}:=\max(G,\exp(\tilde C a_0^q))$. We will show that $\mathbb P(|\xi|\ge a)\le \tilde{G}\exp(-\tilde{C}a^q)$ for all $a\geq 0$.

If $ a\ge a_0$, then $\sigma a-|\mu|\ge \sigma a/2$, so
\[\mathbb P(|\xi|\ge  a)\le \mathbb P(|\ell(\vecz, \vecxi)|\ge \sigma a/2)\le  G\exp(- C(\sigma a/2)^q)
=  G\exp(-( C\sigma^q /2^q)\,a^q)\leq \tilde G\exp(-\tilde C a^q).\]

If $0\le a\le a_0$, then 
\[
\Prob(|\xi|\ge a)\le 1 = \exp(\tilde C a_0^q)\exp(-\tilde C a_0^q) \le \exp(\tilde C a_0^q)\exp(-\tilde C a^q) \le \tilde G\exp(-\tilde C a^q) ,
\]
where the first inequality comes from the monotonicity of $\exp(-\tilde C y^q)$ in $y$. 
Exploiting the fact that $\xi$ has lighter-than-Gaussian tails, Lemma \ref{lemma:equivalence:mgf_tails} ensures that there exists some $\nu>0$ and ${\bar{p}}<2$ such that $\Expect[\exp(t\xi)]\leq 2\exp(\nu t^{\bar{p}})$. We therefore can conclude that:
    \begin{align*}
    \rho_\Prob(\ell(\z,\vecxi)) &= (1/\alpha)\log(\Expect[\exp(\alpha (\mu+\sigma\xi))])=\mu+(1/\alpha)\log(\Expect[\exp(\alpha\sigma\xi)])\\ &\leq \mu+(\nu/\alpha)  (\alpha\sigma)^{\bar{p}}+(1/\alpha)\log(2) = \mu +\mathcal{O}(\sigma^{\bar{p}}).
    \end{align*}

\end{proof}

\begin{lemma}\label{thm:bigOUpperBoundLoss}
    Let Assumption \ref{ass:meanVarControl} be satisfied. Further let $\Q_N$ and its fitting procedure satisfy properties \ref{property:affEquiFit}, \ref{property:unifBoundedMeanFit}, and  \ref{property:unifSubGaussFit}. Then:
    \[\Expect[\rho_{\hat{\Prob}_N}(\ell(\z,\vecxi))+\delta_N(\Q_N)]  \leq \mu+ \mathcal{O}(\mbox{Var}(\ell(\vecz, \vecxi))).\]
\end{lemma}

\begin{proof}
Using similar steps as in the proof of Lemma \ref{thm:bigOmegaLoss}, we start by employing  Lemma \ref{thm:biasCorrectBound} to obtain:
\begin{align*}
\Expect[\rho_{\hat{\Prob}_N}(\ell(\z,\vecxi))+ \delta_N(\Q_N)]
&\leq \Expect[\mu + \log(N)/\alpha +\sigma (M_N - \median(M_{N,\mathbb{Q}_N^\xi}|{\Q_N^\xi})) + (1/\alpha)\Lambda_{\mathbb{Q}_N^\xi}(\sigma\alpha)]\\
&= \mu + \log(N)/\alpha +\sigma \Expect[M_N - \median(M_{N,\mathbb{Q}_N^\xi}|{\Q_N^\xi})] + (1/\alpha)\Expect[\Lambda_{\mathbb{Q}_N^\xi}(\sigma\alpha)].
\end{align*}
We then exploit equation \eqref{eq:affine_equiv} to express the expected cumulant generating functions as follows:
        \begin{align*}
        \Expect[\Lambda_{\mathbb{Q}_N^\xi}(\sigma\alpha)]
        &= \Expect[\sigma\alpha\tilde{\mu}_N]+ \Expect[\Lambda_{\tilde{\mathbb{Q}}_N^\xi}(\sigma\alpha\tilde{\sigma}_N)] \leq  \Expect[\sigma\alpha\tilde{\mu}_N]+ \Expect[\nu_N^2 (\sigma\alpha\tilde{\sigma}_N)^2],%
        \end{align*}
where the inequality comes from properties \ref{property:unifBoundedMeanFit} and \ref{property:unifSubGaussFit}, which imply that $\tilde{\Q}_N^\xi$ is a centered subgaussian distribution, hence by Lemma \ref{lemma:cgf:subgaussian:centered} there exists a ${\nu_N}>0$ such that $\Expect[\exp(t\xi)]\leq \exp(\nu_N^2t^2)$. Therefore, 
\begin{align*}
\Expect[\rho_{\hat{\Prob}_N}(\zeta)+ \delta_N(\Q_N)]
&\leq \mu + \log(N)/\alpha +\sigma \Expect[M_N - \mbox{median}(M_{N,\mathbb{Q}_N^\xi}|{\Q_N^\xi})] + \sigma\Expect[\tilde{\mu}_N]+ \sigma^2\alpha \Expect[\nu_N^2 \tilde{\sigma}_N^2]\\
&=\mu + \mathcal{O}(\sigma^2).
\end{align*}
\end{proof}

\begin{lemma}\label{thm:GaussProp4makeAss3}
    If $\ell(\vecz, \vecxi)$ have lighter-than-Gaussian tails and $\Q_N$ satisfies Property \ref{property:unifSubGaussFit}, then $\Q_N$ satisfies Assumption \ref{assum:fit_tail_bound}.
\end{lemma}
\begin{proof}
    By Lemma \ref{lemma:equivalence:mgf_tails}, there exists a ${\bar{p}}\in(1,2)$ and a $\nu>0$ such that $\Expect_{\Prob}[\exp(t\ell(\z,\vecxi))]\leq 2\exp(\nu t^{\bar{p}})$ for all $t\geq 0$. 
    Thus the first and second moments of $\ell(\vecz, \vecxi)$ are finite and the strong law of large numbers implies $\hat{\mu}_N\rightarrow \Expect[\ell(\vecz, \vecxi)]$ and $(1/N)\sum_{i=1}^N\ell(\vecz, \hat{\vecxi}_i)^2 \to \Expect[(\ell(\vecz, \vecxi))^2]$ almost surely. From continuous mapping theorem \citep{Van_der_Vaart_2000}, we have that $\hat{\sigma}_N = \sqrt{(1/N)\sum_{i=1}^N\ell(\vecz, \hat{\vecxi}_i)^2 -\hat{\mu}_N^2} \rightarrow \sqrt{\Expect[(\ell(\vecz, \vecxi))^2]-(\Expect[\ell(\vecz, \vecxi)])^2}= \sqrt{\mbox{Var}(\ell(\vecz, \vecxi))}$ almost surely. 
    Fixing $\bar{\sigma}:=\sqrt{\mbox{Var}(\ell(\vecz, \vecxi))}+1$ and $\bar{R}:=|\Expect[\ell(\vecz, \vecxi)]|+1$, there must therefore almost surely be some $\tilde{N}\geq 1$ such that $\hat{\sigma}_N\le \bar{\sigma}$ and $|\hat{\mu}_N|\leq \bar{R}$  for all $N\geq \tilde{N}$. We can now show that for all $N\geq \tilde{N}$ and $a\geq 0$:
    \begin{align*}
        &\Prob_{\Q_N}(|\zeta|\geq a)\leq \Prob_{\Q_N}(|\zeta|\geq a - (\bar{R}-|\hat{\mu}_N|)) \leq \Prob_{\Q_N}(|\zeta-\hat{\mu}_N| \geq a - \bar{R})\\
        &\leq \Prob_{\Q_N}(|\zeta-\hat{\mu}_N| \geq \frac{\hat{\sigma}_N}{\bar{\sigma}}(a - \bar{R})) = \Prob_{\Q_N}(|\zeta-\hat{\mu}_N|/\hat{\sigma}_N \geq (a - \bar{R})/\bar{\sigma})\\
        &\leq \left\{\begin{array}{cl}1 & \mbox{ if $a<\bar{R}$}\\ 2\exp(-(a - \bar{R})^2/(\mathfrak c_2^2\bar{\sigma}^2))& \mbox{otherwise}\end{array}\right.\\
        &\leq 2\exp(-(a^2/2 - \bar{R}^2)/(\mathfrak c_2^2\bar{\sigma}^2))\\
        &=G\exp(-a^2/K^2)
    \end{align*}
    where $G:=2\exp(\bar{R}^2/(\mathfrak c_2^2\bar{\sigma}^2))$ and $K:=\sqrt{2}\mathfrak c_2\bar{\sigma}$, and the last inequality comes from the fact that for all $a\ge 0$, we have $(\max(a-\bar R,0))^2\ge \frac{a^2}{2}-\bar R^2$. %
    Clearly, $a\leq \bar{R} \implies a^2/2-\bar{R}^2 \le \bar{R}^2/2-\bar{R}^2\leq 0$, and when $a\geq \bar{R}$, we have that $(a-\bar R)^2\ge a^2/2-\bar{R}^2 \iff a^2/2+2\bar{R}^2-2a\bar{R} \geq 0 \iff (a/\sqrt{2} - \sqrt{2}\bar{R})^2\ge 0$. %

For any $\lambda>0$, one has that  $(a/K-\lambda K/2)^2\ge0$ yields $-a^2/K^2\le -\lambda a+\lambda^2K^2/4$, so $\exp(-a^2/K^2)\le \exp(\lambda^2K^2/4)\exp(-\lambda a)$. With  $\lambda:=\alpha C,$ 
\[\Prob_{\Q_N}(|\zeta|\ge a)\leq G\exp(-a^2/K^2)\le G\exp(\alpha^2 C^2K^2/4)\exp(-a\alpha  C )=:\tilde G\exp(-a \alpha C) \;\forall a\ge0,\] where $\tilde G:=G\exp\big(\alpha^2 C^2K^2/4\big)$.
\end{proof}

\subsection{Proof of Theorem \ref{thm:botstrapconsistent}}
\begin{proof}
We want to show that $\rho_{\hat{\Prob}_N}(\ell(\vecz, \vecxi))+\delta_N(\Q_N) \rightarrow \rho_\Prob(\ell(\vecz, \vecxi))$ almost surely. To do so, we will show that both $\rho_{\hat{\Prob}_N}(\ell(\vecz, \vecxi)) \rightarrow \rho_\Prob(\ell(\vecz, \vecxi))$ and $\delta_N(\Q_N) \rightarrow 0$ almost surely. Indeed, if  both $\rho_{\hat{\Prob}_N}(\ell(\vecz, \vecxi)) \rightarrow \rho_\Prob(\ell(\vecz, \vecxi))$ and $\delta_N(\Q_N) \rightarrow 0$ almost surely, then we can use the fact that the sum of two convergent sequences converges to the sum of their limits to conclude that: 
\begin{align*}
&\Prob(\rho_{\hat{\Prob}_N}(\ell(\vecz, \vecxi)) + \delta_N(\Q_N) \rightarrow \rho_{\Prob}(\ell(\vecz, \vecxi)))\\   &\quad \geq  \Prob(\{\rho_{\hat{\Prob}_N}(\ell(\vecz, \vecxi)) \rightarrow \rho_\Prob(\ell(\vecz, \vecxi))\} \cap  \{\delta_N(\Q_N) \rightarrow 0\})\\
    &\quad =1-\Prob(\{\rho_{\hat{\Prob}_N}(\ell(\vecz, \vecxi)) \not\rightarrow \rho_\Prob(\ell(\vecz, \vecxi))\} \cup  \{\delta_N(\Q_N) \not\rightarrow 0\})\\
    &\quad \geq 1 - (1-\Prob(\rho_{\hat{\Prob}_N}(\ell(\vecz, \vecxi)) \rightarrow \rho_\Prob(\ell(\vecz, \vecxi))))-(1-\Prob(\delta_N(\Q_N) \rightarrow 0))\\
    &\quad =1,
\end{align*}
where the second inequality follows from the union bound.
    
\textbf{Step 1: $\rho_{\hat{\Prob}_N}(\ell(\vecz, \vecxi)) \rightarrow \rho_\Prob(\ell(\vecz, \vecxi))$ almost surely.} Given that $\Expect[\exp(\alpha \ell(\vecz, \vecxi))]$ is finite (see Lemma \ref{assum:tail_bound2}) and each $\hat{\vecxi}_i$ is i.i.d., the strong law of large numbers tells us that:
\begin{align}
    & \frac{1}{N}\sum_{i=1}^N\exp(\alpha \ell(\vecz, \hat{\vecxi}_i)) \rightarrow \Expect[\exp(\alpha\ell(\vecz, \vecxi))]  \text{ almost surely}.\label{eq:slln}
\end{align}
Since the logarithm function is continuous over the strictly positive values, using the continuous mapping theorem \citep{Van_der_Vaart_2000}, we obtain:
\begin{align*}
        & \eqref{eq:slln} \implies \frac{1}{\alpha}\log\left(\frac{1}{N}\sum_{i=1}^N\exp(\alpha \ell(\vecz, \hat{\vecxi}_i))\right) \rightarrow \frac{1}{\alpha}\log\left(\Expect[\exp(\alpha \ell(\vecz, \vecxi))]\right)  \text{ almost surely}.
\end{align*}
Hence, $\rho_{\hat{\Prob}_N}(\ell(\vecz, \vecxi)) \rightarrow \rho_\Prob(\ell(\vecz, \vecxi))$ almost surely.
    
    \textbf{Step 2: $\delta_N(\Q_N) \rightarrow 0$ almost surely.}  
    We will show that $\delta_N(\Q_N) \rightarrow 0$ almost surely by showing that $\exp(-\alpha\delta_N(\Q_N)) \rightarrow 1$ almost surely. 

    Let's consider any sequence $\{\bar{\Q}_N\}_{N=1}^\infty$ that is uniformly exponentially bounded \revised{for $N\geq \bar{N}$}, with each element of the sequence \revised{after $\bar{N}$} satisfying Assumption \revised{\ref{assum:tail_bound2} for the same $G>0$ and $C>2$. We therefore consider the subsequence $\{\bar{\Q}_N\}_{N=\bar{N}}^\infty$ in what follows.} %
    
    For each $N\revised{\geq \bar{N}}$, we have:
    \begin{align*}
        \exp(-\alpha\delta_N(\bar{\Q}_N)) & \ = \ \exp(\alpha \,\mbox{median}(\rho_{\hat{\bar{\Q}}_{N,N}}(\zeta)-\rho_{\bar{\Q}_N}(\zeta)\lvert\bar{\Q}_N))\\
        & \ = \ \mbox{median}(\exp(\alpha\rho_{\hat{\bar{\mathbb Q}}_{N,N}}(\zeta)-\alpha\rho_{\bar{\Q}_N}(\zeta))|\bar{\Q}_N)\\
        & \ = \ \mbox{median}\left(\exp\left(\log\left(\frac{(1/N)\sum_{i=1}^N \exp(\alpha \zeta_i)}{\Expect_{\bar{\Q}_N}[\exp(\alpha \zeta)]}\right) \right)\middle|\bar{\Q}_N\right)\\
        & \ = \ \mbox{median}[X_N|\bar{\Q}_N],
    \end{align*}
    where $X_N:=\frac{(1/N)\sum_{i=1}^N \exp(\alpha \zeta_i)}{\Expect_{\bar{\Q}_N}[\exp(\alpha \zeta)]}$ and \revised{$\hat{\bar{\mathbb Q}}_{N,N}(\zeta)$ denotes the empirical distribution of $N$ samples drawn from $\bar{\Q}_{N}$}. The second equality follows by the monotonicity of the exponential function, and the third equality follows by the definition of the entropic risk and the properties of the logarithm function. Note that  $\{\zeta_i\}_{i=1}^N$ are drawn i.i.d. from $\bar{\Q}_N$. 
    
    To analyze the $\mbox{median}[X_N|\bar{\Q}_N]$, we first compute the mean and variance of $X_N$. It is easy to see that $\Expect_{\bar{\Q}_N}[X_N]=1$, while the variance of $X_N$ can be bounded as follows:
    \begin{align*}
    \mbox{Var}_{\bar{\Q}_N}(X_N)&=\frac{\mbox{Var}_{\bar{\Q}_N}\left(\sum_{i=1}^N \exp(\alpha \zeta_i)\right)}{\left(N\Expect_{\bar{\Q}_N}[\exp(\alpha \zeta)]\right)^2} =\frac{\mbox{Var}_{\bar{\Q}_N}\left(\exp(\alpha \zeta)\right)}{N\left(\Expect_{\bar{\Q}_N}[\exp(\alpha \zeta)]\right)^2}\leq \revised{\frac{(2G+C-2)\exp(\frac{2G}{C})}{N(C-2)}},
    \end{align*}
    where the second equality follows from the fact that  $\{\zeta_i\}_{i=1}^N$ are i.i.d.  The inequality follows since \revised{$\{\bar{\Q}_N\}_{N=\bar{N}}^\infty$  satisfy Assumption \ref{assum:tail_bound2} for the same $G>0$ and $C>2$}, thus Lemma~\ref{thm:boundedExpVar2}  provides bounds for both $\mbox{Var}_{\bar{\Q}_N}\left(\exp(\alpha \zeta)\right)$   and $\Expect_{\bar{\Q}_N}[\exp(\alpha \zeta)]$, resulting in the bound $\frac{\text{Var}_{\bar{\Q}_N}\left(\exp(\alpha\zeta)\right)}{\left(\Expect_{\bar{\Q}_N}[\exp(\alpha\zeta)]\right)^2}\leq \revised{\frac{(2G+C-2)\exp(2G/C)}{C-2}}$.

    We next show that $\mbox{median}[X_N|\bar{\Q}_N]$ is bounded. To this end, consider the  Chebyshev's inequality:
       \begin{align*}
        \Prob_{\bar{\Q}_N}\left(\lvert X_N - 1\rvert\geq k\sqrt{\mbox{Var}_{\bar{\Q}_N}(X_N)}\right) \leq \frac{1}{k^2}. 
    \end{align*}
     Substituting the bound for the $ \mbox{Var}_{\bar{\Q}_N}(X_N)$  and setting $k=2$, which implies an upper tail probability bound of 25\%,   results in \begin{equation*}
        \Prob_{\bar{\Q}_N}\left(|X_N - 1|\geq 2\Delta/\sqrt{N}\right) \leq \frac{1}{4}, 
    \end{equation*}
    where $\Delta:= \revised{\sqrt{(2G+C-2)\exp(\frac{2G}{C})}/\sqrt{C-2}}$. Thus, we   conclude that $\mbox{median}[X_N|\bar{\Q}_N]\in [1-2\Delta/\sqrt{N},\,1+2\Delta/\sqrt{N}]$ since otherwise it would imply that 50\% of the probability is outside this interval (either on the right or the left), which would contradict the fact that the total probability outside the interval is below $1/4$. 
    
    Finally, we show that as $N$ tends to infinity, $\mbox{median}[X_N|\bar{\Q}_N]$ converges to 1 almost surely. For any $\epsilon>0$, there exists  $N_0=4\Delta^2/\epsilon^2$ such that for all $N>\max(N_0, \bar N)$, $|\mbox{median}[X_N|\bar{\Q}_N]-1| \leq \epsilon$ which implies $ \lim_{N\rightarrow \infty} \mbox{median}[X_N|\bar{\Q}_N] = 1$.
    
    We conclude from the above analysis that any sequence $\{\bar{\Q}_N\}_{N=\bar N}^\infty$ that is uniformly exponentially bounded as prescribed in Assumption \ref{assum:fit_tail_bound} must be such that $ \exp(-\alpha\delta_N(\bar{\Q}_N))=\mbox{median}[X_N|\bar{\Q}_N]$ converges to 1 almost surely. Since Assumption \ref{assum:fit_tail_bound} imposes that such sequences are almost surely obtained, we must have that $\delta_N(\Q_N) \rightarrow 0$ almost surely.

    \end{proof}

\subsection{\revised{Proof of Theorem \ref{thm:our_estimator_overestimate_omega}}}
\begin{proof}

 Based on lemmas \ref{thm:bigOmegaLoss} and \ref{thm:bigOLoss}, we thus conclude that 
    \[\Expect[\rho_{\hat{\Prob}_N}(\ell(\vecz,\vecxi))+ \delta_N(\Q_N)] - \rho_\Prob(\ell(\vecz,\vecxi)) = \Omega(\sigma^2) - \mathcal{O} (\sigma^{\bar{p}}) = \Omega(\sigma^2),\]
    since ${\bar{p}}<2$. 
Moreover from Lemma \ref{thm:bigOUpperBoundLoss}, we also have 
\begin{align*}
\Expect[\rho_{\hat{\Prob}_N}(\ell(\vecz,\vecxi))+ \delta_N(\Q_N)] - \rho_\Prob(\ell(\vecz,\vecxi)) &= \mu + \mathcal O(\sigma^2)-\rho_\Prob(\ell(\vecz,\vecxi))\leq \mu + \mathcal O(\sigma^2)-\Expect_\Prob[\ell(\vecz,\vecxi)] \\
&= \mu + \mathcal  O(\sigma^2) -\mu = \mathcal O(\sigma^2),
\end{align*}
since $\rho(X)\geq \Expect[X]$ for any random variable and $\Expect_\Prob[\ell(\vecz,\vecxi)]=\mu$. Hence,
\[| \rho_\Prob(\ell(\vecz,\vecxi)) -\Expect[\rho_{\hat{\Prob}_N}(\ell(\vecz,\vecxi))+ \delta_N(\Q_N)]|\leq \max(\mathcal O(\sigma^2),\Omega(\sigma^2))\leq \mathcal O(\sigma^2).\]

Finally, Lemma \ref{thm:GaussProp4makeAss3} ensures that $\Q_N$ satisfies Assumption \ref{assum:fit_tail_bound}. Hence, based on Theorem \ref{thm:botstrapconsistent} $\rho_{\hat{\Prob}_N}(\ell(\vecz,\vecxi))+\delta_N(\Q_N)$ is asymptotically consistent
\end{proof}
\noindent 

\subsection{\revised{Proof of Theorem \ref{thm:GMM}}}
We will show in turn that both Assumption \ref{assum:fit_tail_bound} and properties \ref{property:affEquiFit}-\ref{property:unifSubGaussFit} are satisfied by $\Q_N$. By theorems \ref{thm:botstrapconsistent} and \ref{thm:our_estimator_overestimate_omega}, this will imply respectively that the estimator is strongly asymptotically consistent, and that equations \eqref{eq:thm:OmegaVar} and \eqref{eq:thm:OVar} are satisfied.

\begin{lemma}
    A $\Q_N$ produced by Algorithm \ref{alg:fittingGMM} satisfies properties \ref{property:affEquiFit} and \ref{property:unifBoundedMeanFit}.
\end{lemma}

\begin{proof}
    For any $(a,b)\in\mathbb{R}\times\mathbb{R}_+$, let $\Q_N$ be obtained from Algorithm \ref{alg:fittingGMM} using $\DataN:=\{\ell(\z, \hat{\vecxi}_i)\}_{i=1}^N$ and $\Q_N'$ be obtained using $\DataN:=\{a+b\ell(\z, \hat{\vecxi}_i)\}_{i=1}^N$. Letting $(\hat{\mu}_N,\hat{\sigma}_N,\bar{\DataN},\bar{\Theta},\btheta^*)$ and $(\hat{\mu}_N',\hat{\sigma}_N',\bar{\DataN}',\bar{\Theta}',{\btheta^*}')$ be the objects assigned as one follows the steps of the algorithm. One can straightforwardly observe that:
    \begin{align*}
        &\hat{\mu}_N' := (1/N)\sum_{\zeta\in\DataN'} \zeta = (1/N)\sum_{\zeta\in\DataN} (a+b\zeta) = a+b\hat{\mu}_N\\
        &\hat{\sigma}_N'  = \sqrt{(1/N)\sum_{\zeta\in\DataN'} (\zeta-\hat{\mu}_N')^2} = \sqrt{(1/N)\sum_{\zeta\in\DataN} (a+b\zeta-a-b\hat{\mu}_N)^2}=b\hat{\sigma}_N\\
        &\bar{\DataN}' :=\{(a+b\ell(\z, \hat{\vecxi}_i)-\hat{\mu}_N')/{\hat{\sigma}_N}'\}_{i=1}^N = \{(a+b\ell(\z,\hat{\vecxi}_i)-a-b\hat{\mu}_N)/{b\hat{\sigma}_N}\}_{i=1}^N=\bar{\DataN}\\
        &\bar{\Theta}' := \{\btheta\in\Theta| \Expect_{\Q^\btheta}[\zeta]= 0\} = \bar{\Theta}.
    \end{align*}
    Given that both $\bar{\DataN}'=\bar{\DataN}$ and $\bar{\Theta}'=\bar{\Theta}$, it must be that ${\btheta^*}'={\btheta^*}$ almost surely. The distribution of $\zeta'\sim\Q_N'$ must therefore be the distribution of $\hat{\mu}_N'+\hat{\sigma}_N'\xi'\stackrel{F}{=}a+b(\hat{\mu}_N+\hat{\sigma}_N\xi')\stackrel{F}{=}a+b\zeta$, with $\xi'\sim \Q^{\btheta^*}$ and $\zeta\sim\Q_N$. This verifies Property \ref{property:affEquiFit}.

    To verify Property \ref{property:unifBoundedMeanFit}, one can simply exploit the fact that $\btheta^*\in\bar{\Theta}$, which ensures that $\Expect_{\Q^{\btheta^*}}[\xi']=0$. This further implies that $\Expect_{\Q_N}[\zeta]=\Expect_{\Q^{\btheta^*}}[\hat{\mu}_N+\hat{\sigma}_N\xi']=\hat{\mu}_N$.

\end{proof}

\begin{lemma}
    If there exists a bound $\epsilon>0$, some $B>-\infty$, %
    and some $\bar{\mathfrak j}\in\mathfrak J$, such that $\pi_{\bar{\mathfrak j}}\geq\epsilon$, $\sigma_{\bar{\mathfrak j}} \geq \epsilon$, and $\mu_{\bar{\mathfrak j}}\geq B$ %
for all $(\pi,\mu,\sigma)\in\Theta$, then a $\Q_N$  produced by Algorithm \ref{alg:fittingGMM} using a GMM within $\Theta$ satisfies Property \ref{property:unifHeavyGaussFit}.
\end{lemma}

\begin{proof}
    We will need \citet[][Proposition 2.1.2]{Vershynin2025}, which states that for $G\sim \mathcal N(0,1)$ and $t\ge 2$ 
\begin{align}\label{eq:lowernound:normal}
\mathbb P(G\ge t)\ &\ge\ \Big(\frac{1}{t}-\frac{1}{t^3}\Big)\frac{1}{\sqrt{2\pi}}\exp(-t^2/2)\notag \\
& \ge \frac{1}{2\sqrt{2\pi}}\frac{1}{t}e^{-t^2/2}
\ge \frac{1}{2\sqrt{2\pi}}e^{-t^2},
\end{align}
where the second inequality comes from 
$1/t-1/t^3=(1/t)(1-1/t^2)\geq 1/(2t)$ since $t\ge 2$, the third inequality comes from $(1/2t)\geq \exp(-t^2/2)$ since for $t\ge 2,\ 1/(2t) \ge e^{-t^2/2} \iff g(t):=t^2/2-\ln(2t)\ge 0,\ g'(t)=t-1/t\ge 2-1/2>0,\ g(2)=2-\ln 4>0 \Rightarrow g(t)\ge g(2)>0$.
    Thus, we have that for $a\geq \bar{a}:=\max(0,  B+2\epsilon)$ %
\begin{align*}
\Prob_{\Q_N}((\zeta-\hat\mu_N)/\hat\sigma_N \ge a) = \Prob_{\Q^{\btheta^*}}(\tilde\xi \ge a)
&= \sum_{\mathfrak j \in \mathfrak J}\pi_{\mathfrak j}\Prob_{\Q^{\btheta^*}}(\tilde\xi\ge a\mid Z=\mathfrak j) 
\geq  \pi_{\bar{\mathfrak  j}} \Prob_{\Q^{\btheta^*}}(\tilde\xi\ge a\mid Z=\bar{\mathfrak j})\\
&=\pi_{\bar{\mathfrak j}} \Prob_{\Q^{\btheta^*}}(G\ge ( a-\mu_{\bar {\mathfrak j}})/\sigma_{\bar{\mathfrak j}}) \\
  & \geq \epsilon  \Prob_{\Q^{\btheta^*}}(G\ge ( a-B)/\epsilon) \\
 &\geq (\epsilon/(2\sqrt{2\pi}))\exp(-((a-B)/\epsilon)^2)\\
 &\geq  \epsilon/(2\sqrt{2\pi}) \exp(-2\left(B/\epsilon \right)^2)\exp(-2 a^2/\epsilon ^2)
\end{align*}
where in the second inequality, we use $\pi_{\bar{\mathfrak j}}\ge \epsilon$, $\sigma_{\bar{\mathfrak j}}\geq \epsilon >0$, $\mu_{\bar{\mathfrak j}} \ge B$, and $a\geq B$, the  third inequality comes from \eqref{eq:lowernound:normal},  the last inequality is due to
$(a-b)^2\leq 2a^2+2b^2$.
With $\mathfrak c_1=\epsilon/(2\sqrt{2\pi}) \exp(-2\left(B/\epsilon\right)^2)>0$ and $\mathfrak c_2=\epsilon/\sqrt{2}>0$, we obtain that for $a\geq \bar{a}:=\max(0, 2\epsilon + B)$:%
\begin{equation}\label{tail:Q:lb}
   \Prob_{\Q_N}((\zeta-\hat\mu_N)/\hat\sigma_N \ge a) \ge \mathfrak c_1 \exp(-a^2/\mathfrak c_2^2).%
\end{equation}
\end{proof} %

\begin{lemma}
    If there exists $B>0$ such that $\max_{\mathfrak j  \in  \mathfrak J} \sigma_{\mathfrak j} \leq B$ and $\max_{\mathfrak j \in \mathfrak J} |\mu_{\mathfrak j}| \leq B$ for all $(\pi,\mu,\sigma)\in\Theta$, then a $\Q_N$  produced by Algorithm \ref{alg:fittingGMM} using a GMM within $\Theta$ satisfies Property \ref{property:unifSubGaussFit}.
\end{lemma}
\begin{proof}
Let $Z$ be the latent variable representing the components of the GMM. Then for all $a\geq 2B$,
\begin{align*}
  \mathbb P_{\Q_N}((\zeta-\hat\mu_N)/\hat\sigma_N \ge a)&=\sum_{\mathfrak j \in \mathfrak J} \pi_{\mathfrak j} \mathbb P(\tilde\xi \ge a \mid Z=\mathfrak j )\\
  &= \sum_{\mathfrak j  \in \mathfrak J} \pi_{\mathfrak j} \mathbb P\Big(G \ge \frac{a - \mu_{\mathfrak j}}{\sigma_{\mathfrak j}}\Big)\\&
   \le  \sum_{\mathfrak j \in \mathfrak J} \pi_{\mathfrak j}  \exp\big(-((a - \mu_{\mathfrak j})/(\sqrt{2}\sigma_{\mathfrak j}))^2\big)\\
  &\le \sum_{\mathfrak j \in \mathfrak J} \pi_{\mathfrak j}  \exp\big(-((a -B)/(\sqrt{2}B))^2\big)\\&
   \le \exp\big(-a^2/(8B^2)\big),
\end{align*}
where $G\sim \mathcal N(0,1)$ and the first inequality uses the standard Gaussian tail bound since $a-\mu_{\mathfrak j}\geq B>0$ and $\sigma_{\mathfrak j}>0$, the second inequality uses $\mu_{\mathfrak j} \le B \iff
(a - \mu_{\mathfrak j}) \ge (a - B)$ and $\sigma_{\mathfrak j} \le B$. The last inequality comes from the equivalences $
a \ge 2B \iff B\leq a/2\iff a- B \ge a/2 \implies
((a - B)/(\sqrt{2}B))^2
  \ge ((a/2)/(\sqrt{2}B))^2= (a/(2\sqrt{2}B))^2$
  
Applying the same argument to $-\tilde\xi$ and using $\mu_{\mathfrak j}\ge -B, \forall\, \mathfrak j \in \mathfrak J$ yields $
  \mathbb P(\tilde\xi \le -a) \le \exp\big(-a^2/(8B^2)\big), \forall a\ge 2B$. Thus, for all $a\ge 2B$, we have that $
  \mathbb P(|\tilde\xi|\ge a) 
  \le 2 \exp\big(-a^2/(8B^2)\big).$

Let $\mathfrak c_2^2 := \max(8B^2,\; \frac{4 B^2}{\log 2})$.

\textbf{Case 1: $a\ge 2B$.} Since $\mathfrak c_2^2 \ge 8B^2$,  we have that:
\begin{align*}
    \mathbb P(|\tilde\xi|\ge a)
\le 2 \exp(-a^2/8B^2)
\le 2 \exp(-a^2/\mathfrak c_2^2).
\end{align*}

\textbf{Case 2: $a\le 2B$.} We combine $\mathbb P(|\tilde\xi|\ge a)\le 1$ with $\mathfrak c_2^2 \ge 4B^2/\log 2$ to have that:
\begin{align*}
      \mathbb P(|\tilde\xi|\ge a) \leq 1 = 2\exp(-\log(2))\le 2\exp(-4B^2/\mathfrak c_2^2)\le 2\exp(-a^2/\mathfrak c_2^2)
\end{align*}
where the last inequality comes from $a\le 2B$. 
Therefore,  we have that for all $a\ge 0$,
\[
 \mathbb P_{\Q_N}(|(\zeta-\hat\mu_N)/\hat\sigma_N| \ge a) = \mathbb P(|\tilde\xi|\ge a) \leq 2\exp(-a^2/\mathfrak c_2^2).
 \]%
\end{proof}

\subsection{\revised{Proof of Proposition \ref{prop:identifiable}}}
\begin{proof}
The fact that $\hat{\zeta}_i\stackrel{F}{=} \hat{\zeta}'_i$ implies $\hat{\rho}_{\mathbb Q, n}\stackrel{F}{=} \hat{\rho}_{\mathbb Q', n}$ follows from construction. We therefore focus on the reverse. Given that $\frac{1}{\alpha}\log(y)$ is a strictly increasing bijection from $(0,\infty)$ to $\mathbb{R}$, it is clear that $\hat{\rho}_{\mathbb Q, n}\stackrel{F}{=} \hat{\rho}_{\mathbb Q', n}$ implies that $\bar{Z}_n\stackrel{F}{=}\bar{Z}_n'$, with $\bar{Z}_n:=(1/n)\sum_{i=1}^n\exp(\alpha \hat{\zeta}_i)>0$ and $\bar{Z}_n':=(1/n)\sum_{i=1}^n\exp(\alpha \hat{\zeta}_i')>0$.
Recalling that for any nonnegative random variable $Z$, its Laplace--Stieltjes transform is 
$L_Z(s):=\mathbb{E}[\exp(-sZ)]$ for $s\ge0$ \citep{Feller1971} and uniquely determines the distribution of $Z$  \cite[see][Theorem 1a, Chapter XIII]{Feller1971}, one can obtain that for all $s\geq 0$:
\begin{align*}
L_{\exp(\alpha\hat{\zeta}_1)}(s)^n &= (\Expect[\exp(-s \exp(\alpha\hat{\zeta}_1))])^n %
=\Expect[\exp(-s \sum_{i=1}^n \exp(\alpha\hat{\zeta}_i))]\\
&= L_{\bar{Z}_n}(sn)=L_{\bar{Z}_n'}(sn)=L_{\exp(\alpha\hat{\zeta}_1')}(s)^n,
\end{align*}
where we exploited the fact that $\{\hat{\zeta}_i\}_{i=1}^n$ are i.i.d. and omitted the repeated steps in terms of $\{\hat{\zeta}_i'\}_{i=1}^n$ to get to the final equality.
Taking the unique positive $n$-th root on the first and last expression yields $
L_{\exp(\alpha\hat{\zeta}_1)}(s) = L_{\exp(\alpha\hat{\zeta}_1')}(s)$ for all $s\ge0$. By the uniqueness theorem for Laplace-Stieltjes transforms \citep[][Theorem 1a, Chapter XIII]{Feller1971}, %
we must therefore have that $
\exp(\alpha\hat{\zeta}_1) \stackrel{F}{=} \exp(\alpha\hat{\zeta}_1')$. Again by the bijection property of $(1/\alpha)\log(y)$, we confirm that $\hat{\zeta}_1 \stackrel{F}{=} \hat{\zeta}_1'$.
\end{proof}

\section{Distributionally Robust Optimization}\label{sec:appenddro}
This appendix supports the DRO model proposed for entropic risk minimization in Section \ref{sec:dro}. 
\subsection{Main Theory}

For the optimal solution $\bm z^*$ of problem \eqref{eq:true:opt} to be well defined, we make the following standard assumptions:
\begin{assum}\label{assum:loss}
    We assume that: 
    \begin{enumerate}[label=(A.\arabic*), ref=(A.\arabic*)]
        \item \label{assum:compact_convex} $\mathcal{Z}$ is a compact and convex set.
        \item \label{assum:convex_loss} $\ell(\vecz, \vecxi)$ is convex in $\z$ for almost every $\vecxi \in \Xi$.
\item \label{ass:lipchitz_loss_xi} $\ell(\vecz, \vecxi)$ is $L$-Lipschitz continuous in $\vecxi$ for all $\z\in \Z$.
        \item \label{ass:lipchitz_loss_z} $\ell(\vecz, \vecxi)$ is $L(\vecxi)$-Lipschitz continuous in $\z$ for all $\vecxi\in \Xi$ with $\Expect_\Prob[L(\vecxi)^q]<\infty$ for all $q\geq1$.
        \item \label{assum:tail_bound_loss_z} $| \ell( \z, \vecxi)| \leq \bar{L}(\vecxi)$ for all $\z \in \mathcal{Z}$ almost surely, with the tail of $\bar{L}(\vecxi)$ exponentially bounded:
        \[
        \Prob\big(\bar{L}(\vecxi) > a\big) \leq G \exp(-a\alpha C) \quad \text{for all } a \geq 0,
        \]
        for some constants $G > 0$ and $C > 2$.
    \end{enumerate}
\end{assum}

As in Assumption \ref{assum:fit_tail_bound}, the  last assumption ensures that 
the mean and variance of the loss $\exp(\alpha \ell(\vecz, \vecxi))$ are finite for each $\z\in\mathcal Z$. We make the technical assumptions \ref{ass:lipchitz_loss_z} and \ref{assum:tail_bound_loss_z} to ensure the convergence of SAA solution to the true risk. %

\begin{proposition} \label{lemma:convrate:saa}
  Suppose that Assumption \ref{assum:loss} holds. 
  Then, for any $\gamma >0$, 
  there exists a constant $A>0$ 
  for which, \begin{align}\label{eq:conv:utility}
\Prob &\left( \left| \rho_{\text{\model{SAA}}}^* - \rho^* \right| \geq  \frac{A}{\sqrt{N\gamma}\alpha \exp(\alpha \rho^*)} \right) \leq  \gamma,
\end{align}
as long as $N$ is sufficiently large. Consequently, $\rho_{\text{\model{SAA}}}\rightarrow \rho^*$ in probability.
\end{proposition}

To prove Proposition \ref{lemma:convrate:saa} which details are included in Appendix \ref{proof:lemma:saa}, we show that $\exp(\alpha \ell(\vecz, \vecxi))$ is $\kappa(\vecxi)$-Lipschitz continuous in $\z$ for all $\vecxi\in \Xi$ with $\kappa(\vecxi)=\alpha L(\vecxi)\exp(\alpha \bar{L}(\vecxi))$.  
Then, we apply the uniform convergence results for heavy tailed distributions in \citet[][Theorem 3.2]{Jiang_Chen_2020} to show the uniform convergence of empirical utility to true utility for all $\z \in \Z$. This theorem only requires that the second moment of $\kappa(\vecxi)$ is finite instead of the usual light-tailed assumptions that don't hold for $\kappa(\vecxi)$. 
Subsequently, we use the properties of logarithm function in the neighborhood of zero to  show the uniform convergence of empirical risk to optimal risk.  This convergence result underpins the prevalent use of the SAA approach for entropic risk minimization problems \citep{Chen_He_2024, Chen_Ramachandra_2024}.
For a detailed examination of the SAA methodology within stochastic programming, see \cite{Shapiro_Dentcheva_2009}.

In the literature, different ambiguity sets have been considered with the Kullback Leibler (KL)-divergence \citep{Hu_Hong_2012} and Wasserstein ambiguity sets \citep{Mohajerin_Esfahani_Kuhn_2018} being the most commonly used \citep{Rahimian_Mehrotra_2022}. For KL-divergence-based ambiguity sets, the standard formulation \citep{Hu_Hong_2012} restricts the worst-case distribution to be absolutely continuous with respect to the empirical distribution, limiting its support to the same points as the empirical distribution. This poses a problem because it prevents the representation of worst-case scenarios that typically occur in the tails of the loss distribution. An alternative formulation does allow worst-case distributions with support beyond the empirical distribution, enabling a richer ambiguity set \citep{Chan_Van_Parys_2024}. However, with an unbounded loss function, this flexibility allows nature to exploit the  tail, resulting in infinite loss for the decision maker. Consequently, KL-divergence-based ambiguity sets are ill-suited to our problem, which involves unbounded support and heavy-tailed losses. This also holds for  type-$p$  Wasserstein ambiguity set with $p<\infty$  due to the following result.
\begin{proposition}\label{prop:infinite_cost}
    The $p$-Wasserstein DRO with entropic risk measure results in unbounded loss if $p<\infty$. 
\end{proposition}
Next, we show that type-$\infty$ Wasserstein ambiguity set is a suitable choice  for problem~\eqref{eq:DRO}. %
\cite{Bertsimas_Shtern_2023} have shown that  problem \eqref{eq:DRO} can be equivalently written as:
 \begin{equation}\label{eq:drotypeinf}
    \min_{\z \in \Z}  \sup_{\Q \in \tilde{\mathcal{B}}_{\infty}(\epsilon)}\frac{1}{\alpha}\log\left(\mathbb{E}_{\Q}[\exp(\alpha \ell(\z,\vecxi))]\right)= \min_{\z \in \Z}\frac{1}{\alpha}\log\left(\frac{1}{N}\sum_{i \in [N]} \sup_{\vecxi:\|\vecxi-\hat{\vecxi}_i\| \leq \epsilon} \exp(\alpha \ell(\vecz, \vecxi))  \right),
\end{equation}
 where the ambiguity set $\tilde{\mathcal{B}}_\infty(\epsilon) $  is defined as: 
 \begin{equation*}
\tilde{\mathcal{B}}_\infty(\epsilon) :=\left\{ \Q \in \mathcal{M}(\Xi) \rvert \exists\,\, \vecxi_i \in \Xi,  \|\vecxi_i-\hat{\vecxi}_i\| \leq \epsilon,\, \forall i\in [N],\, \Q(\vecxi) = \frac{1}{N}\sum_{i=1}^N  \delta_{\vecxi_i}(\vecxi)  \right\}.
  \end{equation*}
  
For piecewise concave loss functions, the following theorem  gives an equivalent reformulation of the DRO problem as the finite dimensional convex optimization problem using Fenchel duality \citep{Ben-Tal_denHertog_2015}.
 \begin{theorem}\label{thm:reg:cone}
Let $\ell(\z,\vecxi) = \max_{j\in[m]} \ell_j(\z,\vecxi)$ where $\ell_j(\z,\vecxi)$ is a concave function in $\vecxi$ for each $j\in [m]$ and $\z \in \Z$. Then, the     DRO problem \eqref{eq:drotypeinf} with type-$\infty$ Wasserstein ambiguity set is equivalent to 
\begin{equation}\label{eq:model:wassinfty}
    \begin{array}{lll}
       \min  &    \frac{1}{\alpha}\log \left(\frac{1}{N}\sum_{i=1}^{N}\exp(\alpha t_i)\right)\\ 
     \text{s.t.} & \bm t\in\mathbb{R}^N,\,\bm z\in\mathcal{Z},\,\bm\varphi_{ij}\in\mathbb{R}^d & \forall i\in [N],\,j\in [m]\\
     &   \V{\varphi}_{ij}^\top \hat{\vecxi}_i - \ell_{j*}(\z,\V{\varphi}_{ij})+ \epsilon \| \V{\varphi}_{ij} \|_{*}\leq t_i &\forall i\in [N],\,j\in [m],
    \end{array}
\end{equation}
where  $\ell_{j*}(\z,\V{\varphi}_{ij}) := \inf_{\vecxi} \{\V{\varphi}_{ij}^\top \vecxi-\ell_j(\z, \vecxi)\}$ is the partial concave conjugate of $\ell_j(\z, \vecxi)$, and $\|\cdot\|_*$ denotes the dual norm.
\end{theorem}
The DRO reformulation and reformulation technique simplify significantly when the loss function is either piecewise linear or linear, rather than piecewise convex in $\bm z$ and concave in $\vecxi$. The following corollary presents these special cases. 

\begin{corollary}\label{Corollary:DRO}
    Let $\ell(\z, \vecxi) := \max_{k\in \mathcal{K}}\left\{a_k(\z^\top \vecxi)+b_k\right\}$ be a piecewise linear function for given parameters $a_k$ and $b_k$.  Then, the DRO problem \eqref{eq:drotypeinf} with type-$\infty$ Wasserstein ambiguity set is equivalent to
    \begin{equation}\label{eq:DRO_piece}
    \begin{aligned}
\rho_{\model{DRO}}:=\quad
\min   \quad &  \frac{1}{\alpha}\log \left(\frac{1}{N}\sum_{i=1}^{N}t_i\right)\\
\text{s.t.}  \quad  & \bm t\in\mathbb{R}^N,\,\z\in \Z\\
\quad & \exp\left(\alpha \left(a_k(\z^\top \hat{\vecxi}_i) +b_k\right) +\epsilon \| a_k\z \|_{*}\right)\leq t_i  & \forall i, k\in \mathcal{K}
\end{aligned}
\end{equation}
which for a linear loss $\ell(\z, \vecxi)=\z^\top \vecxi$ can be further simplified to 
\begin{equation}\label{eq:DRO_lin}
    \min_{\z\in \Z}   \frac{1}{\alpha}\log \left(\frac{1}{N}\sum_{i=1}^{N}\exp(\alpha \z^\top \hat{\vecxi}_i)\right)+\epsilon \norm{\z}_{*}.
\end{equation}
\end{corollary}

    \begin{proof}
        Here, we provide an alternative proof for the piecewise linear loss functions  $\ell(\z, \vecxi) := \max_{k\in \mathcal{K}}\left\{a_k(\z^\top \vecxi)+b_k\right\}$ that does not rely on Fenchel duality \citep{Ben-Tal_denHertog_2015}. 
The supremum of $\exp(\max_{k\in \mathcal{K}}\left\{\alpha\left(a_k(\z^\top \vecxi)+b_k\right)\right\})$ over the set $\{\vecxi:\|\vecxi-\hat{\vecxi}_i\| \leq \epsilon\}$ is given by:
       \begin{align*}
          \sup_{\vecxi:\|\vecxi-\hat{\vecxi}_i\| \leq \epsilon} &\exp\left(\max_{k\in \mathcal{K}}\left\{\alpha\left(a_k(\z^\top \vecxi)+b_k\right)\right\}\right) \\&=  \sup_{\vecxi:\|\vecxi-\hat{\vecxi}_i\| \leq \epsilon} \max_{k\in \mathcal{K}} \left\{\exp\left(\alpha\left(a_k(\z^\top \vecxi)+b_k\right)\right)\right\}\\
          &=\max_{k\in \mathcal{K}} \left\{ \exp\left(\alpha \left(a_k\z^\top\hat{\vecxi}_i +b_k\right)+\alpha \sup_{\vecxi:\|\vecxi\| \leq \epsilon}\left(a_k\z^\top \vecxi\right) \right)\right\}\\&=
          \max_{k\in \mathcal{K}} \left\{ \exp\left(\alpha \left(a_k(\z^\top \hat{\vecxi}_i) +b_k\right) +\alpha \epsilon \| a_k\z \|_{*}\right)\right\},
       \end{align*}
       where the first equality follows from interchanging $\exp$ and $\max$ operations and then using the fact that $\exp(\cdot)$ is increasing in its arguments,  last equality follows from  the definition of the dual norm and $\| \cdot\|_{*}$ denotes the dual norm of $\| \cdot\|$.
       On combining with the objective function in \eqref{eq:drotypeinf}, we obtain:
       \begin{align*}
          &\frac{1}{\alpha}\log\left(\frac{1}{N}\sum_{i =1}^N 
          \max_{k\in \mathcal{K}} \left\{ \exp\left(\alpha \left(a_k(\z^\top \hat{\vecxi}_i) +b_k\right) +\alpha \epsilon \| a_k\z \|_{*}\right)\right\} \right). 
       \end{align*}
       So, with a piecewise linear loss function, problem \eqref{eq:drotypeinf} is equivalent to the  convex optimization problem in \eqref{eq:DRO_piece}.
Further, specializing the result to a linear loss function $\ell(\vecz, \vecxi)=\z^\top \vecxi$, problem \eqref{eq:drotypeinf} is equivalent to the regularized risk-averse SAA problem in \eqref{eq:DRO_lin}.
    \end{proof}

It is interesting to see that for the linear case, the DRO problem  reduces to the regularized SAA problem where the %
regularization penalty   is controlled by the size $\epsilon$ of the ambiguity set  and that the type of penalty depends on the dual of the norm used to define the ambiguity set. To complement these results, we also  %
provide reformulations of the distributionally robust newsvendor  and  regression problems as exponential cone programs.

Our next theorem formalizes that as the sample size $N$ tends to infinity, the DRO value $\rho_{\model{DRO}}^*$ with a properly chosen radius will converge to the true optimal risk $\rho^*$ in probability. 
The proof follows from showing that for  Lipschitz continuous (in $\z$) loss functions, $\rho_\model{SAA}^*\leq \rho_\model{DRO}^*\leq \rho_\model{SAA}^*+L\epsilon$ and using Proposition \ref{lemma:convrate:saa} that establishes that $\rho_\model{SAA}^*$ converges to $\rho^*$ at the rate $\mathcal{O}(1/\sqrt{N})$ for  locally Lipschitz continuous (in $\vecxi$) loss functions. Finally, choosing the radius to decay at the rate $\mathcal{O}(1/\sqrt{N})$ preserves the rate of convergence of SAA.
  \begin{theorem}\label{thm:converge:dro}
Suppose that Assumption \ref{assum:loss} holds. %
Then for any $\gamma>0$ and  %
using $\mathcal{B}_\infty(c/\sqrt{N})$, for some $c>0$, then there exists  a constant \(A > 0\) such that %
\[ \Prob\left(\lvert \rho_\model{DRO}^* - \rho^* \rvert \geq \frac{A}{\sqrt{N\gamma}\alpha \exp(\alpha \rho^*)}+\frac{c}{\sqrt{N}} \right) \leq \gamma, \]
as long as $N$ is sufficiently large. Consequently, $\rho_\model{DRO}^*\rightarrow \rho^*$ in probability.
\end{theorem}
While Theorem~\ref{thm:converge:dro} provides a rate for $\epsilon$ that ensures convergence of the DRO risk to the true optimal risk in probability, the values of the constants depend on the unknown underlying probability distribution.  In practice, $\epsilon$ needs to be estimated using K-Fold CV described in Algorithm \ref{alg:KFold}. However, %
estimating true risk from finite data is challenging. To address this, we %
employ the bias-aware estimation procedure  described in Section~\ref{sec:bias_correct}.

\begin{algorithm}[t]
\caption{K-fold cross validation \label{alg:KFold}}
\begin{algorithmic}[1]
\Function{K-foldCV}{$K, \Dist_N, \epsilon$}
    \State $\mathcal{S} \gets \emptyset$
    \For{$k \gets 1$ \textbf{to} $K$}
        \State $\Dist_{-k} \gets \Dist_N \setminus \Dist_k$ \Comment{Training data (all samples except those in fold $k$)}
       \State $\hat{\Prob}^K_{-k} \gets$ empirical distribution of scenarios in $\Dist_{-k}$
        \State Solve  problem \eqref{eq:model:wassinfty} with distribution $\hat{\Prob}^K_{-k}$ and radius $\epsilon$ to get $\z^*(\hat{\Prob}^K_{-k} , \epsilon)$  
        \State $\mathcal{S} \gets \mathcal{S} \cup \{\ell( \z^*(\hat{\Prob}^K_{-k} , \epsilon), \vecxi) \mid \vecxi \in \Dist_k\}$ 
    \EndFor
    \State \Return $\mathcal{S}, \rho_{k\sim U(K)}(\rho_{\vecxi\sim \hat{\Prob}_{k}^K}(\ell( \z^*(\hat{\Prob}^K_{-k} , \epsilon), \vecxi)))$
\EndFunction
\end{algorithmic}
\end{algorithm}

\subsection{Detailed proofs}
\subsubsection{Proof of Proposition~\ref{lemma:convrate:saa}} \label{proof:lemma:saa}
\begin{proof} The proof relies on the following lemma.

\begin{lemma}\label{lemma:log}
The following inequalities follow from the properties of the logarithm function:
\begin{equation*}
\begin{array}{ll}
    \log(1+\epsilon)\leq \epsilon &\text{ if } \epsilon\geq0, \\
        \log(1-\epsilon)\geq -\epsilon/(1-1/e) & \text{ if } \epsilon\in [0, 1-1/e]. 
\end{array}
\end{equation*}
\end{lemma}
\begin{proof}
    The logarithm function is concave, thus it follows that
    \begin{equation*}
    \log(1+\epsilon) \leq \log(1)+ \epsilon \frac{d\log(x)}{dx}\bigg|_{x=1} = \epsilon.
    \end{equation*}
    Moreover, by concavity of $\log(1-\epsilon)$ for $\epsilon\in [0, 1-1/e]$, we have:
    \begin{equation*}
    \log(1-\epsilon) \geq  \log(1) + \epsilon \frac{\log(1/e)-\log(1)}{1-1/e}= -\frac{\epsilon}{1-1/e}.
    \end{equation*}
Figure~\ref{fig:combinedlog} gives a pictorial representation of the result.
\end{proof}

  We will show that  $\rho_\model{SAA}^* $ converges to $\rho^*$ by demonstrating that $\rho_{\hat{\mathbb P}_N}(\ell(\z, \vecxi))$ converges to $\rho_{{\mathbb P}}(\ell(\z, \vecxi))$ for all $\z\in \Z$, i.e., a uniform rate of convergence. The proof is conducted in two steps. In the first step, we apply \citet[][Theorem 3.2]{Jiang_Chen_2020} to show that the empirical exponential utility, $\Expect_{\hat{\mathbb{P}}_N}[\exp(\alpha \ell(\z, \vecxi))]$, converges uniformly to the true exponential utility, $\Expect_{\mathbb{P}}[\exp(\alpha \ell(\z, \vecxi))]$. In the second step, we leverage the properties of the logarithm function to extend this uniform convergence from the empirical exponential utility to the entropic risk measure. 

To apply \citet[][Theorem 3.2]{Jiang_Chen_2020}, we need to verify several properties for the exponential utility minimization problem. It is easy to see that  $(i)$  $\mathcal{Z}$ is compact by Assumption \ref{assum:loss}\ref{assum:compact_convex}; $(ii)$ the exponential utility $\Expect_{\mathbb{P}}[\exp(\alpha \ell(\z, \vecxi))]$ is continuous in $\bm z$ and  $(iii)$ the empirical distribution $\hat{\Prob}_N$ is constructed using $N$  i.i.d samples from $\Prob$; $(iv)$ Assumption \ref{assum:loss}\ref{assum:tail_bound_loss_z} implies that $\ell(\z,\vecxi)$ satisfies Assumption \ref{assum:tail_bound2} for all $\z\in\mathcal{Z}$ so that Lemma \ref{thm:boundedExpVar2} confirms that $\Expect[\left(\exp(\alpha \ell(\z, \vecxi)\right)^2]<\infty$ for all $\z\in \Z$.
Lastly, $(v)$ the following inequalities verify that $\exp(\alpha \ell(\z, \vecxi))$ is H-calm from above  with Lipschitz modulus $\kappa(\vecxi):= \alpha L(\vecxi)\exp(\alpha \bar{L}(\vecxi))$ and order one: 
\begin{align*} 
\exp(\alpha \ell(\z_2, \vecxi)) - \exp(\alpha \ell(\z_1, \vecxi))   & \ \leq \left(\left.\partial_\z \exp(\alpha \ell(\z, \vecxi))\right|_{\z = \z_2}\right)^\top\left(\z_2-\z_1\right) \\ 
& \ =   \alpha\exp(\alpha \ell(\z_2, \vecxi)) \left( \partial_{\z}\ell(\z_2, \vecxi)\right)^\top\left(\z_2-\z_1\right) \\
& \ \leq \ \alpha\exp(\alpha \ell(\z_2, \vecxi)) \| \partial_\z\ell(\z_2, \vecxi)\|_* \| \z_2-\z_1\| \\
& \ \leq \ \alpha \exp(\alpha \bar{L}(\vecxi)) L(\vecxi)\lVert  \z_2 - \z_1 \rVert. 
\end{align*}
The first inequality follows from the convexity of the exponential utility, where $\partial_\z $ denotes the subgradient operator with respect to $\z$, and the  equality follows by applying the chain rule. The second inequality follows from the definition of the dual norm. Since 
by Assumption~\ref{assum:loss}, $\ell(\z, \vecxi)$ is locally Lipschitz continuous with Lipschitz constant $L(\vecxi)$, and  the dual norm of the subgradient of $\ell(\z, \vecxi)$ is bounded by the Lipschitz constant $L(\vecxi)$ \citep[Lemma 2.6]{Shalev-Shwartz_2012}, we obtain the final inequality.  Finally, we can confirm that the second moment of $\kappa(\vecxi)$ is bounded. Namely, letting $0<\varsigma<(C/2)-1$, based on Hölder's inequality, we have that:
\begin{align*}
\Expect[\kappa(\vecxi)^2]&=\alpha^2 \Expect[L(\vecxi)^2\exp(2\alpha\bar{L}(\vecxi))]\leq \alpha^2 \Expect[L(\vecxi)^{2(1+1/\varsigma)}]^{\varsigma/(1+\varsigma)}\Expect[\exp(2(1+\varsigma)\alpha\bar{L}(\vecxi))]^{1/(1+\varsigma)}
\end{align*}
We can further show using Assumption \ref{assum:loss}\ref{ass:lipchitz_loss_z} that :
\begin{align*}
\Expect[\exp(2(1+\varsigma)\alpha\bar{L}(\vecxi))]&\leq  \int_0^{\infty} \Prob(\exp(2(1+\varsigma)\alpha\bar{L}(\vecxi))> x) dx\\
        & = \int_0^{\infty} \Prob(\exp(2(1+\varsigma)\alpha\bar{L}(\vecxi))> \exp(2(1+\varsigma)\alpha y)) 2(1+\varsigma)\alpha \exp(2(1+\varsigma)\alpha y) dy\\
        & = \int_0^{\infty} \Prob(\bar{L}(\vecxi)> y) 2(1+\varsigma)\alpha \exp(2(1+\varsigma)\alpha y) dy\\
        &\leq 2(1+\varsigma)\alpha \int_0^{\infty} G \exp(-(C-2(1+\varsigma))\alpha y) dy=\frac{2G(1+\varsigma)}{C-2(1+\varsigma)}<\infty,    
\end{align*}
Moreover, $\Expect[L(\vecxi)^{2(1+1/\varsigma)}]$ is finite   based on Assumption \ref{assum:loss}\ref{assum:tail_bound_loss_z}, which allows us to conclude that $\Expect[\kappa(\vecxi)^2]<\infty$.

We now have all the components to apply \citet[][Theorem 3.2]{Jiang_Chen_2020} which states that for any $\gamma>0$, there exists an $A>0$, independent of $N$, such that: 
\begin{equation}\label{eq:conv:utility2}
\Prob \left( \sup_{\z \in \mathcal{Z}} \left| \Expect_{\hat{\Prob}_N}[\exp(\alpha \ell(\z, \vecxi))] - \Expect_{\Prob}[\exp(\alpha \ell(\z, \vecxi))] \right| \geq A/\sqrt{N\gamma}\right) \leq \gamma,
\end{equation}
for sufficiently large $N$, thus concluding that the empirical exponential utility converges uniformly to the true exponential utility, with a rate of convergence of $\mathcal O(1/\sqrt{N})$.

Next, we show that empirical risk converges  uniformly to the true risk by exploiting the properties of the $\log$ function, namely that $\log(1+\epsilon)\leq \epsilon$ for all $\epsilon\geq 0$, and $\log(1-\epsilon)\geq -\epsilon/(1-1/e)) $ for all $\epsilon \in [0, 1-1/e]$. This result is summarized in Lemma~\ref{lemma:log} and can be pictorially verified in Figure~\ref{fig:combinedlog}.

    \begin{figure}[htbp]
    \centering
    \begin{subfigure}[htbp]{0.48\linewidth}
        \centering
        \includegraphics[width=\linewidth]{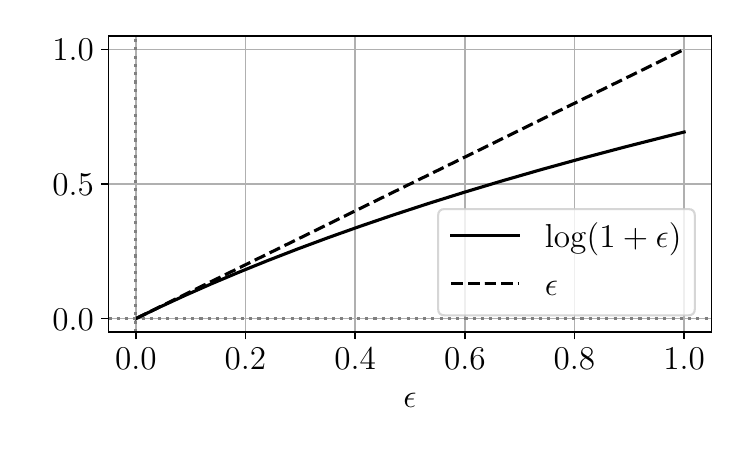}
        \caption{$\log(1+\epsilon)\leq \epsilon,\; \forall \epsilon\geq 0$.}
        \label{fig:loglb}
    \end{subfigure}
    \hfill
    \begin{subfigure}[htbp]{0.48\linewidth}
        \centering
        \includegraphics[width=\linewidth]{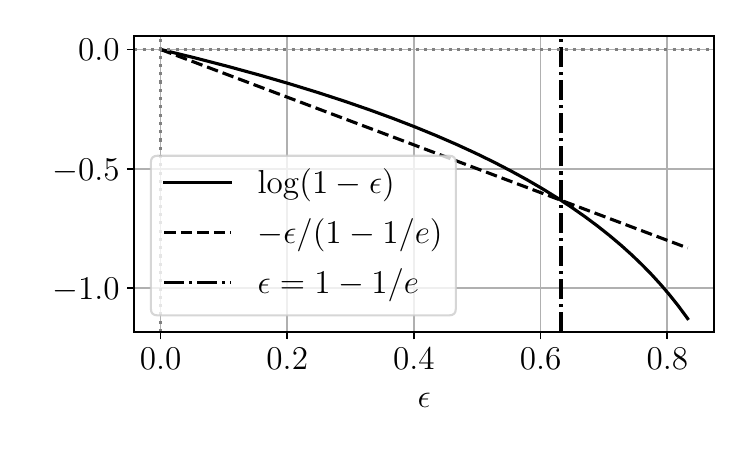}
        \caption{$\log(1-\epsilon)\geq -\epsilon/(1-1/e) \;\forall \epsilon\in [0, 1-1/e]$}
        \label{fig:logub}
    \end{subfigure}
    \caption{Plot of the inequalities that the $\log$ function satisfies around $0$.}
    \label{fig:combinedlog}
\end{figure}

In the subsequent steps of the proof, we consider $N$ to be large enough such that $\frac{A}{\sqrt{N\gamma} \exp(\alpha \rho^*)}  \leq 1-1/e$, with $\rho^*:=\min_{\z\in\Z} (1/\alpha)\log(\Expect_{\Prob}[\exp(\alpha \ell(\z, \vecxi))])$ and focus on scenarios for which  the  absolute difference between the empirical utility and true utility in \eqref{eq:conv:utility2} is upper bounded by $A/(\sqrt{N\gamma})$ for all $\z\in \mathcal{Z}$. Specifically, we will show that it implies  that  the  absolute difference between the empirical risk and true risk is upper bounded by $A/(\alpha \sqrt{N\gamma}\exp(\alpha \rho^*))$ for all $\z\in \mathcal{Z}$.  To simplify notation, we further denote by $\hat{u}_N(\z):= \Expect_{\hat{\Prob}_N}(\exp(\alpha \ell(\z, \vecxi))) $, $u(\z):=\Expect_{\Prob}[\exp(\alpha \ell(\z, \vecxi))]$, $u^*:=\min_{\z\in\Z} \Expect_{\Prob}(\exp(\alpha \ell(\z, \vecxi)))=\exp(\alpha\rho^*)$. For any $\gamma>0$,  from \eqref{eq:conv:utility2}, we have: 
\begin{equation*}
\begin{array}{lll}
 &\lvert \hat{u}_{N}(\z) - u(\z)\rvert <  \frac{A}{\sqrt{N\gamma}}\,& \forall \z\in \Z\\
    \implies& \frac{\lvert\hat{u}_N(\z)-u(\z)\rvert}{u^*}  <  \frac{A}{\sqrt{N\gamma}\exp(\alpha \rho^*)} &\forall \z\in \Z\\
           \implies&  \frac{\lvert\hat{u}_N(\z)-u(\z)\rvert}{u(\z)} <  \frac{A}{\sqrt{N\gamma}\exp(\alpha \rho^*)}& \forall \z\in \Z\\
        \implies& \frac{\hat{u}_N(\z)}{u(\z)}  <  1+\frac{A}{\sqrt{N\gamma}\exp(\alpha \rho^*)} &\forall \z\in \Z\\
   \implies &\frac{1}{\alpha}\log\left(\frac{\hat{u}_N(\z)}{u(\z)}\right)  <  \frac{1}{\alpha}\log\left(1+\frac{A}{\sqrt{N\gamma}\exp(\alpha \rho^*)}\right) &\forall \z\in \Z,
  \end{array}
\end{equation*}
where in the first implication we divided both sides by $u^*$, while the second implication follows since   $u^* \leq \min_{\z\in \Z} u(\z)$. The last implication follows by  applying the logarithm on both sides of the inequality and dividing by $\alpha$.  The last expression can be further simplified  since by the properties of the logarithm function, we have $\log(1+\epsilon)\leq \epsilon$ for all $\epsilon \geq 0$, hence:
\begin{equation} \label{eq:right}
\frac{1}{\alpha}\log\left(\hat{u}_N(\z)\right) -\frac{1}{\alpha}\log\left(u(\z)\right)  <  \frac{A}{\alpha\sqrt{N\gamma}\exp(\alpha \rho^*)}<  \frac{A}{(1-1/e)\alpha\sqrt{N\gamma}\exp(\alpha \rho^*)},\; \forall \z \in \Z.
\end{equation}
 Similarly, for any $\gamma>0$,  from \eqref{eq:conv:utility2} we have%
 \begin{equation*}
\begin{array}{lll}
 &\lvert\hat{u}_N(\z) - u(\z)\rvert <  \frac{A}{\sqrt{N\gamma}}\,& \forall \z\in \Z\\
           \implies&  \frac{\lvert\hat{u}_N(\z)-u(\z)\rvert}{u(\z)} <  \frac{A}{\sqrt{N\gamma}\exp(\alpha \rho^*)}& \forall \z\in \Z\\
        \implies&  -\frac{A}{\sqrt{N\gamma}\exp(\alpha \rho^*)}  < \frac{\hat{u}_N(\z)}{u(\z)}-1,& \forall \z\in \Z  \\
  \implies & \frac{1}{\alpha}\log\left(1-\frac{A}{\sqrt{N\gamma}\exp(\alpha \rho^*)} \right)  <\frac{1}{\alpha}\log\left(\frac{\hat{u}_N(\z)}{u(\z)}\right)& \forall \z\in \Z,\\

  \end{array}
\end{equation*}
where the first and last implication follow  as before. %
The last expression can be further simplified given that $\frac{A}{\sqrt{N\gamma} \exp(\alpha \rho^*)} \leq 1-1/e$ for sufficiently large $N$, which allows us to apply the property of the logarithm function,  $\log(1-\epsilon)\geq -\epsilon/(1-1/e))$ for all $ \epsilon\in [0, 1-1/e]$, to obtain:
\begin{equation}\label{eq:left}
 -\frac{A}{(1-1/e)\alpha\sqrt{N\gamma}\exp(\alpha \rho^*)} <  \frac{1}{\alpha}\log\left(\hat{u}_N(\z)\right) -\frac{1}{\alpha}\log\left(u(\z)\right)\; \forall \z \in \Z.
\end{equation}
We now combine \eqref{eq:right} and \eqref{eq:left} as follows:%
\begin{equation}
    \left| \frac{1}{\alpha}\log\left(\hat{u}_N(\z)\right) -\frac{1}{\alpha}\log\left(u(\z)\right) \right| < \frac{A}{(1-1/e)\alpha\sqrt{N\gamma}\exp(\alpha \rho^*)}\; \forall \z \in \Z.
\end{equation}
From \eqref{eq:conv:utility2}, we obtain:
\begin{align*}
     \Prob \left( \sup_{\z \in \mathcal{Z}} \left| \rho_{\hat{\Prob}_N}(\z)- \rho_{\Prob}(\z) \right| <  \frac{A}{(1-1/e)\alpha\sqrt{N\gamma}\exp(\alpha \rho^*)} \right)  
     &\geq \Prob \left( \sup_{\z \in \mathcal{Z}} \left| \hat{u}_N(\z) - u(\z) \right| < \frac{A}{\sqrt{N\gamma}}\right)\\
     &>1-\gamma,
\end{align*}
which implies that:
\begin{equation*}
\Prob \left( \left| \min_{\z \in \mathcal{Z}}  \rho_{\hat{\Prob}_N}(\z)- \min_{\z \in \mathcal{Z}}\rho_{\Prob}(\z) \right| <  \frac{A}{(1-1/e)\alpha\sqrt{N\gamma}\exp(\alpha \rho^*)}\right) >1-\gamma.
\end{equation*}
The latter can be rewritten using our notation as: 
\begin{equation*}
\Prob \left( \left| \rho_{\model{SAA}}^* - \rho^* \right| \geq  B/(\sqrt{N\gamma}\alpha \exp(\alpha \rho^*)) \right) \leq  \gamma,
\end{equation*}
with $B:=A/(1-1/e)>0$.

To prove the convergence in probability, we simply consider any $\gamma>0$ and $\Delta>0$ and confirm that:
\[\Prob \left( \left| \rho_{\model{SAA}}^* - \rho^* \right| \geq \Delta\right)\leq \gamma,\]
as long as $N$ is large enough for $\frac{A}{\sqrt{N\gamma} \exp(\alpha \rho^*)}  \leq (1-1/e)\min(\Delta,1)$. Indeed, for such $N$, we necessarily have that:
\[\Prob \left( \left| \rho_{\model{SAA}}^* - \rho^* \right| \geq \Delta\right)\leq \Prob \left( \left| \rho_{\model{SAA}}^* - \rho^* \right| \geq \frac{A}{(1-1/e)\sqrt{N\gamma} \exp(\alpha \rho^*)}\right) \leq \gamma. \]
\end{proof}

\subsubsection{Proof of Proposition \ref{prop:infinite_cost} }
\begin{proof}
Let $u(\z)$ denote the worst-case expected utility for a given decision $\z$, that is,
\begin{align}
  u(\z):=\sup_{\Q \in \mathcal{B}(\epsilon)}u_\Q(\ell(\z, \vecxi))=\mathbb{E}_{\Q}[\exp(\alpha \ell(\z,\vecxi))],
\end{align} 
where the ambiguity set $\mathcal{B}(\epsilon)$  of radius $\epsilon\geq 0$ is defined as:
\begin{equation*} 
\mathcal{B}(\epsilon) := \left\{
         \Q \in \mathcal{M}(\Xi)\lvert  \Q\left\{\vecxi \in \Xi\right\} = 1,    \mathcal{W}^p(\Q, \hat{\Prob}_{N})\leq \epsilon \right\}.
\end{equation*}
The type-$p$ Wasserstein distance, $\mathcal{W}^p\left(\Prob_1,\Prob_{2}\right)$, is defined as:
\begin{align*}
&\mathcal{W}^p\left(\Prob_1,\Prob_{2}\right)=\inf_{\pi\in \mathcal{M}(\Xi\times\Xi)}\left\{\left(\int_{\Xi}\int_{\Xi}\|{\vecxi}_1-{\vecxi}_2\|^p\, \pi (d\vecxi_1,d\vecxi_2)\right)^{1/p}\right\},
\end{align*}
where $\pi$ is a joint distribution of ${\bm \eta}_1$ and ${\bm \eta}_2$
with marginals $\Prob_1$ and $\Prob_2$, respectively.
From \citet[][Lemma 2, Proposition 2]{Gao_Kleywegt_2023},  the worst-case utility $u(\z)$ for any $\z\in \Z$ is infinite with $p$-Wasserstein $(p<\infty)$ ambiguity set since the loss function $\exp(\alpha \ell(\z,\vecxi))$ does not satisfy the growth condition, that is, there does not exist any $\vecxi_0 \in \Xi$ and constants $L>0, M>0$ such that $\exp(\alpha \ell(\z,\vecxi)) \leq L \|\vecxi-\vecxi_0\|^p +M$ holds for all $\vecxi \in \Xi$.

Next, we prove by contradiction that the worst-case entropic risk    $\rho(\z):=\sup_{\Q \in \mathcal{B}(\epsilon)}\rho_\Q(\ell(\z, \vecxi))$ is  unbounded for all $\z\in \Z$. Suppose that $\rho(\z)$ is bounded, implying that   there exists an $  M<\infty$ such that:
 \begin{align}
\rho(\z) =  \sup_{\Q \in \mathcal{B}(\epsilon)}\frac{1}{\alpha} \log \left(u_{\Q }(\z)\right) \leq M,
 \end{align}
 which can be equivalently written as: 
\begin{equation}
 u_\Q(\ell(\z, \vecxi)) \leq \exp(\alpha M) \text{ for all } \Q\in \mathcal{B}(\epsilon),
\end{equation}
since the exponential function is a monotonically increasing.  This implies that $u(\z)$ is bounded,   which contradicts the result in \cite{Gao_Kleywegt_2023}. %
Hence, the worst-case  entropic risk $\rho(\z)$ is infinite for all $\z \in \Z$.
 \end{proof}

\subsubsection{Proof of Theorem \ref{thm:reg:cone}}\label{append:sec:reg:cone}
\begin{proof}
       Problem~\eqref{eq:drotypeinf} can be  equivalently written as:
       \begin{equation*}
       \begin{array}{ll}
          \min   &\frac{1}{\alpha}\log\left(\frac{1}{N}\sum_{i=1}^N \exp(\alpha t_i) \right)\\
           \text{s.t.} & \bm t\in\mathbb{R}^N,\,\bm z\in\mathcal{Z} \\
           &\displaystyle t_i \geq \sup_{\vecxi \in  \Xi^i_\epsilon} \ell(\z,\vecxi) \quad      \forall i\in [N],
       \end{array}
       \end{equation*}
       where $\Xi^i_\epsilon = \{\vecxi:\|\vecxi-\hat{\vecxi}_i\| \leq \epsilon\}$. Since $\ell(\z, \vecxi)=\max_{j\in[m]}\ell_j(\z, \vecxi)$, we obtain:
            \begin{equation}
       \begin{array}{ll}
          \min_{\bt}  &\frac{1}{\alpha}\log\left(\frac{1}{N}\sum_{i=1}^N \exp(\alpha t_i) \right)\\
          \text{s.t.} & \bm t\in\mathbb{R}^N, \,\bm z\in\mathcal{Z}, \\
            & \displaystyle t_i \geq \sup_{\vecxi \in  \Xi^i_\epsilon}   \ell_j(\z, \vecxi) \quad\forall i\in[N], j \in [m].
       \end{array}\label{eq:nonlinearcons}
       \end{equation}
       The relative interior of the intersection of set $\Xi^i_\epsilon$ and domain of $\ell_j(\z, \vecxi)$ is non-empty for all $\epsilon \geq0$. So, we can use  Fenchel duality theorem \citep{Ben-Tal_denHertog_2015} to obtain: 
      \begin{subequations}\label{Fenchel}
       \begin{equation}
         \sup_{\vecxi \in \Xi_\epsilon^i} \ell_j(\z, \vecxi)= \inf_{\V{\varphi}_{ij}} \delta^*(\V{\varphi}_{ij} \rvert \Xi_\epsilon^i) - \ell_{j*}(\z,\V{\varphi}_{ij}), 
       \end{equation}
       where $\V{\varphi}_{ij}\in \Re^d$, $\ell_{j*}(\z,\V{\varphi}_{ij}) := \inf_{\vecxi} \{\V{\varphi}_{ij}^\top \vecxi-\ell_j(\z, \vecxi)\}$ is the partial concave conjugate of $\ell_j(\z, \vecxi)$ and $ \delta^*(\V{\varphi}_{ij} \rvert \Xi^i_\epsilon) $ is the support function of $\Xi_\epsilon^i$, i.e.,
              \begin{align}
          & \delta^*(\V{\varphi}_{ij} \rvert \Xi^i_\epsilon) = \sup_{\vecxi \in \Xi^i_\epsilon} \V{\varphi}_{ij}^\top \vecxi = \V{\varphi}_{ij}^\top\hat{\vecxi}_i + \sup_{\V{\zeta}: \|\V{\zeta}\|\leq \epsilon} \V{\varphi}_{ij}^\top\V{\zeta} = \V{\varphi}_{ij}^\top \hat{\vecxi}_i+ \epsilon \| \V{\varphi}_{ij} \|_{*}, 
       \end{align}
       \end{subequations}
       where the last equality follows by the definition of the dual norm. Substituting \eqref{Fenchel} in \eqref{eq:nonlinearcons} results in the following finite dimensional  conic program:
       
       \begin{equation*} 
    \begin{array}{lll}
       \min  &  \displaystyle \frac{1}{\alpha}\log \left(\frac{1}{N}\sum_{i=1}^{N}\exp(\alpha t_i)\right)\\
       
     \text{s.t.}   & \bm t\in\mathbb{R}^N,\,\bm z\in\mathcal{Z},\,\bm\varphi_{ij}\in\Re^d&\forall i\in [N],\,j\in [m]\\
     &\displaystyle  \V{\varphi}_{ij}^\top \hat{\vecxi}_i+ \epsilon \| \V{\varphi}_{ij} \|_{*} - \ell_{j*}(\z,\V{\varphi}_{ij})\leq t_i &\forall i\in [N],\,j\in [m].
    \end{array}
\end{equation*}
    \end{proof}

\subsubsection{Proof of Theorem \ref{thm:converge:dro}}\label{proof:droconverge}
\begin{proof}
We will show that $\rho_\model{DRO}^*$ converges to $\rho^*$ by examining the relationship between $\rho_\model{DRO}^*$  and $\rho_\model{SAA}^*$,  and  $\rho_\model{SAA}^*$ and $\rho^*$. The latter relationship is established in Proposition \ref{lemma:convrate:saa}.

Concerning the relationship between $\rho_\model{SAA}^*$ and $\rho_\model{DRO}^*$ defined using the ambiguity set $\mathcal{B}_\infty(\epsilon)$ with any $\epsilon\geq 0$, 
since $\hat{\Prob}_N \in \mathcal{B}_\infty(\epsilon)$, the following inequality holds: 
\begin{equation*}
\sup_{\Q \in \mathcal{B}_\infty(\epsilon)}\frac{1}{\alpha}\log\left(\Expect_{\Q}[\exp(\alpha \ell(\z,\vecxi))]\right) \geq \frac{1}{\alpha}\log(\Expect_{\hat{\Prob}_N}[\exp(\alpha \ell(\z,\vecxi))])\quad \forall \bm z\in\mathcal{Z},
\end{equation*}
which implies that $\rho_\model{DRO}^*   \geq \rho_\model{SAA}^*$. Moreover, we know from   \cite{Bertsimas_Shtern_2023} that
\begin{equation}\label{eq:worstcaseLip}
    \sup_{\Q \in \mathcal{B}_\infty(\epsilon)}\frac{1}{\alpha}\log\left(\Expect_{\Q}[\exp(\alpha \ell(\z,\vecxi))]\right) = \frac{1}{\alpha}\log\left(\frac{1}{N}\sum_{i\in[N]}\sup_{\vecxi:\|\vecxi-\hat{\vecxi}_i\| \leq \epsilon} \exp(\alpha \ell(\z, \vecxi)) \right).
\end{equation}
In addition, from the L-Lipschitz continuity of $\ell$  in $\vecxi$ for all $\z \in \Z$ (see Assumption \ref{assum:loss}\ref{ass:lipchitz_loss_xi}), we have
\begin{equation*}
    \begin{array}{lll}
         &|\ell(\z, \vecxi) - \ell(\z, \hat{\vecxi}_i)| \leq L \|\vecxi - \hat{\vecxi}_i\| &\forall \bm\eta, \forall i\\
       \implies& |\ell(\z, \vecxi) - \ell(\z, \hat{\vecxi}_i)|  \leq L \epsilon& \forall \vecxi \in \{\vecxi: \|\vecxi - \hat{\vecxi}_i\| \leq  \epsilon\},\forall i \\
      \implies & \displaystyle \sup_{\vecxi:\|\vecxi_i-\hat{\vecxi}\| \leq \epsilon} \exp(\alpha \ell(\z, \vecxi)) \leq \exp(\alpha (\ell(\z, \hat{\vecxi}_i) + L\epsilon))&\forall i, \\
    \end{array}
\end{equation*} 
where the first implication follows from the definition of the ambiguity set. Substituting the resulting  inequality in \eqref{eq:worstcaseLip} results in
 \begin{equation*}
     \begin{array}{lll}
     \displaystyle\sup_{\Q \in \mathcal{B}_\infty(\epsilon)}\frac{1}{\alpha}\log\left(\Expect_{\Q}\left[\exp(\alpha \ell(\z,\vecxi))\right]\right)  &\displaystyle\leq \frac{1}{\alpha}\log\left(\frac{1}{N}\sum_{i\in[N]}\left(\exp(\alpha \ell(\z, \hat\vecxi_i)+ \alpha L\epsilon)\right)\right)\\[3ex]
    &= \displaystyle\frac{1}{\alpha}\log\left(\frac{1}{N}\sum_{i\in[N]}\exp(\alpha \ell(\z,\hat\vecxi_i))\right)+L\epsilon.
    \end{array}
 \end{equation*}
 From the above inequality, it follows that $
\rho_\model{DRO}^* \leq \rho_\model{SAA}^*+  L \epsilon$. Combining with $\rho_{\model{DRO}}^* \geq\rho_{\model{SAA}}^*$, we conclude that    $\rho_\model{SAA}^*\leq \rho_\model{DRO}^* \leq \rho_\model{SAA}^*+L\epsilon$.

We can now establish a high confidence bound on $\lvert \rho_\model{DRO}^* - \rho^* \rvert$ that converges to zero at the rate of $\mathcal O(1/\sqrt{N})$ when using $\mathcal{B}_\infty(c/\sqrt{N})$ for some $c>0$. %
Namely, given some $\gamma>0$, we let $\phi(N, \gamma):=\frac{A}{\sqrt{N\gamma}\alpha \exp(\alpha \rho^*)}=\mathcal O(1/\sqrt{N})$ as defined in Proposition \ref{lemma:convrate:saa}. We can then show:
\begin{align*}
     \Prob\left(\lvert \rho_\model{DRO}^* - \rho^* \rvert \geq \phi(N, \gamma)+Lc/\sqrt{N}  \right) &\leq  \Prob\left( \lvert \rho_\model{DRO}^* - \rho_{\model{SAA}}^*\rvert   + \lvert \rho_{\model{SAA}}^*- \rho^* \rvert \geq \phi(N, \gamma)+Lc/\sqrt{N} \right) \\
     &\leq  \Prob(  \lvert \rho_{\model{SAA}}^*- \rho^* \rvert \geq \phi(N, \gamma)  ) \\
     &\leq \gamma,
\end{align*}
where the first inequality follows from the triangle inequality, the second from $\lvert \rho_\model{DRO}^* - \rho_\model{SAA}^*\rvert\leq Lc/\sqrt{N}$, and the third from Proposition \ref{lemma:convrate:saa} as long as $N$ is large enough.

To prove the convergence in probability, we simply consider any $\gamma>0$ and $\Delta>0$ and confirm that:
\[\Prob \left( \left| \rho_{\model{DRO}}^* - \rho^* \right| \geq \Delta\right)\leq \gamma,\]
as long as $N$ is large enough for Proposition \ref{lemma:convrate:saa} to apply and $\phi(N, \gamma)+Lc /\sqrt{N} \leq \Delta$. Indeed, for such $N$, we necessarily have that:
\[\Prob \left( \left| \rho_\model{DRO}^* - \rho^* \right| \geq \Delta\right)\leq \Prob \left( \left| \rho_\model{DRO}^* - \rho^* \right| \geq \phi(N, \gamma)+Lc /\sqrt{N}\right)\geq \gamma.\]
\end{proof}

\subsubsection{Proof of Proposition \ref{prop:k-fold}}
\begin{proof}
\begin{align*}
    \Expect[\rho_{k\sim U(K)}(\rho_{\vecxi\sim \hat{\mathbb P}_k^K}(\ell(\z^*(\hat{\mathbb P}_{-k}^K,\epsilon), \vecxi)))] &=\Expect\left[\frac{1}{\alpha} \log\left(\frac{1}{K}\sum_{k=1}^K\exp\left(\alpha \rho_{\vecxi\sim \hat{\mathbb P}_k^K}(\ell(\z^*(\hat{\mathbb P}_{-k}^K,\epsilon), \vecxi))\right)\right)\right]\\
    &\leq \frac{1}{\alpha} \log\left(\Expect\left[\frac{1}{K}\sum_{k=1}^K\exp\left(\alpha \rho_{\vecxi\sim \hat{\mathbb P}_k^K}(\ell(\z^*(\hat{\mathbb P}_{-k}^K,\epsilon), \vecxi))\right)\right]\right)\\
    &= \frac{1}{\alpha} \log\left(\frac{1}{K}\sum_{k=1}^K\Expect\left[\exp\left(\alpha \rho_{\vecxi\sim \hat{\mathbb P}_k^K}(\ell(\z^*(\hat{\mathbb P}_{-k}^K,\epsilon), \vecxi))\right)\right]\right)\\    
       &= \frac{1}{\alpha} \log\left(\Expect\left[\exp\left(\alpha \rho_{\vecxi\sim \hat{\mathbb P}_1^K}(\ell(\z^*(\hat{\mathbb P}_{-1}^K,\epsilon), \vecxi))\right)\right]\right)\\       
       &= \rho\left(\rho_{\vecxi\sim \hat{\mathbb P}_1^K}(\ell(\z^*(\hat{\mathbb P}_{-1}^K,\epsilon), \vecxi))\right)\\
&= \rho\left(\rho\left(\rho_{\vecxi\sim \hat{\mathbb P}_1^K}(\ell(\z^*(\hat{\mathbb P}_{-1}^K,\epsilon), \vecxi)\revised{)}\middle|\hat{\mathbb P}_{-1}^K\right)\right)\\  
&= \rho\left(\rho_{\vecxi\sim\mathbb P}\left(\ell(\z^*(\hat{\mathbb P}_{-1}^K,\epsilon), \vecxi)\right)\right)\\  
    &= \rho(\ell(\z^*(\hat{\mathbb P}_{N-N/K},\epsilon), \vecxi)),
\end{align*}
where expectations and $\rho$'s are with respect to randomness in the data $\Dist_N$, except for the last equation where the randomness is in both the data and a new sample $\vecxi\sim\mathbb P$. The first inequality follows from concavity of log function and Jensen's inequality,  then we exploit the fact that each $(\hat{\mathbb P}_{k}^K,\hat{\mathbb P}_{-k}^K)$ pair is identically distributed to $(\hat{\mathbb P}_{1}^K,\hat{\mathbb P}_{-1}^K)$, and finally, we use the tower property of the entropic risk measure. 
 \end{proof}

\section{Additional details}\label{append:details}

\subsection{\revised{Location–scale optimization problem}}\label{sec:appendix:location:scale}
Our setting builds on the location–scale condition of \citet{meyer1987two}, where two cumulative distributions $G_1$ and $G_2$ differ only through parameters $(a,b)$ if $G_1(x)=G_2(a+b x)$ for all $b>0$. Such conditions are classically used to show that expected–utility maximizers select portfolios on the Markowitz mean–variance efficient frontier; see, e.g., \citet{meyer1992sufficient}, who use a Kolmogorov–Smirnov test to assess whether standardized stock–return portfolios share a common distribution. More generally, there is a long tradition of assuming that portfolio returns are fully characterized by their mean and variance, notably under multivariate normal, elliptical, or skew–elliptical models \citep{meyer1992sufficient,Adcock2014MeanVarianceSkew,schuhmacher2021justifying}. In the same spirit, Assumption~\ref{ass:meanVarControl} posits a common location–scale representation for feasible losses, so that each decision corresponds to a pair $(\mu,\sigma)$. For instance, if $\vecxi\in \mathbb{R}^d$ is elliptically symmetric around its mean, i.e., $\vecxi \stackrel{F}{=} \boldsymbol m + A \bm{u}$ with $\bm u\in \mathbb{R}^d$ spherically symmetric, then for each $\vecz$ with $A^\top \vecz \neq 0$, the loss $\ell(\vecz,\vecxi)=\vecz^\top\vecxi$ admits the representation $\ell(\vecz,\vecxi)\stackrel{F}{=} \mu(\vecz) + \sigma(\vecz)\,\xi$
where $\mu(\vecz)=\vecz^\top \boldsymbol m $, $\sigma(\vecz)=\|A^\top \vecz\|$, and $\xi := v(\vecz)^\top \bm u$ with $
v(\vecz):=A^\top \vecz/\|A^\top \vecz\|\in\mathbb{R}^d$. By spherical symmetry of $\bm u$, the distribution of $\xi$ does not depend on $\vecz$. If $A^\top \vecz =0$, then $\sigma(\vecz)=0$, and the loss is degenerate, $\ell(\vecz, \vecxi)=\mu(\vecz)$. This representation allows us to study how the bias of sample-based risk estimators scales with volatility.

\subsection{Fitting a GMM} \label{appen:fit_dist}
\paragraph{Entropic risk matching.}   %

After each update of the parameters of a GMM using the gradient descent procedure described in Algorithm \ref{alg:riskmatching}, the parameters are projected back into the feasible region for a valid GMM. This ensures that the mixing weights $\hat{\bm{\pi}}_{t+1}$ remain valid probabilities (which is achieved using a softmax function), and that the standard deviations $\hat{\bm{\sigma}}_{t+1}$ are positive (enforced by taking the maximum of $\exp(-5)$ and $\hat{\bm{\sigma}}_{t+1}$).
The number of components $J$ in the GMM is selected by  CV based on the Wasserstein distance between the distribution of the entropic risk of samples drawn from  fitted GMM and the distribution of entropic risk constructed from the scenarios in the validation set. In all numerical experiments, we set the maximum iterations $T=30000$ and tolerance $\epsilon=\exp(-9)$.

\paragraph{Matching the extremes.}  
Our proposed procedure is motivated by the  Fisher–Tippett–Gnedenko extreme value theorem \citep{de_Haan_Ferreira_2006} which states that given i.i.d.~samples of $\{\zeta_1, \zeta_2, \cdots, \zeta_{n}\}$ with cumulative distribution function (cdf) given by $F(\cdot)$,  the distribution of the (normalized) maxima $M_n =\max\{\zeta_1, \zeta_2, \cdots, \zeta_{n}\}$ converges to a non-degenerate distribution~$G$:
 \begin{align*}
  \lim_{n \rightarrow \infty} \Prob\left(\frac{M_n-b_n}{a_n}\leq x\right) = \lim_{n \rightarrow \infty}  F(a_n x+b_n)^n\rightarrow G(x),
 \end{align*}  
where $a_n$ and $b_n$ are normalizing sequences of scale and location parameters, respectively, that ensure the limit exists and $F(\cdot)^n$ is the cdf of $M_n$.  The limit distribution $G$ belongs to one of three extreme value distributions--Weibull, Fréchet or Gumbel distribution--depending on the tails of $F(\cdot)$. 

 The distribution of the maxima of $n$ i.i.d samples from a normal distribution $\mathcal{N}(\mu, \sigma)$ is given by $\left(\Phi_{\mu, \sigma}\right)^n$ where $\Phi_{\mu, \sigma}$ is the cdf of a normally distributed random variable with mean $\mu$ and $\sigma$. 
  We find the parameters of the normal distribution by matching the 50th and 90th quantiles  of $F_\mathcal{M}(.)$ to the corresponding quantiles of $\left(\Phi_{\mu, \sigma}\right)^n$: 
  \begin{align*}
      &\mu + \sigma \Phi_{0,1}^{-1}(0.5^{1/n}) = F_\mathcal{M}(0.5), \\
      & \mu + \sigma \Phi_{0,1}^{-1}(0.9^{1/n}) = F_\mathcal{M}(0.9),
  \end{align*}
  where the $p$-th quantile for $Y\sim \mathcal{N}(\mu, \sigma)$  is given by $\mu + \sigma \Phi_{0,1}^{-1}(p)$. Solving the above two equations in $(\mu, \sigma)$ gives an approximate distribution for the tails of the underlying distribution, which depends on the true distribution, the total number of samples $N$, and the number of bins $B$. 
  
To balance the trade-off between the number of bins and the sample size in each bin, we set $B=\sqrt{N}$, which is a reasonable compromise.  A large number of bins provides more independent realizations, reducing estimation bias, while a sufficiently large sample size in each bin ensures that the maxima accurately represent the extremes of the distribution.

\begin{algorithm}[t]
\caption{Fit GMM \revised{in step \ref{algstep:fit} of Algorithm \ref{alg:fittingGMM}} based on the extreme value theory\label{alg:fitGMMevt}}
\begin{algorithmic}[1]
\Function{\model{BS-EVT}}{$\bar{\mathcal{S}}$}
\State Divide scenarios in $\bar{\mathcal{S}}$ into $B$ bins of size $n=N/B$ each
  \State $F_\mathcal{M} \gets$ cdf of maxima in each bin 
  \State Determine $(\mu^e, \sigma^e)$ so that $\Phi_{\mu^e, \sigma^e}^n$ matches  $F_{\mathcal{M}}$ 
  \vspace{0.1cm}
  \State $\hat{\V{\pi}} \gets \begin{pmatrix}0.5 \\ 
  0.5\end{pmatrix}$,
         $\hat{\V{\mu}} \gets \begin{pmatrix}\mu^e \\ %
         -\mu^e\end{pmatrix}$,
         $\hat{\V{\sigma}} \gets \begin{pmatrix}\sigma^e \\ 0 \end{pmatrix}$  
  \State $\V{\theta} \gets (\hat{\V{\pi}}, \hat{\V{\mu}}, \hat{\V{\sigma}})$
  \State \Return $\Q^{\btheta}$
\EndFunction
\end{algorithmic}
\end{algorithm}
\subsection{Differential sampling from GMM} \label{appen:samplegmm}

In Algorithm \ref{alg:riskmatching}, differentiable samples are generated using Algorithm \ref{alg:samplegmm}. This algorithm leverages the reparameterization trick \citep{Kingma_Salimans_2015} for continuous distributions and the Gumbel-Softmax trick \citep{Jang_Gu_2017, Maddison_Mnih_2017} for discrete distributions. The Gumbel-Softmax trick allows approximate, differentiable sampling of mixture components, while Gaussian samples are obtained by combining deterministic transformations of the parameters with random noise. As a result, gradients can flow through both the discrete and continuous sampling steps. This enables the minimization of the Wasserstein distance between the empirical and model-based entropic risk distributions.
\begin{algorithm}[H]
\caption{Differentiable Sampling from GMM}
\begin{algorithmic}[1]
\Function{SampleGMM}{$n$, $\btheta$, $\tau$}
  \State Initialize $\mathcal{S}\gets\emptyset$
  \State Extract mixture weights $\bm \pi\in\Delta_J$, means $ \bm \mu\in\mathbb{R}^J$, and standard deviations $ \bm \sigma\in\mathbb{R}_+^J$ from $\btheta$
  \State Let $\mathfrak J\gets\{1,2,\dots,J\}$
  \For{$i=1$ \textbf{to} $n$}
    \State Sample Gumbel noise $\bm g\in \mathbb{R}^{J}$ with $g_j\sim \mathrm{Gumbel}(0,1)$ i.i.d.
    \State Compute $\bm \phi \gets \log(\bm\pi) + \bm g$ \Comment{Gumbel-Softmax logits}
    \State Compute relaxed component weights $\bm w \gets \mathrm{softmax}(\bm \phi/\tau)\in\Delta_J$
    \State Sample $\bm \varepsilon\in\mathbb{R}^J$ with $\varepsilon_j\sim\mathcal{N}(0,1)$ i.i.d.
    \State Compute component-wise Gaussian draws $z_j \gets \mu_j + \sigma_j \varepsilon_j$ for all $j\in\mathfrak{J}$ \Comment{reparametrize}
    \State Compute $s_i \gets \sum_{j\in\mathfrak{J}} w_j z_j$
    \State Append $s_i$ to $\mathcal{S}$
  \EndFor 
  \State \Return $\mathcal{S}$
\EndFunction
\end{algorithmic}
\label{alg:samplegmm}
\end{algorithm}
\subsection{Parameters in Example \ref{ex:projects}}\label{append:param:example}
The parameters of the GMM in Example \ref{ex:projects} are given by:\begin{align*}
\V{\pi} = \begin{bmatrix}0.16\\
       0.28\\
       0.23\\
       0.20\\
       0.13
\end{bmatrix}, 
\quad  & \V{\mu} = \begin{bmatrix}-19.5\\-19.0\\ -18.5\\ -18.0\\ -17.5
\end{bmatrix}, \quad  \V{\sigma}=\begin{bmatrix}
4/25\\ 1/4\\ 4/9\\ 1\\    4 
\end{bmatrix}.\end{align*}
The expected value of $\xi$ is $-18.57$ and standard deviation is $1.65$.

\begin{figure}[t]
    \centering
    \begin{subfigure}[b]{0.48\linewidth}
        \centering
        \includegraphics[width=\linewidth]{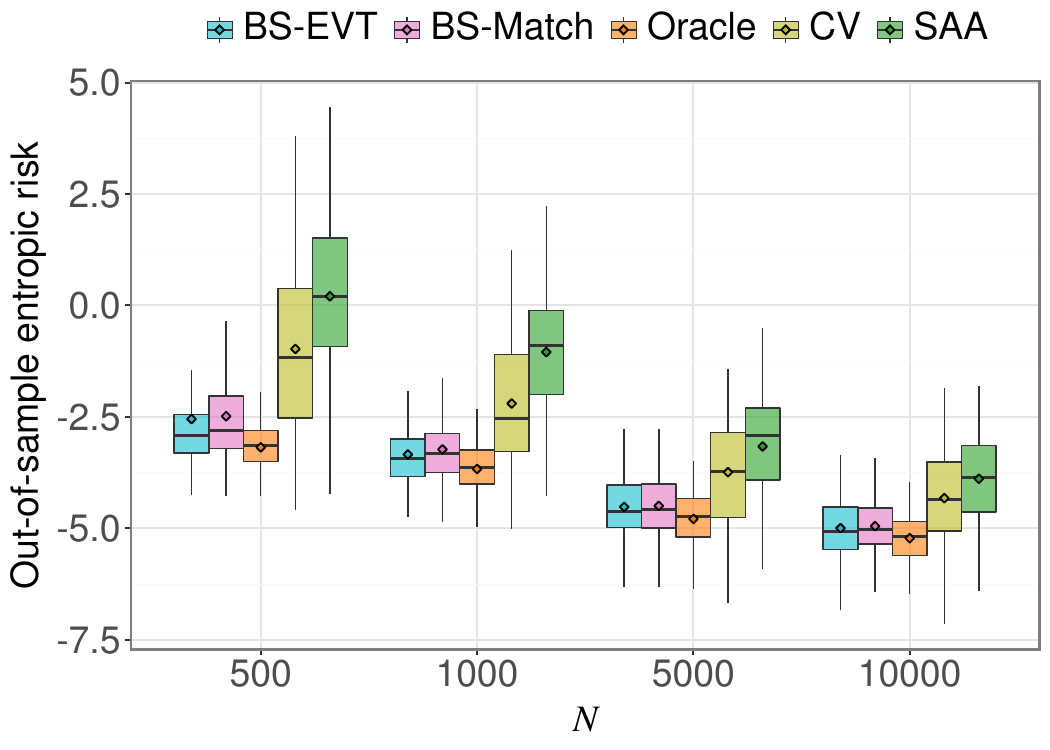}
        \caption{Insurer’s out-of-sample entropic risk.}
        \label{fig:risk_with_N_diff}
    \end{subfigure}
    \hfill
    \begin{subfigure}[b]{0.48\linewidth}
        \centering
        \includegraphics[width=\linewidth]{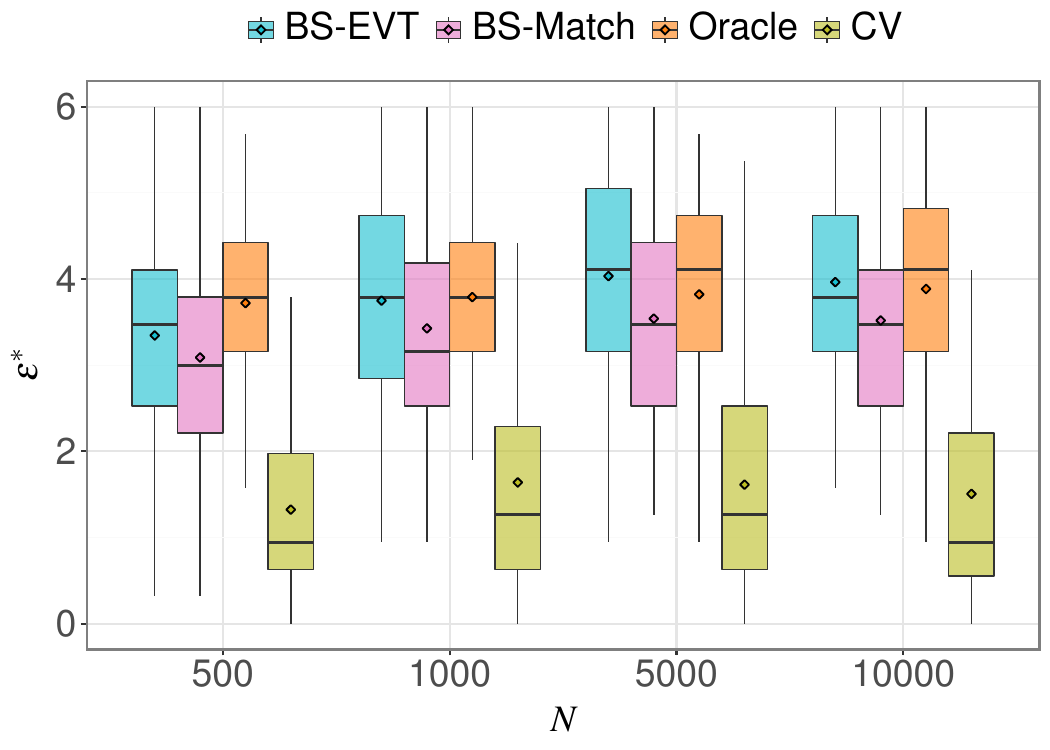}
        \caption{
        Optimal radius $\epsilon^*$ selected by different methods.}
        \label{fig:radius_with_N_diff}
    \end{subfigure}
    \caption{Comparison of the effects of training sample size $N$ on entropic risk (left) and radius $\epsilon^*$ (right). Each household observes samples from different $\Gamma$ marginals and the correlation coefficient $r=0.5$.}
    \label{fig:combined_N_diff}
\end{figure}

\subsection{Households have different marginal distribution}\label{append:diff_marginal}

 In this experiment, we examined the scenario where each household has a distinct Gamma-distributed loss function. The scale, $\kappa$ and location parameters, $\lambda$, of the $\Gamma$-distribution of the loss of the five households are given by $(8, 0.41)$, $(8.5, 0.42)$, $(9, 0.43)$, $(9.5, 0.44)$, and $(10, 0.45)$, respectively. 
 The correlation coefficient was set to $r=0.5$, indicating a moderate positive correlation among the losses. We compared the performance of the proposed methods (\model{BS-EVT} and \model{BS-Match}) with the traditional  CV approach and SAA.

The findings in this case mirror those of the first experiment where households have the same marginal loss distribution: As depicted in Figure \ref{fig:risk_with_N_diff}, \model{BS-EVT} and \model{BS-Match} consistently outperform the traditional \model{CV} method and SAA across different sample sizes. Furthermore, increasing $N$ leads to a reduction in the out-of-sample entropic risk for all methods, with \model{BS-EVT} and \model{BS-Match} showing the most significant improvements. This is due to the fact that \model{CV} and \model{SAA} are overly optimistic and choose smaller radius as compared to \model{Oracle}, whereas the proposed methods \model{BS-EVT} and \model{BS-Match} choose close to optimal radius, see Figure \ref{fig:radius_with_N_diff}.

\subsection{Estimate of entropic risk}\label{append:sec:estimate_out}
Figures \ref{fig:n_500}-\ref{fig:n_10000} present the statistics of the estimate of the out-of-sample risk for different $N\in \{500, 5000, 10000\}$ and radius $\epsilon$ of the ambiguity set in the interval $[0,6]$.  Similar to Figure \ref{fig:est_risk_with_eps_1000}, it can be seen that \model{CV} underestimates the entropic risk for each $\epsilon$. However, \model{BS-Match} and \model{BS-EVT} make better estimation of the variation in the true entropic risk with $\epsilon$, thereby enabling a more informed choice of $\epsilon^*$.

\begin{figure}[htbp]
    \centering
    \begin{subfigure}[htbp]{\linewidth}
        \centering
        \includegraphics[width=0.6\linewidth]{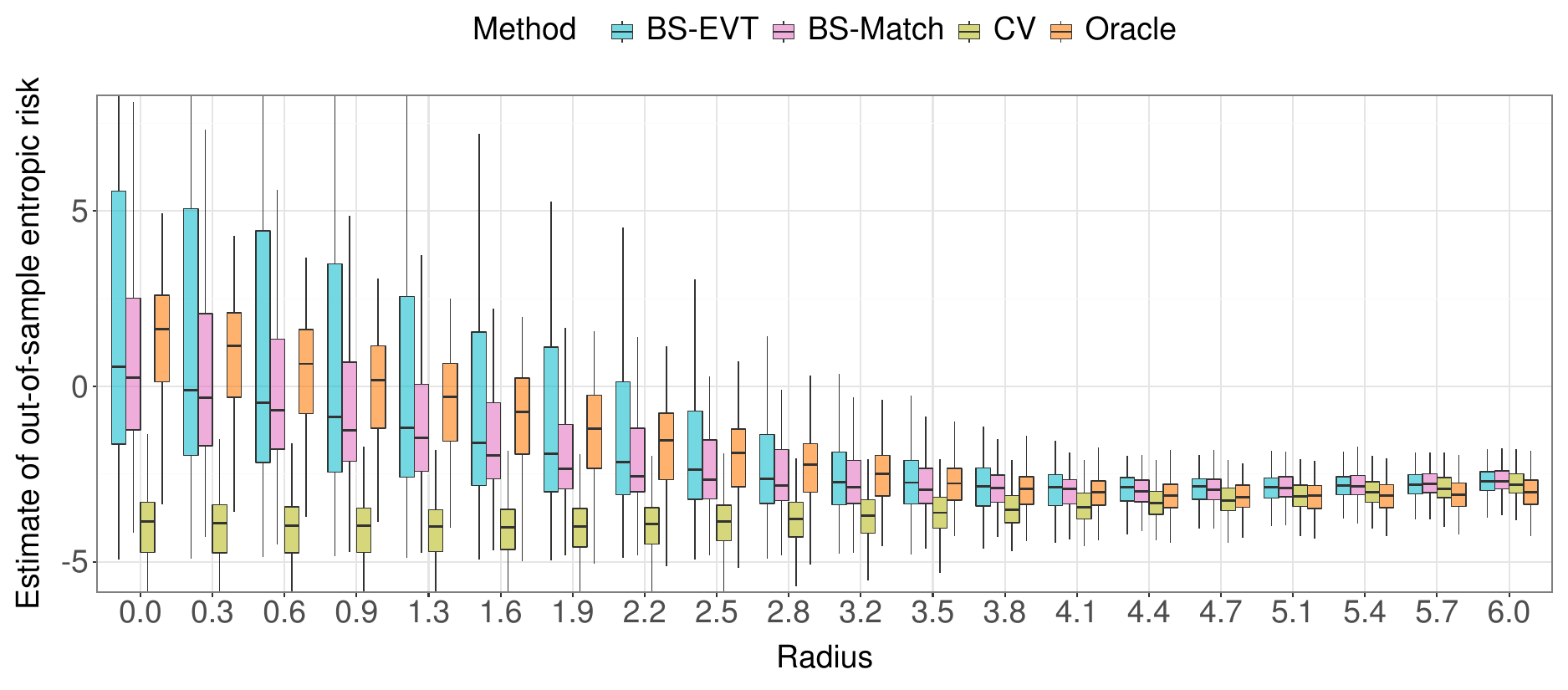}
        \caption{$N=500$}
        \label{fig:n_500}
    \end{subfigure}

    \vspace{0.2cm} %

    \begin{subfigure}[htbp]{\linewidth}
        \centering
        \includegraphics[width=0.6\linewidth]{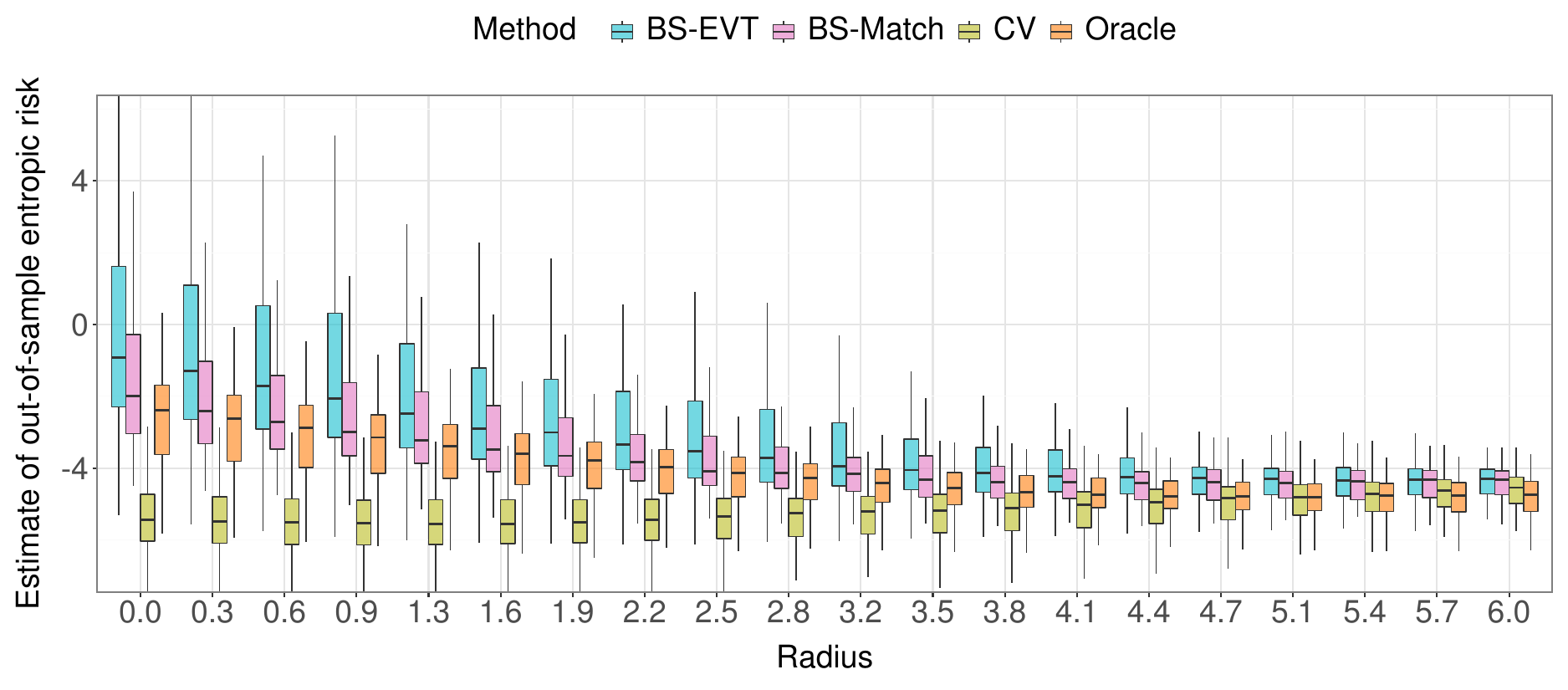}
        \caption{$N=5000$}
        \label{fig:n_1000}
    \end{subfigure}
    
\vspace{0.2cm}
        \begin{subfigure}[htbp]{\linewidth}
        \centering
        \includegraphics[width=0.6\linewidth]{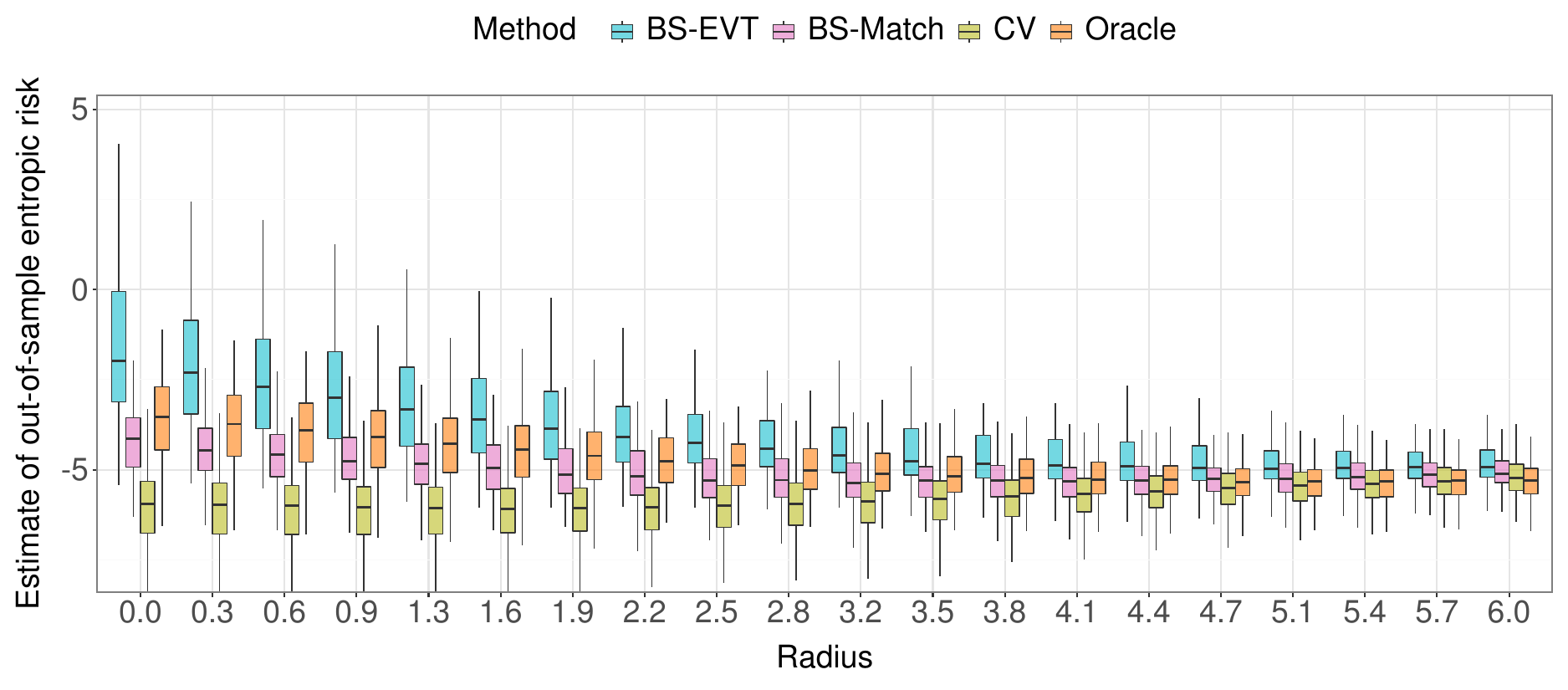}
        \caption{$N=10000$}
        \label{fig:n_10000}
    \end{subfigure}

    \caption{Estimate of entropic risk for different radius and $N$}
    \label{fig:est_rik_N}
\end{figure}

\end{document}